%% file: Doktorarbeit.tex
\documentclass[openright,oneside,fontsize=12pt,bibliography=totoc%
]{scrbook}[2007/12/24]
\pdfoutput=1
\usepackage[    
    natbib={numbers}, 
    hints=false, 
    checktabu=false, 
    draft=false, 
    qserie=false
 ]{hudiss}
                   
\hudisssetup{%
    titlepagefont={\large\sffamily} 
}

\hudissmetadata{%
    authorprefix={Dipl.-Math.}, 
    authorfirstname={Eren}, 
    authorsurname={U\c{c}ar}, 
    doctitle={Spectral invariants for polygons and orbisurfaces}, 
    approvala={Prof. Dr. Dorothee Sch\"uth}, 
  	approvalb={Prof. Dr. Carolyn Gordon},
    approvalc={Prof. Dr. Daniel Grieser},
    degree={doctor rerum naturalium \\ (Dr. rer. nat.)}, 
    subject={Mathematik}, 
    faculty={Mathematisch-Naturwissenschaftlichen Fakult\"at}, 
    university={der Humboldt-Universit\"at zu Berlin}, 
    dean={Prof. Dr. Elmar Kulke}, 
    president={Prof. Dr.-Ing. Dr. Sabine Kunst}, 
    datesubmitted={27.??.????}, 
    dateexam={20. September 2017}, 
    keywordsen={spectral geometry, inverse spectral problem, Laplacian, polygons, hyperbolic plane, heat kernel, spectral invariants, heat invariants, orbifolds, orbisurfaces, generalized Mehler transform}, 
    keywordsde={Spektralgeometrie, inverses Spektralproblem, Laplace-Operator, Polygone, Hyperbolische Ebene, W\"armeleitungskern, Spektralinvarianten, W\"armekoeffizienten, Orbifolds, Orbifl\"achen, Verallgemeinerte Mehler Transformation} 
}

\input{adjustment.tex}

\makeindex

\begin{document}

    \frontmatter
        \maketitle
        \include{abstract}
        \selectlanguage{english} 
        \include{dedication}

        \tableofcontents    
   \mainmatter
        \include{chapter01}
        \include{chapter02}
        \include{chapter03}
        \include{chapter04}

    \backmatter
        \appendix
         \nocite{*}                
        \bibliographystyle{alpha}
        \bibliography{bibliography}
       \include{declaration}    
\end{document}

%% file: adjustment.tex
\usepackage[latin1]{inputenc} 
\usepackage[T1]{fontenc}      
\usepackage{amsmath}  
\usepackage{amsfonts} 
\usepackage{amssymb}  
\usepackage{amsthm}
\usepackage{fancyhdr} 
\usepackage{graphicx} 
\usepackage[nice]{nicefrac}
\allowdisplaybreaks 

\usepackage{MnSymbol} 

\usepackage{relsize} 

\usepackage{hyperref} 

\usepackage{tikz} 
\usetikzlibrary{calc, decorations.markings}
\usepackage{subfig}

\usepackage[activate={true,nocompatibility},final,tracking=true,kerning=true,spacing=true]{microtype}

\newcommand{\LegendreP}[3]{P_{#1}^{#2}{\hspace*{-2mm}\left( #3 \right)}} 
\newcommand{\LegendreQ}[3]{Q_{#1}^{#2}{\hspace*{-2mm}\left(#3\right)}}

\newcommand{\hyperball}[2]{B_{#1}^{\mathbb{H}^2}{\hspace*{-1mm}\left(#2\right)}} 
\newcommand{\genball}[2]{B_{#1}{\hspace*{-0.5mm}\left(#2\right)}} 

\newcommand{\proj}[2]{\pi_{\mathbb{#1}_k}\hspace{-0.5mm}(#2)} 

\DeclareMathOperator*{\Res}{Res} 
\DeclareMathOperator*{\diag}{diag}
\DeclareMathOperator*{\offdiag}{o-diag}

\DeclareMathOperator{\arcosh}{arcosh} 
\DeclareMathOperator{\arsinh}{arsinh}

\newcommand{\I}{\mathrm{I}} 

\newcommand{\Iso}{\mathrm{Iso}} 
\newcommand{\vol}{\mathrm{\,vol}} 

\newcommand{\volume}{\mathrm{vol}} 
\newcommand{\Id}{\mathrm{Id}} 
\newcommand{\diverg}{\mathrm{div}}

\newcommand{\ml}{\mathrm{ml}} 

\theoremstyle{plain}
\newtheorem{theorem}{Theorem}[chapter] 
\newtheorem*{theorem*}{Theorem}  
\newtheorem{lemma}[theorem]{Lemma} 
\newtheorem{corollary}[theorem]{Corollary} 
\newtheorem{proposition}[theorem]{Proposition}
\newtheorem{example}[theorem]{Example}

\theoremstyle{definition}
\newtheorem{definition}[theorem]{Definition}

\theoremstyle{remark}
\newtheorem*{remark}{Remark}



%% file: abstract.tex
\selectlanguage{english}
\begin{abstract}

In this thesis we deal with spectral invariants for polygons and closed orbisurfaces of constant Gaussian curvature. In each case our method is to study the heat kernel and the asymptotic expansion of the heat trace as $t\searrow0$.

 First we investigate hyperbolic polygons, i.e. relatively compact domains in the hyperbolic plane with piecewise geodesic boundary. We compute the asymptotic expansion of the heat trace as $t\searrow 0$ associated to the Dirichlet Laplacian of any hyperbolic polygon, and we obtain explicit formulas for all heat invariants. Our approach to the asymptotic expansion is based on the so-called \emph{principle of not feeling the boundary} due to M. Kac (\citep{Kac}). Analogous results for Euclidean and spherical polygons were known before and are published in \citep{VanDenBerg} and \citep{Watson}, respectively. By comparing the heat invariants for Euclidean, spherical and hyperbolic polygons and by scaling the metric appropriately, we unify these results and deduce the heat invariants for arbitrary polygons. Here, the term polygon refers to any relatively compact domain with piecewise geodesic boundary contained in a complete Riemannian manifold of constant Gaussian curvature. It turns out that the heat invariants provide much information about a polygon, if the curvature does not vanish. For example, then the multiset of all angles which are not equal to $\pi$ and the Euler characteristic of a polygon are spectral invariants.

Furthermore, we compute the asymptotic expansion of the heat trace for any closed Riemannian orbisurface of constant curvature as $t\searrow 0$, and obtain explicit formulas for all heat invariants. To solve this problem, we study two examples of spherical orbisurfaces thoroughly and show how their heat traces are related to each other. This relation together with some observations from \citep{Watson} allow us to compute all heat invariants for arbitrary Riemannian orbisurfaces of constant Gaussian curvature. If the curvature does not vanish, then it is possible to detect interesting information about the topology and the singular set of an orbisurface from the heat invariants. For example, we prove that if two orientable orbisurfaces with the same  curvature $\kappa\neq 0$ are isospectral, then they must be homeomorphic. This result was previously known for $\kappa=-1$ but it was proven differently, namely by using the Selberg trace formula for the wave kernel and Weyl's asymptotic law (\citep{Strohmaier}).

We show that the heat invariants for polygons and orbisurfaces are closely related to each other. For instance, we prove that the contribution of an interior angle $\frac{\pi}{k}$, $k\in\mathbb{N}_{\geq 2}$, to the heat invariants for a polygon is the same as the contribution of a dihedral point of isotropy order $2k$ to the heat invariants for an orbisurface, if the polygon and the orbisurface have the same curvature. This fact together with the formulas from \citep{Watson} provide an alternative proof for the contribution of an interior angle $\frac{\pi}{k}$, $k\in\mathbb{N}_{\geq 2}$, to the heat invariants for a polygon.
 
\end{abstract}

\cleardoublepage

\selectlanguage{ngerman}
\begin{abstract}

In dieser Arbeit besch\"aftigen wir uns mit Spektralinvarianten von Polygonen und geschlossenen Orbifl\"achen konstanter Gau{\ss}kr\"ummung. Unsere Methode ist es jeweils den W\"armeleitungskern und die asymptotische Entwicklung der W\"armespur f\"ur $t\searrow 0$ zu untersuchen.

Als erstes untersuchen wir hyperbolische Polygone, d.h. relativ kompakte Gebiete in der hyperbolischen Ebene mit st\"uckweise geod\"atischem Rand. Wir berechnen die asymptotische Entwicklung der W\"armespur f\"ur $t\searrow 0$ bez\"uglich des Dirichlet-Laplace Operators eines beliebigen hyperbolischen Polygons, und wir erhalten explizite Formeln f\"ur alle W\"armeinvarianten. Unsere Herangehensweise an die asymptotische Entwicklung basiert auf dem so genannten \emph{principle of not feeling the boundary} von M. Kac (\citep{Kac}). Analoge Resultate f\"ur euklidische und sph\"arische Polygone waren vorher bekannt und wurden in \citep{VanDenBerg} bzw. \citep{Watson} ver\"offentlicht. Indem wir die W\"armeinvarianten f\"ur euklidische, sph\"arische und hyperbolische Polygone miteinander vergleichen und die Metrik geeignet skalieren, vereinheitlichen wir diese Resultate und leiten die W\"armeinvarianten f\"ur beliebige Polygone her. Hierbei ist mit dem Begriff Polygon ein relativ kompaktes Gebiet mit st\"uckweise geod\"atischem Rand in einer vollst\"andigen Riemann'schen Mannigfaltigkeit konstanter Gau{\ss}kr\"ummung gemeint. Es stellt sich heraus, dass die W\"armeinvarianten viele Informationen \"uber ein Polygon liefern, falls die Kr\"ummung nicht verschwindet. Zum Beispiel sind dann die Multimenge aller Winkel, die ungleich $\pi$ sind, und die Euler-Charakteristik eines Polygons Spektralinvarianten.

Au{\ss}erdem berechnen wir die asymptotische Entwicklung der W\"armespur von geschlossenen Riemann'schen Orbifl\"achen konstanter Kr\"ummung f\"ur $t\searrow 0$ und erhalten explizite Formeln f\"ur alle W\"armeinvarianten. Um dieses Problem zu l\"osen, untersuchen wir ausf\"uhrlich  zwei Beispiele sph\"arischer Orbifl\"achen und zeigen, wie ihre W\"armespuren in Relation zueinander stehen. Diese Relation zusammen mit einigen Beobachtungen aus dem Artikel \citep{Watson} erm\"oglichen es uns die W\"armeinvarianten beliebiger Riemann'scher Orbifl\"achen konstanter Gau{\ss}kr\"ummung zu berechnen. Falls die Kr\"ummung nicht verschwindet, so kann man interessante Informationen aus den W\"armeinvarianten \"uber die Topologie und die singul\"are Menge einer Orbifl\"ache ermitteln. Beispielsweise zeigen wir, dass falls zwei orientierbare Orbifl\"achen mit derselben Kr\"ummung $\kappa\neq 0$ isospektral sind, sie dann auch hom\"oomorph sein m\"ussen. Dieses Resultat war vorher bekannt f\"ur $\kappa=-1$, wurde jedoch auf andere Weise bewiesen, n\"amlich mittels der Selberg-Spurformel f\"ur den Wellenkern und der Weyl'schen Eigenwertasymptotik (\citep{Strohmaier}).

Wir zeigen, dass die W\"armeinvarianten f\"ur Polygone und Orbifl\"achen eng miteinander verwandt sind. Zum Beispiel beweisen wir, dass der Beitrag eines Innenwinkels $\frac{\pi}{k}$, $k\in\mathbb{N}_{\geq 2}$, zu den W\"armeinvarianten eines Polygons derselbe ist wie der Beitrag eines Diederpunktes mit Isotropieordnung $2k$ zu den W\"armeinvarianten einer Orbifl\"ache, falls das Polygon und die Orbifl\"ache dieselbe Kr\"ummung besitzen. Diese Tatsache gemeinsam mit den Formeln aus \citep{Watson} liefert einen alternativen Beweis f\"ur den Beitrag eines Innenwinkels $\frac{\pi}{k}$, $k\in\mathbb{N}_{\geq 2}$, zu den W\"armeinvarianten eines Polygons.
\end{abstract}
\cleardoublepage

%% file: dedication.tex
\cleardoublepage


\chapter*{Acknowledgements}

I would like to thank my supervisor Prof. Dorothee Sch\"uth for her constant support, scientific criticism, and faith in me. I benefited a lot from her mathematical knowledge and experience; without her guidance this work would not exist.

I would also like to thank my colleagues at the Humboldt-Universit\"at zu Berlin for the pleasant working atmosphere and the helpful discussions on mathematics.

Last but not least I am deeply grateful to my family for all the support and the extra love that they gave me. In particular my parents, G\"ul\c{s}eher and Cengiz, my brother Imran, and my wife Magdalena. This project would not have entered my life without the encouragement of my wife. She shared both the joyful and the seemingly hopeless moments during my research - thank you!

\thispagestyle{empty}
\upshape\cleardoublepage

%% file: chapter01.tex
\chapter{Introduction}

Let $M$ be a two-dimensional smooth connected Riemannian manifold and $\Omega\subset M$ be a non-empty and relatively compact domain. Further, let $\Delta_{\Omega}$ denote the Dirichlet Laplacian for $\Omega$. Then the spectrum of $\Delta_{\Omega}$ consists of a sequence of non-negative eigenvalues without a finite accumulation point and each eigenvalue has a finite multiplicity. We enumerate them according to their multiplicities as follows:
\begin{align}
 0\leq\lambda_1<\lambda_2\leq \lambda_3\leq...\nearrow \infty. \label{eigenvalues}
\end{align}

The spectrum of $\Delta_{\Omega}$ and the geometry of the domain $\Omega$ are closely connected. Historically, this became popularised through a seminal article with the catchy title ``Can one hear the shape of a drum?'' written by M. Kac half a century ago (\citep{Kac}). In his article, M. Kac interprets a domain physically as a drum and the eigenvalues of $\Delta_{\Omega}$ as the pure tones of the drum, if it is fixed along its boundary. On the one hand, the spectrum of the Dirichlet Laplacian is uniquely determined by the shape, or rather by the geometry of the domain. On the other hand, it is possible to deduce geometric properties of the domain if one only knows the sequence \eqref{eigenvalues}. Such properties, i.e. those which are determined by the spectrum, are called \emph{spectral invariants} and are principal objects of study in the realm of inverse spectral geometry. One well-known method to obtain spectral invariants of $\Omega$ is to compute the asymptotic expansion of the heat trace, which lies at the heart of our investigations in this thesis.

The heat trace of $\Omega$ is the function
\begin{align}
\label{heat trace}
Z_{\Omega}(t)=\sum\limits_{i=1}^{\infty} e^{- \lambda_i t}, \quad t>0.
\end{align}
Obviously, the heat trace only depends on the spectrum of $\Delta_{\Omega}$. It is remarkable that, if the boundary $\partial\Omega$ of the domain is non-empty and smooth, then the heat trace has an asymptotic expansion of the form
\begin{align}
\label{heat asymptotics}
Z_{\Omega}(t)\overset{t\downarrow 0}{\sim} \frac{1}{4\pi t}\sum\limits_{k=0}^{\infty} a_k t^{\frac{k}{2}},
\end{align}
where the coefficients $a_k$ in the asymptotic expansion, the so-called \emph{heat invariants}, correspond to geometric properties of $\Omega$. For example, if $\kappa$ denotes the Gaussian curvature and $\kappa_g$ the geodesic curvature, then the first three heat invariants are given as
\begin{align*}
a_0 &= \int\limits_{\Omega} 1\, dA = \vert \Omega \vert, \qquad a_1 = -\frac{\sqrt{\pi}}{2} \int\limits_{\partial \Omega} 1\, ds = -\frac{\sqrt{\pi}}{2}  \vert \partial \Omega \vert , \\
a_2 &= \frac{1}{3}\left( \int\limits_{\Omega}\kappa\, dA + \int\limits_{\partial \Omega} \kappa_g \, ds \right) = \frac{2\pi}{3} \chi\left( \Omega \right),
\end{align*}
where  $\vert \Omega \vert,\, \vert \partial\Omega \vert $ and $\chi(\Omega)$ denote the area of the domain, the length of its boundary and its Euler characteristic, respectively. Thus the area, the perimeter and the Euler characteristic are determined by the spectrum of $\Delta_{\Omega}$. More generally, it is known that any heat invariant can be written as a sum, $a_k = i_k + b_k$, where $i_k$ is given as an integral over $\Omega$ of some polynomial in the Gaussian curvature and its covariant derivatives, whereas $b_k$ is given as an integral over $\partial \Omega$ of some polynomial in the geodesic curvature and its derivatives. Nevertheless almost all of these polynomials, and hence the heat invariants, are still unknown. Using computer calculations, L. Smith has obtained the first seven heat invariants for smooth domains in the Euclidean plane (\cite{Smith}), and, more generally, the first five heat invariants for any smooth domain can be found in \cite{GilkeyBranson}.

In this thesis we first aim to study the asymptotic expansion of the heat trace for \emph{hyperbolic polygons}, i.e. relatively compact domains in the hyperbolic plane with piecewise geodesic boundary. At first sight, the behaviour of the heat trace for hyperbolic polygons as $t\searrow 0$ is similar as for domains with smooth boundary. It has an asymptotic expansion as in \eqref{heat asymptotics} and the coefficients in the asymptotic expansion, still called heat invariants, reflect geometric properties of the domain. However, the above formulas for the heat invariants are generally not valid anymore for domains whose boundary is not smooth. We will prove that the heat trace for hyperbolic polygons has an asymptotic expansion of the form
\begin{align*}
Z_{\Omega}(t) \overset{t\downarrow 0}{\sim} \frac{\vert \Omega \vert}{4\pi t} - \frac{\vert \partial\Omega \vert}{8\sqrt{\pi t}} + \sum\limits_{k=0}^{\infty}\left( i_k + b_k t^{\frac{1}{2}}+ \nu_k \right)t^k,
\end{align*}
where we also obtain explicit formulas for all coefficients (see Theorem \ref{theorem:AsymptoticExpansionHeatTraceHyperbolPolygon}). Qualitatively, $i_k$ and $b_k$ can be written as described above for domains with smooth boundary. But now there is also a contribution to the heat invariants from the vertices of the polygon, which we abbreviate by $\nu_k$.

The analogous problem was solved for Euclidean polygons in \citep{VanDenBerg} and for spherical polygons in \citep{Watson}. The immediate question is how the heat invariants depend on the Gaussian curvature of the polygon? We will answer this question by deriving the asymptotic expansion of the heat trace for polygons in arbitrary two-dimensional complete Riemannian manifolds of constant curvature (see Corollary \ref{corollary:HeatAsymptoticConstantCurvaturePolygon}). We want to emphasise that we not only generalise the ambient space, but also our definition of the term polygon is more general than the standard notion (see Definition \ref{definition:PolygonGeneral}). It turns out that the heat invariants for polygons can be written as polynomials in the Gaussian curvature with coefficients depending on all angles of the polygon. Later on we use orbifold theory as a vehicle to explain this behaviour more deeply in the case of polygons with interior angles of $\frac{\pi}{k}$, $k\in\mathbb{N}_{\geq 2}$.

Through the heat invariants we will deduce some interesting geometric consequences for a polygon of constant curvature when its spectrum is given. In particular we will prove the following:

\begin{itemize}
\item[$\bullet$] It is possible to detect the volume, perimeter and curvature of a polygon from its spectrum (see Corollary \ref{corollary:SpectralInvariantsVolumePerimeterCurvature}).
\item[$\bullet$] If the curvature of a polygon is nonzero, then the multiset of all its angles which are not equal to $\pi$, as well as its Euler characteristic are spectral invariants (see Theorem \ref{theorem:SpectralInvariantsAnglesEulerCharacteristic}).
\item[$\bullet$] A polygon $\Omega$ with at least one angle not equal to $\pi$ and a smooth domain $D$ can never be \emph{isospectral}, i.e. they can not have the same spectrum, if $\chi(\Omega)\geq \chi(D)$ (see Corollary \ref{corollary:PolygonNotIsospectralSmoothDomain} (ii)). 
\item[$\bullet$] We prove that hyperbolic and spherical triangles are determined up to isometry by their spectrum within the class of all polygons (see Corollary \ref{corollary:IsospectralTriangle})
\item[$\bullet$] Regular hyperbolic and spherical polygons are determined up to isometry by their spectrum within the class of all hyperbolic and spherical polygons. Moreover, the spectrum determines whether hyperbolic or spherical polygons are convex (see Corollary \ref{corollary:ExistencePolygonsSpectrallyDetermined}).
\item[$\bullet$] A simply connected Euclidean polygon can never be isospectral to a two-dimensional compact manifold with nonempty smooth boundary. Moreover, a polygon with zero curvature can not be isospectral to a smooth domain in the Euclidean plane (see Corollary \ref{corollary:SmoothDomainsSimplyConnectedPolygonalDomainsNeverIsospectral}).
\end{itemize}

In connection with the situation described above, we also study the heat invariants for closed two-dimensional Riemannian orbifolds of constant curvature. Orbifolds can be regarded as a natural generalisation of manifolds, and as such the class of orbifolds naturally possesses a greater variety than that of manifolds. Roughly speaking, an orbifold consists of a regular manifold possibly with an additional set of various singular points, where everything is related to each other through an orbifold structure. For example, consider a hyperbolic polygon $\Omega$ whose angles are equal to rational multiples of $\pi$. Such a polygon can be captured accurately as a closed orbifold, such that the regular part is equal to $\Omega$, the singular set is equal to $\partial\Omega$ and all geometric properties are described properly within the orbifold structure. Over the last 15 years, orbifolds were studied intensively by spectral geometers and many objects, ideas and tools known from manifold theory were translated into the realm of orbifolds (see \citep{Gordon08}, \citep{Gordon12}). 

Suppose $\mathcal{O}$ is any Riemannian orbisurface, i.e. a two-dimensional closed Riemannian orbifold, and let $\Delta_{\mathcal{O}}$ be its Laplacian. Again, the spectrum of $\Delta_{\mathcal{O}}$ consists of a non-negative sequence of eigenvalues without a finite accumulation point. All eigenvalues have a finite multiplicity and we list them like in \eqref{eigenvalues} as $0=\mu_1\leq \mu_2\leq...\nearrow \infty$. The heat trace $Z_{\mathcal{O}}(t):=\sum_{i=1}^{\infty}e^{-\mu_i t}$ of the orbifold has an asymptotic expansion as $t\searrow 0$ exactly as in \eqref{heat asymptotics}. The coefficients in this asymptotic expansion, the heat invariants, contain geometric information about the orbifold. We will compute the asymptotic expansion of $Z_{\mathcal{O}}(t)$ as $t\searrow 0$ for any orbisurface of constant curvature. Thereby we obtain explicit formulas for all heat invariants in terms of the curvature and the singular points of the orbifold (see Theorem \ref{theorem:OrbifoldAsymptotikKonstanteKruemmung}).

It is well-known that two-dimensional orbifolds can have three types of singular points, called \emph{mirror points}, \emph{dihedral points}, and \emph{cone points}. The heat invariants provide information about the singular set and the topology of the orbisurface. We will show in particular the following: 
\begin{itemize}
\item[$\bullet$]  The spectrum of any orbisurface $\mathcal{O}$ with constant nonzero curvature fixes the value of the sum $M+2N$, where $M$ denotes the number of all dihedral points and $N$ denotes the number of all cone points of $\mathcal{O}$.  Furthermore, the spectrum determines the multiset $\{m_1,...,m_M, n_1,n_1,...,n_N,n_N\}$, where $2m_1,...,2m_M$ denote the orders of the dihedral points and $n_1,...,n_N$ denote the orders of the cone points of $\mathcal{O}$ (see Corollary \ref{corollary:SpectralInvariantsOrbisurfaces} (iii)).
\item[$\bullet$]  If two orientable orbisurfaces with constant curvature $\kappa\neq 0$ are isospectral, then the orbifolds have the same Euler characteristic, and the underlying topological spaces are homeomorphic (see Corollary \ref{corollary:OrientableOrbisurfacesSameUnderlyingSpace}).  This generalises a result of \citep{Strohmaier} in which the same result was proven in case of constant curvature $\kappa=-1$ using the Selberg trace formula for the wave kernel.
\item[$\bullet$] Within various classes of orbifolds, the spectrum determines the singular set and the Euler characteristic of the orbifold as well as the Euler characteristic of the underlying topological space (see Corollary \ref{corollary:IsospectralOrbifoldsNoConeNoDihedral} and Corollary \ref{corollary:SpectrumOfNonOrientableOrbisurfaces}).
\end{itemize}

\bigskip
The contents of this thesis are structured as follows. In the second chapter, we provide the requisite background in spectral geometry and about Legendre functions (more precisely, so-called associated Legendre functions). Then we investigate and solve a class of integrals involving the associated Legendre functions of the second kind. The solutions of these integrals seem to be unknown to date and some of the solutions will be used in the subsequent chapter in order to compute the Green's function and heat kernel associated to an arbitrary hyperbolic wedge.

The third chapter is the main part of this thesis. In the first section, we obtain explicit formulas for the Green's function and the heat kernel of a hyperbolic wedge by solving the corresponding boundary value problems. These formulas are used in the second section in order to compute the asymptotic expansion of the heat trace for any hyperbolic polygon. Explicit formulas for all heat invariants are given as well. In the third section, we use the heat invariants for Euclidean polygons (from \citep{VanDenBerg}) and spherical polygons (from \citep{Watson}) to compute the heat invariants for polygons of arbitrary constant curvature. We then draw some conclusions from the heat invariants, i.e. we discuss what can be said about a polygon knowing all its heat invariants. Finally, in the last section, we investigate once again the heat trace for hyperbolic polygons with interior angles $\frac{\pi}{k}$, for some $k\in\mathbb{N}_{\geq 2}$. By using Sommerfeld's method of images we obtain another and more elementary formula for the heat kernel of any wedge with angle $\frac{\pi}{k}$, $k\in\mathbb{N}_{\geq 2}$. Using this formula we compute the contributions to the heat invariants from the vertices of a polygon once again. They will now appear rather differently as finite trigonometric sums. By comparison with the previous formulas, we obtain remarkable identities for these trigonometric sums.

The final chapter deals with two-dimensional orbifolds, also called orbisurfaces. Here, we first investigate two examples of spherical orbisurfaces and, in particular, how their heat traces are related to each other. Then we use this relation and some observations from the article \citep{Watson} to compute the heat invariants for two-dimensional Riemannian orbifolds of constant curvature. We will obtain formulas for all heat invariants in terms of the Gaussian curvature and the singular strata of the orbifold. In the last section, we discuss some applications of the heat invariants and show how the orbifold theory is linked to the results of Chapter \ref{chapter:polygons}. For example, we prove that the contribution of an interior angle $\frac{\pi}{k}$, $k\in\mathbb{N}_{\geq 2}$, to the heat invariants for a polygon is the same as the contribution of a dihedral point of isotropy order $2k$ to the heat invariants for an orbisurface, if the polygon and the orbisurface have the same curvature. This fact together with the formulas from \citep{Watson} provide an alternative proof for the contribution of an angle $\frac{\pi}{k}$, $k\in\mathbb{N}_{\geq 2}$, to the heat invariants for a polygon.

%% file: chapter02.tex
\chapter{Preliminaries}
\label{chapter:Preliminaries}

\section{Background in spectral geometry}
\label{section:Background}

In this section we want to recall some basic definitions and well-known facts to lay the foundations for this thesis. In particular we will clarify our notation along the way. Since all of the material in this section is well documented in the literature, we do not need to go into the details. Instead we will always point out where to find the corresponding topics in the literature. 

Let $M$ be a two-dimensional smooth connected Riemannian manifold and $\Omega\subset M$ be any non-empty domain, i.e. an open and connected subset. Note that the case $\Omega=M$ is included here. Let $\Delta_{\Omega}$ be the (positive semi-definite) Dirichlet Laplacian for $\Omega$ with respect to the Riemannian measure (see \cite{Grigoryan}).

A central function for our purposes is the heat kernel associated to the Dirichlet Laplacian $\Delta_{\Omega}$, which we want to introduce first. Since we will consider no other operator than the Dirichlet Laplacian for $\Omega$, we will call it simply the heat kernel of $\Omega$.

\begin{definition} (see \cite{Grigoryan}, Definition 9.1.)
\label{definition:FundamentalSolution}
A smooth function $u:\Omega\times (0,\infty)\rightarrow \mathbb{R}$ is called a \emph{fundamental solution} to the heat equation at the point $y\in \Omega$ if it satisfies the following conditions:
\begin{align}
\label{equation:FundamentalSolution}
\begin{cases}
\text{ }\text{ }\left( \partial_t + \Delta\right) u(x,t) = 0, &\quad \forall\, (x,t)\in \Omega\times (0,\infty),\\
\lim\limits_{t\searrow 0} \int\limits_{\Omega}u(x,t) f(x) dx = f(y),&\quad \forall\, f\in C_{c}^{\infty}(\Omega).
\end{cases}
\end{align}
As usual, $C_{c}^{\infty}(\Omega)$ denotes the space of all real-valued, smooth and compactly supported functions on $\Omega$. Further, $\Delta:C^{\infty}(\Omega)\rightarrow C^{\infty}(\Omega)$, $\Delta f :=-\diverg(\nabla f)$ denotes the Laplacian of $\Omega$, where $\nabla$ and $\diverg$ are taken with respect to the Riemannian metric of $\Omega$ induced by $M$.
\end{definition}

The heat kernel of $\Omega$ can be characterised as the minimal non-negative fundamental solution to the heat equation in the sense of Definition \ref{definition:HeatKernel} below. For the existence of such a minimal fundamental solution we refer to \citep{Grigoryan}. The uniqueness follows directly from the minimality property.

\begin{definition} (see \cite{Grigoryan}, Theorem 9.5.)
\label{definition:HeatKernel}
The heat kernel of $\Omega$ is defined as the smooth function
\begin{align*}
K_{\Omega}:\Omega\times\Omega\times (0,\infty)\rightarrow \mathbb{R},\quad (x,y,t)\mapsto K_{\Omega}(x,y;t), 
\end{align*}
such that for any $y\in\Omega$ the function $K_{\Omega}\left(\cdot, y;\cdot \right)$ is a non-negative solution to \eqref{equation:FundamentalSolution}, and it is \emph{minimal} in the following sense: If $u(x,t)$ is another smooth and non-negative solution to \eqref{equation:FundamentalSolution}, then $u(x,t)\geq K_{\Omega}\left(x, y;  t \right)$ for all $x\in\Omega$, $t>0$.
\end{definition}

Another very important function is the heat trace, defined for relatively compact domains and closely related to the heat kernel.

\begin{definition}
\label{definition:heat trace}
Let $\Omega\subset M$ be a relatively compact domain, i.e. a domain with compact closure. The \emph{heat trace} of $\Omega$ is defined as the function $Z_{\Omega}:(0,\infty)\rightarrow\mathbb{R}$ given by
\begin{align}
\label{equation:heat trace}
Z_{\Omega}(t):=\int\limits_{\Omega} K_{\Omega}(x,x;t) dx, \quad t>0.
\end{align}
\end{definition}

The integral in \eqref{equation:heat trace} is convergent since $0\leq K_{\Omega}(x,x;t)\leq K_{M}(x,x;t)$  for all $t>0$, $x\in\Omega$ and $x\mapsto K_{M}(x,x;t)$ is integrable as a continuous function on the compact domain $\overline{\Omega}$. For the second estimate note that $K_{\Omega}(\cdot,x;\cdot)$ and $K_{M}(\cdot,x;\cdot)_{\mid \Omega\times (0,\infty)}$ are both fundamental solutions at $x\in\Omega$ and $K_{\Omega}(\cdot,x;\cdot)$ satisfies the minimality property of Definition \ref{definition:HeatKernel}.

\begin{proposition} \emph{(see \citep[Theorem 10.13]{Grigoryan})}
\label{proposition:heat trace bounded domain}
If the domain $\Omega$ is relatively compact, then the following holds:
\begin{itemize}
\item[$(i)$] The spectrum of $\Delta_{\Omega}$ consists of a sequence of non-negative eigenvalues without a finite accumulation point and each eigenvalue has a finite multiplicity.

Suppose the eigenvalues $\left( \lambda_i \right)_{i=1}^{\infty}$ are enumerated with multiplicity such that
\begin{align}
0\leq \lambda_1<\lambda_2\leq \lambda_3\leq...\nearrow  \infty.
\end{align}
There exists an orthonormal basis $\lbrace \varphi_i \rbrace_{i=1}^{\infty}$ for $L^2\left( \Omega \right)$ such that each function $\varphi_i$ is an eigenfunction of $\Delta_{\Omega}$ corresponding to the eigenvalue $\lambda_i$.
\item[$(ii)$] The heat kernel $K_{\Omega}(x,y;t)$ is given by
\begin{align}
\label{equation:HeatKernelBoundedDomain}
K_{\Omega}(x,y;t) = \sum\limits_{i=1}^{\infty} e^{-\lambda_i t}\varphi_i(x)\varphi_i(y),
\end{align}
where the series converges absolutely and uniformly on $\Omega\times\Omega\times [\epsilon,\infty)$ for any $\epsilon>0$.
\end{itemize}
\end{proposition}

\begin{remark}
\label{remark:heatkernelproperties}
Because $\Omega$ is connected (by definition), the first eigenvalue must be simple, i.e. $\lambda_1<\lambda_2$ (see \citep[Corollary 10.12]{Grigoryan}). Further, one can show that $\lambda_1>0$ if $M\backslash \overline{\Omega}$ is non-empty (see \citep[Theorem 10.22]{Grigoryan}). 

If $\Omega$ is a relatively compact domain with \emph{smooth} boundary $\partial \Omega$, then the heat kernel can be extended continuously to $\overline{\Omega}\times \overline{\Omega} \times (0,\infty)$ by setting $K_{\Omega}(x,y;t):=0$ whenever $x$ or $y$ lies on the boundary (see \citep[Section 2.3]{GrigoryanEstimate} and \citep{Dodziuk}).
\end{remark}

\begin{corollary}
\label{corollary:HeatTraceBoundedDomain}
If $\Omega$ is relatively compact, then the heat trace is related to the spectrum of $\Delta_{\Omega}$ by
\begin{align}
\label{equation:HeatTraceBoundedDomain}
Z_{\Omega}(t)=\sum\limits_{i=1}^{\infty}e^{-\lambda_i t}.
\end{align}
\end{corollary}

\begin{proof}
Since for any fixed $t>0$, the series in \eqref{equation:HeatKernelBoundedDomain} converges uniformly on $\Omega\times\Omega$, we obtain
\begin{align*}
Z_{\Omega}(t)&=\int\limits_{\Omega} \sum\limits_{i=1}^{\infty} e^{-\lambda_i t}\varphi_i(x)^2dx = \sum\limits_{i=1}^{\infty} \int\limits_{\Omega} e^{-\lambda_i t}\varphi_i(x)^2dx =\sum\limits_{i=1}^{\infty}  e^{-\lambda_i t}.
\end{align*}
The last equality follows because $\Vert \varphi_i \Vert_{L^2(\Omega)}=1$ by definition. 
\end{proof}

We will use the Laplace transform in Section \ref{section:Green} to solve boundary value problems for the Green's function and the heat kernel. Therefore, we want to summarise some fundamental facts of the Laplace transform in advance, referring to \cite{DaviesIntegral} and \cite{Doetsch} for more details.

\begin{definition}
\label{definition:LaplaceTransform}
Let $f:(0,\infty)\rightarrow\mathbb{R}$ be a continuous function. The \emph{Laplace integral} of $f$ with parameter $s\in\mathbb{C}$ is defined formally as
\begin{align*}
\int\limits_{0}^{\infty} e^{-st}f(t) dt. 
\end{align*}
Suppose the Laplace integral of $f$ \emph{exists} for at least one value of $s$, i.e. the function $(0,\infty)\ni t\mapsto e^{-st}f(t)\in\mathbb{C}$ is integrable. Then the \emph{Laplace transform} of $f$ is the function $F$ given by
\begin{align}
\label{equation:LaplaceTransform}
F(s):=\int\limits_{0}^{\infty} e^{-st}f(t)dt,
\end{align} 
which is defined for all values of $s\in\mathbb{C}$ such that the Laplace integral exists. We write $\mathcal{L}\lbrace f \rbrace$ for the Laplace transform of $f$, such that $\mathcal{L}\lbrace f \rbrace(s):=F(s)$.
\end{definition}

In practice, the Laplace transform is often convergent and holomorphic on a half-plane as in \eqref{equation:HalfPlane} below.

\begin{definition}
\label{definition:HalfPlane}
For any $\delta\in\mathbb{R}$ we use the following notation:
\begin{align}
\label{equation:HalfPlane}
\mathcal{H}_{>\delta}:=\{\,z\in\mathbb{C} \mid  \Re(z)>\delta \,\},
\end{align}
where $\Re(z)$ denotes the real part of $z$.
\end{definition}

\begin{proposition}
\label{proposition:InverseLaplaceTransform}
Let $f:(0,\infty)\rightarrow\mathbb{R}$ be continuous and suppose the Laplace integral of $f$ is absolutely convergent at some point $\delta\in\mathbb{R}$. Then the Laplace integral is absolutely convergent for all $s\in \mathcal{H}_{>\delta}$ and the Laplace transform $F=\mathcal{L}\lbrace f \rbrace$ is holomorphic on the half-plane $\mathcal{H}_{>\delta}$. Moreover, the Laplace transform is unique, i.e. if $g:(0,\infty)\rightarrow \mathbb{R}$ is continuous and $\mathcal{L}\lbrace g \rbrace = \mathcal{L}\lbrace f \rbrace$, then $f(t)=g(t)$ for all $t\in (0,\infty)$.

If $f$ is even continuously differentiable, then for any $\varepsilon>\delta$ and any $t\in(0,\infty)$,
\begin{align}
\label{equation:InverseLaplaceTransform}
f(t)=\frac{1}{2\pi i}\lim\limits_{N\rightarrow\infty}\int\limits_{\varepsilon-iN}^{\varepsilon+iN}e^{st} F(s) ds.
\end{align}
\end{proposition}

For the result about uniqueness of the Laplace transform, we refer to \citep[Satz 9 on p. 79]{Doetsch}. The other results of Proposition \ref{proposition:InverseLaplaceTransform} are also contained in \citep{Doetsch}, but more elegant proofs can be found in \citep[Section 2.1 and 3.3]{DaviesIntegral}. 

We emphasise that the integral in \eqref{equation:InverseLaplaceTransform} is taken along an arbitrary straight line in the half-plane $\mathcal{H}_{>\delta}$, parallel to the imaginary axis. Furthermore, both endpoints of the integral tend to infinity simultaneously. 

\begin{definition}
\label{definition:InverseLaplaceTransform}
A continuous function $f:(0,\infty)\rightarrow\mathbb{R}$ is called the \emph{inverse Laplace transform} of a function $F$, if $\mathcal{L}\{ f\} = F$. Then we use the notation $\mathcal{L}^{-1}\lbrace F \rbrace:=f$.
\end{definition} 

If the inverse Laplace transform of a given function is not known already, one can use \eqref{equation:InverseLaplaceTransform} for its calculation. Equation \eqref{equation:InverseLaplaceTransform} is called the \emph{complex inversion formula} and is actually also valid in more general situations (see \citep[Satz 1 on p. 210]{Doetsch}).

Next, we want to consider the Laplace transform of the heat kernel $K_{\Omega}(x,y;\cdot)$. It does not always exist if $x = y$ and thus we introduce the following sets.

\begin{definition}
\label{definition:OffDiagonal}
The sets
\begin{align*}
\diag(\Omega):&=\{\, \left(x,y\right)\in\Omega\times\Omega \mid x=y \,\}, \\
\offdiag(\Omega):&=\{\, \left(x,y\right)\in\Omega\times\Omega \mid x\neq y \,\}
\end{align*} 
are called the \emph{diagonal} and \emph{off-diagonal} of the cartesian product $\Omega\times\Omega$, respectively.
\end{definition}

By definition, $K_{\Omega}(x,y;t)$ is smooth on $\Omega\times\Omega\times (0,\infty)$. Moreover, we can extend the heat kernel smoothly to $\offdiag(\Omega)\times \mathbb{R}$ by setting $K_{\Omega}(x,y;t):=0$ for all $t\leq 0$ (see \citep[Corollary 9.21]{Grigoryan}). In particular, for any $(x,y)\in\offdiag(\Omega)$ the function $K_{\Omega}(x,y;\cdot)$ can be extended continuously at $t=0$ and one can prove that the Laplace integral of $K_{\Omega}(x,y;\cdot)$ exists for all $s>0$ (see \citep[Exercise 9.10]{Grigoryan}). Thus, by Proposition \ref{proposition:InverseLaplaceTransform}, the Laplace transform $s\mapsto \mathcal{L}\{
K_{\Omega}(x,y;\cdot)\}(s)$ is holomorphic on the domain $\mathcal{H}_{>0}$.

\begin{definition}
\label{definition:Greens}
The Laplace transform of $K_{\Omega}(x,y;t)$ is called the \emph{resolvent kernel} or \emph{Green's function} of $\Omega$. More precisely, the Green's function is defined as the function
\begin{align}
\label{equation:Greens}
\begin{split}
&G_{\Omega}: \offdiag(\Omega) \times \mathcal{H}_{>0}\rightarrow \mathbb{C}, \\
&\left((x,y),s\right)\mapsto G_{\Omega}(x,y;s):=\int\limits_{0}^{\infty} e^{-st}K_{\Omega}(x,y;t) dt.
\end{split}
\end{align}
\end{definition}

\begin{proposition}
\label{proposition:PDEGreen}
\emph{(\citep[Exercise 9.10]{Grigoryan})} For any $s>0$, the Green's function $G_{\Omega}\left( \cdot,\cdot;s \right)$ is a non-negative and smooth function on $\offdiag(\Omega)$. Furthermore, for any $y\in\Omega$ and any $s > 0$ the function $u:\Omega\backslash \{y\}\ni x\mapsto G_{\Omega}(x,y;s)\in\mathbb{R}$ belongs to $L^1(\Omega)$ and
\begin{align}
\label{equation:PDEGreen}
\int\limits_{\Omega} u(x)\cdot \left( s + \Delta \right)f(x)\, dx &=f(y)\, \text{ for all } f\in C_c^{\infty}(\Omega).
\end{align}
In particular, $\left( s + \Delta \right)u(x)=0$ for all $x\in\Omega\backslash \{y\}$.
\end{proposition}

\section{Associated Legendre functions}
\label{section:Legendre}

In this section, we want to introduce the so-called associated Legendre functions of the first and second kinds and summarise some of their properties. Then we investigate a special class of integrals involving associated Legendre functions, which seem to be unsolved up to now. Some of these integrals are important for our purposes, because they will naturally appear in the next chapter, where we will obtain explicit formulas for the Green's function and the heat kernel of a wedge. Therefore we will explain our method for solving those integrals in detail. Further, we want to remark that these integrals can be rewritten as so-called \emph{generalised Mehler transforms}. From this point of view we basically compute new generalised Mehler transforms. Despite the fact that the generalised Mehler transform has been known for some time (see \cite{Lowndes}, \cite{Sneddon}, \cite{Rosenthal}), apparently there hardly exist calculated examples compared with other integral transformations (see \cite{Oberhettinger} for a list of known generalised Mehler transforms).

We want to stress that our discussion of the associated Legendre functions will not be complete, and that extensive treatises can be found in the literature on associated Legendre functions (see for example \citep[Chapter $8.7-8.8$]{Gradshteyn}, \cite{Hobson}, \cite{Virchenko}, \citep[ Chapter III]{Erdelyi}, \citep[Chapter $14$]{NIST}, \citep[Chapter $5$]{Olver}, and \citep[Chapter $8$]{Temme}).

\begin{definition}
\label{definition:LegendreDGL}
The \emph{associated Legendre equation} is an ordinary differential equation, defined as
\begin{align}
\label{equation:LegendreDGL}
\left( 1-z^2 \right)\frac{d^2 u}{dz^2} - 2z \frac{d u}{dz} + \left( \nu \left( \nu+1 \right) - \frac{\mu}{1-z^2} \right) u = 0,
\end{align}
with parameters $\nu, \mu \in \mathbb{C}$ and variable $z\in \mathbb{C}\backslash\lbrace{ -1, +1 \rbrace}$. Any solution of \eqref{equation:LegendreDGL} is called an \emph{associated Legendre function}.
\end{definition}

In particular, we are interested in two special solutions, namely, the associated Legendre functions of the first and second kinds. These can be defined in several equivalent ways; we choose to define them via the hypergeometric function.

\begin{definition}
\label{definition:Pochhammer}
For any $z\in\mathbb{C}, k\in\mathbb{N}_0$, the \emph{Pochhammer symbol} is defined as
\begin{align*}
\left( z \right)_{k}:=\begin{cases}
1 , & \text{if $k=0$,}\\
z(z+1)(z+2)\cdots (z+k-1), & \text{if $k\geq 1$.}
\end{cases}
\end{align*}
It follows that $\left( -n \right)_k = 0$ if and only if $n\in\mathbb{N}_0,k\in\mathbb{N}$ with $k>n$. The Pochhammer symbol is related to the gamma function $\Gamma(z)$ through the equation
\begin{align*}
(z)_k = \frac{\Gamma(z+k)}{\Gamma(z)},
\end{align*}
which is easily established using the identity $\Gamma(z+1)=z\cdot \Gamma(z)$.
\end{definition}

The following facts about the hypergeometric function, i.e. everything up to Definition \ref{definition:LegendreFirstSecond} below, can be found for example in \citep[Section 9 of Chapter 5]{Olver}.

\begin{definition} 
\label{definition:Hypergeom}
Let $a,b,c \in\mathbb{C}$ such that $c\notin \lbrace 0,-1,-2,... \rbrace$, and let $z\in\mathbb{C}\backslash [1,\infty)$. The \emph{hypergeometric function} $F(a,b;c;z)$ is defined through
\begin{align*}
F(a,b;c;z):=\sum\limits_{k=0}^{\infty}\frac{(a)_k  (b)_k}{(c)_k \text{ } k!} z^k,\quad \text{for }\vert z \vert<1,
\end{align*}
and via analytic continuation for $\vert z \vert \geq 1$ with $z\notin [ 1,\infty)$. 

The function
\begin{align*}
\frac{1}{\Gamma(c)}F(a,b;c;z)
\end{align*}
is an entire function in all three parameters $a,b,c$ and is analytic in $z\in\mathbb{C}\backslash [1,\infty)$. As usual, this means that the function is analytic wherever it is defined directly and can be extended analytically to where it is not formally defined. In particular, we think of the above function as analytically extended in $c\in\lbrace 0,-1,-2,... \rbrace$, even though $F(a,b;c;z)$ is formally not defined for those values of $c$.
 Furthermore, the hypergeometric function is a solution of the \emph{hypergeometric differential equation}:
\begin{align*}
z(1-z)\frac{d^2 u}{dz^2} + \left( c-\left(a+b+1 \right)z \right)\frac{d u}{dz}- ab\cdot u = 0.
\end{align*}
\end{definition}

We use the hypergeometric function to introduce now the most important functions in this section. For powers, we always use the standard branch: $z^{\mu}=e^{\mu\log(z)}$ with $\log(z):=\log(\vert z \vert) + i\arg(z)$ for $z\in\mathbb{C}\backslash (-\infty,0],$ with $\arg(z)\in(-\pi,\pi)$.

\begin{definition} Let $z\in\mathbb{C}\backslash (-\infty, 1].$
\label{definition:LegendreFirstSecond}
The \emph{associated Legendre function of the first kind} is defined by
\begin{align}
\label{equation:LegendreFirst}
P_{\nu}^{\mu}(z) := \left( \frac{z+1}{z-1} \right)^{\frac{\mu}{2}}\frac{1}{\Gamma(1-\mu)}F\left( -\nu, \nu + 1; 1-\mu;\frac{1-z}{2} \right),
\end{align}
where $\nu,\mu\in\mathbb{C}$.
The \emph{associated Legendre function of the second kind} is defined by
\begin{align}
\label{equation:LegendreSecond}
Q_{\nu}^{\mu}(z) := \frac{e^{\mu\pi i} \Gamma\left( \nu+\mu+1 \right)\sqrt{\pi} \left( z^2-1 \right)^{\frac{\mu}{2}}}{2^{\nu +1} z^{\nu+\mu+1}\cdot \Gamma\left(\nu + \frac{3}{2}\right)}F\left(\frac{\nu+\mu+2}{2}, \frac{\nu+\mu+1}{2}; \nu+\frac{3}{2};\frac{1}{z^2} \right),
\end{align}
where $\mu,\nu\in\mathbb{C}$ such that $\mu+\nu\notin\lbrace -1,-2,-3,,... \rbrace$.

For $\mu=0$ we use $P_{\nu}(z) := P_{\nu}^{0}(z)$ and $Q_{\nu}(z) := Q_{\nu}^{0}(z)$. These functions are also known as \emph{Legendre functions of the first and second kinds}, respectively.

Note that for $z\in\mathbb{C}\backslash (-\infty,1]$, both $\frac{1-z}{2}$ and $\frac{1}{z^2}$ are in $\mathbb{C}\backslash [1,\infty)$, so that Definition \ref{definition:Hypergeom} indeed applies. Moreover, for these $z$ we have $\frac{z+1}{z-1}\in\mathbb{C}\backslash (-\infty,0]$, so $\big( \frac{z+1}{z-1} \big)^{\frac{\mu}{2}}$ is defined as explained above. The same holds for $\left(z^2-1\right)^{\frac{\mu}{2}}$ if we read it as $(z+1)^{\frac{\mu}{2}}\cdot (z-1)^{\frac{\mu}{2}}$.
\end{definition} 
The associated Legendre functions of the first and second kinds are linearly independent solutions of \eqref{equation:LegendreDGL}, and are analytic in the parameters $\mu$ and $\nu$, whenever defined, as well as in the variable $z.$ Definition \ref{definition:LegendreFirstSecond} was first introduced by Hobson, it can be found in his classical treatise \cite{Hobson} and is commonly used in the literature (e.g. in \cite{Gradshteyn}, \cite{Erdelyi}, \cite{Temme} etc.). Some authors, however, define the associated Legendre functions differently, in particular, the function of the second kind. For example, E. Barnes and G. Watson define $Q_{\nu}^{\mu}(z)$ differently in \cite{BarnesLegendre} and \cite{WatsonAsymp}, respectively. Also \cite{Olver}, \cite{NIST} do not prefer to work with $Q_{\nu}^{\mu}(z)$ as defined in \eqref{equation:LegendreSecond}.

There are plenty of remarkable relations and identities between the asociated Legendre functions of the first and second kinds (see \cite{Gradshteyn}, \cite{Erdelyi}). For the convenience of the reader we summarise those which will be important for us.

\begin{lemma}
\label{lemma:MagischeFormeln}
The associated Legendre functions of the first and second kinds satisfy the following identities for all allowed values of $\nu, \mu, z,\omega$:
\begin{align}
P_{\nu}^{\mu}(z) &=P_{-\nu-1}^{\mu}(z), \label{equation:ReflectionLegendreP} \\
Q_{\nu}^{\mu}(z) &= e^{2i\mu\pi}\frac{\Gamma\left( \nu+\mu+1 \right)}{\Gamma\left( \nu-\mu+1 \right)}Q_{\nu}^{-\mu}(z), \label{equation:MagischeFormel1}\\
P_{\nu}^{-\mu}(z) &= \frac{\Gamma\left( \nu-\mu+1 \right)}{\Gamma\left( \nu+\mu+1 \right)}\left(P_{\nu}^{\mu}(z) - \frac{2}{\pi}e^{-i\mu\pi}\sin\left( \mu\pi \right) Q_{\nu}^{\mu}(z) \right),\label{equation:MagischeFormel2} \\
Q_{\nu}^{-\mu}(z)Q_{\nu}^{\mu}(\omega) &= Q_{\nu}^{-\mu}(\omega)Q_{\nu}^{\mu}(z), \label{equation:SymmetrieProduktLegendreQ}\\
-2\frac{\sin\left( \mu\pi \right)}{\pi}Q_{\nu}^{-\mu}(z)Q_{\nu}^{\mu}(\omega) &= e^{-i\mu\pi}P_{\nu}^{-\mu}(\omega)Q_{\nu}^{\mu}(z) - e^{i\mu\pi}P_{\nu}^{\mu}(\omega)Q_{\nu}^{-\mu}(z). \label{equation:MagischeFormel3}
\end{align}
\end{lemma}

\begin{proof}
Equation \eqref{equation:ReflectionLegendreP} is stated in \citep[formula 8.731 5]{Gradshteyn} and is a consequence of \eqref{equation:LegendreFirst} and the symmetry of the hypergeometric function with respect to the first two variables, i.e.  $F(a,b;c;z)=F(b,a;c;z)$. The identities \eqref{equation:MagischeFormel1} and  \eqref{equation:MagischeFormel2} can be found in \citep[formula 8.736 4]{Gradshteyn} and  \citep[formula 8.736 1]{Gradshteyn}, respectively. Equation \eqref{equation:SymmetrieProduktLegendreQ} follows from \eqref{equation:MagischeFormel1}, when applied twice.

In order to prove \eqref{equation:MagischeFormel3}, we first apply \eqref{equation:MagischeFormel1}, \eqref{equation:MagischeFormel2} and get:
\begin{align*}
P_{\nu}^{-\mu}(\omega)Q_{\nu}^{\mu}(z) = e^{2i\mu\pi}P_{\nu}^{\mu}(\omega)Q_{\nu}^{-\mu}(z) - 2e^{i\mu\pi}\frac{\sin\left( \mu\pi \right)}{\pi}Q_{\nu}^{-\mu}(z)Q_{\nu}^{\mu}(\omega).
\end{align*}
If we multiply both sides of the above equation with $e^{-i\mu\pi}$ and rearrange the terms, we obtain \eqref{equation:MagischeFormel3}.
\end{proof}

Another remarkable connection between the associated Legendre functions of the first and second kinds is given by \emph{Whipple's formula} (see \citep[formulas (13) and (14) on p. 141]{Erdelyi}).
\begin{lemma}
\label{lemma:Whipple}
For all $z\in\mathbb{C}\backslash (-\infty,1]$ with $\Re(z)>0$ the following relations hold:
\begin{align}
e^{-i\mu\pi}Q_{\nu}^{\mu}(z)&=\sqrt{\frac{\pi}{2}}\Gamma\left( \nu+\mu+1 \right)\frac{1}{\left( z^2-1 \right)^{\frac{1}{4}}}P_{-\mu-\frac{1}{2}}^{-\nu-\frac{1}{2}}\left( \frac{z}{\left( z^2 -1 \right)^{\frac{1}{2}}} \right),\label{equation:Whipple1} \\
P_{\nu}^{\mu}(z)&=ie^{i\nu\pi}\sqrt{\frac{2}{\pi}}\cdot \frac{1}{\Gamma\left(-\nu-\mu\right)}\cdot \frac{1}{\left( z^2-1 \right)^{\frac{1}{4}}}Q_{-\mu-\frac{1}{2}}^{-\nu-\frac{1}{2}}\left( \frac{z}{\left( z^2 -1 \right)^{\frac{1}{2}}} \right). \label{equation:Whipple2}
\end{align}
\end{lemma}

In order to investigate certain integrals later in this section, it is helpful to know the following asymptotic behaviour of the associated Legendre functions of the first and second kinds.

\begin{lemma}
\label{lemma:LegendreAsymp}
Let $a>0$.
\begin{itemize}
\item[$(i)$] For any $\mu\in\mathbb{C}$:
\begin{align} P_{\nu}^{\mu}\left( \cosh(a) \right) = &\frac{\Gamma\left( \nu+1 \right)}{\Gamma\left( \nu-\mu+1 \right)} \cdot \frac{1}{\sqrt{2\pi (\nu+1)\sinh(a)}} \nonumber\\
& \cdot \left( e^{\left( \nu+\frac{1}{2} \right)a} + e^{-\pi i \left( \mu-\frac{1}{2} \right)-\left( \nu+\frac{1}{2} \right)a} \right)\left( 1+O\left( \frac{1}{\vert \nu \vert} \right) \right) \label{equation:LegendreAsymp1}
\end{align}
as $\vert \nu \vert \rightarrow\infty$ with $\Re\left(\nu \right)>-1$.
\item[$(ii)$] For any $\mu\in\mathbb{C}$:
\begin{align}
P_{-\frac{1}{2}+i\rho}^{\mu}\left(\cosh(a)\right) = \rho^{\mu-\frac{1}{2}}\sqrt{\frac{2}{\pi\sinh(a)}}\cos\left( a\rho+\frac{\pi}{4}\left( 2\mu-1 \right) \right)\left( 1+O\left(\frac{1}{\rho}\right) \right) \label{equation:LegendreAsymp2}
\end{align}
as $\rho\rightarrow\infty$ with $\rho\in\mathbb{R}$.
\item[$(iii)$] For any $\mu\in\mathbb{C}$ and $\delta\in (0,\pi)$:
\begin{align}
Q_{\nu}^{\mu}\left(\cosh(a)\right) = &\sqrt{\frac{\pi}{2\sinh(a)}}\nu^{\mu-\frac{1}{2}}e^{i\mu\pi-a\left( \nu+\frac{1}{2} \right)}\left(1+O\left(\frac{1}{\vert \nu \vert}\right)\right) \label{equation:LegendreAsymp3}
\end{align}
as $ \vert \nu \vert \rightarrow\infty$ with $\vert \arg(\nu) \vert<\pi-\delta $.
\end{itemize}
\end{lemma}

\begin{proof}
All three asymptotic estimates are basically stated at the beginning of $\S 8$ in \citep{Virchenko}.

A proof of \eqref{equation:LegendreAsymp1} can be found in \cite{Goetze}, which is also referred to in \citep{Virchenko}. At this point we want to remark that both sources \cite{Virchenko} and \cite{Goetze} state more general results. They deal with asymptotic estimates for so-called \emph{generalised associated Legendre functions} of \emph{first} and \emph{second kinds}, denoted by $P_{\nu}^{n, m}(z)$ and $Q_{\nu}^{n, m}(z)$ respectively. For $n=m=\mu$ these functions reduce to $P_{\nu}^{\mu}(z)$, respectively $Q_{\nu}^{\mu}(z)$.

We show how equation \eqref{equation:LegendreAsymp2} can be deduced from \eqref{equation:LegendreAsymp1}. Let $\rho\in\mathbb{R}$ and set $\nu = -\frac{1}{2}+i\rho$ in equation \eqref{equation:LegendreAsymp1} to obtain
\begin{align*}
P_{-\frac{1}{2}+i\rho}^{\mu}\left( \cosh(a)\right) = \frac{\Gamma(i\rho+\frac{1}{2})}{\Gamma(i\rho+\frac{1}{2}-\mu)} \left(\frac{1}{2}+ i\rho \right)^{-\frac{1}{2}} &\sqrt{\frac{2}{\pi\sinh(a)}}\,\, \cdot  \\
&\cdot \frac{1}{2}\left( e^{i\rho a}+ e^{-\pi i\left(\mu-\frac{1}{2}\right)-i\rho a} \right)\left( 1+ O\left( \frac{1}{\rho} \right) \right)
\end{align*}
as $\rho\rightarrow\infty$. Now use the following well-known asymptotic estimate for the quotient of two gamma functions (see \citep[formula (12)]{Tricomi} or \citep[formula (11) on p. 33]{Luke1}):
\begin{align}
\frac{\Gamma\left( z+\alpha \right)}{\Gamma\left( z+\beta \right)} = z^{\alpha-\beta}\left(1 + O\left( \frac{1}{\vert z\vert}\right) \right) \label{equation:GammaQuotAsymp}
\end{align}  
as $z\rightarrow\infty$ with $\vert arg(z) \vert\leq \pi-\epsilon$ for some $\epsilon>0$.
Thus
\begin{align*}
P_{-\frac{1}{2}+i\rho}^{\mu}\left(\cosh(a)\right) = \left(i\rho\right)^{\mu}\left( \frac{1}{2} + i\rho \right)^{-\frac{1}{2}}&\sqrt{\frac{2}{\pi\sinh(a)}}\,\, \cdot \\
&\cdot \frac{1}{2}\left( e^{i\rho a}+ e^{-\pi i\left(\mu-\frac{1}{2}\right)-i\rho a} \right)\left( 1+ O\left( \frac{1}{\rho} \right) \right)
\end{align*}
as $\rho\rightarrow\infty$. Further, we observe that 
\begin{align*}
\left( \frac{1}{2} + i\rho \right)^{-\frac{1}{2}} &= \left( i\rho \right)^{-\frac{1}{2}}\left( 1+ \frac{1}{2 i \rho} \right)^{-\frac{1}{2}}  = i^{-\frac{1}{2}} \rho^{-\frac{1}{2}} \left( 1+O\left( \frac{1}{\rho} \right) \right)
\end{align*}
as $\rho\rightarrow\infty$. Hence,
\begin{align*}
P_{-\frac{1}{2}+i\rho}^{\mu}\left(\cosh(a)\right) = \rho^{\mu-\frac{1}{2}}\sqrt{\frac{2}{\pi\sinh(a)}}\cdot \frac{i^{\mu-\frac{1}{2}}}{2}\left( e^{i\rho a}+ e^{-\pi i\left(\mu-\frac{1}{2}\right)-i\rho a} \right)\left( 1+ O\left( \frac{1}{\rho} \right) \right)
\end{align*}
as $\rho\rightarrow\infty$.
An easy calculation shows that 
\begin{align*}
\frac{i^{\mu-\frac{1}{2}}}{2}\left( e^{i\rho a}+ e^{-\pi i\left(\mu-\frac{1}{2}\right)-i\rho a} \right)& = \frac{1}{2}e^{i\frac{\pi}{2}(\mu-\frac{1}{2})}\left( e^{i\rho a}+ e^{-\pi i\left(\mu-\frac{1}{2}\right)-i\rho a} \right)\\
& =  \frac{1}{2}\left( e^{i\rho a + i\frac{\pi}{4}(2\mu-1)}+ e^{-i\frac{\pi}{4} \left(2\mu - 1\right)-i\rho a} \right) \\
& = \cos\left( a\rho+\frac{\pi}{4}\left( 2\mu-1 \right) \right),
\end{align*}
which closes the proof of equation \eqref{equation:LegendreAsymp2}.

Equation \eqref{equation:LegendreAsymp3} is proven in Lemma $2$ of \S 8 in \cite{Virchenko}, even though the statement of the lemma as well as the proof given there contain some typos. The statement given in \cite{Virchenko} is mostly correct, but the restriction $\vert \arg(\nu) \vert<\pi-\delta$ in the asymptotic estimate misses there.

\end{proof}

\begin{lemma}
\label{lemma:LegendreProdAsymp}
Let $\nu\in \mathbb{C}$ and $z,\omega\in (1,\infty)$ be given. Further, let $\tilde{z},\tilde{\omega}\in (0,\infty)$ be such that 
\begin{align}
\label{equation:TildeDefinitionen}
\cosh(\tilde{z})=\frac{z}{\left( z^2-1\right)^{\nicefrac{1}{2}}}, \quad \cosh(\tilde{\omega})=\frac{\omega}{\left( \omega^2-1\right)^{\nicefrac{1}{2}}}.
\end{align}
Then
\begin{itemize}
\item[$1)$]
\begin{align} 
e^{-i\pi\mu} P_{\nu}^{-\mu}\left( \omega \right)Q_{\nu}^{\mu}\left( z \right) = \frac{e^{- \tilde{\omega}\mu}}{2\mu } \left( e^{\mu\tilde{z}}+e^{i\pi(\nu+1)}e^{-\mu\tilde{z}} \right) \left( 1+O\left(\frac{1}{\vert \mu \vert}\right) \right) \label{equation:AsympLegendreProdKond1}
\end{align}
as $\vert \mu \vert\rightarrow\infty$ with $\Re(\mu)>-\frac{1}{2}$.
\item[$2)$] 
\begin{align}
\frac{\sin(i\pi\rho)}{\pi} Q_{\nu}^{-i\rho}\left( z \right)Q_{\nu}^{i\rho}\left( \omega \right) = &\frac{i}{ \rho} \cos\left(\tilde{z}\rho-\frac{\pi}{2}\left( \nu+1 \right)\right)\cdot \nonumber\\
&\cdot \cos\left(\tilde{\omega}\rho-\frac{\pi}{2}\left( \nu+1 \right)\right)\left( 1+O\left( \frac{1}{\rho} \right) \right) \label{equation:AsympLegendreProdKond2}
\end{align}
as $\rho\rightarrow\infty$ with $\rho\in \mathbb{R}$.
\end{itemize}
\end{lemma}

\begin{proof}
Since $\cosh^2(s)-\sinh^2(s)=1$ for all $s\in\mathbb{C}$, it follows that
\begin{align}
\label{equation:BeziehungTilde}
\frac{1}{\sqrt{\sinh(\tilde{z})}}=\left( z^2-1 \right)^{\nicefrac{1}{4}}, \quad \frac{1}{\sqrt{\sinh(\tilde{\omega})}}=\left( \omega^2-1 \right)^{\nicefrac{1}{4}}.
\end{align}
Let us consider situation $1)$ first.

When we use Whipple's formulas \eqref{equation:Whipple1}, \eqref{equation:Whipple2} and then \eqref{equation:ReflectionLegendreP}, we get
\begin{align*}
 e^{-i\pi \mu}P_{\nu}^{-\mu}\left( \omega \right) & Q_{\nu}^{\mu}\left( z \right) \\
&=\frac{\Gamma\left( \mu+\nu+1 \right)}{\Gamma\left( \mu-\nu \right)}\frac{ie^{i\nu\pi}}{\left( z^2-1 \right)^{\nicefrac{1}{4}}\left( \omega^2-1 \right)^{\nicefrac{1}{4}}} P_{\mu-\frac{1}{2}}^{-\nu-\frac{1}{2}}\left( \cosh(\tilde{z}) \right) Q_{\mu-\frac{1}{2}}^{-\nu-\frac{1}{2}}\left( \cosh(\tilde{\omega}) \right).
\end{align*}
We can now apply the asymptotic estimates \eqref{equation:LegendreAsymp1}, \eqref{equation:LegendreAsymp3} to the right-hand side of the above equation. When we do so and then use the relation \eqref{equation:BeziehungTilde}, we obtain for $\vert \mu \vert\rightarrow\infty$ with $\Re(\mu)>-\frac{1}{2}$:
\begin{align}
e^{-i\pi \mu}P_{\nu}^{-\mu}\left( \omega \right) Q_{\nu}^{\mu}\left( z \right) &=\frac{\Gamma\left( \mu+\frac{1}{2} \right)}{\Gamma\left( \mu-\nu\right)} \frac{e^{-\tilde{\omega}\mu} \left( \mu-\frac{1}{2} \right)^{-\nu-1}}{2\sqrt{\mu+\frac{1}{2}} } \left( e^{\mu\tilde{z}}+e^{i\pi(\nu+1)}e^{-\mu\tilde{z}} \right) \left( 1+O\left(\frac{1}{\vert \mu \vert}\right) \right) \nonumber \\
&=\frac{\Gamma\left( \mu+\frac{1}{2} \right)}{\Gamma\left( \mu-\nu\right)} \frac{e^{-\tilde{\omega}\mu} \mu^{-\nu-\frac{3}{2}}}{2} \left( e^{\mu\tilde{z}}+e^{i\pi(\nu+1)}e^{-\mu\tilde{z}} \right) \left( 1+O\left(\frac{1}{\vert \mu \vert}\right) \right). \label{equation:AsymptProd9}
\end{align}
Applying the asymptotic estimate \eqref{equation:GammaQuotAsymp} for the quotient of two gamma functions, we get
\begin{align}
\frac{\Gamma\left( \mu+\frac{1}{2} \right)}{\Gamma\left( \mu-\nu\right)} = \mu^{\nu+\frac{1}{2}}\left( 1+O\left( \frac{1}{\vert \mu \vert} \right) \right) \label{equation:GammaQuotSpez}
\end{align}
as $\mu\rightarrow\infty$ with $\Re(\mu)>-\frac{1}{2}$. Lastly, when we combine \eqref{equation:GammaQuotSpez} and \eqref{equation:AsymptProd9}, we obtain
\begin{align*}
e^{-i\pi \mu}P_{\nu}^{-\mu}\left( \omega \right)Q_{\nu}^{\mu}\left( z \right) = \frac{e^{-\tilde{\omega}\mu}}{2 \mu } \left( e^{\mu\tilde{z}}+e^{i\pi(\nu+1)}e^{-\mu\tilde{z}} \right) \left( 1+O\left(\frac{1}{\vert \mu  \vert}\right) \right),
\end{align*}
as $\mu\rightarrow\infty$ with $\Re(\mu)>-\frac{1}{2}$. This shows $1)$. 

For $2)$, we first use Whipple's formula \eqref{equation:Whipple1} twice and \eqref{equation:ReflectionLegendreP} once, to obtain
\begin{align}
Q_{\nu}^{-i\rho}(z)Q_{\nu}^{i\rho}(\omega) =\, &\frac{\pi}{2}\Gamma\left( i\rho+\nu+1 \right)\Gamma\left( -i\rho+\nu+1 \right)\frac{1}{\left(z^2-1\right)^{\nicefrac{1}{4}}\left(\omega^2-1\right)^{\nicefrac{1}{4}}}\cdot  \nonumber \\
& \cdot P_{-\frac{1}{2}+i\rho}^{-\nu-\frac{1}{2}}\left( \cosh(\tilde{z})  \right)P_{-\frac{1}{2}+i\rho}^{-\nu-\frac{1}{2}}\left( \cosh(\tilde{\omega}) \right). \label{equation:ConditionsSchritt1.1}
\end{align}
It is known (see e.g. \cite{Lebedev} on p. 15) that for any $\varepsilon \in (0,\pi)$ and $\alpha\in\mathbb{C}$,
\begin{align}
\label{equation:AsymptotikGamma}
\Gamma\left( z+\alpha \right) = e^{\left( z+\alpha-\frac{1}{2} \right)\log(z)-z+\frac{1}{2}\log(2\pi)}\left(1+O\left(\frac{1}{\vert z \vert}\right)\right),
\end{align}
as $\vert z \vert\rightarrow\infty$ with $\vert \arg(z) \vert<\pi-\varepsilon$.

Hence we obtain that, for $\rho\rightarrow\infty$ with $\rho\in \mathbb{R}$,
\begin{align*}
\Gamma\left(\pm i\rho+\nu+1\right) &= e^{\left( \pm i\rho+\nu+\frac{1}{2} \right)\log(\pm i\rho)}e^{\mp i\rho}\cdot \sqrt{2\pi}\left(1+O\left(\frac{1}{ \rho }\right)\right),
\end{align*}
and therefore
\begin{align}
\label{equation:ProduktGamma}
\Gamma\left( i\rho+\nu+1\right) \Gamma\left( -i\rho+\nu+1\right) = \rho^{2\nu+1} e^{-\pi\rho}\cdot 2\pi \left(1+O\left(\frac{1}{ \rho }\right)\right).
\end{align}

On the other hand, when we use \eqref{equation:LegendreAsymp2} and \eqref{equation:BeziehungTilde}, we get
\begin{align}
&\frac{\pi}{2}\frac{1}{\left(z^2-1\right)^{\nicefrac{1}{4}}\left(\omega^2-1\right)^{\nicefrac{1}{4}}}\cdot 
P_{-\frac{1}{2}+i\rho}^{-\nu-\frac{1}{2}}\left( \cosh(\tilde{z}) \right)P_{-\frac{1}{2}+i\rho}^{-\nu-\frac{1}{2}}\left( \cosh(\tilde{\omega})\right) \nonumber \\
&= \rho^{-2\nu-2}\cos\left( \tilde{z}\rho-\frac{\pi}{2}\left( \nu+1 \right) \right)\cos\left( \tilde{\omega}\rho-\frac{\pi}{2}\left( \nu+1 \right) \right)\left( 1+O\left( \frac{1}{\rho} \right) \right) \label{equation:ConditionsBeweis2}
\end{align}
as $\rho\rightarrow\infty$ with $\rho \in \mathbb{R}$. From \eqref{equation:ConditionsBeweis2}, \eqref{equation:ProduktGamma} and \eqref{equation:ConditionsSchritt1.1} we obtain
\begin{align*}
Q_{\nu}^{-i\rho}(z)Q_{\nu}^{i\rho}(\omega) = \frac{2\pi}{\rho} e^{-\pi \rho} \cos\left( \tilde{z}\rho-\frac{\pi}{2}\left( \nu+1 \right) \right)\cos\left( \tilde{\omega}\rho-\frac{\pi}{2}\left( \nu+1 \right) \right)\left( 1+O\left( \frac{1}{\rho} \right) \right)
\end{align*}
as $\rho\rightarrow\infty$ with $\rho\in \mathbb{R}$. Hence, as $\rho\rightarrow\infty$ in $\mathbb{R}$:
\begin{align*}
&\frac{\sin(\pi i \rho)}{\pi} Q_{\nu}^{-i\rho}(z)Q_{\nu}^{i\rho}(\omega)   \\
&= \left( 1-e^{-2\pi\rho} \right)\frac{i}{\rho} \cos\left( \tilde{z}\rho-\frac{\pi}{2}\left( \nu+1 \right) \right)\cos\left( \tilde{\omega}\rho-\frac{\pi}{2}\left( \nu+1 \right) \right)\left( 1+O\left( \frac{1}{\rho} \right) \right)  \\
&=\frac{i}{\rho} \cos\left( \tilde{z}\rho-\frac{\pi}{2}\left( \nu+1 \right) \right)\cos\left( \tilde{\omega}\rho-\frac{\pi}{2}\left( \nu+1 \right) \right)\left( 1+O\left( \frac{1}{\rho} \right) \right). 
\end{align*}
\end{proof}

In the last part of this section we want to discuss a class of integrals involving associated Legendre functions of the second kind. We start with a rather technical lemma.
\begin{lemma}
\label{lemma:Schritt1}
Let $\nu\in\mathbb{C}$ with $\Re(\nu)>-1$ and $z,\omega\in\mathbb{C}\backslash\left(-\infty, 1 \right]$. Suppose further that we are given the following situation:
\begin{itemize}
\item[$(i)$] Let $g$ be a meromorphic function on $\mathbb{C}$ and let $S\subset \mathbb{C}$ denote the set of all poles of $g$. Suppose that $g$ is complex differentiable at all points on the imaginary axis, possibly except for a simple pole at the origin, and that $g$ is an odd function, i.e. $g(-s)=-g(s)$ for all $s\in\mathbb{C}\backslash S$.
\item[$(ii)$] Fix some $\varepsilon\in (0,1)$ such that $0<2\varepsilon<  1+\Re(\nu) $ and $B_{2\varepsilon}(0)\cap S\backslash\lbrace{ 0\rbrace}=\emptyset$, where $B_{2\epsilon}(0)$ denotes the open disc of radius $2\varepsilon$ centered at the origin. 
\item[$(iii)$] Let $N:=\left(N_k \right)_{k\in\mathbb{N}}\subset \left[1,\infty \right)$ be an unbounded and monotonically increasing sequence.
\end{itemize}  

Then there exists a function $\varphi:\mathbb{N}\rightarrow \mathbb{C}$ with $\lim_{k\rightarrow\infty}\varphi(k)=0$ and such that for any $k\in\mathbb{N}$ the following holds:
If $C_{k}$ denotes any injective and piecewise differentiable curve from $-iN_k$ to $iN_k$ such that all other points on the curve lie in the domain $\lbrace z\in\mathbb{C}: \Re(z)>0 \rbrace\backslash S$ \emph{(}see Fig. $\ref{Skizze:LemmaSchritt1}$\emph{)}, then we have:
\begin{align}
\label{equation:Schritt1}
2\left( \sum\limits_{p\in S\left( C_{k} \right)} \emph{Res}\left( g;p \right) e^{-i\pi p} P_{\nu}^{-p}(\omega)Q_{\nu}^{p}(z) \right)& + \emph{Res}\left( g;0 \right) P_{\nu}( \omega )Q_{\nu}(z) \nonumber \\
=\frac{2}{\pi}\int\limits_{\frac{\varepsilon}{N_k}}^{N_k} \frac{\sin\left(\pi i\rho \right)}{\pi}g(i\rho)Q_{\nu}^{-i\rho}(z)  Q_{\nu}^{i\rho}(\omega ) d\rho &+ \frac{1}{\pi i}\int\limits_{C_{k}}g(s)e^{-i\pi s}P_{\nu}^{-s}( \omega )  Q_{\nu}^{s}(z) ds  \nonumber \\
&  + \varphi(k),
\end{align}
where $S\left( C_{k} \right)$ denotes the set of all poles of $g$ contained in the bounded domain which is enclosed by the imaginary axis and the curve $C_{k}$.
\end{lemma}

\begin{figure} [ht] 
 \centering
\begin{tikzpicture}[decoration={markings,							
mark=at position 1cm with {\arrow[line width=1pt]{>}},
mark=at position 3.5cm with {\arrow[line width=1pt]{>}},
mark=at position 5.5cm with {\arrow[line width=1pt]{>}},
mark=at position 7.85cm with {\arrow[line width=1pt]{>}},
mark=at position 9cm with {\arrow[line width=1pt]{>}},
mark=at position 10.4cm with {\arrow[line width=1pt]{>}},
mark=at position 12.3cm with {\arrow[line width=1pt]{>}}
}
]
\draw[help lines,->] (-3,0) -- (3,0) coordinate (xaxis);
\draw[help lines,->] (0,-3) -- (0,3) coordinate (yaxis);

\path[draw,line width=0.8pt,postaction=decorate] (0,-2) node[left] {$-iN_k$} ..controls (2,-1.8) and (2.9,-1.4) .. (2.2,-0.9) ..controls (0.7,0) and (1.2,0.5) ..(1.4,0.6)..controls (2.3, 1) and (2.7, 1.5)..(0,2) node[left] {$iN_k$} ;


\draw[dashed] (-1.4,3) -- (-1.4, 0); 
\draw[dashed] (-1.4,-0.5) -- (-1.4, -3);

\draw[dashed] (0,0) circle (0.7); 


\draw (0,0) node[circle,fill,inner sep=1.5pt]{};
\draw (0.5,0.9) node[circle,fill,inner sep=1.5pt]{};
\draw (-0.5,-0.9) node[circle,fill,inner sep=1.5pt]{};
\draw (1.8,1) node[circle,fill,inner sep=1.5pt]{};
\draw (-1.8,-1) node[circle,fill,inner sep=1.5pt]{};
\draw (0.9,-1.6) node[circle,fill,inner sep=1.5pt]{};
\draw (-0.9,1.6) node[circle,fill,inner sep=1.5pt]{};
\draw (2.2,0) node[circle,fill,inner sep=1.5pt]{};
\draw (-2.2,0) node[circle,fill,inner sep=1.5pt]{};

\node[below] at (xaxis) {$\Re(z)$};
\node[left] at (yaxis) {$\Im(z)$};
\node at (2.1,0.5) {$C_{k}$};
\node[left] at (0,0.3) {\scriptsize{$2\varepsilon$}};
\node[below] at (-1.4,0) {\scriptsize $-1-\Re(\nu)$};
\end{tikzpicture}
\caption{Situation of Lemma \ref{lemma:Schritt1}} \label{Skizze:LemmaSchritt1}
\end{figure}
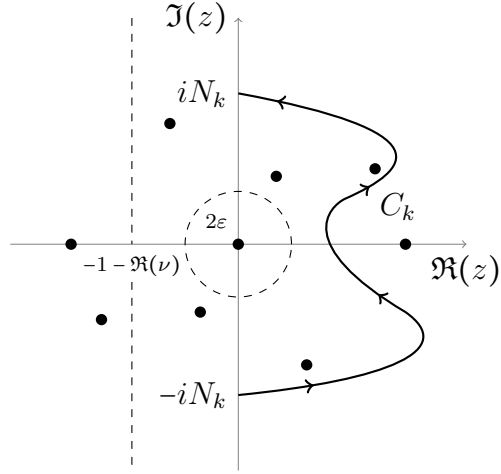

\begin{proof}
Let $k\in\mathbb{N}$ and $\omega,z\in \mathbb{C}\backslash (-\infty,1]$ be arbitrary. First we close up the curve $C_k$ in two different ways and obtain two piecewise smooth curves $\gamma_{k}^1$ and $\gamma_{k}^2$, as shown in Fig. \ref{Skizze:Gamma12}. That is, $\gamma_{k}^1$ joins the point $-iN_k$ with $iN_k$ along $C_k$, then goes from $iN_k$ to $i\frac{\varepsilon}{N_k}$ in a straight line, moves around the origin from $i\frac{\varepsilon}{N_k}$ to $-i\frac{\varepsilon}{N_k}$ on a semicircle centered at the origin and with positive orientation, and eventually joins $-i\frac{\varepsilon}{N_k}$ to $-iN_k$ by a straight line. The second curve is defined similarly as the first one, with only a single difference. The curve $\gamma_{k}^2$ again moves around the origin from $i\frac{\varepsilon}{N_k}$ to $-i\frac{\varepsilon}{N_k}$ on a semicircle centered at the origin, but now with the opposite orientation.

\begin{figure} [ht] 
  \subfloat[The closed path $\gamma_{k}^1$.]{
\begin{tikzpicture}[decoration={markings,							
mark=at position 1cm with {\arrow[line width=1pt]{>}},
mark=at position 3.5cm with {\arrow[line width=1pt]{>}},
mark=at position 5.5cm with {\arrow[line width=1pt]{>}},
mark=at position 7.85cm with {\arrow[line width=1pt]{>}},
mark=at position 9cm with {\arrow[line width=1pt]{>}},
mark=at position 10.4cm with {\arrow[line width=1pt]{>}},
mark=at position 12.3cm with {\arrow[line width=1pt]{>}}
}
]
\draw[help lines,->] (-3,0) -- (3,0) coordinate (xaxis);
\draw[help lines,->] (0,-3) -- (0,3) coordinate (yaxis);

\path[draw,line width=0.8pt,postaction=decorate] (0,-2)..controls (2,-1.8) and (2.9,-1.4) .. (2.2,-0.9) ..controls (0.7,0) and (1.2,0.5) ..(1.4,0.6)..controls (2.3, 1) and (2.7, 1.5)..(0,2) node[left] {$iN_k$} -- (0,0.4) node[right] {$i\frac{\varepsilon}{N_k}$} arc (90:270:0.4) -- (0,- 0.4) -- (0,-2) node[left] {$-iN_k$};


\draw[dashed] (-1.4,3) -- (-1.4, 0); 
\draw[dashed] (-1.4,-0.5) -- (-1.4, -3);


\draw (0,0) node[circle,fill,inner sep=1.5pt]{};
\draw (0.5,0.9) node[circle,fill,inner sep=1.5pt]{};
\draw (-0.5,-0.9) node[circle,fill,inner sep=1.5pt]{};
\draw (1.8,1) node[circle,fill,inner sep=1.5pt]{};
\draw (-1.8,-1) node[circle,fill,inner sep=1.5pt]{};
\draw (0.9,-1.6) node[circle,fill,inner sep=1.5pt]{};
\draw (-0.9,1.6) node[circle,fill,inner sep=1.5pt]{};
\draw (2.2,0) node[circle,fill,inner sep=1.5pt]{};
\draw (-2.2,0) node[circle,fill,inner sep=1.5pt]{};

\node[below] at (xaxis) {$\Re(z)$};
\node[left] at (yaxis) {$\Im(z)$};
\node at (2.1,0.5) {$C_{k}$};
\node[below] at (-1.4,0) {\scriptsize $- 1-\Re(\nu) $};
\end{tikzpicture}
}
\quad
  \subfloat[The closed path $\gamma_{k}^2$.]{
\begin{tikzpicture}[decoration={markings,							
mark=at position 1cm with {\arrow[line width=1pt]{>}},
mark=at position 3.5cm with {\arrow[line width=1pt]{>}},
mark=at position 5.5cm with {\arrow[line width=1pt]{>}},
mark=at position 7.85cm with {\arrow[line width=1pt]{>}},
mark=at position 9cm with {\arrow[line width=1pt]{>}},
mark=at position 10.4cm with {\arrow[line width=1pt]{>}},
mark=at position 12.3cm with {\arrow[line width=1pt]{>}}
}
]
\draw[help lines,->] (-3,0) -- (3,0) coordinate (xaxis);
\draw[help lines,->] (0,-3) -- (0,3) coordinate (yaxis);

\path[draw,line width=0.8pt,postaction=decorate] (0,-2)..controls (2,-1.8) and (2.9,-1.4) .. (2.2,-0.9) ..controls (0.7,0) and (1.2,0.5) ..(1.4,0.6)..controls (2.3, 1) and (2.7, 1.5)..(0,2) node[left] {$iN_k$} -- (0,0.4) node[left] {$i\frac{\varepsilon}{N_k}$} arc (90:-90:0.4) -- (0,- 0.4) -- (0,-2) node[left] {$-iN_k$};


\draw[dashed] (-1.4,3) -- (-1.4, 0); 
\draw[dashed] (-1.4,-0.5) -- (-1.4, -3);


\draw (0,0) node[circle,fill,inner sep=1.5pt]{};
\draw (0.5,0.9) node[circle,fill,inner sep=1.5pt]{};
\draw (-0.5,-0.9) node[circle,fill,inner sep=1.5pt]{};
\draw (1.8,1) node[circle,fill,inner sep=1.5pt]{};
\draw (-1.8,-1) node[circle,fill,inner sep=1.5pt]{};
\draw (0.9,-1.6) node[circle,fill,inner sep=1.5pt]{};
\draw (-0.9,1.6) node[circle,fill,inner sep=1.5pt]{};
\draw (2.2,0) node[circle,fill,inner sep=1.5pt]{};
\draw (-2.2,0) node[circle,fill,inner sep=1.5pt]{};

\node[below] at (xaxis) {$\Re(z)$};
\node[left] at (yaxis) {$\Im(z)$};
\node at (2.1,0.5) {$C_{k}$};
\node[below] at (-1.4,0) {\scriptsize $- 1-\Re(\nu)$};
\end{tikzpicture}

}
\caption{\label{Skizze:Gamma12}}
\end{figure}

By definition, both curves $\gamma_{k}^1$ and $\gamma_{k}^2$ can be written as a composition of the following curves:
\begin{align*}
\gamma_{k}^1 = C_k + \alpha_{k} + \beta_{k}^1 + \delta_k, \\
\gamma_{k}^2 = C_k + \alpha_{k} + \beta_{k}^2 + \delta_k,
\end{align*}
where
\begin{align*}
\alpha_k: \left[ -N_k,-\frac{\varepsilon}{N_k} \right]\rightarrow \mathbb{C},&\quad \rho\mapsto -\rho i; \\
\delta_k: \left[ \frac{\varepsilon}{N_k}, N_k \right]\rightarrow \mathbb{C},&\quad \rho\mapsto -\rho i; \\
\beta_{k}^1 : \left[ 0,\pi \right]\rightarrow \mathbb{C},&\quad \theta\mapsto \frac{\varepsilon}{N_k}ie^{i\theta}; \\
\beta_{k}^2 : \left[ 0,\pi \right]\rightarrow \mathbb{C}, &\quad \theta\mapsto \frac{\varepsilon}{N_k}ie^{-i\theta}.
\end{align*}
We observe that the function $u\mapsto P_{\nu}^{-u}(\omega )Q_{\nu}^{u} ( z ) $ is holomorphic on the half plane $\lbrace u\in\mathbb{C} : \Re(u)>-1-\Re\left( \nu \right) \rbrace$. Hence, we can rewrite the left-hand side of $\eqref{equation:Schritt1}$ by using first the residue theorem and then the above decomposition of $\gamma_{k}^1, \gamma_{k}^2$ as follows:
\begin{align}
\label{equation:Schritt1Bew1}
&\qquad 2\left( \sum\limits_{p \in S\left( C_k \right)} \Res\left( g;p \right) e^{-i\pi p} P_{\nu}^{-p}\left( \omega \right)Q_{\nu}^{p}\left( z \right) \right) + \Res\left( g;0 \right) P_{\nu}\left( \omega \right)Q_{\nu}\left( z \right) \nonumber \\
&=\frac{1}{2\pi i}\oint\limits_{\gamma_{k}^1} g(s) e^{-i\pi s} P_{\nu}^{-s}\left( \omega \right)Q_{\nu}^{s}\left( z \right) ds + \frac{1}{2\pi i}\oint\limits_{\gamma_{k}^2} g(s) e^{-i\pi s} P_{\nu}^{-s}\left( \omega \right)Q_{\nu}^{s}\left( z \right)  ds \nonumber \\
&= \frac{1}{\pi i}\int\limits_{\alpha_k} g(s) e^{-i\pi s} P_{\nu}^{-s}\left( \omega \right)Q_{\nu}^{s}\left( z \right) ds + \frac{1}{\pi i}\int\limits_{\delta_k} g(s) e^{-i\pi s} P_{\nu}^{-s}\left( \omega \right)Q_{\nu}^{s}\left( z \right)  ds \nonumber \\
&\quad +\frac{1}{\pi i}\int\limits_{C_k}g(s)e^{-i\pi s}P_{\nu}^{-s}\left( \omega \right)  Q_{\nu}^{s}\left( z \right) ds + \varphi(k),
\end{align}
with
\begin{align}
\label{equation:DefFehlerterm}
\varphi(k):= \frac{1}{2\pi i}\int\limits_{\beta_{k}^1} g(s) e^{-i\pi s} P_{\nu}^{-s}\left( \omega \right)Q_{\nu}^{s}\left( z \right) ds + \frac{1}{2\pi i}\int\limits_{\beta_{k}^2} g(s) e^{-i\pi s} P_{\nu}^{-s}\left( \omega \right)Q_{\nu}^{s}\left( z \right)  ds.
\end{align}
We observe that we can sum up the first two summands of \eqref{equation:Schritt1Bew1} as follows:
\begin{align*}
&\frac{1}{\pi i}\int\limits_{\alpha_k} g(s) e^{-i\pi s} P_{\nu}^{-s}\left( \omega \right)Q_{\nu}^{s}\left( z \right) ds + \frac{1}{\pi i}\int\limits_{\delta_k} g(s) e^{-i\pi s} P_{\nu}^{-s}\left( \omega \right)Q_{\nu}^{s}\left( z \right)  ds \\
= &-\frac{1}{\pi}\int\limits_{-N_k}^{-\frac{\varepsilon}{N_k}} g(-i\rho) e^{-i\pi \left(-i\rho\right)} P_{\nu}^{i\rho}\left( \omega \right)Q_{\nu}^{-i\rho}\left( z \right) d\rho - \frac{1}{\pi}\int\limits_{\frac{\varepsilon}{N_k}}^{N_k} g(-i\rho) e^{i\pi \left( i\rho \right)} P_{\nu}^{i\rho}\left( \omega \right)Q_{\nu}^{-i\rho}\left( z \right)  d\rho \\
= &-\frac{1}{\pi}\int\limits_{\frac{\varepsilon}{N_k}}^{N_k} g(i\rho) e^{-i\pi \left(i\rho\right)} P_{\nu}^{-i\rho}\left( \omega \right)Q_{\nu}^{i\rho}\left( z \right) d\rho - \frac{1}{\pi}\int\limits_{\frac{\varepsilon}{N_k}}^{N_k} g(-i\rho) e^{i\pi \left( i\rho \right)} P_{\nu}^{i\rho}\left( \omega \right)Q_{\nu}^{-i\rho}\left( z \right)  d\rho \\
= & -\frac{1}{\pi}\int\limits_{\frac{\varepsilon}{N_k}}^{N_k} g\left( i\rho \right)\left( e^{-i\pi\left(i\rho \right)}P_{\nu}^{-i\rho}\left( \omega \right)Q_{\nu}^{i\rho}\left( z \right) - e^{i\pi\left( i\rho \right)}P_{\nu}^{i\rho}\left( \omega \right)Q_{\nu}^{-i\rho}\left( z \right) \right)d\rho,
\end{align*}
where in the last equality we used that the function $g$ is odd. Now we can use \eqref{equation:MagischeFormel3} to simplify the  expression in paranthesis under the integral sign and obtain, in total:
\begin{align}
\label{equation:Schritt1Bew2}
& \frac{1}{\pi i}\int\limits_{\alpha_k} g(s) e^{-i\pi s} P_{\nu}^{-s}\left( \omega \right)Q_{\nu}^{s}\left( z \right) ds + \frac{1}{\pi i}\int\limits_{\delta_k} g(s) e^{-i\pi s} P_{\nu}^{-s}\left( \omega \right)Q_{\nu}^{s}\left( z \right)  ds \nonumber \\
=\, &\frac{2}{\pi}\int\limits_{\frac{\varepsilon}{N_k}}^{N_k} \frac{\sin\left( i\rho \pi \right)}{\pi}g\left( i\rho \right)Q_{\nu}^{-i\rho} \left( z \right) Q_{\nu}^{i\rho}\left( \omega \right) d\rho.
\end{align}
Thus the claimed equation \eqref{equation:Schritt1} follows from \eqref{equation:Schritt1Bew1} together with \eqref{equation:Schritt1Bew2}.

It remains to show that $\lim_{k\rightarrow\infty}\varphi(k)=0$, where $\varphi(k)$ is defined as in \eqref{equation:DefFehlerterm}. When we use the parametrisation of $\beta_{k}^{1}$ and $\beta_{k}^2$ given above, we obtain
\begin{align*}
\varphi\left(k \right) = \, &\frac{1}{2\pi}\int\limits_{0}^{\pi} g\left( i\frac{\varepsilon}{N_k}e^{i\theta} \right) e^{-i\pi \left(i\frac{\varepsilon}{N_k}e^{i\theta}\right)} \LegendreP{\nu}{-i\frac{\varepsilon}{N_k}e^{i\theta}}{\omega} \cdot \LegendreQ{\nu}{i\frac{\varepsilon}{N_k}e^{i\theta}}{z}\cdot \left( i\frac{\varepsilon }{N_k}e^{i\theta} \right) d\theta  \\
& -\frac{1}{2\pi} \int\limits_{0}^{\pi}  g\left( i\frac{\varepsilon}{N_k}e^{-i\theta} \right) e^{-i\pi \left(i\frac{\varepsilon}{N_k}e^{-i\theta}\right)} \LegendreP{\nu}{-i\frac{\varepsilon}{N_k}e^{-i\theta}}{\omega} \cdot \LegendreQ{\nu}{i\frac{\varepsilon}{N_k}e^{-i\theta}}{z}\cdot \left( i\frac{\varepsilon }{N_k}e^{-i\theta} \right)  d\theta.
\end{align*}
By assumption, the meromorphic function $g$ has either a simple pole at the origin or is complex differentiable there. Hence the limit of the function $s g(s)$ exists as $s\rightarrow 0$ and is given by $\lim_{s\rightarrow 0} sg(s) = \Res\left( g;0 \right)\in\mathbb{C}$. Thus the integrand $s\mapsto sg(s)e^{-i\pi s}P_{\nu}^{-s}(\omega)Q_{\nu}^{ s}(z)$ is continuous in $B_{\varepsilon}(0)$, in particular at the origin $s=0$, and is bounded on $B_{\varepsilon}(0)$. By Lebesgue's dominated convergence theorem and since $\lim_{k\rightarrow\infty} N_k=\infty$, we have
\begin{align*}
\lim\limits_{k\rightarrow \infty}\varphi(k) &= \frac{1}{2\pi}\int\limits_{0}^{\pi} \Res(g;0)P_{\nu}(\omega)Q_{\nu}(z) d\theta - \frac{1}{2\pi}\int\limits_{0}^{\pi} \Res(g;0)P_{\nu}(\omega)Q_{\nu}(z) d\theta  \\
&= 0.
\end{align*} 
\end{proof}

The following corollary follows immediately from Lemma \ref{lemma:Schritt1}.

\begin{corollary}
\label{corollary:Corollary1Schritt1}
Let $\nu\in\mathbb{C}$ with $\Re(\nu)>-1$ and $z,\omega\in\mathbb{C}\backslash\left( -\infty, 1 \right]$. Further, let $g$ be a meromorphic function with the same properties as in Lemma \emph{\ref{lemma:Schritt1}} and suppose:
\begin{enumerate}
\item[$1)$] There exists a sequence of curves $\left( C_k\right)_{k\in\mathbb{N}}$, each of them as in Lemma $\emph{\ref{lemma:Schritt1}}$, and such that
\begin{align}
\lim\limits_{k\rightarrow \infty}\int\limits_{C_k} g(s)e^{-i\pi s}P_{\nu}^{-s}(\omega)Q_{\nu}^{s}(z) ds = 0. \label{equation:Schritt1.1}
\end{align}
\item[$2)$] The following integral converges in $\mathbb{C}$:
\end{enumerate}
\begin{align}
\int\limits_{0}^{\infty} \frac{\sin(\pi i \rho)}{\pi} g\left(i\rho\right)Q_{\nu}^{-i\rho}(z)Q_{\nu}^{i\rho}(\omega) d\rho. \label{equation:Schritt1.2}
\end{align}
Then it follows that
\begin{align}
\label{equation:Corollary1Schritt1}
&\frac{2}{\pi}\int\limits_{0}^{\infty} \frac{\sin(\pi i \rho)}{\pi} g\left(i\rho\right)Q_{\nu}^{-i\rho}(z)Q_{\nu}^{i\rho}(\omega) d\rho \nonumber \\
= & \lim\limits_{k\rightarrow\infty} \left( 2 \sum\limits_{p\in S\left( C_k \right)} \emph{Res}\left( g;p \right) e^{-i\pi p} P_{\nu}^{-p}(\omega)Q_{\nu}^{p}(z) \right)  + \emph{Res}\left( g;0 \right) P_{\nu}( \omega )Q_{\nu}(z).
\end{align} 
\end{corollary}

It is an interesting problem to investigate for which functions $g$ as in Lemma \ref{lemma:Schritt1} the conditions $1)$ and $2)$ of Corollary \ref{corollary:Corollary1Schritt1} are satisfied. We will not pursue this problem in full generality. Instead, since it is sufficient for our purposes, we will restrict the parameters $\omega, z$ to real values. 

The following lemma establishes a class of functions $g$ for which condition $2)$ of Corollary \ref{corollary:Corollary1Schritt1} is satisfied.

\begin{lemma}
\label{lemma:ConditionsSchritt1}
Let $\nu\in \mathbb{C}$ with $\Re(\nu)\notin \{ -1, -2, -3,... \}$ \emph{(}e.g. $\Re(\nu)>-1$\emph{)} and let $z,\omega\in \left(1,\infty \right)$. Let $g$ be as in Lemma \emph{\ref{lemma:Schritt1}}.
\begin{itemize}
\item[$(i)$] Suppose $g(i\rho)=O\left(\frac{1}{\rho}\right)$, as $\rho\rightarrow\infty$ in $\mathbb{R}$. Then the integral \eqref{equation:Schritt1.2} is absolutely convergent, and hence convergent.

\item[$(ii)$] Suppose there exists some constant $C\in\mathbb{C}\backslash\{ 0 \}$ such that $g(i\rho)=C\left(1+O\left(\frac{1}{\rho}\right)\right)$, as $\rho\rightarrow\infty$ in $\mathbb{R}$. Then the integral \eqref{equation:Schritt1.2} is convergent for $\omega\neq z$, but divergent for $\omega = z$. Further, it is never absolutely convergent.
\end{itemize}
\end{lemma}

\begin{proof}
Let us abbreviate the integrand in \eqref{equation:Schritt1.2} by
\begin{align*}
f(\rho):=\frac{\sin(\pi i \rho)}{\pi} g\left(i\rho\right)Q_{\nu}^{-i\rho}(z)Q_{\nu}^{i\rho}(\omega),\quad \rho\in (0,\infty).
\end{align*}
The function $f(\rho)$ is continuous on $(0,\infty)$ and can be extended continuously to $[0,\infty)$ by the assumptions on $g$.  Therefore, we only need to investigate how the integrand behaves asymptotically as $\rho\rightarrow\infty$ in $\mathbb{R}$. We will see that both $(i)$ and $(ii)$ follow from \eqref{equation:AsympLegendreProdKond2} by elementary calculations.

First consider case $(i)$.  Note that the function 
\begin{align*}
\mathbb{R}\ni \rho \mapsto \cos\left( \tilde{z}\rho-\frac{\pi}{2}\left( \nu+1 \right) \right)\cos\left( \tilde{\omega}\rho-\frac{\pi}{2}\left( \nu+1 \right) \right)\in\mathbb{C}
\end{align*}
is bounded, where $\tilde{z}, \tilde{\omega}\in (0,\infty)$ are defined by \eqref{equation:TildeDefinitionen}. Thus it follows from \eqref{equation:AsympLegendreProdKond2} that
\begin{align*}
\left\vert f(\rho) \right\vert = O\left( \frac{1}{\rho^2} \right)
\end{align*}
as $\rho\rightarrow\infty$ in $\mathbb{R}$. Hence the integral \eqref{equation:Schritt1.2} is absolutely convergent.

Now consider case $(ii)$. In this case it follows from \eqref{equation:AsympLegendreProdKond2} that
\begin{align}
f(\rho) = \underbrace{C\frac{i}{\rho} \cos\left( \tilde{z}\rho-\frac{\pi}{2}\left( \nu+1 \right) \right)\cos\left( \tilde{\omega}\rho-\frac{\pi}{2}\left( \nu+1 \right) \right)}_{=:\psi(\rho)}\left( 1+O\left( \frac{1}{\rho} \right) \right) \label{equation:ConditionBeweis4}
\end{align}
as $\rho\rightarrow\infty$ in $\mathbb{R}$. Note that $z\neq\omega$ if and only if $\tilde{z}\neq\tilde{\omega}$, which follows directly from \eqref{equation:TildeDefinitionen}.

We claim that if $z\neq \omega$ and $K>0$, then the function $\psi$ is integrable over the interval $[K,\infty)$, where ``integrable'' is meant here in the sense that the integral converges. Furthermore, $\psi$ is not integrable over $[K,\infty)$ if $z=\omega$, and for any values of $z,\omega\in (1,\infty)$ it is not absolutely integrable over $[K,\infty)$. 

Before we prove this claim, let us discuss why statement $(ii)$ follows from this behaviour of $\psi $. It follows from \eqref{equation:ConditionBeweis4} that there exists a function $\phi:[\Lambda,\infty)\rightarrow\mathbb{R}$ with $\Lambda>0$ such that for some constant $D>0$ we have $\vert \phi(\rho) \vert\leq \frac{D}{\rho}$  and $f(\rho)=\psi(\rho) + \psi(\rho)\cdot \phi(\rho)$ for all $\rho\in [\Lambda,\infty)$. Hence there exists some constant $E>0$ such that $\vert \psi(\rho)\cdot \phi(\rho)\vert \leq \frac{E}{\rho^2}$ for all $\rho\in [\Lambda,\infty)$ and the function $\psi\cdot \phi$ is continuous because $f,\psi$ are continuous and $\psi \cdot \phi  = f  - \psi $. Thus $\psi \cdot \phi $ must be (absolutely) integrable over $[\Lambda,\infty)$. Therefore it follows from the equation $f=\psi+\psi\cdot \phi$ that $f$ is integrable over $[\Lambda,\infty)$ if and only if $\psi$ is integrable over $[\Lambda,\infty)$. This shows that if the above claim is true, then the function $f$ is integrable if $z\neq \omega$, but it is not integrable if $z=\omega$.

Further, if the above claim is true then the integral of $f$ is never absolutely convergent. This is easily seen, since it follows from \eqref{equation:ConditionBeweis4} that
\begin{align*}
\left\vert f(\rho) \right\vert = \left\vert \psi(\rho)\right\vert \left( 1+o(1) \right)
\end{align*}
as $\rho\rightarrow\infty$ in $\mathbb{R}$.
Hence for any $K>0$ the function $\vert f \vert$ is integrable over $[K,\infty)$ if and only if $\left\vert \psi\right\vert$ is integrable over $[K,\infty)$.

It remains to prove the above claim for $\psi$. Let $K>0$ and suppose $z\neq \omega$ and thus $\tilde{z}\neq \tilde{\omega}$. To keep the notation short, let us introduce the following abbreviations:
\begin{align*}
c:=-\frac{\pi}{2}\left( \nu + 1 \right); \quad q(\rho):=\cos\left( \tilde{z}\rho +c \right)\cos\left( \tilde{\omega}\rho +c \right).
\end{align*}
 One can easily show that the function $q$ has the following antiderivative:
\begin{align*}
Q(\rho) := \frac{1}{\tilde{z}^2-\tilde{\omega}^2}\left( \tilde{z}\cdot \sin\left( \tilde{z}\rho +c \right)\cos\left( \tilde{\omega}\rho +c  \right) - \tilde{\omega} \cdot \cos\left( \tilde{z}\rho +c \right)\sin\left( \tilde{\omega}\rho +c  \right)\right).
\end{align*}
This antiderivative is obviously bounded. Hence, using integration by parts,
\begin{align*}
\int\limits_{K}^{\infty} \psi(\rho) d\rho &= \lim\limits_{L\rightarrow\infty} \Bigg(  C \cdot i\left( \frac{Q(L)}{L} - \frac{Q(K)}{K} \right) + C \cdot i \int\limits_{K}^{L} \frac{Q(\rho)}{\rho^2} d\rho \Bigg) \\
&= - C \cdot i \left( \frac{Q(K)}{K}  - \int\limits_{K}^{\infty} \frac{Q(\rho)}{\rho^2} d\rho \right).
\end{align*}
Observe that the remaining integral on the right-hand side is even absolutely convergent. Thus the integral on the left-hand side must be convergent as well.

Suppose now $z=\omega$, and thus $\tilde{z}=\tilde{\omega}$. Again, one can easily show that for all $\tau\in\mathbb{C}$ the function $\mathbb{R}\ni\rho\mapsto \cos^2\left( \tilde{z}\rho +\tau \right)\in\mathbb{C}$ has the following antiderivative:
\begin{align}
\label{equation:AntiderivativeCosineSquare}
\rho\mapsto R(\rho;\tau):=\frac{\sin\left( \tilde{z}\rho +\tau  \right)\cos\left( \tilde{z}\rho +\tau  \right)}{2\tilde{z}} + \frac{\rho}{2}.
\end{align}
When we apply integration by parts, we obtain:
\begin{align*}
\int\limits_{K}^{\infty} \psi(\rho) d\rho = \lim\limits_{L\rightarrow\infty} \Bigg(  C \cdot i\left( \frac{R(L;c)}{L} - \frac{R(K;c)}{K} \right) +& \frac{C \cdot i}{2\tilde{z}}  \int\limits_{K}^{L} \frac{\sin\left( \tilde{z}\rho +c  \right)\cos\left( \tilde{z}\rho +c  \right)}{ \rho^2} d\rho \Bigg. \\
\Bigg. +& \frac{C\cdot i}{2}\ln\left( \frac{L}{K} \right) \Bigg).
\end{align*}
The right-hand side is not convergent as $L\rightarrow\infty$, because we have $\lim_{L\rightarrow\infty}\ln\left( \frac{L}{K} \right)=\infty$, and all the other terms on the right-hand side converge to some complex number for $L\rightarrow\infty$. Therefore, the left-hand side must be divergent as well. In other words, $\psi$ is not integrable if $z=\omega$.

It remains to show that $\psi$ is not absolutely integrable over $[K,\infty)$ for all values of $z,\omega\in (1,\infty)$. It is easy to show that $\vert \cos(s) \vert\geq \vert \cos(\Re(s)) \vert$ for all $s\in\mathbb{C}$. Thus we have 
\begin{align*}
\vert \psi(\rho) \vert \geq \frac{\vert C \vert}{\rho}\cdot \vert \cos\left( \tilde{z}\rho +\Re(c) \right)\vert \cdot \vert\cos \left( \tilde{\omega}\rho +\Re(c) \right) \vert,\quad \forall \rho\in(0,\infty).
\end{align*}
 If $z=\omega$ then we have for all $L>K$:
 \begin{align*}
\int\limits_{K}^{L} \vert \psi(\rho) \vert \, d\rho &\geq \vert C \vert  \int\limits_{K}^{L} \frac{1}{\rho}\cdot \cos^2\left( \tilde{z}\rho +\Re(c) \right) d\rho \\
&= \vert C \vert  \left( \frac{R(L;\Re(c))}{L} - \frac{R(K;\Re(c))}{K} \right) + \frac{ \vert C \vert }{2\tilde{z}}  \int\limits_{K}^{L} \frac{\sin\left( \tilde{z}\rho +c  \right)\cos\left( \tilde{z}\rho +c  \right)}{ \rho^2} d\rho \Bigg. \\
\Bigg. &\qquad\qquad\qquad\qquad\qquad\qquad\quad\quad\,\,  + \frac{\vert C \vert}{2}\ln\left( \frac{L}{K} \right) \Bigg),
 \end{align*}
where $\rho\mapsto R(\rho;\Re(c))$ is the function defined in \eqref{equation:AntiderivativeCosineSquare} for $\tau=\Re(c)$. The  right-hand side diverges to $\infty$ as $L\rightarrow\infty$. Thus $\psi$ is also not absolutely integrable.

Suppose now $z\neq\omega$. Then we have for all $L>K$:
\begin{align*}
\int\limits_{K}^{L} \vert \psi(\rho) \vert\,  d\rho \geq \vert C \vert\cdot \int\limits_{K}^{L} \frac{1}{\rho}\cdot \cos^2\left( \tilde{z}\rho +\Re(c) \right)\cos^2\left( \tilde{\omega}\rho + \Re(c) \right)  d\rho.
\end{align*}
The integral on the right-hand side diverges to $\infty$ as $L\rightarrow\infty$: An easy calculation yields that the function $\mathbb{R}\ni \rho\mapsto \cos^2\left( \tilde{z}\rho +\Re(c) \right)\cos^2\left( \tilde{\omega}\rho + \Re(c) \right)\in\mathbb{R}$ has an elementary antiderivative of the form
\begin{align*}
S(\rho):=\frac{\rho}{4} + u(\rho),
\end{align*}
where $u$ is a bounded function. An explicit formula for $u$ is given by
\begin{align*}
u(\rho):=\frac{1}{\tilde{\omega}^2-\tilde{z}^2}& \Bigg( \frac{\tilde{\omega}}{2}\cos^2\left( \tilde{z}\rho +\Re(c) \right)\cos\left( \tilde{\omega}\rho +\Re(c) \right)\sin\left( \tilde{\omega}\rho +\Re(c) \right) - \\
&\qquad - \frac{\tilde{z}}{2}\cos^2\left( \tilde{\omega}\rho+\Re(c) \right)\cos\left( \tilde{z}\rho +\Re(c) \right)\sin\left( \tilde{z}\rho +\Re(c) \right)  +\\
&\qquad +\frac{\tilde{\omega}^2}{4\tilde{z}}\cos\left( \tilde{z}\rho +\Re(c) \right)\sin\left( \tilde{z}\rho +\Re(c) \right) - \\
&\qquad  -\frac{\tilde{z}^2}{4\tilde{\omega}}\cos\left( \tilde{\omega}\rho +\Re(c) \right)\sin\left( \tilde{\omega}\rho +\Re(c) \right)  \Bigg).
\end{align*}
Thus, with integration by parts, we obtain
\begin{align*}
\lim\limits_{L\rightarrow\infty} \int\limits_{K}^{L} \frac{1}{\rho}\cdot \cos^2\left( \tilde{z}\rho +\Re(c) \right)&\cos^2\left( \tilde{\omega}\rho + \Re(c) \right)  d\rho = \\
&= \lim\limits_{L\rightarrow\infty} \left( \left( \frac{S(L)}{L} - \frac{S(K)}{K} \right) + \int\limits_{K}^{L} \frac{u(\rho)}{\rho^2} d\rho + \frac{1}{4}\ln\left( \frac{L}{K} \right)\right) = \infty.
\end{align*}
Consequently, it follows that $\lim_{L\rightarrow\infty} \int_{K}^{L} \vert \psi(\rho) \vert d\rho = \infty$ and thus our proof is complete.
\end{proof}

Now we want to consider functions $g$ for which condition $1)$ of Corollary \ref{corollary:Corollary1Schritt1} is satisfied. We consider only a special class of curves $C_k$, which is no problem because $1)$ is a condition of existence for the $C_k$.

First we observe that there exists always a sequence $\left( N_k \right)$ as in Lemma \ref{lemma:Schritt1} such that all curves $C_k$ can be chosen as semicircles centered at the origin: If the function $g$ is an entire function or has at most finitely many poles, then this is clear. Suppose now that $g$ has infinitely many poles. By definition, all meromorphic functions have only isolated poles, and hence the set of all poles is at most countable. Let $\lbrace z_i \rbrace_{i=1}^{\infty}$ be the set of all poles of $g$ and let $\lbrace \vert z_i \vert \rbrace_{i=1}^{\infty}$ be the set of all magnitudes of poles. Choose a sequence $\left( N_k \right)_{k\in\mathbb{N}}$ as in Lemma \ref{lemma:Schritt1} such that $ N_k \in [1,\infty)\backslash \lbrace \vert z_i \vert \rbrace_{i=1}^{\infty}$ for all $k\in\mathbb{N}$, and choose $C_k$ as semicircles centered at the origin with radius $N_k$.

\begin{lemma}
Let $g$ be a function as in Lemma $\ref{lemma:Schritt1}$. Let $\nu \in \mathbb{C}$ with $\Re(\nu)>-1$ and  $\omega, z\in (1,\infty)$ with $\omega<z$.
Suppose that $(C_k)_{k\in\mathbb{N}}$, is a sequence of curves as in Lemma $\ref{lemma:Schritt1}$, each of which is a semicircle centered at the origin. Suppose further that $g$ is bounded on the union of the images of all $C_k$. Then $\lim_{k\rightarrow\infty}\int_{C_k} g(s) e^{-i\pi s}P_{\nu}^{-s}(\omega)Q_{\nu}^{s}(z) ds = 0$. In particular, condition $1)$ of Corollary $\ref{corollary:Corollary1Schritt1}$ is satisfied.
\end{lemma}

\begin{proof}
Let $\tilde{z}, \tilde{\omega}\in (0,\infty)$ be associated with $z, \omega$ as in \eqref{equation:TildeDefinitionen}. Observe that the function 
\begin{align*}
x\mapsto \frac{x}{\left( x^2-1 \right)^{\nicefrac{1}{2}}},\quad x>1,
\end{align*}
is monotonically decreasing on $ \left( 1,\infty \right)$; hence $\tilde{z}<\tilde{\omega}$.

Let $\left( N_k \right)_{k=1}^{\infty}\subset \left[1,\infty \right)$ be the unbounded and monotonically increasing sequence such that $C_k$ runs from $-iN_k$ to $iN_k$. Let $C_k$ be parametrised as
\begin{align*}
C_k:\left[ -\frac{\pi}{2}, \frac{\pi}{2} \right]\rightarrow \mathbb{C}, \quad \theta\mapsto N_k\cdot e^{i\theta}.
\end{align*}
By assumption, there exists some $A>0$ such that $\vert g(s)\vert\leq A$ for all $s\in \cup_{k=1}^{\infty} C_k([-\frac{\pi}{2}, \frac{\pi}{2}])$. Further, by \eqref{equation:AsympLegendreProdKond1} there exists some $K>0$ such that for all $s \in\mathbb{C}$ with $\Re(s)\geq 0$ and  $\vert s \vert>K$:
\begin{align}
\left\vert e^{-i\pi s}P_{\nu}^{-s}(\omega)Q_{\nu}^{s}(z) \right\vert &\leq \frac{1}{\vert s \vert} \left( \vert e^{-s(\tilde{\omega}-\tilde{z})} \vert + e^{-\pi \Im(\nu)}\cdot \vert  e^{-s\left( \tilde{z}+ \tilde{\omega} \right)} \vert \right) \leq \frac{1+e^{-\pi\Im(\nu)}}{\vert s \vert} \cdot  e^{-\Re(s)(\tilde{\omega}-\tilde{z})}.
\end{align}
Therefore there exist some $q\in\mathbb{N}$ such that for all $k\geq q$:
\begin{align*}
\left\vert \int\limits_{C_k} g(s) e^{-i\pi s}P_{\nu}^{-s}(\omega)Q_{\nu}^{s}(z) ds \right\vert &\leq \left( 1+e^{-\pi\Im(\nu)} \right)A \int\limits_{-\frac{\pi}{2}}^{\frac{\pi}{2}} e^{-N_k\cos(\theta) \left( \tilde{\omega}- \tilde{z} \right)} d\theta \\
&= \left( 1+e^{-\pi\Im(\nu)} \right)2A \int\limits_{0}^{\frac{\pi}{2}} e^{-N_k\cos(\theta) \left( \tilde{\omega}- \tilde{z} \right)} d\theta.
\end{align*}
To finish the proof of this corollary, we will show that the last integral converges to $0$ as $k\rightarrow\infty$. 

Fix some $\delta \in (0,\frac{\pi}{2})$. Then 
\begin{align*}
\int\limits_{0}^{\frac{\pi}{2}} e^{-N_k\cos(\theta) \left( \tilde{\omega}- \tilde{z} \right)} d\theta =  \underbrace{\int\limits_{0}^{\frac{\pi}{2}-\delta} e^{-N_k\cos(\theta) \left( \tilde{\omega}- \tilde{z} \right)} d\theta}_{=:I_1} +  \underbrace{\int\limits_{\frac{\pi}{2}-\delta}^{\frac{\pi}{2}} e^{-N_k\cos(\theta) \left( \tilde{\omega}- \tilde{z} \right)} d\theta}_{=:I_2}.
\end{align*}
The first term converges obviously to $0$, since (recall $\tilde{z}<\tilde{\omega}$)
\begin{align*}
I_1\leq \left( \frac{\pi}{2}-\delta \right) e^{-N_k\cos(\frac{\pi}{2}-\delta) \left( \tilde{\omega}- \tilde{z} \right)}\longrightarrow 0\quad \text{ as } k\rightarrow\infty.
\end{align*}
The second integral converges to $0$ as well. In fact, with a simple substitution we get
\begin{align*}
I_2= \int\limits_{0}^{\delta}e^{-N_k\cos(-\theta+\frac{\pi}{2}) \left( \tilde{\omega}- \tilde{z} \right)} d\theta = \int\limits_{0}^{\delta}e^{-N_k\sin(\theta) \left( \tilde{\omega}- \tilde{z} \right)} d\theta.
\end{align*}
Now observe that $\sin(\theta)$ is concave on $\left[ 0,\frac{\pi}{2} \right]$, and thus $\sin(\theta)\geq \frac{2}{\pi}\theta$ for all $\theta\in \left[ 0,\frac{\pi}{2} \right]$. In particular $e^{-N_k\sin(\theta) \left( \tilde{\omega}- \tilde{z} \right)}\leq e^{-N_k \frac{2}{\pi}\theta \left( \tilde{\omega}- \tilde{z} \right)}$. Hence
\begin{align*}
I_2 \leq \int\limits_{0}^{\delta} e^{-N_k \frac{2}{\pi}\theta \left( \tilde{\omega}- \tilde{z} \right)} d\theta = \frac{\pi}{2 \left( \tilde{\omega}- \tilde{z} \right) N_{k}} \left( 1- e^{-N_k \frac{2}{\pi}\delta \left( \tilde{\omega}- \tilde{z} \right)} \right)\longrightarrow 0 \quad \text{ as } k\rightarrow \infty.
\end{align*}
\end{proof}

To get an overview of the preceding discussion, we summarise the relevant parts in a single theorem.

\begin{theorem}
\label{theorem:TheoremZentralIntegral}
Let $z,\omega\in (1,\infty)$ with $\omega<z$. Let $\nu\in \mathbb{C}$ with $\Re(\nu)>-1$. Suppose that $g$ is an odd and meromorphic function on $\mathbb{C}$, which is complex differentiable at all points on the imaginary axis except possibly with a simple pole at the origin. Suppose further:
\begin{itemize}
\item[$(i)$] Either $g(i\rho) = O\left( \frac{1}{\rho}\right)$ as $\rho \rightarrow \infty$ in $\mathbb{R}$ \emph{(}in which case we will say ``$g$ is of type $I"$\emph{)}, or there exists some constant  $C\in\mathbb{C}\backslash\{ 0 \}$ with $g(i\rho) = C \left( 1+ O\left( \frac{1}{\rho}\right) \right)$ as $\rho \rightarrow \infty$ in $\mathbb{R}$ \emph{(}``$g$ is of type $II"$\emph{)}.
\item[$(ii)$] There exists an unbounded and monotonically increasing sequence $\left( N_{k} \right)_{k\in\mathbb{N}}\subset [1,\infty)$ such that $g$ is bounded on $\cup_{k=1}^{\infty} S^{1}(N_k)$, where  $ S^{1}(N_k)$ denotes the circle of radius $N_k$ centered at the origin.
\end{itemize}
Then
\begin{align}
\label{equation:TheoremZentralIntegral}
&\frac{2}{\pi}\int\limits_{0}^{\infty} \frac{\sin(\pi i \rho)}{\pi} g\left(i\rho\right)Q_{\nu}^{-i\rho}(z)Q_{\nu}^{i\rho}(\omega) d\rho \nonumber \\
= & \lim\limits_{k\rightarrow\infty} \left( 2 \sum\limits_{p\in B^{+}\left( N_k \right)} \emph{Res}\left( g;p \right) e^{-i\pi p} P_{\nu}^{-p}(\omega)Q_{\nu}^{p}(z) \right)  + \emph{Res}\left( g;0 \right) P_{\nu}( \omega )Q_{\nu}(z),
\end{align} 
where $B^{+}\left( N_k \right):=\lbrace z\in\mathbb{C} : \vert z \vert < N_k,\text{ }\Re(z)>0 \rbrace$.

Further, we know that the integral on the left-hand side of \eqref{equation:TheoremZentralIntegral} is absolutely convergent for all $z,\omega\in (1,\infty)$ and $\nu\in\mathbb{C}$ with $\Re(\nu)\notin \{ -1, -2, -3,... \}$, if $g$ is of type $I$. If, instead, $g$ is of type $II$, then this integral is convergent if and only if $z\neq \omega$, and it is never absolutely convergent.
%
\end{theorem}

With this theorem, we can solve many integrals involving associated Legendre functions. 

\begin{example}
\label{example:LegendreIntegral}
Let $\nu\in \mathbb{C}$ with $\Re(\nu)>-1$ and $\omega,z\in (1,\infty)$ with $\omega<z$. The following examples of functions $g$ satisfy $\vert g(s) \vert = O\left( \frac{1}{\vert s \vert} \right)$ as $\vert s \vert\rightarrow\infty$ with $s\in\mathbb{C}$. In particular, these $g$ will satisfy the above conditions $(i), (ii)$ \emph{(}more precisely, they are of type $I$\emph{)} so we can apply Theorem $\ref{theorem:TheoremZentralIntegral}$ to solve the corresponding integrals. In order to simplify the integrands, we use the identity $\sin(is)=i\sinh(s)$ for $s\in\mathbb{C}$.
\begin{itemize}
\item[$(a)$] If $g(s)=\frac{1}{s}$, then:
\begin{align}
\label{equation:ErstesWichtigesIntegral}
\frac{2}{\pi}\int\limits_{0}^{\infty} \frac{\sinh(\pi \rho)}{\pi} \cdot \frac{1}{ \rho} Q_{\nu}^{-i\rho}(z)Q_{\nu}^{i\rho}(\omega) d\rho = P_{\nu}( \omega )Q_{\nu}(z).
\end{align}
\label{example:LegendreIntegral1}
\item[$(b)$] If $g(s)=\frac{s}{(s-1)(s+1)}=\frac{s}{s^2-1}$, then:
\begin{align*}
\frac{2}{\pi}\int\limits_{0}^{\infty} \frac{\sinh(\pi \rho)}{\pi} \cdot \frac{ \rho}{\rho^2+1} Q_{\nu}^{-i\rho}(z)Q_{\nu}^{i\rho}(\omega) d\rho = -P_{\nu}^{-1}(\omega)Q_{\nu}^{1}(z).
\end{align*}
\item[$(c)$] We can generalise the above examples. Let $\mathcal{H}_{>0}=\lbrace s\in\mathbb{C} \mid \Re(s)>0 \rbrace$. For any $n\in\mathbb{N}_{0}$ and pairwise different $a_1,...,a_n\in\mathcal{H}_{>0}$, we define the following polynomials:
\begin{align*}
p\left(s \mid a_1,...,a_n\right):= \prod\limits_{i=1}^{n}\left( s-a_i \right)\left( s + a_i \right) = \prod\limits_{i=1}^{n}\left( s^2-a_i^2 \right).
\end{align*}
As usual we define the empty product as $1$, e.g. 
\begin{align*}
\prod_{\substack{i=1\\i\neq 1}}^{1}(...):=1 ,\quad \prod_{i=1}^{0}(...):=1.
\end{align*}
 Thus, if $n=0$, then $p\left(s \mid a_1,...,a_n\right):=1.$

Now let $m,n\in\mathbb{N}_{0}$ be given. We fix pairwise different $b_1,...,b_m\in\mathcal{H}_{>0}$ and pairwise different  $a_1,...,a_n\in\mathcal{H}_{>0}$. Then the following functions satisfy the conditions of Theorem $\ref{theorem:TheoremZentralIntegral}$:
\begin{align*}
g_1(s):=\frac{1}{s}\cdot \frac{p\left( s\mid a_1,...,a_n \right)}{p\left( s\mid b_1,...,b_m \right)}\quad \text{if }\, m\geq n; \quad 
g_2(s):= s \cdot \frac{p\left( s\mid a_1,...,a_n \right)}{p\left( s\mid b_1,...,b_m \right)},\quad \text{if }\, m > n.
\end{align*}
For $g_1$ we obtain the integrals
\begin{align*}
\frac{2}{\pi}\int\limits_{0}^{\infty}& \frac{\sinh(\pi  \rho)}{\pi} \cdot \frac{(-1)^{n+m}}{ \rho}\cdot \left( \frac{\prod\limits_{i=1}^{n}\left( \rho^2+a_i^2\right)}{\prod\limits_{i=1}^{m}\left(\rho^2+b_i^2\right)}\right) Q_{\nu}^{-i\rho}(z)Q_{\nu}^{i\rho}(\omega) d\rho = \\
& \left(\mathlarger{\sum\limits_{k=1}^{m}} \left(\frac{\prod\limits_{i=1}^{n}\left( b_k^2-a_i^2 \right)}{\prod\limits_{\substack{i=1\\i\neq k}}^{m}\left(b_k^2-b_i^2\right) }\right) \frac{e^{-i\pi b_k}}{b_k} P_{\nu}^{-b_k}( \omega )Q_{\nu}^{b_k}(z)\right) + (-1)^{n+m} \left( \frac{\prod\limits_{i=1}^{n}a_i^2}{\prod\limits_{i=1}^{m}b_i^2 }\right)P_{\nu}( \omega )Q_{\nu}(z).
\end{align*}
For $g_2$ we obtain
\begin{align*}
\frac{2}{\pi}\int\limits_{0}^{\infty}& \frac{\sinh(\pi  \rho)}{\pi} (-1)^{n+m} \cdot \rho\cdot \left( \frac{\prod\limits_{i=1}^{n}\left( \rho^2+a_i^2\right)}{\prod\limits_{i=1}^{m}\left(\rho^2+b_i^2\right)}\right) Q_{\nu}^{-i\rho}(z)Q_{\nu}^{i\rho}(\omega) d\rho = \\
& -\left(\mathlarger{\sum\limits_{k=1}^{m}} \left(\frac{\prod\limits_{i=1}^{n}\left( b_k^2-a_i^2 \right)}{\prod\limits_{\substack{i=1\\i\neq k}}^{m}\left(b_k^2-b_i^2\right) }\right) e^{-i\pi b_k}P_{\nu}^{-b_k}( \omega )Q_{\nu}^{b_k}(z)\right).
\end{align*}
\item[$(d)$] Let $c\in\mathcal{H}_{>0}$, $n\in\mathbb{N}$, $m\in\mathbb{N}_{0}$ such that $2m+1<4n,$ and set $g(s):=\frac{s^{2m+1}}{(s-c)^{2n}(s+c)^{2n}}$. The only poles of $g$ are of order $2n$ at the points $c$ and $-c$. In order to compute the residue of $g$ at $c$, note that
\begin{align*}
\Res(g;c)&= \frac{1}{(2n-1)!}\cdot \frac{d^{2n-1}}{ds^{2n-1}}_{\Big| s=c}  \left\{ \frac{s^{2m+1}}{(s+c)^{2n}} \right\},\\
\Res(g;-c)&= \frac{1}{(2n-1)!}\cdot \frac{d^{2n-1}}{ds^{2n-1}}_{\Big| s=-c} \left\{ \frac{s^{2m+1}}{(s-c)^{2n}} \right\},
\end{align*}
and thus we have $\Res(g;c) = \Res(g;-c)$. Consider the curve $S_R:[0,2\pi]\ni \theta \mapsto R\cdot e^{i\theta}\in\mathbb{C}$ with $R>0$. Then we have by the residue theorem:
\begin{align*}
2\cdot \Res(g;c) &=   \frac{1}{2\pi i} \lim\limits_{R\rightarrow\infty} \int\limits_{S_R} g(s) ds \\
&=  \frac{1}{2\pi} \lim\limits_{R\rightarrow\infty} \int\limits_{0}^{2\pi} \frac{R^{2m+1}}{R^{4n-1}} \cdot \frac{ e^{i (2m+2)\theta}}{\left( e^{i\theta} - \frac{c}{R}  \right)^{2n}\left( e^{i\theta} + \frac{c}{R} \right)^{2n}} d\theta = 
\begin{cases}
0, \text{ if } 2m+1 < 4n-1 \\
1, \text{ if } 2m+1 = 4n-1.
\end{cases}
\end{align*}
Thus we have

\begin{align*}
\frac{2}{\pi}\int\limits_{0}^{\infty} \frac{\sinh(\pi \rho)}{\pi} \cdot \frac{(-1)^{m+1}\rho^{2m+1}}{\left( \rho^2+c^2 \right)^{2n}} &Q_{\nu}^{-i\rho}(z)Q_{\nu}^{i\rho}(\omega) d\rho = \nonumber \\
&= \begin{cases}
0, &\text{if } 2m+1 < 4n-1 \\
e^{-i\pi c}P_{\nu}^{-c}( \omega )Q_{\nu}^{c}(z), &\text{if } 2m+1 = 4n-1.
\end{cases}
\end{align*}
\end{itemize}
\end{example}

%
\begin{remark}
As we will show in the next chapter (see Lemma \ref{lemma:LemmaInverseLaplaceLegendreQ} $(ii)$), equation \eqref{equation:ErstesWichtigesIntegral} is also valid if $\omega = z$ and $\nu\in (-1,\infty)$. That case will be very useful as we shall see.
\end{remark}

We want to stress a special case of the previous theorem by the following corollary. 
\begin{corollary}
\label{corollary:Corollary1Schritt2}
Let $\nu\in \mathbb{C}$ with $\Re(\nu)>-1$ and $\omega,z\in (1,\infty)$ with $\omega<z$. Let $f:\mathbb{C}\rightarrow\mathbb{C}$ be an entire and even function, i.e. $f(s)=f(-s)$ for all $s\in\mathbb{C}$. 
Consider the function 
\begin{align*}
g(s):=\frac{\pi}{\sin(\pi s)}\cdot f(s),
\end{align*}
and the sequence $(N_k)_{k\in\mathbb{N}}$ with $N_k:=k+\frac{1}{2}$. 
Suppose that the conditions $(i)$ and $(ii)$ of Theorem $\ref{theorem:TheoremZentralIntegral}$ are satisfied with this particular $g$ and sequence $\left(N_k\right)_{k\in\mathbb{N}}$. Then we have
\begin{align}
\label{eqiaton:Corollary2Schritt1}
\frac{2}{\pi}\int\limits_{0}^{\infty} f\left(i\rho\right)Q_{\nu}^{-i\rho}(z)Q_{\nu}^{i\rho}(\omega) d\rho = 2 \left( \sum\limits_{m=1}^{\infty} f(m) P_{\nu}^{-m}(\omega)Q_{\nu}^{m}(z) \right)  + f(0) P_{\nu}( \omega )Q_{\nu}(z).
\end{align} 
\end{corollary}

\begin{proof}
Observe that $s\mapsto \frac{\pi}{\sin(\pi s)}$ is an odd function and meromorphic on $\mathbb{C}$. The only singularities are simple poles at all integer points $p\in\mathbb{Z}$ with residue $\Res\left( s\mapsto \frac{\pi}{\sin(\pi s)};p \right)=\left(-1\right)^{p}$. Thus $g$ satisfies the assumptions of Theorem \ref{theorem:TheoremZentralIntegral} with $\Res\left( g;p \right)=\left( -1 \right)^p f(p)$. Therefore \eqref{eqiaton:Corollary2Schritt1} follows when we apply \eqref{equation:TheoremZentralIntegral} to this case.
\end{proof}

Let us apply Corollary \ref{corollary:Corollary1Schritt2} to solve an important integral, which will reappear in the next chapter.

\begin{example}
\label{example:IntegralFuerGreenSumme}
Let $\nu\in \mathbb{C}$ with $\Re(\nu)>-1$ and $\omega,z\in (1,\infty)$ with $\omega<z$. Further, let $f(s):=\cos\left( s\cdot \left( \pi-\theta \right) \right)$, with $\theta\in\left[ 0, 2\pi \right)$. We set $g(s):=\frac{\pi}{\sin(\pi s)}\cdot f(s)$ and $N_k:=k+\frac{1}{2}$ as in Corollary \emph{\ref{corollary:Corollary1Schritt2}}. We first check if the conditions $(i)$ and $(ii)$ of Theorem \emph{\ref{theorem:TheoremZentralIntegral}} are satisfied. 

Note that for $\rho\in\mathbb{R}$ we have
\begin{align*}
g(i\rho) = \frac{\pi}{i} \cdot \frac{\cosh((\pi-\theta)\rho)}{\sinh(\pi\rho)} = \frac{\pi}{i} \cdot \frac{e^{-\theta \rho }+ e^{-(2\pi-\theta)\rho}}{1-e^{-2\pi\rho}}.
\end{align*}
If $\theta\in (0,2\pi)$ then this is in $O(e^{-\alpha\rho})$ as $\rho\rightarrow\infty$ with $\alpha:=\min\lbrace \theta, 2\pi-\theta \rbrace>0$. In particular, it is in $O\left( \frac{1}{\rho} \right)$ as $\rho\rightarrow\infty$, so $g$ is of type I. If $\theta = 0$, then 
\begin{align*}
g(i\rho) = \frac{\pi}{i}\cdot \frac{1 + e^{-2\pi\rho}}{1-e^{-2\pi\rho}} = \frac{\pi}{i} \cdot \left( 1+ \frac{2e^{-2\pi \rho}}{1-e^{-2\pi\rho}} \right).
\end{align*}
In particular, $g$ is of type II \emph{(}with $C:=\frac{\pi}{i}$ in Theorem \emph{\ref{theorem:TheoremZentralIntegral}} $(i)$\emph{)}. Thus, condition $(i)$ of Theorem \emph{\ref{theorem:TheoremZentralIntegral}} is satisfied.

Note that $g$ is bounded on the set $\cup_{k=1}^{\infty} S^1(N_k)$ if and only if $g^2$ is bounded on this set. Furthermore, if $x:=\Re(s),\, y:=\Im(s)$ so that $s=x+iy$, then a straightforward calculation shows:
\begin{align*}
\vert g(s) \vert^2 =\pi^2 \cdot \frac{\vert \cos((\pi-\theta)(x+iy))\vert^2}{\vert \sin(\pi(x+iy))\vert^2} = \pi^2 \cdot \frac{\cosh(2(\pi-\theta)y) + \cos(2(\pi-\theta)x)}{\cosh(2 \pi y) - \cos(2 \pi x)}.
\end{align*}
Let $y_0>0$ be such that 
\begin{align*}
\left\lbrace \, s\in S^1(N_k) \, \Big| \, \vert \Im(s) \vert\leq y_0 \,\right\rbrace \subset \left\lbrace  s\in S^1(N_k) \, \Big| \, \vert \Re(s) \vert \in \left[ k+\frac{1}{3},k+\frac{1}{2}  \right]  \right\rbrace
\end{align*}
for all $k\in\mathbb{N}$ \emph{(}or equivalently, for $k=1$\emph{)}. Then on $\cup_{k=1}^{\infty} \{\, s\in S^1(N_k) \mid \vert \Im(s) \vert\leq y_0  \,\}$ we have $\cos(2\pi x)\in  \left[ -1, -\frac{1}{2}\right]$ and thus
\begin{align*}
\vert g(s) \vert^2 \leq 2\pi^2 \cdot \frac{\cosh(2(\pi-\theta)y)}{\cosh(2\pi y)}\leq 2\pi^2.
\end{align*}
On $\{\, s\in\mathbb{C} \mid \vert \Im(s) \vert\geq y_0 \,\}$ we have
\begin{align*}
\vert g(s) \vert^2 \leq \pi^2 \cdot \frac{\cosh(2(\pi-\theta)y) + 1}{\cosh(2 \pi y) - 1}.
\end{align*}
Since $y\mapsto \pi^2 \cdot \frac{\cosh(2(\pi-\theta)y) + 1}{\cosh(2 \pi y) - 1}$ has a limit for $y\rightarrow\infty$ it is bounded on $[y_0,\infty)$. We conclude that $g^2$ is bounded on $\{\, s\in\mathbb{C} \mid \vert \Im(s) \vert\geq y_0 \,\}$, too. In particular, condition $(ii)$ of Theorem \emph{\ref{theorem:TheoremZentralIntegral}} is satisfied by the function $g$ with the above $N_k$. Therefore, by Corollary \emph{\ref{corollary:Corollary1Schritt2}} we get:
\begin{align*}
\frac{2}{\pi}\int\limits_{0}^{\infty} \cosh\left(\rho \left( \pi-\theta  \right) \right)Q_{\nu}^{-i\rho}(z) & Q_{\nu}^{i\rho}(\omega) d\rho   \\
&= 2 \left( \sum\limits_{m=1}^{\infty} \cos\left(m \left( \pi-\theta  \right) \right) P_{\nu}^{-m}(\omega)Q_{\nu}^{m}(z) \right) +  P_{\nu}( \omega )Q_{\nu}(z)  \\
&= 2 \left( \sum\limits_{m=1}^{\infty} (-1)^m\cos\left(m \theta  \right) P_{\nu}^{-m}(\omega)Q_{\nu}^{m}(z) \right) +  P_{\nu}( \omega )Q_{\nu}(z)  \\
&=Q_{\nu}\left( \omega z -\sqrt{\omega^2-1} \sqrt{z^2-1}\cos\left( \theta \right) \right),
\end{align*}
where the last equality is a classical addition formula \emph{(}see e.g. \emph{\citep[$8.795\text{ } 2$]{Gradshteyn}}\emph{)}.
\end{example}

In the above example we needed to assume $\omega<z$ in order to apply Corollary \ref{corollary:Corollary1Schritt2}, but it turns out that the above formula is also valid for more general values of $\omega$ and $z$. Before we generalise that formula, we will prove some estimates in the following lemma, which will also be useful later.

\begin{lemma}
\label{lemma:EstimateProductLegendreFunctions}
Let $\omega, z\in (1,\infty)$ and $\rho\in [0,\infty)$ be arbitrary. 
\begin{itemize}
\item[$(i)$]
For all $\nu\in\mathbb{C}$ with $\Re(\nu)\geq -\frac{1}{2} $:
\begin{align}
\label{equation:EstimateProductLegendreFunctions}
\vert  Q_{\nu}^{-i\rho}(z)Q_{\nu}^{i\rho}(\omega)\vert \leq \frac{\pi^4}{\sqrt{z-1} \cdot \sqrt{\omega-1}} \cdot \frac{\left( \vert \nu \vert + 1 +\rho\right)^{2\vert \nu \vert +1}}{\vert \Gamma(\nu+1)\vert^2}\cdot e^{-\pi\rho}.
\end{align}
\item[$(ii)$] For all $\nu\in\mathbb{C}$ with $\Re(\nu)\geq 0 $:
\begin{align}
\label{equation:EstimateProductLegendreFunctionsForSmallValuesOfZ}
\vert  Q_{\nu}^{-i\rho}(z)Q_{\nu}^{i\rho}(\omega)\vert \leq \ln\left( \frac{z+1}{z-1} \right)\ln\left( \frac{\omega+1}{\omega-1} \right) \cdot \frac{\pi^2 \left( \vert \nu \vert + 1 +\rho\right)^{2\vert \nu \vert +1}}{\vert \Gamma(\nu+1)\vert^2}\cdot e^{-\pi\rho}.
\end{align}
\item[$(iii)$] For all $\nu\in\mathbb{C}$ with $-1<\Re(\nu)<-\frac{1}{2}$:
\begin{align}
\label{equation:EstimateProductLegendreFunctionsGreaterRange}
\vert  Q_{\nu}^{-i\rho}(z)Q_{\nu}^{i\rho}(\omega)\vert \leq \frac{\pi^3 e^{\frac{1}{3(\Re(\nu)+1)}} }{\vert \Gamma(\nu+1)\vert^2 \cdot ((z-1)(\omega-1))^{\Re(\nu)+1} (\Re(\nu)+1)^{1-2\Re(\nu)} }\, e^{-\pi\rho}.
\end{align}
\end{itemize}
\end{lemma}

\begin{proof}
Let $\omega, z\in (1,\infty)$, $\rho\in [0,\infty)$, and $\nu\in\mathbb{C}$ with $\Re(\nu)> -1 $ be given. From \citep[formula $(5)$ on p. $155$]{Erdelyi} we know that for all $\mu \in\mathbb{C}$ with $\Re(\mu)\geq 0$:
\begin{align}
\label{equation:RepresentationLegendreQIntegral}
 Q_{\nu}^{\mu}(z) = e^{\mu\pi i}\cdot \frac{1}{2^{\nu+1}}\cdot \frac{\Gamma(\nu + 1 +\mu)}{\Gamma(\nu + 1)} \cdot \frac{1}{(z^2-1)^{\frac{\mu}{2}}} \cdot \int\limits_{0}^{\pi} \frac{\sin(t)^{2\nu+1}}{(z+\cos(t))^{-\mu+\nu+1}}\, dt.
\end{align}
Hence,
\begin{align}
\label{equation:AbschaetzungProduktLegendre1}
\vert Q_{\nu}^{-i\rho}(z)  Q_{\nu}^{i\rho}(\omega) \vert \leq & \frac{1}{4^{\Re(\nu)+1}}\cdot \frac{\vert \Gamma(\nu + 1 + i\rho) \Gamma(\nu + 1 - i\rho)\vert}{\vert \Gamma(\nu + 1)\vert^2 }  \cdot \int\limits_{0}^{\pi} \frac{\sin(t)^{2\Re(\nu)+1}}{(z+\cos(t))^{\Re(\nu)+1}}\, dt \, \cdot   \nonumber \\
& \cdot \int\limits_{0}^{\pi} \frac{\sin(t)^{2\Re(\nu)+1}}{(\omega+\cos(t))^{\Re(\nu)+1}}\, dt.
\end{align}

$(i)$. Suppose $\Re(\nu)\geq -\frac{1}{2}$. Then we can rewrite the integrals appearing above as follows:
\begin{align}
\int\limits_{0}^{\pi} \frac{\sin(t)^{2\Re(\nu)+1}}{(z+\cos(t))^{\Re(\nu)+1}}\, dt &= \int\limits_{0}^{\pi} \left(\frac{1-\cos^2(t)}{z+\cos(t)}\right)^{\Re(\nu)+\frac{1}{2}} \cdot \frac{1}{\sqrt{z+\cos(t)}}\, dt \nonumber \\
&=\int\limits_{0}^{\pi}\left( \frac{1+\cos(t)}{z+\cos(t)}\right)^{\Re(\nu)+\frac{1}{2}}  \cdot \frac{ \left( 1-\cos(t) \right)^{\Re(\nu)+\frac{1}{2}}}{\sqrt{z+\cos(t)}}\, dt. \label{equation:AbschaetzungProduktLegendreZwischenschritt}
\end{align}
Using the estimates $0\leq \frac{1+\cos(t)}{z+\cos(t)}<1$, $0\leq 1-\cos(t)\leq 2$, and $0< \frac{1}{\sqrt{z+\cos(t)}}\leq\frac{1}{\sqrt{z-1}}$ for all $t\in [0,\pi]$, we obtain from \eqref{equation:AbschaetzungProduktLegendreZwischenschritt}:
\begin{align}
\label{equation:AbschaetzungProduktLegendre2}
\int\limits_{0}^{\pi} \frac{\sin(t)^{2\Re(\nu)+1}}{(z+\cos(t))^{\Re(\nu)+1}}\, dt \, \leq \, 2^{\Re(\nu)+\frac{1}{2}} \cdot \frac{\pi}{\sqrt{z-1}}.
\end{align}
Hence, from \eqref{equation:AbschaetzungProduktLegendre1} and \eqref{equation:AbschaetzungProduktLegendre2} we get
\begin{align}
\label{equation:EstimateProductLegendreFunctionsBeweisZwischenschritt}
\vert  Q_{\nu}^{-i\rho}(z)Q_{\nu}^{i\rho}(\omega)\vert \leq \frac{\pi^2}{\sqrt{z-1} \cdot \sqrt{\omega-1}} \cdot \frac{\vert \Gamma(\nu+1+i\rho)\Gamma(\nu+1-i\rho) \vert}{2\cdot \vert \Gamma(\nu+1)\vert^2}.
\end{align}

%


Moreover, it is well-known (see \citep[formula $5.6.9$ on p. 138]{NIST}) that for all $z\in \mathbb{C}$ with $\Re(z) > 0$:
\begin{align}
\label{equation:GammaFunctionEstimateGeneral}
\vert \Gamma(z) \vert \leq \sqrt{2\pi} \cdot \vert z \vert ^{\Re(z)-\frac{1}{2}} \cdot e^{-\frac{\pi}{2}\vert \Im(z) \vert} \cdot e^{\frac{1}{6 \vert z \vert}}.
\end{align} 
Thus, 
\begin{align}
\label{equation:GammaImaginaryEstimate}
\vert \Gamma(\nu+1+i\rho)\Gamma(\nu+1-i\rho) \vert &\leq  2\pi e^{\frac{2}{3}}\cdot \left( \vert \nu \vert + 1 + \rho \right)^{2\Re(\nu)+1}\cdot e^{-\frac{\pi}{2}(\vert \Im(\nu) + \rho \vert + \vert \Im(\nu) - \rho \vert)} \nonumber \\
&\leq  2\pi^2 \cdot \left( \vert \nu \vert + 1 + \rho \right)^{2\vert \nu \vert+1} \cdot e^{-\pi\rho}.
\end{align}
Hence, $(i)$ follows from \eqref{equation:GammaImaginaryEstimate} and \eqref{equation:EstimateProductLegendreFunctionsBeweisZwischenschritt}.

$(ii).$ Suppose $\Re(\nu)\geq 0$. Similar as above, we can estimate the integrals on the right-hand side of \eqref{equation:AbschaetzungProduktLegendre1} as follows:
\begin{align*}
\int\limits_{0}^{\pi} \frac{\sin(t)^{2\Re(\nu)+1}}{(z+\cos(t))^{\Re(\nu)+1}}\, dt \, & = \int\limits_{0}^{\pi} \left( \frac{1+\cos(t)}{z+\cos(t)}\right)^{\Re(\nu)} \cdot (1-\cos(t))^{\Re(\nu)}\cdot \frac{\sin(t)}{z+\cos(t)} \, dt \\
& \leq 2^{\Re(\nu)} \int\limits_{0}^{\pi} \frac{\sin(t)}{z+\cos(t)} \, dt = 2^{\Re(\nu)} \ln\left(\frac{z+1}{z-1}\right).
\end{align*}
Hence, $(ii)$ follows from the above estimate combined with \eqref{equation:AbschaetzungProduktLegendre1} and \eqref{equation:GammaImaginaryEstimate}.

$(iii)$. Suppose $-1<\Re(\nu)<-\frac{1}{2}$, and thus $-1<2\Re(\nu)+1<0$. In that case, we estimate the integrals on the right-hand side of \eqref{equation:AbschaetzungProduktLegendre1} as follows:
\begin{align*}
\int\limits_{0}^{\pi} \frac{\sin(t)^{2\Re(\nu)+1}}{(z+\cos(t))^{\Re(\nu)+1}}\, dt \, &\leq \frac{1}{(z-1)^{\Re(\nu)+1}}\int_{0}^{\pi}\sin(t)^{2\Re(\nu)+1} \, dt  \\
&\, =\frac{2}{(z-1)^{\Re(\nu)+1}}\int_{0}^{\frac{\pi}{2}}\sin(t)^{2\Re(\nu)+1} \, dt.
\end{align*}
Note that $[0,\frac{\pi}{2}]\ni t\mapsto \sin(t)\in\mathbb{R}$ is concave and thus $\sin(t)\geq \frac{2}{\pi}t$. Hence, from the above estimate, we obtain
\begin{align*}
\int\limits_{0}^{\pi} \frac{\sin(t)^{2\Re(\nu)+1}}{(z+\cos(t))^{\Re(\nu)+1}}\, dt \,&\,\leq  \frac{2}{(z-1)^{\Re(\nu)+1}}\int_{0}^{\frac{\pi}{2}} \left(\frac{2}{\pi}t\right)^{2\Re(\nu)+1} \, dt \\
&\, =  \frac{\pi}{2\cdot (z-1)^{\Re(\nu)+1} (\Re(\nu)+1)}.
\end{align*}
Using that upper bound and the estimate \eqref{equation:AbschaetzungProduktLegendre1}  we have 
\begin{align*}
\vert  Q_{\nu}^{-i\rho}(z)Q_{\nu}^{i\rho}(\omega)\vert \leq \frac{\pi^2}{2(z-1)^{\Re(\nu)+1} (\omega-1)^{\Re(\nu)+1} (\Re(\nu)+1)^2} \cdot \frac{\vert\Gamma(\nu+1+i\rho)\Gamma(\nu+1-i\rho) \vert }{\vert \Gamma(\nu+1)\vert^2 }.
\end{align*}
Furthermore, from \eqref{equation:GammaFunctionEstimateGeneral} we obtain
\begin{align*}
\vert \Gamma(\nu+1+i\rho)\Gamma(\nu+1-i\rho) \vert \leq (2\pi) e^{\frac{1}{3(\Re(\nu)+1)}} (\Re(\nu)+1)^{2\Re(\nu)+1} e^{-\pi\rho},
\end{align*}
and therefore
\begin{align*}
\vert  Q_{\nu}^{-i\rho}(z)Q_{\nu}^{i\rho}(\omega)\vert &\leq \frac{\pi^3 e^{\frac{1}{3(\Re(\nu)+1)}}}{\vert \Gamma(\nu+1)\vert^2 (z-1)^{\Re(\nu)+1} (\omega-1)^{\Re(\nu)+1} (\Re(\nu)+1)^{1-2\Re(\nu)} }\cdot e^{-\pi \rho}.
\end{align*}
\end{proof}

\begin{corollary}
\label{corollary:IntegralProduktLegendreFuerGreensFunction}
Let $\nu\in \mathbb{C}$ with $\Re(\nu)>-1$, $\theta\in [0,2\pi)$ and $\omega,z\in (1,\infty)$ such that $\omega\neq z$. Then
\begin{align}
\label{example:IntegralFuerGreen}
\frac{2}{\pi}\int\limits_{0}^{\infty} \cosh\left(\rho \left( \pi-\theta  \right) \right)Q_{\nu}^{-i\rho}(z)Q_{\nu}^{i\rho}(\omega) d\rho = Q_{\nu}\left( \omega z -\sqrt{\omega^2-1} \sqrt{z^2-1}\cos\left( \theta \right) \right).
\end{align}
Furthermore, the above equality holds for all $\theta\in (0,2\pi)$ and $\omega=z\in (1,\infty)$.
\end{corollary}

\begin{proof}
Suppose $\theta\in [0,2\pi)$ and $\omega,z\in (1,\infty)$ such that $\omega\neq z$. For $\omega<z$ the proof was given in Example \ref{example:IntegralFuerGreenSumme}. In addition to that, observe that both sides of \eqref{example:IntegralFuerGreen} are symmetric in the variables $\omega,z$. The right-hand side is obviously symmetric in those variables, and the left-hand side is symmetric since the term $Q_{\nu}^{-i\rho}(z)Q_{\nu}^{i\rho}(\omega)$ of the integrand is symmetric due to \eqref{equation:SymmetrieProduktLegendreQ}. Hence the equation holds also if $z<\omega$.

Now suppose that we have $\theta\in (0,2\pi)$ and $\omega = z\in (1,\infty)$. Observe that the function on the right-hand side of \eqref{example:IntegralFuerGreen} is continuous in $z\in (1,\infty)$ and thus for any sequence $\left( z_n \right)_{n\in\mathbb{N}}\subset (1,\infty)$ with $z_n\neq z$ for all $n\in\mathbb{N}$ and  $\lim_{n\rightarrow\infty} z_n = z$ we have:
\begin{align*}
 Q_{\nu}\left( \omega z -\sqrt{\omega^2-1}\sqrt{z^2-1}\cos\left( \theta \right) \right) &= \lim\limits_{n\rightarrow\infty} Q_{\nu}\left( \omega z_n -\sqrt{\omega^2-1} \sqrt{z_n^2-1}\cos\left( \theta \right) \right) \\
 &=\lim\limits_{n\rightarrow\infty} \frac{2}{\pi}\int\limits_{0}^{\infty} \cosh\left(\rho \left( \pi-\theta  \right) \right)Q_{\nu}^{-i\rho}(z_n)Q_{\nu}^{i\rho}(\omega) d\rho \\
&= \frac{2}{\pi}\int\limits_{0}^{\infty} \cosh\left(\rho \left( \pi-\theta  \right) \right)Q_{\nu}^{-i\rho}(z)Q_{\nu}^{i\rho}(\omega) d\rho,
\end{align*}
where we obtain the last equality by Lebesgue's dominated convergence theorem: The integrand is continuous with respect to $z\in (1,\infty)$, and the sequence of integrands is dominated by some integrable function due to Lemma \ref{lemma:EstimateProductLegendreFunctions} $(i)$ and $(iii)$. In fact, by Lemma \ref{lemma:EstimateProductLegendreFunctions} and the estimate $\cosh(\rho(\pi-\theta))\leq e^{\, \rho\vert \pi-\theta \vert}$ for all $\rho\geq 0$, there exist $C,\varepsilon>0$ such that for all $\rho\geq 0$ and $n\in\mathbb{N}$: $\vert \cosh\left(\rho \left( \pi-\theta  \right) \right)Q_{\nu}^{-i\rho}(z_n)Q_{\nu}^{i\rho}(\omega) \vert\leq C \cdot e^{-\rho\varepsilon}$. This upper bound is obviously integrable over $\rho\in[0,\infty)$.
\end{proof}

\begin{remark}
One can consider the above integrals as certain so-called generalised Mehler transforms (see e.g. \cite{Oberhettinger}). One only needs to apply Whipple's formula \eqref{equation:Whipple1} to the term $Q_{\nu}^{-i\rho}(z)$ on the left-hand side of \eqref{equation:TheoremZentralIntegral}. The corresponding ordinary Mehler transform is obtained if we additionally set $\nu = -\frac{1}{2}$. Note that if we apply Whipple's formula \eqref{equation:Whipple1} also to the term $Q_{\nu}^{i\rho}(\omega)$ on the left-hand side of \eqref{equation:TheoremZentralIntegral}, then some terms under the integral cancel each other out and the integrand becomes shorter. In particular, if $\nu = -\frac{1}{2}$, then the integrand can be simplified further due to the formula (see \citep[8.334 2]{Gradshteyn})
\begin{align*}
\Gamma\left(\frac{1}{2} + i\rho \right)\Gamma\left(\frac{1}{2} - i\rho \right) = \frac{\pi}{\cos(\pi i\rho)},\quad \rho\in \mathbb{R}.
\end{align*}
\end{remark}

\begin{lemma}
\label{lemma:HolomorphieIntegralLegendreProdukt}
Let $\omega, z\in (1,\infty)$. Further, let $h:[0,\infty)\rightarrow\infty$ be continuous such that there exists some $\varepsilon > 0$ with $h(\rho)=O(e^{(\pi-\varepsilon)\rho})$ as $\rho\rightarrow\infty$. Then the function
\begin{align*}
F:\mathcal{H}_{>-\frac{1}{2}} \ni \nu\mapsto \int\limits_{0}^{\infty} Q_{\nu}^{-i\rho}(z)Q_{\nu}^{i\rho}(\omega) h(\rho) d\rho \in\mathbb{C}
\end{align*}
is holomorphic on $\mathcal{H}_{>-\frac{1}{2}}=\{\, z\in\mathbb{C} \mid \Re(z)>-\frac{1}{2}\,  \}$.
\end{lemma}

\begin{proof}
Let $f(\nu,\rho):=Q_{\nu}^{-i\rho}(z)Q_{\nu}^{i\rho}(\omega) h(\rho)$ for all $\nu\in \mathcal{H}_{>-\frac{1}{2}}$ and $\rho \in [0,\infty)$. Because of \citep[Satz $5.8$ on p. 148]{Elstrodt} it suffices to show for all non-empty compact sets $K\subset \mathcal{H}_{>-\frac{1}{2}}$ the existence of some integrable function $g_K:[0,\infty)\rightarrow\mathbb{R}$ with $\sup_{\nu\in K}\vert f(\nu,\rho) \vert\leq g_K(\rho)$ for all $\rho\in [0,\infty)$. Note that the other conditions of \citep[Satz $5.8$ on p. 148]{Elstrodt} are obviously satisfied by $f$, namely $f(\nu,\cdot)$ is integrable for all $\nu\in \mathcal{H}_{>-\frac{1}{2}}$ because of Theorem \ref{theorem:TheoremZentralIntegral} and $f(\cdot, \rho)$ is holomorphic on $\mathcal{H}_{>-\frac{1}{2}}$ for all $\rho\geq 0$.

Let $K\subset \mathcal{H}_{>-\frac{1}{2}}$ be an arbitrary non-empty compact set. From \eqref{equation:EstimateProductLegendreFunctions} and the assumptions on $h$, there exist constants $\tilde{C}, \varepsilon>0$ such that for all $\nu \in K$ and $\rho\in[0,\infty)$:
\begin{align}
\label{equation:BeweisHolomorphie1}
\vert f(\nu,\rho) \vert \leq \tilde{C}\cdot \frac{(\vert\nu \vert + 1 +\rho)^{2\vert \nu \vert+1}}{\vert \Gamma(\nu+1)\vert^2} \cdot  e^{ -\varepsilon \rho}.
\end{align}

With $D:=\sup_{\nu\in K }\left\lbrace \frac{1}{\vert \Gamma(\nu+1)\vert^2} \right\rbrace$, $M:=\sup_{\nu\in K}\{ \vert \nu \vert \}$ we have for all $\nu\in K$ and $\rho\in[0,\infty)$:
\begin{align*}
\vert f(\nu,\rho) \vert \leq \tilde{C} D   \left( M + 1 + \rho \right)^{2M+1}\cdot e^{-\varepsilon \rho}\leq C\cdot e^{-\frac{\varepsilon}{2}\rho}=: g_K(\rho),
\end{align*}
where $C:=\sup_{\rho \geq 0} \{ \tilde{C}D(M+1+\rho)^{2M+1}\cdot e^{-\frac{\varepsilon}{2}\rho} \}\in (0,\infty)$. That function $g_K$ is obviously integrable over $[0,\infty)$ and satisfies $\sup_{\nu\in K}\vert f(\nu,\rho) \vert\leq g_K(\rho)$ for all $\rho\in [0,\infty)$. Hence, $F$ must be holomorphic.
\end{proof}

\begin{lemma}
\label{lemma:SecondEstimateProductLegendreFunctions}
Let $\nu\in (0,\infty)$ be fixed.
\begin{itemize}
\item[$(i)$] For any $b \in (0,\infty)$ there exist constants $C,D>0$ \emph{(}depending on $\nu$ and b\emph{)} such that for all $a\in (0,\infty)$ and $\mu>0$:
\begin{align}
\label{equation:SecondEstimateProductLegendreFunctions}
\vert P_{\nu}^{-\mu}(\cosh(a))Q_{\nu}^{\mu}(\cosh(b)) \vert \leq Ce^{\left(\nu+\frac{1}{2}\right)a} \cdot  \mu^{\nu+\frac{1}{2}} \cdot \left(D\cdot \sinh\left(\frac{a}{2}\right)\right)^{\mu}.
\end{align}
\item[$(ii)$] For all $a,b\in (0,\infty)$
\begin{align}
\label{equation:SecondEstimateProductLegendreFunctions2}
\vert P_{\nu}(\cosh(a))Q_{\nu}(\cosh(b)) \vert \leq e^{\nu(a-b)}\cdot e^{\frac{a}{2}} \ln\left(\frac{\cosh(b)+1}{\cosh(b)-1}\right).
\end{align}
\end{itemize}
\end{lemma}

\begin{proof}
From \citep[formula 8.715 1]{Gradshteyn} we know that for all $\mu \geq 0$ and $a>0$:
\begin{align*}
P_{\nu}^{-\mu}(\cosh(a)) = \frac{\sqrt{2}}{\sqrt{\pi}\sinh(a)^{\mu}\cdot \Gamma(\frac{1}{2} + \mu)} \int\limits_{0}^{a} \cosh\bigg(\left(\nu+\frac{1}{2}\right)x\bigg) \cdot (\cosh(a)-\cosh(x))^{\mu-\frac{1}{2}}\, dx,
\end{align*}
and therefore
\begin{align}
\label{equation:SecondEstimateProductLegendreFunctionsBeweis1}
\vert P_{\nu}^{-\mu}(\cosh(a)) \vert \leq \sqrt{\frac{2}{\pi}} \cdot \frac{e^{\left(\nu+\frac{1}{2}\right)a} (\cosh(a)-1)^{\mu}}{\sinh(a)^{\mu}\cdot \vert \Gamma(\frac{1}{2} + \mu)\vert} \int\limits_{0}^{a} \frac{1}{\sqrt{\cosh(a)-\cosh(x)}}\, dx. 
\end{align}
Since $\cosh(t)=2 \sinh^2\left(\frac{t}{2}\right) + 1$ for all $t\in\mathbb{R}$, the above integral can be estimated as follows:
\begin{align*}
\int\limits_{0}^{a} \frac{1}{\sqrt{\cosh(a)-\cosh(x)}}\, dx &= \frac{1}{\sqrt{2}}\int\limits_{0}^{a} \frac{1}{\sqrt{\sinh^2\left(\frac{a}{2}\right)-\sinh^2\left(\frac{x}{2}\right)}}\, dx \\
&\leq \frac{1}{\sqrt{2\sinh\left(\frac{a}{2}\right)}} \int\limits_{0}^{a} \frac{1}{\sqrt{\sinh\left(\frac{a}{2}\right)-\sinh\left(\frac{x}{2}\right)}}\, dx,
\end{align*}
and using $\sinh\left(\frac{a}{2}\right)-\sinh\left(\frac{x}{2}\right)\geq \frac{a}{2}-\frac{x}{2}$ for all $x\in (0,a)$ we obtain:
\begin{align}
\label{equation:SecondEstimateProductLegendreFunctionsBeweis2}
\int\limits_{0}^{a} \frac{1}{\sqrt{\cosh(a)-\cosh(x)}}\, dx \leq \frac{1}{\sqrt{\sinh\left(\frac{a}{2}\right)}} \int\limits_{0}^{a} \frac{1}{\sqrt{a-x}}\, dx = 2\sqrt{\frac{a}{\sinh\left(\frac{a}{2}\right)}}\leq 2  \sqrt{2}.
\end{align}
Because of \eqref{equation:SecondEstimateProductLegendreFunctionsBeweis1} and \eqref{equation:SecondEstimateProductLegendreFunctionsBeweis2} we have for all $\mu \geq 0$ and $a>0$:
\begin{align}
\label{equation:BeweisAbschaetzungLegendrePGrob}
\vert P_{\nu}^{-\mu}(\cosh(a)) \vert \leq \frac{3\cdot  e^{\left(\nu+\frac{1}{2}\right)a} }{ \vert \Gamma(\frac{1}{2} + \mu)\vert} \cdot \left(\frac{\cosh(a)-1}{\sinh(a)}\right)^{\mu}\leq \frac{3\cdot  e^{\left(\nu+\frac{1}{2}\right)a} }{ \vert \Gamma(\frac{1}{2} + \mu)\vert} \cdot \left(\sinh\left(\frac{a}{2}\right)\right)^{\mu},
\end{align}
where, for the last equality, we used $\frac{\cosh(a)-1}{\sinh(a)}=\frac{2\sinh^2(\frac{a}{2})}{\sinh(a)}<\sinh(\frac{a}{2})$.

 Further, from \eqref{equation:RepresentationLegendreQIntegral} (or see \citep[formula (5) on p. 155]{Erdelyi}), for all $\mu \geq 0$ and $b>0$:
\begin{align}
\label{equation:SecondEstimateProductLegendreFunctionsBeweis3}
\vert Q_{\nu}^{\mu}(\cosh(b))\vert = \frac{  \Gamma(\nu+1+\mu) }{\Gamma(\nu+1)2^{\nu+1}\sinh(b)^{\mu}}\int\limits_{0}^{\pi} \frac{\sin(t)^{2\nu+1}\cdot(\cosh(b)+\cos(t))^{\mu}}{(\cosh(b)+\cos(t))^{\nu+1}} \,dt.
\end{align}
$(i)$. We will estimate the integrand given above. Note that for any $z>1$ the function $(-1,1)\ni x\mapsto \frac{1-x^2}{z+x}\in\mathbb{R}$ has a global maximum at $x=-z+\sqrt{z^2-1}$, where the maximal value is $2z-2\sqrt{z^2-1}$. Hence, with $z:=\cosh(b)$ we have for all $t\in(0,\pi)$ 
\begin{align}
\label{equation:AbschaetzungUnterIntegralFein}
 \frac{\sin(t)^{2\nu+1}}{(\cosh(b)+\cos(t))^{\nu+1}} &= \left(\frac{1-\cos^2(t)}{\cosh(b)+\cos(t)}\right)^{\nu}\cdot \frac{\sin(t)}{\cosh(b)+\cos(t)} \leq \frac{2^{\nu}e^{-b\nu}\sin(t)}{(\cosh(b)-1)}.
\end{align}
Thus, we can estimate the integral on the right-hand side of \eqref{equation:SecondEstimateProductLegendreFunctionsBeweis3} as follows:
\begin{align*}
\int\limits_{0}^{\pi} \frac{\sin(t)^{2\nu+1}\cdot(\cosh(b)+\cos(t))^{\mu}}{(\cosh(b)+\cos(t))^{\nu+1}} \,dt \,&\leq \frac{2^{\nu}}{(\cosh(b)-1)}\int\limits_{0}^{\pi} \sin(t)(\cosh(b)+\cos(t))^{\mu} \,dt \\
&=\frac{2^{\nu}}{(\cosh(b)-1)}\cdot \frac{(\cosh(b)+1)^{\mu+1}-(\cosh(b)-1)^{\mu+1}}{\mu+1} \\
& \leq \frac{2^{\nu+1}  \cosh(b)^{\mu+1} 2^{\mu}}{(\cosh(b)-1)}.
\end{align*}
Using \eqref{equation:SecondEstimateProductLegendreFunctionsBeweis3} we have
\begin{align}
\label{equation:SecondEstimateProductLegendreFunctionsBeweis4}
\vert Q_{\nu}^{\mu}(\cosh(b))\vert \leq \frac{  \Gamma(\nu+1+\mu) }{\Gamma(\nu+1)}\cdot \frac{\cosh(b)}{\cosh(b)-1}\cdot \left(2\coth(b)\right)^{\mu}.
\end{align}

Let $\tilde{C}>0$ be some constant such that $\big\vert\frac{\Gamma(\nu+1+\mu)}{ \Gamma(\frac{1}{2} + \mu)}\big\vert \leq \tilde{C}\cdot \mu^{\nu+\frac{1}{2}}$ for all $\mu>0$, where such a constant exists because of \eqref{equation:GammaQuotAsymp}. If we set $C:=\tilde{C}\cdot \frac{3 \cosh(b)}{(\cosh(b)-1)\Gamma(\nu+1)}$ and $D:=2\coth(b)$, then the lemma follows from \eqref{equation:BeweisAbschaetzungLegendrePGrob} and \eqref{equation:SecondEstimateProductLegendreFunctionsBeweis4}.

$(ii)$. Similarly, for $\mu=0$ we obtain from equation \eqref{equation:SecondEstimateProductLegendreFunctionsBeweis3}:
\begin{align*}
\vert Q_{\nu}(\cosh(b))\vert &= \frac{1}{2^{\nu+1}}\int\limits_{0}^{\pi} \frac{\sin(t)^{2\nu+1}}{(\cosh(b)+\cos(t))^{\nu+1}} \,dt \\
&= \frac{1}{2^{\nu+1}}\int\limits_{0}^{\pi} \left( \frac{1-\cos^2(t)}{\cosh(b)+\cos(t)}\right)^{\nu} \frac{\sin(t)}{\cosh(b)+\cos(t)} \,dt \\
& \leq \frac{e^{-b \nu }}{2}\int\limits_{0}^{\pi} \frac{\sin(t)}{ \cosh(b)+\cos(t)} \,dt \\
& = \frac{e^{-b \nu }}{2}\ln\left(\frac{\cosh(b)+1}{\cosh(b)-1}\right).
\end{align*}
The statement follows from the above estimate together with \eqref{equation:BeweisAbschaetzungLegendrePGrob}.
\end{proof}

%% file: chapter03.tex
\chapter{Spectral invariants for polygons}
\label{chapter:polygons}

The main purpose of this chapter is to compute the heat invariants for any hyperbolic polygon. Let us start with some basic definitions to clarify the terminology. 

\begin{definition}
\label{definition:Polygon}
A \emph{hyperbolic polygon} is defined as a relatively compact domain in the hyperbolic plane $\mathbb{H}^2$ whose boundary is a union of finitely many geodesic segments.

Analogously, the terms \emph{Euclidean polygon} and \emph{spherical polygon} are defined as relatively compact domains  in the Euclidean plane $\mathbb{R}^2$ and the unit sphere $\mathbb{S}^2$, respectively, with a piecewise geodesic boundary.

\end{definition}

We will have more to say about polygons below, in particular about the definition of \emph{edges}, \emph{vertices} and \emph{angles} of a polygon. But for the moment it is sufficient to think of all these notions on an intuitive level as we are used to. The following subsets of the hyperbolic plane will play an important role for our purposes.


\begin{definition}
\label{definition:Wedge}
A domain $W\subset \mathbb{H}^2$ is called a \emph{hyperbolic wedge} if there exist some $\gamma\in (0,2\pi]$ and polar coordinates $(a,\alpha)$ with respect to some base point $P\in\mathbb{H}^2$ (where the angle $\alpha$ is measured with respect to a chosen geodesic ray emanating from $P$ and a chosen orientation), such that $W$ is parametrised as
\begin{align*}
W=\{\, (a,\alpha) \mid 0<a<\infty,\, 0<\alpha <\gamma \,\}.
\end{align*}
The point $P$ is called the \emph{vertex} and $\gamma$ is called the \emph{angle} of the wedge. Thus, loosely speaking, a wedge is any domain bounded by two geodesic rays emanating from one point, its vertex. 

Again, the terms \emph{Euclidean wedge} and \emph{spherical wedge} are defined analogously. Note that spherical wedges are relatively compact.
\end{definition} 


Before we start our study of the heat invariants for hyperbolic polygons, we want to informally point out some previous publications related to our problem. Most of the works in this direction deal with Euclidean polygons, such that the study of spectral invariants for Euclidean polygons has a long and also intense history. Early publications focusing more or less on Euclidean polygons include \cite{Weyl}, \cite{CourantHilbert}, \cite{BaileyBrownell}, \cite{Fedosov}, \cite{Kac}, \cite{McKean}. Regarding the heat invariants for Euclidean polygons, a rigorous and complete investigation can be found in \cite{VanDenBerg}. M. van den Berg and S. Srisatkunarajah proved that the heat trace $Z_{\Omega_{\mathbb{R}}}(t)$ for any Euclidean polygon $\Omega_{\mathbb{R}}$ with $M$ vertices and interior angles $\gamma_1,...,\gamma_M$ has an asymptotic expansion of the form
\begin{align}
\label{equation:EuclideanAsymp}
Z_{\Omega_{\mathbb{R}}}(t) = \frac{\vert \Omega_{\mathbb{R}} \vert}{4\pi t} - \frac{\vert \partial\Omega_{\mathbb{R}} \vert}{8\sqrt{\pi t}} + \sum\limits_{i=1}^{M}\frac{\pi^2-\gamma_i^2}{24\pi\gamma_i} + O\left( e^{-\frac{c}{t}} \right), 
\end{align}
as $t\searrow 0$, with some constant $c>0$. Their approach to prove this asymptotic estimate was basically a combination of two previous results. Firstly, they used the so-called \emph{principle of not feeling the boundary} due to M. Kac (see \cite{Kac}). By this principle Kac approximated the heat kernel $K_{\Omega_{\mathbb{R}}}(x,x;t)$ by simpler functions with the same asymptotic expansion as $t\searrow 0$. Roughly speaking, if the point $x\in\Omega_{\mathbb{R}}$ is near to a vertex of the polygon, then the heat kernel $K_{\Omega_{\mathbb{R}}}$ is approximated by the heat kernel of a Euclidean wedge; if the point $x\in\Omega_{\mathbb{R}}$ is close to the boundary but not near to a vertex, it is approximated by the heat kernel of a half-plane; if the point is far from the boundary $\partial\Omega_{\mathbb{R}}$, the heat kernel of the whole plane is used as an approximation. Unfortunately, Kac was not able to find a simple enough formula for the heat kernel of a Euclidean wedge; he could only find a complicated integral expression for it. Secondly, M. van den Berg and S. Srisatkunarajah used a result due to D.B. Ray. It is claimed in \cite{McKean} that Ray obtained the constant term on the right-hand side of \eqref{equation:EuclideanAsymp} using an explicit formula for the Green's function of a wedge. This formula for the Green's function is stated in \cite{McKean} (and in \cite{VanDenBerg}) without proof and Ray never published his results. Using this formula as a starting point, van den Berg and Srisatkunarajah could deduce a simpler expression for the heat kernel of a wedge compared to that of Kac.

The principle of not feeling the boundary can be used in the same way to solve the analogous problem for hyperbolic polygons, as we will show in the first two sections of this chapter. The big challenge for hyperbolic polygons, as it was for Euclidean polygons, is to obtain a workable formula for the heat kernel of a wedge (see Section \ref{section:Green}).

The heat invariants for spherical polygons were published in \cite{Watson}, even though that article contains many inaccuracies and typos. Watson's approach is similar to that of \cite{VanDenBerg} for Euclidean polygons. He, too, uses the principle of not feeling the boundary to approximate the heat trace function. One advantage in the spherical case is that any spherical wedge, also commonly known as \emph{spherical lune}, is a relatively compact domain and all eigenvalues of its Dirichlet Laplacian are known. Explicit formulas for all eigenvalues of any spherical lune are stated in \cite{Gromes}, such that Watson could use them in his work \cite{Watson}. He shows that the asymptotic expansion of the heat trace $Z_{\Omega_{\mathbb{S}}}(t)$ for any spherical polygon $\Omega_{\mathbb{S}}$ has the form
\begin{align}
\label{equation:SphereAsymp}
Z_{\Omega_{\mathbb{S}}}(t) \overset{t\downarrow 0}{\sim} \frac{\vert \Omega_{\mathbb{S}} \vert}{4\pi t} - \frac{\vert \partial\Omega_{\mathbb{S}} \vert}{8\sqrt{\pi t}} + \sum\limits_{k=0}^{\infty}\left( i_k^{ \mathbb{S} } + b_k^{\mathbb{S}}\cdot t^{\frac{1}{2}} + \nu_k^{\mathbb{S}}  \right)t^k,
\end{align}
where $i_{k}^{ \mathbb{S} },b_{k}^{ \mathbb{S} },\nu_{k}^{ \mathbb{S} }$ stand for contributions from the interior, the boundary, and the vertices, respectively, and he obtains explicit formulas for all coefficients appearing in \eqref{equation:SphereAsymp}.

The asymptotic expansion of the heat trace for hyperbolic polygons is very similar to the spherical case. We will prove in Section \ref{section:ExpansionTrace} (see Theorem \ref{theorem:AsymptoticExpansionHeatTraceHyperbolPolygon}) that the heat trace $Z_{\Omega_{\mathbb{H}}}(t)$ for any hyperbolic polygon $\Omega_{\mathbb{H}}$ has an asymptotic expansion as in \eqref{equation:SphereAsymp}, i.e.
\begin{align}
\label{equation:HyperbAsymp}
Z_{\Omega_{\mathbb{H}}}(t) \overset{t\downarrow 0}{\sim} \frac{\vert \Omega_{\mathbb{H}} \vert}{4\pi t} - \frac{\vert \partial \Omega_{\mathbb{H}} \vert}{8\sqrt{\pi t}} + \sum\limits_{k=0}^{\infty}\left( i_k^{ \mathbb{H} } + b_k^{\mathbb{H}}\cdot t^{\frac{1}{2}} + \nu_k^{\mathbb{H}}  \right)t^k.
\end{align}
We will also obtain explicit formulas for all coefficients in \eqref{equation:HyperbAsymp}. The coefficients in \eqref{equation:HyperbAsymp} and \eqref{equation:SphereAsymp} are very similar to each other, only the sign changes in every other coefficient. More precisely, we have $i_k^{ \mathbb{H}} = (-1)^{k+1} i_k^{ \mathbb{S}}$, $b_k^{ \mathbb{H}} = (-1)^{k+1} b_k^{ \mathbb{S}}$ and $ \nu_k^{ \mathbb{H}}=\left(-1\right)^{k}  \nu_k^{ \mathbb{S}}$ for all $k\in \mathbb{N}_{0}$. Morally this is ``obvious'' for $i_k^{\mathbb{H}}$ and $b_k^{\mathbb{H}}$ because of the principle of not feeling the boundary and general facts from \citep{GilkeyBranson} (and \citep{Watson}); thus our results confirm Watson's formulas. It is remarkable for $\nu_k^{\mathbb{H}}$, since there is no general theory known for the contributions of vertices to the heat invariants like for the contributions from the interior and the contributions from a smooth boundary.

Let us return to Definition \ref{definition:Polygon} and examine some aspects of hyperbolic polygons. The motivation for that definition comes from the article \citep{Watson} in which the heat invariants are stated for spherical polygons as defined above. Thus it appears natural to treat the analogous setting in the hyperbolic case, too. Moreover, even though there is no definition given of the term \emph{polygon} in the article \citep{VanDenBerg}, the results and the reasoning of that article are basically valid for Euclidean polygons in the sense of Definition \ref{definition:Polygon}. It should be noted that the authors of the article \citep{VanDenBerg} seem to consider more general polygonal domains of the Euclidean plane than it is done in the other articles cited above. For example, in contrast to the other articles, they seem to include polygons with angles of measure $2\pi$ (see \citep[formulas $(1.6)$, $(2.2)$ or $(2.14)$]{VanDenBerg}). Note that the angle $2\pi$ is excluded in \citep[formula $(4a)$]{McKean} even though the contrary is claimed in \citep[formula $(1.6)$]{VanDenBerg}. 

If the boundary of a hyperbolic polygon is the union of (finitely many) piecewise geodesic simple closed curves, as it is seen in Fig. \ref{ExamplesSimplePolygons}, then it is evident how to define the edges, vertices, and angles. However, our definition of the term hyperbolic polygon is more general and also includes some pathological cases which are usually not treated as polygons (see Fig. \ref{Polygons}). In order to get an idea of the definitions we have marked the six polygons of Fig. \ref{ExamplesSimplePolygons} and Fig. \ref{Polygons} as follows: the vertices are marked by a dot, the (interior) angles are denoted by $\gamma_1, \gamma_2, \gamma_3,...$, and the edges correspond to the straight line segments between two dots.


\begin{figure} [ht] 
 \centering
\begin{tikzpicture}

\node[circle,fill,inner sep=1pt] at (-5,-3) {};
\node[circle,fill,inner sep=1pt] at (-2,-3) {};
\node[circle,fill,inner sep=1pt] at (-2,0) {};
\node[circle,fill,inner sep=1pt] at (-4,-0.5) {};
\node[circle,fill,inner sep=1pt] at (-5,0) {};

\draw (-5,-3) -- (-2,-3) -- (-2,0) -- (-4,-0.5) -- (-5,0) -- (-5,-3);

\node[circle,fill,inner sep=1pt] at (-0.5,-3) {};
\node[circle,fill,inner sep=1pt] at (2.5,-3) {};
\node[circle,fill,inner sep=1pt] at (2,1) {};
\node[circle,fill,inner sep=1pt] at (0,0.8) {};

\draw (-0.5,-3) -- (2.5,-3) -- (2,1) -- (0,0.8) -- (-0.5,-3);


\node[circle,fill,inner sep=1pt] at (1,-2.3) {};
\node[circle,fill,inner sep=1pt] at (1.5,0) {};
\node[circle,fill,inner sep=1pt] at (0.8,-0.3) {};
\node[circle,fill,inner sep=1pt] at (0.2,-1) {};

\draw  (1,-2.3) -- (1.5,0) -- (0.8,-0.3) -- (0.2,-1) -- (1,-2.3);


\node[circle,fill,inner sep=1pt] at (4,-3) {};
\node[circle,fill,inner sep=1pt] at (8,-3) {};
\node[circle,fill,inner sep=1pt] at (7,1) {};

\node[circle,fill,inner sep=1pt] at (5.6,-2.5) {};
\node[circle,fill,inner sep=1pt] at (6.6,-2.5) {};
\node[circle,fill,inner sep=1pt] at (6.6, -1.7) {};
\node[circle,fill,inner sep=1pt] at (5.6,-1.7) {};

\node[circle,fill,inner sep=1pt] at (7.1,-1.1) {};
\node[circle,fill,inner sep=1pt] at (6.8,-0.2) {};
\node[circle,fill,inner sep=1pt] at (6.4,-0.8) {};

\draw (4,-3) -- (8,-3) -- (7,1) -- (4,-3); 

\draw (5.6,-2.5) -- (6.6,-2.5) -- (6.6, -1.7) -- (5.6,-1.7) -- (5.6,-2.5); 

\draw  (7.1,-1.1) -- (6.8,-0.2) -- (6.4,-0.8) -- (7.1,-1.1); 


\draw[dashed] (-4.5,-3) arc (0:90:5mm);
\node[above right] at (-5.05,-3.05) {\tiny $\gamma_1$};

\draw[dashed] (-2.5,-3) arc (180:90:5mm);
\node[above left] at (-1.95,-3.05) {\tiny $\gamma_2$};

\draw[dashed] (-2,-0.5) arc (270:197:5mm);
\node[below] at (-2.15,-0.05) {\tiny $\gamma_3$};

\draw[dashed] (-3.66,-0.41) arc (14:-205:3.5mm);
\node[below] at (-4,-0.5) {\tiny $\gamma_4$};

\draw[dashed] (-5,-0.5) arc (270:330:5mm);
\node[below] at (-4.82,-0.05) {\tiny $\gamma_5$};


\draw[dashed] (0,-3) arc (0:80:5mm);
\node[above] at (-0.3,-3) {\tiny $\gamma_1$};

\draw[dashed] (2,-3) arc (180:100:5mm);
\node[above] at (2.3,-3) {\tiny $\gamma_2$};

\draw[dashed] (2,0.5) arc (270:190:5mm);
\node[below] at (1.8,1) {\tiny $\gamma_3$};

\draw[dashed] (0,0.3) arc (-90:7:5mm);
\node[below] at (0.2,0.8) {\tiny $\gamma_4$};


\draw[dashed] (0.84,-2.05) arc (-238.4:77.74:3mm);
\node[below] at (1,-2.2) {\tiny $\gamma_5$};

\draw[dashed] (0.4,-0.77) arc (49.4:301.6:3mm);
\node[left] at (0.33,-1) {\tiny $\gamma_6$};

\draw[dashed] (1.08,-0.18) arc (23.2:229.4:3mm);
\node[above] at (0.75,-0.4) {\tiny $\gamma_7$};

\draw[dashed] (1.44,-0.29) arc (-102.3:190.9:3mm);
\node[above] at (1.6,-0.1) {\tiny $\gamma_8$};


\draw[dashed] (4.5,-3) arc (0:50:5mm);
\node[above] at (4.3,-3.1) {\tiny $\gamma_1$};

\draw[dashed] (7.5,-3) arc (180:105:5mm);
\node[above] at (7.78,-3.05) {\tiny $\gamma_2$};

\draw[dashed] (7.12,0.51) arc (293.58:240:5mm);
\node[below] at (6.95,0.85) {\tiny $\gamma_3$};


\draw[dashed] (5.9,-2.5) arc (360:90:3mm);
\node[below] at (5.5,-2.4) {\tiny $\gamma_4$};

\draw[dashed] (6.3,-2.5) arc (-180:90:3mm);
\node[below] at (6.7,-2.4) {\tiny $\gamma_7$};

\draw[dashed] (6.6, -2) arc (-90:180:3mm);
\node[above] at (6.7, -1.8) {\tiny $\gamma_6$};

\draw[dashed] (5.9,-1.7) arc (0:270:3mm);
\node[above] at (5.5,-1.8) {\tiny $\gamma_5$};


\draw[dashed] (6.57,-0.55) arc (56.3:337:3mm);
\node[left] at (6.58,-0.9) {\tiny $\gamma_8$};

\draw[dashed] (6.89,-0.48) arc (-71.56:236.3:3mm);
\node[above] at (6.8,-0.3) {\tiny $\gamma_{9}$};

\draw[dashed] (6.82,-0.98) arc (-203:108.4:3mm);
\node[below] at (7.1,-1.05) {\tiny $\gamma_{10}$};

\end{tikzpicture}
\caption{ Simple polygons } \label{ExamplesSimplePolygons}
\end{figure}
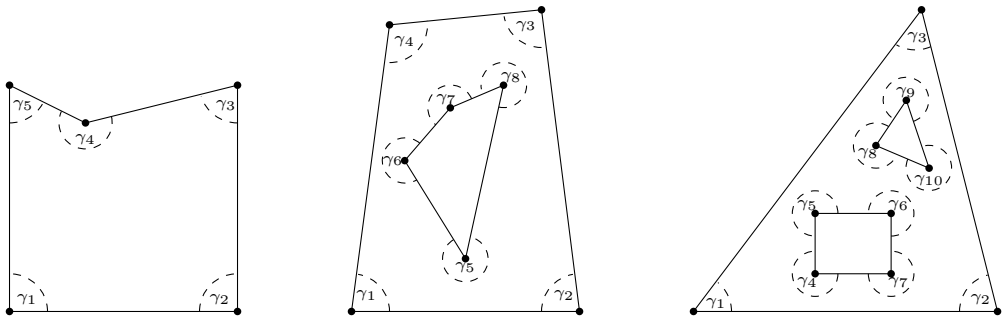

\begin{figure} [ht] 
 \centering
\begin{tikzpicture}


\draw (-5,-3) -- (-2,-3) -- (-2,0) -- (-5,-0.5) -- (-5,-3); 
\draw (-4.4,-1.9) -- (-3.5,-1.9) -- (-3.2,-0.9) -- (-3.5,-1.9) -- (-2.5,-1.9); 

\draw (0,-3) -- (2,-3) -- (2,1) -- (0,1) -- (0,-1) -- (1,-1) -- (0,-1) -- (0,-3);

\draw (4,-3) -- (8,-3) -- (8,1) -- (4,1) -- (4,-3); 
\draw (5.3, -2) -- (7.3,-0.5) -- (7.3,-2)--(5.3,-0.5); 


\node[circle,fill,inner sep=1pt] at (-5,-3) {};
\node[circle,fill,inner sep=1pt] at (-2,-3) {};
\node[circle,fill,inner sep=1pt] at (-2,0) {};
\node[circle,fill,inner sep=1pt] at (-5,-0.5) {};

\draw[dashed] (-4.5,-3) arc (0:90:5mm);
\node[above right] at (-5.05,-3.05) {\tiny $\gamma_1$};

\draw[dashed] (-2.5,-3) arc (180:90:5mm);
\node[above left] at (-1.95,-3.05) {\tiny $\gamma_2$};

\draw[dashed] (-2,-0.5) arc (258:197:5mm);
\node[below] at (-2.15,-0.05) {\tiny $\gamma_3$};

\draw[dashed] (-5,-1) arc (-90:8:5mm);
\node[below] at (-4.8,-0.5) {\tiny $\gamma_4$};

\node[circle,fill,inner sep=1pt] at (-4.4,-1.9) {};
\node[circle,fill,inner sep=1pt] at (-3.5,-1.9) {};
\node[circle,fill,inner sep=1pt] at (-3.2,-0.9) {};
\node[circle,fill,inner sep=1pt] at (-2.5,-1.9) {};

\draw[dashed] (-4.1,-1.9) arc (0:360:3mm);
\node[left] at (-4.3,-1.9) {\tiny $\gamma_5$};

\draw[dashed] (-3.9,-1.9) arc (-180:0:4mm);
\node[below] at (-3.5,-1.85) {\tiny $\gamma_{10}$};

\draw[dashed] (-3.1,-1.9) arc (0:73.33:4mm);
\node[above] at (-3.3,-2) {\tiny $\gamma_8$};

\draw[dashed] (-3.385,-1.517) arc (73.33:180:4mm);
\node[above] at (-3.65,-1.95) {\tiny $\gamma_6$};

\draw[dashed] (-3.286,-1.187) arc (-106.67:253.33:3mm);
\node at (-3.15,-0.75) {\tiny $\gamma_7$};

\draw[dashed] (-2.8,-1.9) arc (-180:180:3mm);
\node[right] at (-2.6,-1.9) {\tiny $\gamma_9$};


\node[circle,fill,inner sep=1pt] at (0,-3) {};
\node[circle,fill,inner sep=1pt] at (2,-3) {};
\node[circle,fill,inner sep=1pt] at (2,1) {};
\node[circle,fill,inner sep=1pt] at (0,1) {};
\node[circle,fill,inner sep=1pt] at (0,-1) {};
\node[circle,fill,inner sep=1pt] at (1,-1) {}; 

\draw[dashed] (0.5,-3) arc (0:90:5mm);
\node[above] at (0.2,-3) {\tiny $\gamma_1$};

\draw[dashed] (1.5,-3) arc (180:90:5mm);
\node[above] at (1.8,-3) {\tiny $\gamma_2$};

\draw[dashed] (2,0.5) arc (270:180:5mm);
\node[below] at (1.8,1) {\tiny $\gamma_3$};

\draw[dashed] (0,0.5) arc (-90:0:5mm);
\node[below] at (0.2,1) {\tiny $\gamma_4$};

\draw[dashed] (0,-0.5) arc (90:-90:5mm);
\node[above] at (0.2,-1) {\tiny $\gamma_5$};

\draw[dashed] (0.65,-1) arc (-180:180:3.5mm);
\node[right] at (0.9,-1) {\tiny $\gamma_6$};

\node[below] at (0.2,-1) {\tiny $\gamma_7$};

\node[circle,fill,inner sep=1pt] at (4,-3) {};
\node[circle,fill,inner sep=1pt] at (8,-3) {};
\node[circle,fill,inner sep=1pt] at (8,1) {};
\node[circle,fill,inner sep=1pt] at (4,1) {};

\node[circle,fill,inner sep=1pt] at (5.3, -2) {};
\node[circle,fill,inner sep=1pt] at (7.3,-0.5) {};
\node[circle,fill,inner sep=1pt] at (7.3,-2) {};
\node[circle,fill,inner sep=1pt] at (5.3,-0.5) {};
\node[circle,fill,inner sep=1pt] at (6.3,-1.25) {};

\node[below] at (5.2, -1.9) {\tiny $\gamma_5$};
\node[left] at (6.35,-1.25) {\tiny $\gamma_{6}$};
\node[above] at (5.2,-0.6) {\tiny $\gamma_{7}$};
\node[above] at (6.35,-1.25) {\tiny $\gamma_8$};
\node[above] at (7.4,-0.6) {\tiny $\gamma_9$};
\node[below] at (7.4,-2) {\tiny $\gamma_{10}$};
\node[below] at (6.35,-1.25) {\tiny $\gamma_{11}$};

\draw[dashed] (4.5,-3) arc (0:90:5mm);
\node[above] at (4.2,-3) {\tiny $\gamma_1$};

\draw[dashed] (7.5,-3) arc (180:90:5mm);
\node[above] at (7.8,-3) {\tiny $\gamma_2$};

\draw[dashed] (8,0.5) arc (270:180:5mm);
\node[below] at (7.8,1) {\tiny $\gamma_3$};

\draw[dashed] (4,0.5) arc (-90:0:5mm);
\node[below] at (4.2,1) {\tiny $\gamma_4$};

\draw[dashed] (5.6, -1.8) arc (35.5:395.5:4mm);
\draw[dashed] (6.65, -1) arc (38:315:4mm);
\draw[dashed] (5.6,-0.7) arc (-35.5:324.5:4mm);
\draw[dashed] (7.3,-0.9) arc (-90:217:4mm);
\draw[dashed] (7.3,-1.6) arc (90:-217:4mm);

\end{tikzpicture}
\caption{ Polygons } \label{Polygons}
\end{figure}

Note that the number of edges, vertices, and angles of any polygon in Fig. \ref{ExamplesSimplePolygons} is the same, whereas it is not the same for the polygons of Fig. \ref{Polygons} (e.g. each polygon of Fig. \ref{Polygons} has more angles than vertices). Furthermore, angles of measure $\pi$ and $2\pi$ may appear in general. For example, in the left polygon of Fig. \ref{Polygons} the angle $\gamma_{10}$ has measure $\pi$ and $\gamma_5,\gamma_7,\gamma_9$ are of size $2\pi$. We proceed with precise definitions of edges, vertices, and angles.
 
\begin{definition}
\label{definition:HyperbolDistanceAndGeodesicDisc}
For any two points $x,y\in\mathbb{H}^2$ we denote the (hyperbolic) distance between them by $d(x,y)$.
Further, for any $r>0$ and $P\in\mathbb{H}^2$ let 
\begin{align}
\hyperball{r}{P}:=\{\, x\in\mathbb{H}^2 \mid d(x,P)<r \,\}
\end{align}
denote the geodesic disc in the hyperbolic plane of radius $r$ and center $P$. 
\end{definition}

\begin{definition}
\label{definition:EdgeVertex}
Let $\Omega\subset\mathbb{H}^2$ be a hyperbolic polygon and let $P\in\partial\Omega$ be a boundary point. If there exists some $r>0$ such that $\hyperball{r}{P}\cap\partial\Omega$ is exactly one diameter of $\hyperball{r}{P}$, i.e. exactly one geodesic segment of length $2\cdot r$ and midpoint $P$, then we call $P$ a \emph{smooth boundary point} of $\Omega$. Otherwise $P$ is called a \emph{vertex} of $\Omega$.

A geodesic segment is called an \emph{edge} of $\Omega$, if it is contained in $\partial\Omega$ and both of its endpoints are vertices and there is no other vertex lying on it (except for the endpoints).

Hence $\Omega$ has finitely many vertices and the boundary $\partial \Omega$ consists of the union of finitely many edges which have no points in common except for vertices.
\end{definition}


\begin{definition}
\label{definition:Angle}
Let  $\Omega\subset\mathbb{H}^2$ be a hyperbolic polygon and let $P\in\partial\Omega$ be a vertex. Let $E_1, E_2$ be (not necessarily different) edges of $\Omega$ having $P$ as an endpoint. Let $W$ be a wedge with vertex $P$ and bounded by geodesic rays which contain $E_1$ and $E_2$, respectively. We choose polar coordinates $(a,\alpha)$ such that we have
\begin{align*}
W=\{\, (a,\alpha) \mid 0<a<\infty,\, 0<\alpha <\gamma \,\}
\end{align*}
for some $\gamma\in (0,2\pi]$. If there exists a $r>0$ such that
\begin{align*}
W\cap \hyperball{r}{P} \cap \Omega = \{\, (a,\alpha) \mid 0<a<r,\, 0<\alpha <\gamma \,\},
\end{align*}
then we call $W$ an \emph{angle} of $\Omega$ at $P$. For convenience we refer to $W$ by referring to the corresponding angle $\gamma$ of $W$.  
\end{definition}

Lastly, let us briefly discuss the meaning of the constants $\vert \Omega_{\mathbb{H}} \vert$ and $\vert \partial \Omega_{\mathbb{H}} \vert$ appearing in \eqref{equation:HyperbAsymp}. As usual, the former denotes the volume (or rather the area) of the polygon $\Omega_{\mathbb{H}}$. However, $\vert \partial \Omega_{\mathbb{H}} \vert$ does not always correspond to the length of the boundary $\partial\Omega_{\mathbb{H}}$, but is defined as follows.

\begin{definition}
\label{definition:VolumenVerallgemeinerteFlaeche}
Let $\Omega$ be a hyperbolic polygon. Suppose $\tilde{M}\in\mathbb{N}$ is the number of its edges and $E_1,...,E_{\tilde{M}}$ denote all edges of $\Omega$. Further, let $L(E_j)\in (0,\infty)$ be the length of the edge $E_j$ for all $j=1,...,\tilde{M}$. Let $k\in \{1,...,\tilde{M}\}$ be arbitrary and let $Q_k$ denote the midpoint of the edge $E_k$. If there exists some $r>0$ such that $\hyperball{r}{Q_k}\cap \Omega$ is the disjoint union of two (open) half-discs, then we define $\vert E_k \vert:=2\cdot L(E_k)$, and otherwise we set $\vert E_k \vert:=L(E_k)$. We now define $\vert \partial\Omega \vert:=\sum_{j=1}^{\tilde{M}} \vert E_j \vert$.

Note that if the boundary of $\Omega$ is the union of piecewise geodesic simple closed curves, then we have $\vert E_j \vert = L(E_j)$ for all $j=1,...,\tilde{M}$ so that $\vert \partial \Omega \vert$ indeed corresponds to the length of the boundary.
\end{definition}

The constants $\vert \partial \Omega_{\mathbb{R}} \vert$ and $\vert \partial \Omega_{\mathbb{S}} \vert$ appearing in the asymptotic expansions \eqref{equation:EuclideanAsymp} and \eqref{equation:SphereAsymp}, respectively, must be interpreted in an analogous manner so that the formulas are correct in general.

\section{Green's function and heat kernel of a wedge}
\label{section:Green}

Let $W\subset \mathbb{H}^2$ be a hyperbolic wedge with vertex $P \in\mathbb{H}^2$ and interior angle $\gamma\in (0,2\pi]$. Our aim in this section is to study the function $Z_{\gamma}$ defined by
\begin{align*}
Z_{\gamma}(t;r):= \int\limits_{B_r^{\mathbb{H}^2}\hspace{-0.5mm}(P) \cap W} K_W(x,x;t) dx, \quad r,t>0.
\end{align*}
Note that the heat kernel $K_W$ is the one for the entire wedge, while the integration is done over a bounded sector only.
In order to study $Z_{\gamma}$, we first investigate the heat kernel $K_W$ (recall Definition \ref{definition:HeatKernel}) and the Green's function $G_W$ (recall Definition \ref{definition:Greens}) of the wedge and establish explicit formulas for them as generalised Mehler transforms. These formulas in turn will lead to an explicit formula for $Z_{\gamma}(t;r)$, which will be used in Section \ref{section:ExpansionTrace} to compute the heat invariants for all hyperbolic polygons.


It will be useful to introduce the following shifted functions.

\begin{definition}
\label{definition:ShiftedFunctions1}
For any domain $\Omega\subset \mathbb{H}^2$, we call the functions $K_{\Omega}^{\nicefrac{1}{4}}:\Omega\times\Omega\times (0,\infty)\rightarrow\mathbb{R}$ and $G_{\Omega}^{\nicefrac{1}{4}}:\offdiag(\Omega)\times \mathcal{H}_{>\frac{1}{4}}\rightarrow\mathbb{C}$ defined by
\begin{align}
K_{\Omega}^{\nicefrac{1}{4}}(x,y;t):&=e^{\frac{1}{4}t}\cdot K_{\Omega}(x,y;t),\label{equation:ShiftedHeat}\\
G_{\Omega}^{\nicefrac{1}{4}}(x,y;s):&=\mathcal{L}\lbrace K_{\Omega}^{\nicefrac{1}{4}}(x,y;t)\rbrace (s) \label{equation:ShiftedGreen}
\end{align}
the \emph{shifted} heat kernel and the \emph{shifted} Green's function of $\Omega$, respectively.
\end{definition}

Like the heat kernel, the function $K_{\Omega}^{\nicefrac{1}{4}}$ satisfies the conditions of Definition \ref{definition:HeatKernel} if the Laplacian $\Delta$ is replaced in condition \eqref{equation:FundamentalSolution} by the \emph{shifted} Laplacian $\Delta^{\nicefrac{1}{4}}:=\Delta-\frac{1}{4}$. Further, for each $s>\frac{1}{4}$, the shifted Green's function  $G_{\Omega}^{\nicefrac{1}{4}}$ satisfies Proposition \ref{proposition:PDEGreen} if the Laplacian is replaced by the shifted Laplacian. The reason is that  $s+\Delta^{\nicefrac{1}{4}}=\left(s-\frac{1}{4}\right)+\Delta$ and
\begin{align}
G_{\Omega}^{\nicefrac{1}{4}}(x,y;s) &= \int\limits_{0}^{\infty} e^{-st }e^{\frac{1}{4}t}\cdot K_{\Omega}(x,y;t) dt= \int\limits_{0}^{\infty} e^{-\left(s-\frac{1}{4}\right)t }\cdot K_{\Omega}(x,y;t) dt \nonumber \\
& = G_{\Omega}\left(x,y;s-\frac{1}{4}\right). \label{equation:ConnectionShiftGreen}
\end{align}
Most of the formulas below are stated in terms of the shifted functions, since then they become shorter. The formulas can always be rewritten into corresponding formulas for the Green's function by \eqref{equation:ConnectionShiftGreen} and for the heat kernel using \eqref{equation:ShiftedHeat}.

The heat kernel $K_{\mathbb{H}^2}$ for the hyperbolic plane is given by the following formulas (see, e.g., \citep[(12) and (13) on p. 246]{Chavel}):
\begin{align}
K_{\mathbb{H}^2}(x,y;t)&= \frac{1}{2\pi}\int\limits_{0}^{\infty} e^{-\left( \frac{1}{4} + \rho^2 \right)t}P_{-\frac{1}{2}+i\rho}\left( \cosh(d(x,y)) \right)\rho\tanh\left( \pi\rho \right)d\rho \label{equation:HyperHeat1}\\
&=\frac{\sqrt{2}}{(4\pi t)^{\nicefrac{3}{2}}}e^{-\frac{t}{4}}\int\limits_{d(x,y)}^{\infty} \frac{\rho e^{-\frac{\rho^2}{4t}}}{\sqrt{\cosh(\rho)-\cosh(d(x,y))}}d\rho, \label{equation:HyperHeat2}
\end{align}
where $P_{-\frac{1}{2}+i\rho}$ is the Legendre function of the first kind from Definition \ref{definition:LegendreFirstSecond}.
We start with a helpful new formula for the Green's function $G_{\mathbb{H}^2}$ of the hyperbolic plane in terms of polar coordinates.

\begin{proposition}
\label{proposition:GreenPlane}
Let $s\in\mathbb{C}$ be such that $\Re(s) > \frac{1}{4}$, and let $x,y\in\mathbb{H}^2$ with $x\neq y$ be given. The shifted Green's function for the hyperbolic plane is given by the equation
\begin{align}
G_{\mathbb{H}^2}^{\nicefrac{1}{4}}(x,y;s)&=\frac{1}{2\pi}Q_{\sqrt{s}-\frac{1}{2}}\left( \cosh(d(x,y)) \right). \label{equation:GreenPlane1}
\end{align}
Suppose, further, that we have chosen polar coordinates in $\mathbb{H}^2$ \emph{(}associated with a chosen orientation and a chosen geodesic ray\emph{)} such that $x=(a,\alpha)$ and $y=(b,\beta)$, where $a,b\in(0,\infty)$ and $\alpha,\beta\in [0,2\pi)$. Then
\begin{align}
G_{\mathbb{H}^2}^{\nicefrac{1}{4}}(x,y;s)=\frac{1}{\pi^2} \int\limits_{0}^{\infty} Q_{\sqrt{s}-\frac{1}{2}}^{-i\rho}\left(\cosh(a)\right)Q_{\sqrt{s}-\frac{1}{2}}^{i\rho}\left(\cosh(b)\right) \cosh\left(\rho\cdot \left( \pi-\vert \alpha-\beta \vert \right)\right)d\rho. \label{equation:GreenPlane2}
\end{align}
\end{proposition}

\begin{proof}
From \eqref{equation:HyperHeat1} we obtain
\begin{align}
\label{equation:IntegralformelShiftedHeatPlane}
K_{\mathbb{H}^2}^{\nicefrac{1}{4}}(x,y;t)=\frac{1}{2\pi}\int\limits_{0}^{\infty} e^{-  \rho^2 t}P_{-\frac{1}{2}+i\rho}\left( \cosh(d(x,y)) \right)\rho\tanh\left( \pi\rho \right)d\rho.
\end{align}
Thus using the Fubini-Tonelli theorem we get
\begin{align*}
G_{\mathbb{H}^2}^{\nicefrac{1}{4}}(x,y;s)&=\int\limits_{0}^{\infty} e^{-st}\left(\frac{1}{2\pi}\int\limits_{0}^{\infty} e^{-  \rho^2 t }P_{-\frac{1}{2}+i\rho}\left( \cosh(d(x,y)) \right)\rho\tanh\left( \pi\rho \right)d\rho \right) dt\\
&= \frac{1}{2\pi}\int\limits_{0}^{\infty} \left( \int\limits_{0}^{\infty}  e^{-\left( s+  \rho^2\right) t } dt\right) P_{-\frac{1}{2}+i\rho}\left( \cosh(d(x,y)) \right)\rho\tanh\left( \pi\rho \right)d\rho \\
&= \frac{1}{2\pi}\int\limits_{0}^{\infty} \frac{\rho\tanh\left( \pi\rho \right)}{s+\rho^2} P_{-\frac{1}{2}+i\rho}\left( \cosh(d(x,y)) \right)d\rho \\
&= \frac{1}{2\pi} Q_{\sqrt{s}-\frac{1}{2}}\left( \cosh(d(x,y)) \right),
\end{align*}
where the last equality can be found for example in \citep[7.213]{Gradshteyn} or \citep[p. 20]{Oberhettinger}. Let us explain shortly why the Fubini-Tonelli theorem is applicable, which we used in the second equality above. Because of \eqref{equation:LegendreAsymp2} and since $\rho\mapsto P_{-\frac{1}{2}+i\rho}\left( \cosh(d(x,y))\right)$ is continuous, there exists some constant $C>0$ such that for all $\rho\in (0,\infty)$: $\vert P_{-\frac{1}{2}+i\rho}\left( \cosh(d(x,y))\right) \vert\leq \frac{C}{\sqrt{\rho}}$. Hence,
\begin{align*}
\int\limits_{0}^{\infty}\int\limits_{0}^{\infty} \vert e^{-\left( s+  \rho^2\right) t }P_{-\frac{1}{2}+i\rho}\left( \cosh(d(x,y))\right) \rho\tanh\left( \pi\rho \right) \vert\, dt d\rho &\leq C  \int\limits_{0}^{\infty}\int\limits_{0}^{\infty} e^{-\left( \frac{1}{4} +  \rho^2\right)t} \cdot \sqrt{\rho}\, dt \, d\rho \\
&=C\int\limits_{0}^{\infty} \frac{\sqrt{\rho}}{\frac{1}{4}+\rho^2}\, d\rho<\infty.
\end{align*}
This completes the proof of equation \eqref{equation:GreenPlane1}.

The second formula \eqref{equation:GreenPlane2} follows from \eqref{equation:GreenPlane1} and Corollary \ref{corollary:IntegralProduktLegendreFuerGreensFunction} with $z:=\cosh(a)$, $\omega:=\cosh(b)$, and $\theta:=\vert \alpha-\beta \vert$, since the distance is given in the above polar coordinates by the following formula (see \citep[Theorem 2.2.1 (i)]{Buser}):
\begin{align}
\label{equation:HyperbolicDistancePolarCoordinates}
\cosh(d(x,y))=\cosh(a)\cosh(b)-\sinh(a)\sinh(b)\cos(\vert \alpha-\beta \vert).
\end{align}
Note that the assumptions of Corollary \ref{corollary:IntegralProduktLegendreFuerGreensFunction} are satisfied here because $x\neq y$ implies that either $a\neq b$ or $\vert \alpha-\beta \vert\in (0,2\pi)$.
\end{proof} 

\begin{remark}
Note that the function on the right-hand side of \eqref{equation:GreenPlane1} is defined for all $s\in\mathbb{C}\backslash (-\infty,0]$ and is also holomorphic on that domain. Hence, the function $\mathbb{C}\backslash (-\infty,0] \ni s\mapsto \frac{1}{2\pi}Q_{\sqrt{s}-\frac{1}{2}}\left( \cosh(d(x,y)) \right)\in\mathbb{C}$ is an analytic continuation of $\mathcal{H}_{>\frac{1}{4}}\ni s\mapsto G_{\mathbb{H}^2}^{\nicefrac{1}{4}}(x,y;s)\in\mathbb{C}$. Obviously, the same applies to the right-hand side of \eqref{equation:GreenPlane2}.
\end{remark}

Next, we focus on the heat kernel $K_W$ and the Green's function $G_W$ of the hyperbolic wedge. Before we give a rigorous treatment of those functions, we will consider a heuristic derivation of the formula \eqref{equation:GreenWedge} for $G_W^{\nicefrac{1}{4}}$ given below. We want to point out that our heuristic derivation is analogous to a derivation given in \cite{Srisatkunarajah}. In his doctoral thesis, Srisatkunarajah has worked out a proof of the formula due to D.B. Ray for the Green's function of a Euclidean wedge (as mentioned, this formula is published in \cite{McKean} and \cite{VanDenBerg}). However, the proof given in \cite{Srisatkunarajah} is incomplete and has some argumentative gaps, which we will provide in the hyperbolic case by Lemma \ref{lemma:PropertiesOfH} and Theorem \ref{theorem:GreenWedge}. Note that our formula for the Green's function of a hyperbolic wedge is similar to Ray's formula. The only difference is that the Bessel functions appearing in Ray's formula are replaced properly with the associated Legendre functions of the second kind. However, in the Euclidean case all integrals appearing in the derivation and involving Bessel functions are well-known, such that their solutions could be used in \cite{Srisatkunarajah}. We, instead, will have to refer to Section \ref{section:Legendre} to solve the corresponding integrals for the associated Legendre functions of the second kind.

As mentioned, we proceed with a largely heuristic discussion of the Green's function for $W$ in order to motivate the formula \eqref{equation:GreenWedge} below. Recall that $\Delta=-\diverg \circ \nabla$ denotes the Laplacian. Suppose the heat kernel of the wedge $W$ can be written as
\begin{align}
\label{equation:AnsatzW-Kern}
K_W(x,y;t) = K_{\mathbb{H}^2}(x,y;t)-h(x,y;t),
\end{align}
where $h:\overline{W}\times W\times[0,\infty)\rightarrow\mathbb{R}$ is a function satisfying for all $y\in W$, the following conditions:
\begin{align}
\label{equation:PDE1}
\begin{cases}
h\left( \cdot,y;\cdot \right):\overline{W}\times [0,\infty)\rightarrow \mathbb{R}& \text{ is continuous, } \\
\left( \partial_t + \Delta \right)h(x,y;t) = 0& \text{ for all }x\in W \text{ and } t>0, \\
h(x,y;0)=0&  \text{ for all }x\in W, \\
h(x,y;t) = K_{\mathbb{H}^2}(x,y;t)& \text{ for all }x\in \partial W \text{ and } t>0.
\end{cases}
\end{align}

\begin{remark}
Suppose $h(x,y;t):=K_{\mathbb{H}^2}(x,y;t)-K_W(x,y;t)$ for all $x,y\in W$ and $t>0$. If for all $y\in W$ the function $K_W(\cdot, y;\cdot)$ can be extended continuously to $\overline{W}\times (0,\infty)$ by $K_W (x,y;t):=0$ for all $x\in \partial W$ and $t>0$, then $h$ indeed satisfies all conditions of \eqref{equation:PDE1}. (Of course one needs to extend for all $y\in W$ the function $h(\cdot,y;\cdot)$ appropriately.) Note that $h(\cdot,y;\cdot)$ can be extended continuously to $W\times \{0\}$ by $h(x,y;0):=0$ for all $x\in W$ (see \citep[Corollary $9.21$ and Exercise $9.7$]{Grigoryan}). 
\end{remark}

When we multiply both sides of equation \eqref{equation:AnsatzW-Kern} with $e^{\frac{t}{4}}$, we obtain the corresponding equation for the shifted functions. With $h^{\nicefrac{1}{4}}(x,y;t):=e^{\frac{t}{4}}h(x,y;t)$, we have
\begin{align}
\label{equation:AnsatzW-KernShift}
K_{W}^{\nicefrac{1}{4}}(x,y;t) = K_{\mathbb{H}^2}^{\nicefrac{1}{4}}(x,y;t)-h^{\nicefrac{1}{4}}(x,y;t)\quad \text{ for all } t>0\,\text{ and } x,y\in W.
\end{align}
When we apply the Laplace transform on both sides of \eqref{equation:AnsatzW-KernShift}, we obtain the equation
\begin{align}
\label{equation:AnsatzGreenShift}
G_W^{\nicefrac{1}{4}}(x,y;s) = G_{\mathbb{H}^2}^{\nicefrac{1}{4}}(x,y;s) - H^{\nicefrac{1}{4}}(x,y;s)
\end{align}
for all $(x,y)\in \offdiag(W)$ and $s\in\mathbb{C}$ with $\Re(s)>\frac{1}{4}$, where
\begin{align}
\label{equation:GreenShifted}
H^{\nicefrac{1}{4}}(x,y;s):=\mathcal{L}\lbrace h^{\nicefrac{1}{4}}(x,y;t) \rbrace (s).
 \end{align}
For any $s>\frac{1}{4}$ and $y\in W$ the function $H^{\nicefrac{1}{4}}(\cdot ,y;s)$ solves the following boundary value problem: 
\begin{align}
\label{equation:PDEShifted-H}
\begin{cases}
H^{\nicefrac{1}{4}}(\cdot,y;s): \overline{W}\rightarrow \mathbb{R} &\text{ is continuous and bounded,}\\
\left( s + \Delta-\frac{1}{4} \right)H^{\nicefrac{1}{4}}\left( x,y;s \right) = 0 &\text{ for all } x\in W,\\
H^{\nicefrac{1}{4}}\left( x,y;s \right) = G_{\mathbb{H}^2}^{\nicefrac{1}{4}}\left( x,y;s \right) & \text{ for all } x\in \partial W.
\end{cases}
\end{align}

We want to find a solution to the problem \eqref{equation:PDEShifted-H} in terms of polar coordinates, so let us choose for the rest of this section polar coordinates such that
\begin{align*}
W=\{\, (a,\alpha) \mid 0< a <\infty, \, 0<\alpha < \gamma \,\}.
\end{align*}
The Laplacian is given in polar coordinates by the formula
\begin{align*}
-\Delta = \frac{1}{\sinh(a)}\frac{\partial}{\partial a}\left\lbrace \sinh(a)\frac{\partial}{\partial a} \right\rbrace + \frac{1}{\sinh^2(a)}\frac{\partial^2}{\partial^2 \alpha},
\end{align*}
which follows from the well-known representation of the Laplacian in local coordinates (see e.g. \citep[formula $(3.40)$]{Grigoryan}) and since the Riemannian metric $g$ of the hyperbolic plane is represented in polar coordinates by $g=da^2 + \sinh^2(a)d\alpha^2$ (see \citep[formula $(3.70)$]{Grigoryan}). In the following, let $y=(b,\beta)\in W$ and $s>\frac{1}{4}$ both be fixed.

Firstly, we use the method of separation of variables to find product solutions to the partial differential equation (PDE) in \eqref{equation:PDEShifted-H}, i.e. we look for solutions of the form $u(a,\alpha)=v(a)\cdot w(\alpha)$. Substituting into the PDE, we obtain for all points $x=(a,\alpha)$ such that $u(a,\alpha)\neq 0:$
\begin{align*}
&\left( s + \Delta -\frac{1}{4}\right) u(a,\alpha) = 0  \\
\Leftrightarrow &\left( s-\frac{1}{4} \right)v(a)w(\alpha) - \frac{w(\alpha)}{\sinh(a)}\frac{\partial}{\partial a}\left\lbrace \sinh(a)v'(a) \right\rbrace  - \frac{1}{\sinh^2(a)} w''(\alpha)v(a)  = 0 \\
\Leftrightarrow &\left( s-\frac{1}{4} \right)\sinh^2(a) - \frac{\sinh(a)}{v(a)}\frac{\partial}{\partial a}\left\lbrace \sinh(a)v'(a) \right\rbrace = \frac{w''(\alpha)}{w(\alpha)}.
\end{align*} 

Therefore, both sides must be equal to some separation constant $c\in\mathbb{R}$. We assume $c$ to be non-negative and write $c=\rho^2$ with $\rho\geq 0$. The PDE now reduces to the following two ordinary differential equations:

\begin{itemize}
\item[I)] $w''(\alpha)=\rho^2 \cdot w(\alpha)$,
\item[II)] $\left( -s+\frac{1}{4} \right)\sinh^2(a)\cdot v(a) + \sinh(a) \frac{\partial}{\partial a}\left\lbrace \sinh(a)v'(a) \right\rbrace + \rho^2\cdot v(a) = 0.$
\end{itemize}

We first solve the second equation. When we multiply this equation by $\frac{1}{\sinh^2(a)}$, it can be written equivalently as:
\begin{align*}
v''(a) + \frac{\cosh(a)}{\sinh(a)}v'(a) + \left( \left( -s+\frac{1}{4} \right) - \frac{\left(i\rho\right)^2}{\sinh^2(a)} \right)\cdot v(a) = 0.
\end{align*}
Consider $\tilde{v}:=v\circ \arcosh :(1,\infty)\rightarrow\mathbb{C}$. When we substite $\tilde{v}$ into the above equation, we get the following equivalent differential equation:
\begin{align*}
(1-z^2)\tilde{v}''(z) -2z \tilde{v}'(z) + \left( \left( s - \frac{1}{4} \right) - \frac{\left(i\rho\right)^2}{1-z^2} \right)\cdot \tilde{v}(z) = 0\, \text{ for all } z\in (1,\infty).
\end{align*}
If $\nu:=\sqrt{s} - \frac{1}{2} $, then $\nu\cdot(\nu+1) = s-\frac{1}{4}$. This differential equation for $\tilde{v}$ is nothing else than the associated Legendre equation stated in \eqref{equation:LegendreDGL} with parameters $\nu=\sqrt{s}-\frac{1}{2}$ and $\mu=i\rho$. For all $\rho\geq 0$ two linearly independent solutions are (see \citep[\S 14.2(iii)]{NIST})
\begin{align*}
z\mapsto P_{\sqrt{s} - \frac{1}{2}}^{i\rho}(z)\,\text{ and }\, z\mapsto Q_{\sqrt{s} - \frac{1}{2} }^{-i\rho}(z).
\end{align*}
On the other hand, two linearly independent solutions for equation I) are the functions $\alpha\mapsto\cosh(\rho \alpha)$ and $\alpha\mapsto \sinh(\rho \alpha)$. Thus we have the following product solutions to the PDE:
\begin{align*}
x=(a,\alpha)\mapsto \left( \tilde{A}_1\cosh(\rho \alpha)+ \tilde{A}_2\sinh(\rho \alpha) \right)\cdot \left( \tilde{B}_1 P_{\sqrt{s} - \frac{1}{2}}^{i\rho}(\cosh(a)) + \tilde{B}_2 Q_{\sqrt{s} - \frac{1}{2}}^{-i\rho}(\cosh(a)) \right)
\end{align*}
with arbitrary constants $\tilde{A}_1, \tilde{A}_2, \tilde{B}_1, \tilde{B}_2\in\mathbb{C}$.

Secondly, we construct from the product solutions above a suitable solution $u=H^{\nicefrac{1}{4}}(\cdot,y;s):\overline{W}\rightarrow\mathbb{R}$ to the PDE which also meets the boundary condition in \eqref{equation:PDEShifted-H}. Since 
\begin{align*}
\lim\limits_{a\rightarrow\infty} \vert P_{\sqrt{s} - \frac{1}{2}}^{i\rho}(\cosh(a)) \vert =\infty
\end{align*}
(see e.g. \citep[14.8.12]{NIST}), we set $\tilde{B}_1=0$. Further, because of the formula \eqref{equation:GreenPlane2} for the boundary condition we try, by using the superposition principle, 
\begin{align*}
u(a,\alpha):=\int\limits_{0}^{\infty}\left( A_1(\rho)\cosh(\rho \alpha)+ A_2(\rho)\sinh(\rho \alpha) \right) Q_{\sqrt{s} - \frac{1}{2}}^{-i\rho}(\cosh(a)) d\rho.
\end{align*}
We determine suitable functions $A_1(\rho),\, A_2(\rho)$ from the boundary conditions. The boundary $\partial W$ of the wedge contains the two rays emanating from $P$, which are described in polar coordinates by
\begin{align*}
 \{\, x=(a,\alpha)\mid 0<a<\infty\, ;\, \alpha=0\, \text{ or }\, \alpha=\gamma \,\}.
\end{align*}
For $\alpha=0$ we have
\begin{align*}
u(a,0) &= \int\limits_{0}^{\infty}  A_1(\rho)Q_{ \sqrt{s} - \frac{1}{2}}^{-i\rho}(\cosh(a)) d\rho,\\
G_{\mathbb{H}^2}^{\nicefrac{1}{4}}((a,0),(b,\beta);s) &= \frac{1}{\pi^2}\int\limits_{0}^{\infty}\cosh(\rho(\pi-\beta))Q_{\sqrt{s} - \frac{1}{2}}^{-i\rho}(\cosh(a))Q_{ \sqrt{s} - \frac{1}{2}}^{i\rho}(\cosh(b)) d\rho.
\end{align*}
Therefore we set
\begin{align*}
A_1(\rho):=\frac{1}{\pi^2}\cosh(\rho(\pi-\beta))Q_{\sqrt{s} - \frac{1}{2} }^{i\rho}(\cosh(b)),
\end{align*}
such that the boundary condition at $\alpha=0$ is satisfied.

Moreover,
\begin{align*}
u(a,\gamma) &= \int\limits_{0}^{\infty} \lbrace A_1(\rho)\cosh(\rho\gamma) + A_2(\rho)\sinh(\rho\gamma)\rbrace Q_{ \sqrt{s} - \frac{1}{2}}^{-i\rho}(\cosh(a)) d\rho,\\
G_{\mathbb{H}^2}^{\nicefrac{1}{4}}((a,\gamma),(b,\beta);s) &= \frac{1}{\pi^2}\int\limits_{0}^{\infty}\cosh(\rho(\pi-\vert \gamma - \beta \vert))Q_{\sqrt{s} - \frac{1}{2}}^{-i\rho}(\cosh(a))Q_{ \sqrt{s} - \frac{1}{2}}^{i\rho}(\cosh(b)) d\rho.
\end{align*}
Obviously $\vert \gamma - \beta \vert=\gamma-\beta$ and thus we set:
\begin{align*}
A_2(\rho):=\frac{1}{\pi^2}Q_{\sqrt{s} - \frac{1}{2} }^{i\rho}(\cosh(b))\left(\frac{\cosh(\rho(\pi-(\gamma-\beta))) - \cosh(\rho(\pi-\beta))\cosh(\rho\gamma)}{\sinh(\rho\gamma)}\right).
\end{align*}
Finally we obtain the following candidate for a solution of \eqref{equation:PDEShifted-H}, defined for all $(a,\alpha)\in \overline{W}\backslash\{ P \}$ as:
\begin{align*}
u(a,\alpha) =&\, \frac{1}{\pi^2} \int\limits_{0}^{\infty} Q_{ \sqrt{s} - \frac{1}{2}}^{-i\rho}(\cosh(a))  Q_{ \sqrt{s} - \frac{1}{2}}^{i\rho}(\cosh(b))\Bigg( \cosh(\rho(\pi-\beta))\cosh(\rho\alpha) \\
& + \left( \cosh(\rho(\pi-(\gamma-\beta))) - \cosh(\rho(\pi-\beta))\cosh(\rho\gamma) \right)\frac{\sinh(\rho\alpha)}{\sinh(\rho\gamma)} \Bigg) d\rho.
\end{align*}
Now observe that the long expression in brackets above can be simplified, since
\begin{align}
&\cosh(\rho(\pi-\beta))\cosh(\rho\alpha) + \left( \cosh(\rho(\pi-(\gamma-\beta))) - \cosh(\rho(\pi-\beta))\cosh(\rho \gamma) \right)\frac{\sinh(\rho\alpha)}{\sinh(\rho\gamma)} \nonumber \\
&=\frac{\sinh(\pi\rho)}{\sinh(\gamma\rho)}\cosh(\rho(\gamma-\alpha-\beta)) - \frac{\sinh((\pi-\gamma)\rho)}{\sinh(\gamma\rho)}\cosh((\alpha-\beta)\rho). \label{equation:IdentitaetCosHyperbolicusMonster}
\end{align}
Thus we set for all $x=(a,\alpha)\in \overline{W}\backslash\{ P \}$:
\begin{align}
\label{equation:SolutionH}
H^{\nicefrac{1}{4}} (x,y;s) := & \, \frac{1}{\pi^2} \int\limits_{0}^{\infty} Q_{ \sqrt{s} - \frac{1}{2}}^{-i\rho}(\cosh(a))  Q_{ \sqrt{s} - \frac{1}{2}}^{i\rho}(\cosh(b)) \Bigg( \frac{\sinh(\pi\rho)}{\sinh(\gamma\rho)}\cosh(\rho(\gamma-\alpha-\beta)) \nonumber \\
&- \frac{\sinh((\pi-\gamma)\rho)}{\sinh(\gamma\rho)}\cosh((\alpha-\beta)\rho) \Bigg)d\rho.
\end{align}
Now, a candidate for a formula (in polar coordinates) of the shifted Green's function $G_W^{\nicefrac{1}{4}}$ can be deduced using \eqref{equation:AnsatzGreenShift}, \eqref{equation:GreenPlane2} and \eqref{equation:SolutionH}. That formula is stated explicitly in Theorem \ref{theorem:GreenWedge} below. In order to give a rigorous proof we first study the function $H^{\nicefrac{1}{4}}$, which is done in the following lemma. 

\begin{lemma}
\label{lemma:PropertiesOfH}
Consider the function 
\begin{align*}
H^{\nicefrac{1}{4}}:W\times W\times\mathbb{C}\backslash(-\infty,0]\ni (x,y, s)\mapsto H^{\nicefrac{1}{4}}(x,y;s)\in \mathbb{C},
\end{align*}
where $H^{\nicefrac{1}{4}}(x,y;s)$ is defined by the right-hand side of \eqref{equation:SolutionH} with $x=(a,\alpha)$ and $y=(b,\beta)$ \emph{(}with respect to the polar coordinates chosen above\emph{)}. Then:
\begin{itemize}
\item[$(i)$] For all $x,y\in W$ the function $\mathbb{C}\backslash(-\infty,0]\ni s\mapsto H^{\nicefrac{1}{4}}(x,y;s)\in \mathbb{C}$ is holomorphic.
\item[$(ii)$] For all $s\in\mathbb{C}$ with $\Re(s)>0$ the function
\begin{align*}
W\times W \ni (x,y)\mapsto H^{\nicefrac{1}{4}}(x,y;s)\in \mathbb{R}
\end{align*}
is continuous. Moreover, $H^{\nicefrac{1}{4}}(x,y;s)$ is real valued for all $s>0$ and $x,y\in W$. Further, for all $s>\frac{1}{4}$ and $y\in W$ the function $W\ni x\mapsto H^{\nicefrac{1}{4}}(x,y;s)\in\mathbb{R}$ is smooth and $(s+\Delta^{\nicefrac{1}{4}})H^{\nicefrac{1}{4}}(\cdot,y;s) \equiv 0$ \emph{(}recall that $\Delta^{\nicefrac{1}{4}}=\Delta-\frac{1}{4}$\emph{)}.
\item[$(iii)$] For all $s>\frac{1}{4}$ and $y\in W$ the function $W\ni x\mapsto H^{\nicefrac{1}{4}}(x,y;s)\in \mathbb{R}$ can be extended continuously to $\overline{W}$. That extension \emph{(}which we also denote by $H^{\nicefrac{1}{4}}$\emph{)} satisfies $H^{\nicefrac{1}{4}}(x,y;s)=G_{\mathbb{H}^2}^{\nicefrac{1}{4}}(x,y;s)$ for all $x\in \partial W$. Furthermore, we have $0 < H^{\nicefrac{1}{4}}(x,y;s)$ for all $x\in W$ and $H^{\nicefrac{1}{4}}(x,y;s)\leq G_{\mathbb{H}^2}^{\nicefrac{1}{4}}(x,y;s)$ for all $x\in W\backslash\{y\}$.
\end{itemize}
\end{lemma} 

\begin{proof}
$(i)$. Let $x,y\in W$ be arbitrary. Obviously, the function $H^{\nicefrac{1}{4}}(x,y;\cdot)$ can be written as a composition of holomorphic functions due to Lemma \ref{lemma:HolomorphieIntegralLegendreProdukt}. Note that the function $ s\mapsto \sqrt{s}-\frac{1}{2}$ is holomorphic on the domain $\mathbb{C}\backslash(-\infty,0]$, where that domain is mapped onto $\mathcal{H}_{>-\frac{1}{2}}$. Therefore, Lemma \ref{lemma:HolomorphieIntegralLegendreProdukt} can indeed be applied here; the other conditions of Lemma \ref{lemma:HolomorphieIntegralLegendreProdukt} are obviously satisfied with $h(\rho) := \frac{1}{\pi^2} \cdot \big( \frac{\sinh(\pi\rho)}{\sinh(\gamma\rho)}\cosh(\rho(\gamma-\alpha-\beta))- \frac{\sinh((\pi-\gamma)\rho)}{\sinh(\gamma\rho)}\cosh((\alpha-\beta)\rho) \big)$.

$(ii)$. Suppose $s\in\mathbb{C}$ with $\Re(s)>0$ is given. We write $H^{\nicefrac{1}{4}}(\cdot,\cdot ;s)$ as
\begin{align*}
H^{\nicefrac{1}{4}}(x,y ;s) = \int\limits_{0}^{\infty} \psi(x,y,\rho) d\rho \,\, \text{ for all } x,y\in W
\end{align*}
with
\begin{align*}
\psi(x,y,\rho):=  \frac{1}{\pi^2} \, Q_{ \sqrt{s} - \frac{1}{2}}^{-i\rho}(\cosh(a))  Q_{ \sqrt{s} - \frac{1}{2}}^{i\rho}(\cosh(b)) &\Bigg( \frac{\sinh(\pi\rho)}{\sinh(\gamma\rho)}\cosh(\rho(\gamma-\alpha-\beta)) \\
&- \frac{\sinh((\pi-\gamma)\rho)}{\sinh(\gamma\rho)}\cosh((\alpha-\beta)\rho) \Bigg)
\end{align*}
for all $x,y\in W$ and $\rho>0$. We also set $\psi(x,y,0):=\lim_{\rho\searrow 0}\psi(x,y,\rho)$ for all $x,y\in W$, where the limit obviously exists.

First of all, note that for all $x,y\in W$ the function $\rho\mapsto \psi(x,y,\rho)$ is absolutely integrable over $[0,\infty)$ (see Theorem \ref{theorem:TheoremZentralIntegral}). Thus, the functions
\begin{align*}
q: W\times W \ni (x,y)\mapsto \int\limits_{0}^{\infty} \vert \psi(x,y,\rho) \vert \, d\rho \in [0,\infty)
\end{align*}
and $H^{\nicefrac{1}{4}}(\cdot,\cdot;s)$ are continuous, which is an immediate consequence of Lebesgue's dominated convergence theorem and \eqref{equation:EstimateProductLegendreFunctions}.

Suppose now $s>0$. Note that $Q_{ \sqrt{s} - \frac{1}{2}}^{-i\rho}(\cosh(a))  Q_{ \sqrt{s} - \frac{1}{2}}^{i\rho}(\cosh(b))$ is real valued for all $a,b>0$ and $\rho\geq 0$. This can be seen from \eqref{equation:ConditionsSchritt1.1} and \citep[\S 14.20]{NIST} (see the comment preceding formula $14.20.6$ of \citep{NIST}). Thus, $\psi(x,y,\rho)$ is real valued for all $x,y\in W$ and $\rho\geq 0$ and, consequently, $H^{\nicefrac{1}{4}}(x,y ;s)$ is also real valued for all $x,y\in W$. Also note that for all $\rho\geq0$ the function $W\times W \ni (x,y)\mapsto \psi(x,y,\rho)\in\mathbb{R}$ is smooth. In the following, we deal with the other (less obvious) properties of $H^{\nicefrac{1}{4}}$ stated in $(ii)$.

We suppose now $s>\frac{1}{4}$. First, we show that for all $x\in W$ the function $H^{\nicefrac{1}{4}}(x,\cdot ;s)$ is smooth, which can be seen as follows: For all $x\in W$ the function $W \ni y \mapsto H^{\nicefrac{1}{4}}(x,y;s)\in\mathbb{R}$ belongs to $L_{\text{loc}}^2(W)$ as any continuous function, this means (by definition of the space $L_{\text{loc}}^2(W)$, see \citep[p. 98]{Grigoryan}), $H^{\nicefrac{1}{4}}(x,\cdot;s)\in L^2(U)$ for any relatively compact open set $U\subset W$. Moreover, for all $f\in C_{c}^{\infty}(W)$ and $x\in W$
\begin{align}
\int\limits_W H^{\nicefrac{1}{4}}(x,y;s) \cdot (sf+\Delta^{\nicefrac{1}{4}} f)(y) \,dy &= \int\limits_W \left( \int\limits_{0}^{\infty} \psi(x,y,\rho) d\rho \right) \cdot (sf+\Delta^{\nicefrac{1}{4}} f) (y) \, dy \nonumber \\
& =  \int\limits_{0}^{\infty}  \int\limits_W \psi(x,y,\rho) \cdot (sf+\Delta^{\nicefrac{1}{4}} f) (y) \, dy  \,    d\rho \nonumber \\
& = \int\limits_{0}^{\infty}  \int\limits_W  \underbrace{(s\psi(x,\cdot,\rho) + \Delta^{\nicefrac{1}{4}} \psi(x,\cdot,\rho))}_{\equiv 0}(y) \cdot f(y) \, dy  \,    d\rho \nonumber \\
&= 0. \label{equation:HSmoothFunction}
\end{align}
Note that we used the Fubini-Tonelli theorem for the second equality which is possible because $q(x,\cdot)$ is continuous for all $x\in W$, and $\vert(s+\Delta^{\nicefrac{1}{4}})f \vert$ is compactly supported as well as bounded, and thus
\begin{align*}
\int\limits_W  \int\limits_{0}^{\infty} \vert \psi(x,y,\rho)\cdot(sf+\Delta^{\nicefrac{1}{4}} f)(y)\vert  \,d\rho \,  dy = \int\limits_W q(x,y) \cdot \vert(s f+\Delta^{\nicefrac{1}{4}} f)(y) \vert \, dy <\infty.
\end{align*}
Further, we used ``Green's formula'' for the third equality (see \citep[Theorem 3.16]{Grigoryan}). Lastly, note that $(s+\Delta^{\nicefrac{1}{4}})\psi(x,\cdot,\rho)\equiv 0$ by the discussion preceding Lemma \ref{lemma:PropertiesOfH}. Thus, by elliptic regularity (see \citep[Corollary 7.3]{Grigoryan}), the function $y\mapsto H^{\nicefrac{1}{4}}(x,y;s)$ is smooth. 

Note that, since $H^{\nicefrac{1}{4}}(x,y;s)=H^{\nicefrac{1}{4}}(y,x;s)$ for all $x,y\in W$ (recall equation \eqref{equation:SymmetrieProduktLegendreQ}), the function $W\ni x\mapsto H^{\nicefrac{1}{4}}(\cdot,y;s)\in\mathbb{R}$ is smooth for all $y\in W$, too.

Furthermore, by \eqref{equation:HSmoothFunction} and Green's formula, it follows that $(s+\Delta^{\nicefrac{1}{4}})H^{\nicefrac{1}{4}}(x,\cdot;s)\equiv 0$ for arbitrary $x\in W$. More precisely, for all $x\in W$ and $f\in C_{c}^{\infty}(W)$:
\begin{align*}
0 = \int\limits_W H^{\nicefrac{1}{4}}(x,y;s) \cdot (sf+\Delta^{\nicefrac{1}{4}} f)(y) \,dy = \int\limits_W f(y) \cdot (sH^{\nicefrac{1}{4}}(x,\cdot;s)+\Delta^{\nicefrac{1}{4}} H^{\nicefrac{1}{4}}(x,\cdot;s))(y) dy.
\end{align*}
Because $f\in C_{c}^{\infty}(W)$ was arbitrary and $H^{\nicefrac{1}{4}}(x,\cdot;s)$ is smooth, it follows for all $x\in W$: $(s+\Delta^{\nicefrac{1}{4}})H^{\nicefrac{1}{4}}(x,\cdot;s)\equiv 0$ (see \citep[Lemma 3.13]{Grigoryan}). Finally, because of the symmetry $H^{\nicefrac{1}{4}}(x,y;s)=H^{\nicefrac{1}{4}}(y,x;s)$ for all $x,y\in W$, we also have $(s+\Delta^{\nicefrac{1}{4}})H^{\nicefrac{1}{4}}(\cdot,y;s)\equiv 0$ for arbitrary $y\in W$.

$(iii)$. Let $s>\frac{1}{4}$ and let $y\in W$ be fixed. We want to show that the function $W\ni x\mapsto H^{\nicefrac{1}{4}}(x,y;s)\in\mathbb{R}$ can be extended continuously to $\overline{W}$. Suppose $x\in\partial W$ is any boundary point other than the vertex of $W$, i.e. in polar coordinates $x=(a,\alpha)$ with $a\in (0,\infty)$ and $\alpha\in \{ 0, \gamma \}$. For those boundary points, we define the value of $H^{\nicefrac{1}{4}}(x,y;s)$ by the formula given on the right-hand side of \eqref{equation:SolutionH}. Note that the integral which is obtained after setting $\alpha=0$ or $\alpha=\gamma$ on the right-hand side of \eqref{equation:SolutionH} is absolutely convergent (see Theorem \ref{theorem:TheoremZentralIntegral}). Further, that continuation indeed provides a continuous function $W\backslash\{P\}\ni x\mapsto H^{\nicefrac{1}{4}}(x,y;s)\in\mathbb{R}$ (recall that $P$ denotes the vertex of the wedge). This can easily be seen by using Lebesgue's dominated convergence theorem and \eqref{equation:EstimateProductLegendreFunctions}. Lastly, note that $H^{\nicefrac{1}{4}}(x,y;s) = G_{\mathbb{H}^2}^{\nicefrac{1}{4}}(x,y;s)$ for all $x\in \partial W\backslash\{P\}$, which is obvious when using \eqref{equation:IdentitaetCosHyperbolicusMonster}.

On the contrary, it is more challenging to show that $H^{\nicefrac{1}{4}}(\cdot,y;s)$ can also be extended continuously at the vertex $P$. Note for example, that the formula on the right-hand side of \eqref{equation:SolutionH} has no meaning for $a=0$ because $Q^{-i\rho}_{\sqrt{s}-\frac{1}{2}}(z)$ is formally not defined for $z=1$ (also the limit of $Q^{-i\rho}_{\sqrt{s}-\frac{1}{2}}(\cosh(a))$ as $a\searrow 0$ does not exist). However, we will show that
\begin{align}
\label{equation:WertSpitzenpunkt}
\lim\limits_{x\rightarrow P} H^{\nicefrac{1}{4}}(x,y;s) = \frac{1}{2\pi} Q_{\sqrt{s}-\frac{1}{2}}(\cosh(b)).
\end{align}
For that purpose, let us write $H^{\nicefrac{1}{4}}(x,y;s) = I_1(x) - I_2(x)$ for all $x\in W$ with
\begin{align*}
I_1(x):&=\frac{1}{\pi^2}\int\limits_{0}^{\infty} Q_{ \sqrt{s} - \frac{1}{2}}^{-i\rho}(\cosh(a))  Q_{ \sqrt{s} - \frac{1}{2}}^{i\rho}(\cosh(b))\cdot \frac{\sinh(\pi\rho)}{\sinh(\gamma\rho)}\cosh(\rho(\gamma-\alpha-\beta)) d\rho,\\
I_2(x):&=\frac{1}{\pi^2}\int\limits_{0}^{\infty} Q_{ \sqrt{s} - \frac{1}{2}}^{-i\rho}(\cosh(a))  Q_{ \sqrt{s} - \frac{1}{2}}^{i\rho}(\cosh(b))\cdot \frac{\sinh((\pi-\gamma)\rho)}{\sinh(\gamma\rho)}\cosh((\alpha-\beta)\rho) d\rho.
\end{align*}
First, we will consider the limit of $I_1(x)$ as $x\rightarrow P$. Note that $I_1$ can be written as
\begin{align*}
I_1(x) = \frac{1}{2}\cdot \frac{2}{\pi} \int\limits_{0}^{\infty} Q_{ \sqrt{s} - \frac{1}{2}}^{-i\rho}(\cosh(a))  Q_{ \sqrt{s} - \frac{1}{2}}^{i\rho}(\cosh(b))\cdot \frac{\sin(i\pi\rho)}{\pi} \cdot \frac{\cos(i\rho(\gamma-\alpha-\beta))}{\sin(i\gamma\rho)} d\rho.
\end{align*}
Since we are only interested in the limit of $I_1(x)$ as $x=(a,\alpha)\rightarrow P$ we may assume in the followig $a<b$. The idea is to apply Theorem \ref{theorem:TheoremZentralIntegral} with $g(z)=\frac{\cos(z(\gamma-\alpha-\beta))}{\sin(z\gamma)}$. The only singularities of $g$ are simple poles at $k\frac{\pi}{\gamma}$ with $k\in\mathbb{Z}$ with residue $\Res(g;k\frac{\pi}{\gamma})=\frac{(-1)^k}{\gamma}\cos\left(\frac{\pi k}{\gamma}(\gamma-\alpha-\beta)\right)$. Hence, by Theorem \ref{theorem:TheoremZentralIntegral},
\begin{align}
I_1(x) =\, &\sum\limits_{k=1}^{\infty} \frac{(-1)^k}{\gamma} \cos\left(\frac{\pi k}{\gamma}(\gamma-\alpha-\beta)\right)e^{-i\pi k \frac{\pi}{\gamma}}P_{ \sqrt{s} - \frac{1}{2}}^{-k\frac{\pi}{\gamma}}(\cosh(a))  Q_{ \sqrt{s} - \frac{1}{2}}^{k\frac{\pi}{\gamma}}(\cosh(b)) \,+ \nonumber\\
&+ \frac{1}{2\gamma} P_{ \sqrt{s} - \frac{1}{2}}(\cosh(a))  Q_{ \sqrt{s} - \frac{1}{2}}(\cosh(b)). \label{equation:TheoremAufI1}
\end{align}
Note that
\begin{align}
\label{equation:BehaviourSingularityOneLegendreP}
\lim\limits_{a\searrow 0}P_{ \sqrt{s} - \frac{1}{2}}(\cosh(a)) = 1,
\end{align}
(see \citep[formula 14.8.7 on p. 361]{NIST}) and thus
\begin{align*}
\lim\limits_{a\searrow 0} \frac{1}{2\gamma} P_{ \sqrt{s} - \frac{1}{2}}(\cosh(a))  Q_{ \sqrt{s} - \frac{1}{2}}(\cosh(b)) = \frac{1}{2\gamma} Q_{ \sqrt{s} - \frac{1}{2}}(\cosh(b)).
\end{align*}
Now, we want to show that the series appearing on the right-hand side of \eqref{equation:TheoremAufI1} converges to zero as $a\searrow 0$, uniformly for all $\alpha\in (0,\gamma)$. By Lemma \ref{lemma:SecondEstimateProductLegendreFunctions} $(i)$, there exist constants $C,D>0$ such that for all $\mu>0$ and $a\in\mathbb{R}$ with $0<a< 2\arsinh\left(\frac{1}{2D}\right) <\frac{1}{D}$ (the last inequality follows because $\arsinh(x)<x$ for all $x>0$):
\begin{align*}
\vert  P_{\sqrt{s}-\frac{1}{2}}^{-\mu}(\cosh(a))Q_{\sqrt{s}-\frac{1}{2}}^{\mu}(\cosh(b)) &\vert \leq Ce^{\sqrt{s}a} \cdot  \mu^{\sqrt{s}} \cdot \left(D\sinh\left(\frac{a}{2}\right)\right)^{\mu}\\
&\leq Ce^{\frac{\sqrt{s}}{D}}\left(\frac{1}{2}\right)^{\frac{\mu}{2}} \cdot  \mu^{\sqrt{s}} \cdot \left(D\sinh\left(\frac{a}{2}\right)\right)^{\frac{\mu}{2}}\\
&\leq \hat{C}\cdot \left(D\sinh\left(\frac{a}{2}\right)\right)^{\frac{\mu}{2}},
\end{align*}
where $\hat{C}:=Ce^{\frac{\sqrt{s}}{D}}\cdot \sup_{\mu>0}\big\{ \left(\frac{1}{2}\right)^{\frac{\mu}{2}}\mu^{\sqrt{s}} \,\big\} \in (0,\infty)$. Thus, for all $\alpha\in (0,\gamma)$ and $a\in\mathbb{R}$ with $0<a< 2  \arsinh\left(\frac{1}{2D}\right)$, the absolute value of the series in \eqref{equation:TheoremAufI1} is bounded from above by
\begin{align*}
&\quad\, \sum\limits_{k=1}^{\infty} \bigg\vert \frac{(-1)^k}{\gamma} \cos\left(\frac{\pi k}{\gamma}(\gamma-\alpha-\beta)\right)e^{-i\pi k \frac{\pi}{\gamma}}P_{ \sqrt{s} - \frac{1}{2}}^{-k\frac{\pi}{\gamma}}(\cosh(a))  Q_{ \sqrt{s} - \frac{1}{2}}^{k\frac{\pi}{\gamma}}(\cosh(b)) \bigg\vert \\
&\leq \, \frac{\hat{C}}{\gamma}\sum\limits_{k=1}^{\infty}\left(\left(D\sinh\left(\frac{a}{2}\right)\right)^{\frac{\pi}{2\gamma}}\right)^{k} = \frac{\hat{C}}{\gamma} \cdot \frac{ \left( D\sinh\left(\frac{a}{2}\right)\right)^{\frac{\pi}{2\gamma}}}{1-\left(D\sinh\left(\frac{a}{2}\right)\right)^{\frac{\pi}{2\gamma}}}.
\end{align*}
Obviously, that upper bound holds uniformly for all $\alpha\in(0,\gamma)$ and converges to $0$ as $a\searrow 0$. We conclude 
\begin{align}
\label{equation:LimitI1}
\lim\limits_{x\rightarrow P}I_1(x)=\frac{1}{2\gamma} Q_{ \sqrt{s} - \frac{1}{2}}(\cosh(b)).
\end{align}

Next, we consider $I_2(x)$ as $x\rightarrow P$. If $\gamma=\pi$ then $I_2(x) = 0$ for all $x\in W$ and thus \eqref{equation:WertSpitzenpunkt} follows. Suppose now $\gamma\neq \pi$. We write $I_2$ as
\begin{align*}
I_2(x) = \frac{1}{\pi}\int\limits_{0}^{\infty}Q_{ \sqrt{s} - \frac{1}{2}}^{-i\rho}(\cosh(a))  Q_{ \sqrt{s} - \frac{1}{2}}^{i\rho}(\cosh(b)) \frac{\sin(i\pi \rho)}{\pi} \cdot \frac{\sin((\pi-\gamma)i\rho)\cos((\alpha-\beta)i\rho)}{\sin(i\pi \rho)\sin(\gamma i\rho)} d\rho.
\end{align*}
This time, we will apply Theorem \ref{theorem:TheoremZentralIntegral} with
\begin{align*}
g(z)=\frac{\sin((\pi-\gamma)z)\cos((\alpha-\beta)z)}{\sin(\pi z)\sin(\gamma z)} = \frac{\cos(\gamma z)\cos((\alpha-\beta)z)}{\sin(\gamma z)}-\frac{\cos(\pi z)\cos((\alpha-\beta)z)}{\sin(\pi z)}.
\end{align*}
The singularities of that function are simple poles and the set of all poles is given by $ \mathbb{Z} \cup \big\{ \, k\frac{\pi}{\gamma} \mid k\in\mathbb{Z} \,\big\}$. Let $(p_{\ell})_{\ell=1}^{\infty}$ be the sequence of all positive poles such that $0<p_1<p_2<p_3<...,$ and thus
\begin{align*}
I_2(x)=&\sum\limits_{\ell=1}^{\infty} \text{Res}\left( g;p_{\ell} \right) e^{-i\pi p_{\ell}}P_{\sqrt{s}-\frac{1}{2}}^{-p_{\ell}}(\cosh(a))Q_{\sqrt{s}-\frac{1}{2}}^{p_{\ell}}(\cosh(b)) \\
&+ \frac{1}{2}\left( \frac{1}{\gamma} - \frac{1}{\pi} \right)P_{\sqrt{s}-\frac{1}{2}}(\cosh(a))Q_{\sqrt{s}-\frac{1}{2}}(\cosh(b)).
\end{align*}
As above, one can show that the series on the right-hand side converges to $0$ as $a\searrow 0$, uniformly for all $\alpha\in (0,\gamma)$. Just note that the set of all residue $\{ \text{Res}\left( g;p_{\ell} \right)  \}_{\ell\in\mathbb{N}}$ is bounded. Hence, we obtain together with \eqref{equation:BehaviourSingularityOneLegendreP}:
\begin{align*}
\lim\limits_{x\rightarrow P}I_2(x) = \frac{1}{2}\left( \frac{1}{\gamma} - \frac{1}{\pi} \right)Q_{\sqrt{s}-\frac{1}{2}}(\cosh(b)).
\end{align*}

In particular, we have shown that
\begin{align*}
\lim\limits_{x\rightarrow P} H^{\nicefrac{1}{4}}(x,y;s) = \lim\limits_{x\rightarrow P} ( I_1(x) - I_2(x) ) = \frac{1}{2\pi}Q_{\sqrt{s}-\frac{1}{2}}(\cosh(b))=:H^{\nicefrac{1}{4}}(P,y;s).
\end{align*}
Note that we also have $\frac{1}{2\pi}Q_{\sqrt{s}-\frac{1}{2}}(\cosh(b))=G_{\mathbb{H}^2}^{\nicefrac{1}{4}}(P,y;s)$, which follows directly from \eqref{equation:GreenPlane1} because of $d(P,y)=b$.


It remains to show $0 < H^{\nicefrac{1}{4}}(x,y;s)$ for all $x\in W$ and $H^{\nicefrac{1}{4}}(x,y;s)\leq G_{\mathbb{H}^2}^{\nicefrac{1}{4}}(x,y;s)$ for all $x\in W\backslash\{y\}$. To prove those estimates, we will use the following properties of $H^{\nicefrac{1}{4}}$ and $G_{\mathbb{H}^2}^{\nicefrac{1}{4}}$, respectively. First, for all sequences $(x_n)_{n\in\mathbb{N}}$ such that $x_n\in W\backslash\{ y \}$ and $\lim_{n\rightarrow \infty}d(x_n,P)=\infty$: $\lim_{n\rightarrow\infty}G_{\mathbb{H}^2}^{\nicefrac{1}{4}}(x_n,y;s)=0$ as well as $\lim_{n\rightarrow\infty}H^{\nicefrac{1}{4}}(x_n,y;s)=0$. The former limit is easy to check since $G_{\mathbb{H}^2}^{\nicefrac{1}{4}}(x_n,y;s)=\frac{1}{2\pi}Q_{\sqrt{s}-\frac{1}{2}}(\cosh(d(x_n,y)))$ by \eqref{equation:GreenPlane1}, $\lim_{n\rightarrow \infty}d(x_n,y)=\infty$ by our assumptions on $(x_n)_{n\in\mathbb{N}}$, and $Q_{\sqrt{s}-\frac{1}{2}}(z_n)$ converges to $0$ for all sequences $(z_n)_{n\in\mathbb{N}}$ with $z_n\in (1,\infty)$ and $\lim_{n\rightarrow\infty} z_n =\infty$ (see \citep[14.8.15]{NIST}). For the other limit note that, for all sequences $x_n=(a_n,\alpha_n)\in W$ as above, we have by \eqref{equation:EstimateProductLegendreFunctions} the estimate $\vert H^{\nicefrac{1}{4}}(x_n,y;s)\vert\leq \frac{\hat{D}}{\sqrt{a_n-1}}$ for some constant $\hat{D}>0$. Since $\lim_{n\rightarrow \infty}d(x_n,P)=\infty$, we have $\lim_{n\rightarrow\infty}a_n = \infty$ and thus $\lim_{n\rightarrow\infty}H^{\nicefrac{1}{4}}(x_n,y;s)=0$. Second, note that the heat kernel $K_{\mathbb{H}^2}$ is strictly positive, which can be seen, for example, from \eqref{equation:HyperHeat2}. Consequently $G_{\mathbb{H}^2}^{\nicefrac{1}{4}}(x,y;s)>0$ for all $x\in W\backslash\{ y\}$.

To prove positivity of $H^{\nicefrac{1}{4}}$, let us assume that there exists some $x_{0}\in W$ with $H^{\nicefrac{1}{4}}(x_{0},y;s) < 0$. We choose some radius $R>0$ such that $d(x_0,P) < R$ and $H^{\nicefrac{1}{4}}(x_{0},y;s) < H^{\nicefrac{1}{4}}(x,y;s)$ for all $x \in W$ with $d(x,P) \geq R$. The latter condition can be achieved due to $\lim_{d(x,P)\rightarrow\infty}H^{\nicefrac{1}{4}}(x,y;s) = 0$ as shown above. Now, with $U:=B_R(P)\cap W$, we have $H^{\nicefrac{1}{4}}(\cdot,y;s)_{\vert U}\in C(\overline{U})\cap C^{\infty}(U)$ (which means, by definition, the function belongs to $C^{\infty}(U)$ and can be extended continuously to the closure $\overline{U}$). Furthermore, by our assumptions on $R$, we have $H^{\nicefrac{1}{4}}(x_0,y;s)< H^{\nicefrac{1}{4}}(x,y;s)$ for all $x\in \partial U$ (recall that $H^{\nicefrac{1}{4}}(x,y;s)=G_{\mathbb{H}^2}^{\nicefrac{1}{4}}(x,y;s)>0$ for all $x\in \partial W$). This is a contradiction to the elliptic minimum principle (see \citep[Corollary 8.16]{Grigoryan} and recall that $H^{\nicefrac{1}{4}}(\cdot,y;s)$ satisfies $(sH^{\nicefrac{1}{4}}(\cdot,y;s)+\Delta^{\nicefrac{1}{4}} H^{\nicefrac{1}{4}}(\cdot,y;s))(x)=0$ for all $x\in U$). Thus, we conclude $H^{\nicefrac{1}{4}}(x,y;s)\geq 0$ for all $x\in W$. Moreover, by the \emph{strong} elliptic minimum principle (see \citep[Corollary 8.14]{Grigoryan}), it follows that $H^{\nicefrac{1}{4}}(x,y;s)> 0$ for all $x\in W$.

The second estimate is shown similarly. Note that $u:W\backslash\{y\}\ni x\mapsto G_{\mathbb{H}^2}^{\nicefrac{1}{4}}(x,y;s)-H^{\nicefrac{1}{4}}(x,y;s) \in \mathbb{R}$ is smooth and can be extended continuously to $\overline{W}\backslash\{y\}$ by $u(x):=0$ for all $x\in\partial W$. Moreover, $\lim_{d(x,P)\rightarrow\infty}u(x)=0$ (as shown above) and $\lim_{x\rightarrow y}u(x)=\infty$. The latter limit holds because $H^{\frac{1}{4}}(\cdot,y;s)$ is a bounded function on $W$ and $\lim_{x\rightarrow y} G_{\mathbb{H}^2}^{\nicefrac{1}{4}}(x,y;s) = \lim_{x\rightarrow y}\frac{1}{2\pi}Q_{\sqrt{s}-\frac{1}{2}}(\cosh(d(x,y)))=\infty$, where the last equality follows from \citep[formula (12.23)]{Olver}. Now, suppose there exists some $x_{0}\in W\backslash\{ y \}$ such that $u(x_0)<0$. We choose radii $R, \tau>0$ so that the following conditions are satisfied: $\overline{B_{\tau}(y)}\subset W$, $u(x)>0$ for all $x\in \overline{B_{\tau}(y)}\backslash\{y \}$, $u(x_0)<u(x)$ for all $x\in W$ with $d(x,P)\geq R$, and $x_0 \in \left(W\cap B_R(P) \right) \backslash \overline{B_{\tau}(y)}=:V$. Obviously, $V$ is a relatively compact open set and $u\in C(\overline{V})\cap C^{\infty}(V)$ satisfies $(s+\Delta^{\nicefrac{1}{4}})u(x) = 0$ for all $x\in V$. Moreover $x_0\in V$ is some interior point with the property $u(x_0)<u(x)$ for all $x\in\partial V$, which is a contradiction to the elliptic minimum principle  as above (see \citep[Corollary 8.16]{Grigoryan}). Thus, we conclude $u \geq 0$, or equivalently, $G_{\mathbb{H}^2}^{\nicefrac{1}{4}}(x,y;s) \geq H^{\nicefrac{1}{4}}(x,y;s)$ for all $x\in W\backslash\{ y\}$.
\end{proof}

\begin{theorem}
\label{theorem:GreenWedge}
For all $(x,y)\in\offdiag(W)$ and $s\in\mathbb{C}$ with $\Re(s)>\frac{1}{4}$:
\begin{align}
\label{equation:GreenWedge}
\begin{split}
G_W^{\nicefrac{1}{4}}(x,y;s)&=\frac{1}{\pi^2}\int\limits_{0}^{\infty} Q_{\sqrt{s}-\frac{1}{2}}^{-i\rho}(\cosh(a))Q_{\sqrt{s}-\frac{1}{2}}^{i\rho}(\cosh(b))\cdot \bigg( \cosh(\rho(\pi-\vert \alpha-\beta \vert))-    \\
& \frac{\sinh(\pi\rho)}{\sinh(\gamma\rho)}\cosh(\rho\left( \gamma-\alpha-\beta \right)) + \frac{\sinh(\rho(\pi-\gamma))}{\sinh(\gamma\rho)}\cosh((\alpha-\beta)\rho)\bigg) d\rho,
\end{split}
\end{align}
where $x=(a,\alpha)$ and $y=(b,\beta)$ \emph{(}with respect to the polar coordinates chosen above\emph{)}.
\end{theorem}

\begin{proof}
For brevity, we denote the right-hand side of \eqref{equation:GreenWedge} by $\tilde{G}_W^{\nicefrac{1}{4}}(x,y;s)$. More precisely,
\begin{align*}
\tilde{G}_W^{\nicefrac{1}{4}}:&\offdiag(W) \times \mathcal{H}_{>\frac{1}{4}} \rightarrow\mathbb{C},\\
 &((x,y),s)\mapsto \tilde{G}_W^{\nicefrac{1}{4}}(x,y;s):=G_{\mathbb{H}^2}^{\nicefrac{1}{4}}(x,y;s) - H^{\nicefrac{1}{4}}(x,y;s),
\end{align*} 
where $H^{\nicefrac{1}{4}}$ is defined as in \eqref{equation:SolutionH}. Recall the following facts: For all $x,y\in W$ with $x\neq y$ the functions $G_{\mathbb{H}^2}^{\nicefrac{1}{4}}(x,y;\cdot)$ and $G_W^{\nicefrac{1}{4}}(x,y;\cdot)$ are holomorphic on $\mathcal{H}_{>\frac{1}{4}}$ (recall Proposition \ref{proposition:InverseLaplaceTransform}). Similary, for all $x,y\in W$ the function $H^{\nicefrac{1}{4}}(x,y;\cdot)$ is also holomorphic on $\mathcal{H}_{>\frac{1}{4}}$ which follows from Lemma \ref{lemma:PropertiesOfH} $(i)$. Hence, it suffices to prove the equality $G_W^{\nicefrac{1}{4}}(x,y;s) = \tilde{G}_W^{\nicefrac{1}{4}}(x,y;s)$ for all $s\in(\frac{1}{4},\infty)$ and $x,y\in W$ with $x\neq y$.

Let $s>\frac{1}{4}$ be fixed and let $f\in C_{c}^{\infty}(W)$ be given such that $f(x)\geq 0$ for all $x\in W$. We extend that function to $\mathbb{H}^2$ by $f(y):=0$ for all $y\in\mathbb{H}^2\backslash W$ and we denote the extended function also by $f$, so that it belongs to $C_{c}^{\infty}(\mathbb{H}^2)$. We define
\begin{align*}
&u_W:W\ni x\mapsto \int\limits_{W}G_W^{\nicefrac{1}{4}}(x,y;s)f(y)\, dy\in\mathbb{R} \\
\text{ and }\,\, &u_{\mathbb{H}^2}:\mathbb{H}^2\ni x\mapsto \int\limits_{\mathbb{H}^2}G_{\mathbb{H}^2}^{\nicefrac{1}{4}}(x,y;s)f(y) \, dy\in\mathbb{R},
\end{align*}
which are known to be smooth functions (see \citep[Theorem $8.7 (ii)$]{Grigoryan}). Moreover, one can show that $(s+\Delta^{\nicefrac{1}{4}})u_W(x)=f(x)$ for all $x\in W$, respectively $(s+\Delta^{\nicefrac{1}{4}})u_{\mathbb{H}^2}(x)=f(x)$ for all $x\in \mathbb{H}^2$ (see \citep[Theorem $8.4 (b)$]{Grigoryan}). Similarly, we define
\begin{align*}
\Phi:W\ni x\mapsto \int\limits_{W}H^{\nicefrac{1}{4}}(x,y;s)f(y) \, dy\in\mathbb{R}.
\end{align*}
Note that $\Phi$ is smooth and satisfies $(s+\Delta^{\nicefrac{1}{4}})\Phi(x)=0$ for all $x\in W$, which follows e.g. from Lemma \ref{lemma:PropertiesOfH} $(ii)$, \citep[Satz $5.7$ Zusatz on p. 148]{Elstrodt} and the fact that $f$ is compactly supported in $W$ (by assumption).

Next, we define
\begin{align*}
&\tilde{u}_W:W\ni x\mapsto \int\limits_{W}\tilde{G}_W^{\nicefrac{1}{4}}(x,y;s)f(y)\, dy\in\mathbb{R}
\end{align*}
and we will show $u_W(x)=\tilde{u}_W(x)$ for all $x\in W$ by applying the result in \citep[Exercise 8.2]{Grigoryan} (where the notation given in \citep[Exercise 8.2]{Grigoryan} corresponds to our notation as follows: $\alpha:=s-\frac{1}{4}$, $R_{\alpha}f:=u_W$ and $u:=\tilde{u}_W$). Let us show that the conditions of \citep[Exercise 8.2]{Grigoryan} are indeed satisfied by $\tilde{u}_W$:

First, note that $\tilde{u}_W = u_{\mathbb{H}^2}-\Phi$ and thus, by the discussion above, $\tilde{u}_W$ is smooth and satisfies $(s+\Delta^{\nicefrac{1}{4}})\tilde{u}_W(x)=f(x)$ for all $x\in W$. Moreover, $\tilde{u}_W$ is non-negative, since for all $x\in W$
\begin{align*}
\tilde{u}_W(x)=\int\limits_{W}\tilde{G}_W^{\nicefrac{1}{4}}(x,y;s)f(y) dy = \int\limits_{W}(G_{\mathbb{H}^2}^{\nicefrac{1}{4}}(x,y;s) - H^{\nicefrac{1}{4}}(x,y;s)) \cdot f(y) dy,
\end{align*}
where $f(y)\geq 0$ for all $y\in W$ (by assumption), and $\left(G_{\mathbb{H}^2}^{\nicefrac{1}{4}} - H^{\nicefrac{1}{4}}\right)(x,y;s)\geq 0$ for all $(x,y)\in\offdiag(W)$ by Lemma \ref{lemma:PropertiesOfH} $(iii)$. 
Lastly, we need to check whether $\lim_{n\rightarrow\infty}\tilde{u}_W(x_n)=0$ for all sequences $(x_n)_{n\in\mathbb{N}}\subset W$ such that either there exists some $x_{\ast}\in \partial W$ with $\lim_{n\rightarrow \infty}x_n=x_{\ast}$ or $\lim_{n\rightarrow\infty}d(P,x_n)=\infty$. That property follows immediately from Lebesgue's dominated convergence theorem and the following facts:

Recall that for all $y\in W$ the function $\tilde{G}_W^{\nicefrac{1}{4}}(\cdot,y;s):W\backslash\{ y \} \rightarrow \mathbb{R}$ can be extended continuously to $\overline{W}\backslash\{ y \}$ by $\tilde{G}_W^{\nicefrac{1}{4}}(x,y;s):=0$ for all $x\in\partial W$ (see Lemma \ref{lemma:PropertiesOfH} $(iii)$). Further, we have shown in the proof of Lemma \ref{lemma:PropertiesOfH} $(iii)$ that $\lim_{d(x,P)\rightarrow\infty}\tilde{G}_W^{\nicefrac{1}{4}}(x,y;t)=0$ for all $y\in W$. Lastly, for all $\varepsilon >0$ there exists some constant $D>0$ such that for all $x,y\in W$ with $d(x,y)\geq \varepsilon$:
\begin{align*}
0\leq \tilde{G}_W^{\nicefrac{1}{4}}(x,y;s)\leq G_{\mathbb{H}^2}^{\nicefrac{1}{4}}(x,y;s) \leq D.
\end{align*}
The first estimate follows from Lemma \ref{lemma:PropertiesOfH} $(iii)$ and for the second estimate note that there exists some constant $C>0$ such that
\begin{align}
\label{equation:EstimateHeatKernelPlane}
K_{\mathbb{H}^2} (x,y;t) \leq \frac{C}{t} \cdot e^{-\frac{d(x,y)^2}{8t}},\quad \forall x,y\in\mathbb{H}^2,\, t>0
\end{align}
(see e.g. \citep[Lemma 7.4.26]{Buser}).
Thus, it follows for all $x,y\in \mathbb{H}^2$ with $d(x,y)\geq \epsilon$:
\begin{align*}
G_{\mathbb{H}^2}^{\nicefrac{1}{4}}(x,y;s) &= \int\limits_{0}^{\infty} e^{(\frac{1}{4}-s)t} K_{\mathbb{H}^2} (x,y;t) dt \leq \int\limits_{0}^{\infty} e^{(\frac{1}{4}-s)t} \frac{C}{t} \cdot e^{-\frac{\varepsilon^2}{8t}} dt=:D.
\end{align*}

Thus, by Lebesgue's dominated convergence theorem we have $\lim_{n\rightarrow\infty}\tilde{u}_W(x_n)=0$ for any sequence as above. Using \citep[Exercise 8.2]{Grigoryan} we have $u_W(x)=\tilde{u}_W(x)$ for all $x\in W$. 

Since $f\in C_{c}^{\infty}(W)$ with $f\geq 0$ was arbitrary, we conclude that for all $x\in W$: $G_W^{\nicefrac{1}{4}}(x,y;s)=G_{\mathbb{H}^2}^{\nicefrac{1}{4}}(x,y;s)-H^{\nicefrac{1}{4}}(x,y;s)$ for almost all $y\in W$. Further, for any $x\in W$, both sides are continuous in $y\in W\backslash\{x\}$ and thus they must be equal for all $y\in W\backslash\{x\}$.

%

\end{proof}

By definition, the (shifted) Green's function is defined as the Laplace transform of the (shifted) heat kernel. Thus we obtain a formula for the (shifted) heat kernel from Theorem \ref{theorem:GreenWedge}.

\begin{corollary}
\label{corollary:FormelHeatKernelWedgeShift}
For all $t>0$ and $x\in W$
\begin{align}
\label{equation:FormelHeatKernelWedgeShift}
&K_W^{\nicefrac{1}{4}}(x,x;t) = K_{\mathbb{H}^2}^{\nicefrac{1}{4}}(x,x;t) \, - \mathcal{L}^{-1}\left\lbrace s\mapsto H^{\nicefrac{1}{4}}(x,x;s)\right\rbrace (t),
\end{align}
where $H^{\nicefrac{1}{4}}$ is the same function as in Lemma \emph{\ref{lemma:PropertiesOfH}}.
 
Moreover, for all $x\in W$ the function $(0,\infty)\ni t\mapsto  K_{\mathbb{H}^2}^{\nicefrac{1}{4}}(x,x;t) - K_W^{\nicefrac{1}{4}}(x,x;t) \in\mathbb{R}$ is non-negative, can be extended continuously at $t=0$, its Laplace integral is \emph{(}absolutely\emph{)} convergent and equal to $H^{\nicefrac{1}{4}}(x,x;s)$ for all $s\in\mathcal{H}_{>\frac{1}{4}}$. The above inverse Laplace transform can be computed by the inversion formula \eqref{equation:InverseLaplaceTransform} with $\varepsilon>\frac{1}{4}$.
\end{corollary}

\begin{proof}
Let $x\in W$ be arbitrary. By definition, for all $y\in W$ with $y\neq x$ and for all $s\in\mathcal{H}_{>\frac{1}{4}}$:
\begin{align*}
G_{W}^{\nicefrac{1}{4}}(x,y;s)=\mathcal{L}\lbrace K_{W}^{\nicefrac{1}{4}}(x,y;\cdot) \rbrace(s)\, \text{ and }\, G_{\mathbb{H}^2}^{\nicefrac{1}{4}}(x,y;s) = \mathcal{L}\lbrace K_{\mathbb{H}^2}^{\nicefrac{1}{4}}(x,y;\cdot) \rbrace(s).
\end{align*}
Furthermore, by Theorem \ref{theorem:GreenWedge}, we have $H^{\nicefrac{1}{4}}(x,y;s)=G_{\mathbb{H}^2}^{\nicefrac{1}{4}}(x,y;s)-G_{W}^{\nicefrac{1}{4}}(x,y;s)$ for all $y\in W$ with $y\neq x$ and $s\in\mathcal{H}_{>\frac{1}{4}}$. Thus, for all $y\in W$ with $y\neq x$ and $t>0$:
\begin{align*}
\mathcal{L}^{-1}\left\lbrace s\mapsto H^{\nicefrac{1}{4}}(x,y;s)\right\rbrace (t) = K_{\mathbb{H}^2}^{\nicefrac{1}{4}}(x,y;t) - K_W^{\nicefrac{1}{4}}(x,y;t).
\end{align*}
We will show that the above equation is also valid for $y=x$. For that purpose we consider $u:W\times [0,\infty)\rightarrow \mathbb{R}$ defined for all $(y,t)\in W\times [0,\infty)$ as
\begin{align*}
u(y,t):=\begin{cases}
K_{\mathbb{H}^2}^{\nicefrac{1}{4}}(x,y;t)-K_{W}^{\nicefrac{1}{4}}(x,y;t),&\text{ if }t>0\\
0,& \text{ if }t=0.
\end{cases}
\end{align*}
Obviously, for all $y\in W$ and $t>0$ we have $0\leq K_{W}^{\nicefrac{1}{4}} (x,y;t)\leq K_{\mathbb{H}^2}^{\nicefrac{1}{4}}(x,y;t)$, and thus $0\leq u(y,t)\leq K_{\mathbb{H}^2}^{\nicefrac{1}{4}}(x,y;t)$. Moreover, $u$ is continuous (see \citep[Corollary $9.21$ and Exercise $9.7$]{Grigoryan}). Hence, the Laplace transform of $u(y,\cdot)$ exists for all $y\in W$ and $s\in\mathcal{H}_{>\frac{1}{4}}$. Just note that for all $y\in W$ and $s\in\mathcal{H}_{>\frac{1}{4}}$:
\begin{align*}
\int\limits_{0}^{\infty}e^{-st}u(y,t)dt \leq  \int\limits_{0}^{1}e^{-st}u(y,t)dt+\int\limits_{1}^{\infty}e^{-st}K_{\mathbb{H}^2}^{\nicefrac{1}{4}}(x,y;t)dt,
\end{align*}
where the first integral is convergent because of the continuity of $u(y,\cdot)$ on $[0,\infty)$, and the second integral is convergent due to \eqref{equation:EstimateHeatKernelPlane}.

Let $(y_n)_{n\in\mathbb{N}}\subset W\backslash\{x\}$ be a sequence such that $\lim_{n\rightarrow\infty}y_n = x$. Then for all $s\in\mathcal{H}_{>\frac{1}{4}}$:
\begin{align*}
\mathcal{L}\left\lbrace u(x,\cdot)\right\rbrace  (s) &=\int\limits_{0}^{\infty} e^{-st}(K_{\mathbb{H}^2}^{\nicefrac{1}{4}}(x,x;t)-K_{W}^{\nicefrac{1}{4}}(x,x;t)) dt\\
&=\lim\limits_{n\rightarrow\infty}\int\limits_{0}^{\infty}e^{-st}(K_{\mathbb{H}^2}^{\nicefrac{1}{4}}(x,y_n;t)-K_{W}^{\nicefrac{1}{4}}(x,y_n;t)) dt\\
&=\lim\limits_{n\rightarrow\infty} (G_{\mathbb{H}^2}^{\nicefrac{1}{4}}(x,y_n;s)-G_{W}^{\nicefrac{1}{4}}(x,y_n;s)) =\lim\limits_{n\rightarrow\infty} H^{\nicefrac{1}{4}}(x,y_n;s) \\
&=H^{\nicefrac{1}{4}}(x,x;s).
\end{align*}
Thus, by Proposition \ref{proposition:InverseLaplaceTransform}, $u(x,t)= \mathcal{L}^{-1}\{s\mapsto H^{\nicefrac{1}{4}}(x,x;s)\}(t)$ for all $t>0$. Note that we used continuity of $H^{\nicefrac{1}{4}}$ for the last equation (see Lemma \ref{lemma:PropertiesOfH} $(ii)$), and, for the second equation, we used Lebesgue's dominated convergence theorem which is allowed because of the following: First, for all compact $K\subset W$ with $x\in K$, $u_{\vert K\times [0,1]}$ is bounded by continuity of $u$ and also the function $K\times [1,\infty)\ni (y,t)\mapsto e^{-\frac{1}{4}t}u(y,t)\in\mathbb{R}$ is bounded because of $0\leq e^{-\frac{1}{4}t}u(y,t)\leq K_{\mathbb{H}^2}(x,y;t)$ and \eqref{equation:EstimateHeatKernelPlane}. Thus, $K\times [0,\infty)\ni (y,t)\mapsto e^{-\frac{1}{4}t}u(y,t)\in\mathbb{R}$ is bounded as well. In particular, there exists some constant $C>0$ such that for all $t\geq 0$ and $n\in\mathbb{N}$: $\vert e^{-st} u(x,y_n;t) \vert \leq C\cdot e^{(\frac{1}{4}-\Re(s))t}$. That upper bound is integrable over $[0,\infty)$ since, by assumption, $\Re(s)>\frac{1}{4}$.

Lastly, note that the inversion formula \eqref{equation:InverseLaplaceTransform} holds with $\varepsilon>\frac{1}{4}$ because of Proposition \ref{proposition:InverseLaplaceTransform} and because the Laplace integral of $u(x,\cdot)$ is absolutely convergent for all $s>\frac{1}{4}$, as explained above. 
\end{proof}


Now that we have established the above formula for the heat kernel $K_W$ we are in a position to investigate thoroughly the function stated at the beginning of this section. According to the discussion before, we introduce the shifted function
\begin{align}
\label{equation:DefinitionHeatTraceShift}
Z_{\gamma}^{\nicefrac{1}{4}}(t;r):=e^{\frac{t}{4}}\cdot Z_{\gamma}(t;r) = \int\limits_{B_r^{\mathbb{H}^2}\hspace{-0.5mm}(P)\cap W} K_W^{\nicefrac{1}{4}}(x,x;t) dx
\end{align}
and consequently maintain to work with the shifted function $Z_{\gamma}^{\nicefrac{1}{4}}(t;r)$ instead of $Z_{\gamma}(t;r)$ itself. We have the following formula for $Z_{\gamma}^{\nicefrac{1}{4}}(t;r):$

\begin{theorem}
\label{theorem:HeatTraceShift}
Let $r>0$ be arbitrary. For all $t>0$:
\begin{align}
\label{equation:HeatTraceShift}
\int\limits_{B_r^{\mathbb{H}^2}\hspace{-0.5mm}(P)\cap W} K_W^{\nicefrac{1}{4}}(x,x;t) dx = &\int\limits_{0}^{r} \int\limits_{0}^{\gamma} K_{\mathbb{H}^2}^{\nicefrac{1}{4}}((a,\alpha),(a,\alpha);t) \sinh(a) d\alpha da \, - \nonumber \\
& - 2\int\limits_{0}^{r}\int\limits_{0}^{\frac{\pi}{2}} K_{\mathbb{H}^2}^{\nicefrac{1}{4}}((a,\alpha),(a,-\alpha);t) \sinh(a) d\alpha da  \nonumber\\
& + \frac{\gamma}{2\pi}\int\limits_{0}^{\infty} \frac{e^{-\frac{u^2}{4t}}}{\sqrt{4\pi t}} \cdot \frac{e^{\frac{u}{2}}}{e^{u}-1} \cdot \left( \frac{\pi}{\gamma}\coth\left(  \frac{\pi}{\gamma} \cdot \frac{u}{2} \right) - \coth\left(\frac{u}{2}\right) \right)du  \nonumber \\
&-A_{\gamma}(t;r),
\end{align}
where
\begin{align}
\label{equation:DefinitionFunktionA_gamma}
A_{\gamma}(t;r):=\frac{\gamma}{\pi^2} \mathcal{L}^{-1}\Bigg\lbrace s\mapsto  \int\limits_{r}^{\infty} \int\limits_{0}^{\infty} Q_{ \sqrt{s} - \frac{1}{2}}^{-i\rho}(\cosh(a))&Q_{ \sqrt{s} - \frac{1}{2}}^{i\rho}(\cosh(a)) \cdot \nonumber \\
& \cdot\frac{\sinh(\rho (\pi-\gamma))}{\sinh(\rho \gamma)} \sinh(a)  d\rho\, da \Bigg\rbrace  (t).
\end{align}
\end{theorem}

\begin{proof}
In polar coordinates, the domain of integration is parametrised as
\begin{align*}
\hyperball{r}{P} \cap W = \{\, (a,\alpha) \mid 0<a<r, \, 0<\alpha<\gamma \,\}
\end{align*}
and the volume form is given by $dx=\sinh(a) da d\alpha$. From \eqref{equation:FormelHeatKernelWedgeShift} we get for all $t>0$:
\begin{align}
\label{equation:ShiftedHeatTraceWedgeDiagonal}
 Z_{\gamma}^{\nicefrac{1}{4}}(t;r) = &\int\limits_{0}^{r} \int\limits_{0}^{\gamma}  K_{\mathbb{H}^2}^{\nicefrac{1}{4}}((a,\alpha),(a,\alpha);t)\cdot \sinh(a) d\alpha da \nonumber \\
& - \int\limits_{0}^{r} \int\limits_{0}^{\gamma} \mathcal{L}^{-1}\Bigg\lbrace s\mapsto H^{\nicefrac{1}{4}}((a,\alpha),(a,\alpha);s) \Bigg\rbrace (t) \cdot \sinh(a) d\alpha da.
\end{align}
The first summand on the right-hand side (RHS) of equation \eqref{equation:ShiftedHeatTraceWedgeDiagonal} is exactly the same as the first term on the RHS of \eqref{equation:HeatTraceShift}. Let us simplify the second summand on the RHS of \eqref{equation:ShiftedHeatTraceWedgeDiagonal}. In order to get a better overview we postpone some of the lengthy calculations involved in the remaining part. We will fill in these gaps later by the Lemmas \ref{lemma:LemmaInverseLaplaceLegendreQ} -- \ref{lemma:InverseLaplaceTransformPsiFunktion} to which we will refer here whenever needed.

First, let $s>\frac{1}{4}$ be given. Then $Q_{ \sqrt{s} - \frac{1}{2}}^{-i\rho}(\cosh(a))\cdot Q_{ \sqrt{s} - \frac{1}{2}}^{i\rho}(\cosh(a))$ is non-negative for all $a>0$ and $\rho\geq 0$. This follows from \eqref{equation:ConditionsSchritt1.1} and \citep[\S 14.20]{NIST}. Thus, by \eqref{equation:SolutionH} and the Fubini-Tonelli theorem:
\begin{align}
&\int\limits_{0}^{r} \int\limits_{0}^{\gamma} H^{\nicefrac{1}{4}}((a,\alpha), (a,\alpha);s) \cdot \sinh(a) d\alpha \, da \nonumber \\
&=  \frac{1}{\pi^2}  \int\limits_{0}^{r} \int\limits_{0}^{\gamma} \int\limits_{0}^{\infty} Q_{ \sqrt{s} - \frac{1}{2}}^{-i\rho}(\cosh(a))  Q_{ \sqrt{s} - \frac{1}{2}}^{i\rho}(\cosh(a)) \Bigg( \frac{\sinh(\pi\rho)}{\sinh(\gamma\rho)}\cosh(\rho(\gamma-2\alpha)) \nonumber \nonumber \\
&\qquad\qquad\qquad\qquad\qquad\qquad\qquad\qquad\qquad\qquad\quad - \frac{\sinh((\pi-\gamma)\rho)}{\sinh(\gamma\rho)} \Bigg)d\rho \, \sinh(a) d\alpha \, da \nonumber \\
&=\frac{1}{\pi^2}  \int\limits_{0}^{r}  \int\limits_{0}^{\infty} Q_{ \sqrt{s} - \frac{1}{2}}^{-i\rho}(\cosh(a))  Q_{ \sqrt{s} - \frac{1}{2}}^{i\rho}(\cosh(a)) \int\limits_{0}^{\gamma} \Bigg( \frac{\sinh(\pi\rho)}{\sinh(\gamma\rho)} \cosh(\rho(\gamma-2\alpha)) \nonumber \nonumber \\
&\qquad\qquad\qquad\qquad\qquad\qquad\qquad\qquad\qquad\qquad\quad - \frac{\sinh((\pi-\gamma)\rho)}{\sinh(\gamma\rho)} \Bigg) d\alpha \, d\rho \, \sinh(a)  \, da \nonumber \\
&=\frac{1}{\pi^2}  \int\limits_{0}^{r}  \int\limits_{0}^{\infty} Q_{ \sqrt{s} - \frac{1}{2}}^{-i\rho}(\cosh(a))  Q_{ \sqrt{s} - \frac{1}{2}}^{i\rho}(\cosh(a)) \Bigg( \frac{\sinh(\pi\rho)}{\rho} \nonumber \\
&\qquad\qquad\qquad\qquad\qquad\qquad\qquad\qquad\qquad\qquad\quad - \gamma \frac{\sinh((\pi-\gamma)\rho)}{\sinh(\gamma\rho)} \Bigg)  d\rho \, \sinh(a)  \, da \nonumber \\
&=\frac{1}{2}  \int\limits_{0}^{r} P_{ \sqrt{s} - \frac{1}{2}}(\cosh(a))  Q_{ \sqrt{s} - \frac{1}{2}}(\cosh(a))\sinh(a) da \nonumber \\
& \qquad - \frac{ \gamma}{\pi^2}  \int\limits_{0}^{r}  \int\limits_{0}^{\infty} Q_{ \sqrt{s} - \frac{1}{2}}^{-i\rho}(\cosh(a))  Q_{ \sqrt{s} - \frac{1}{2}}^{i\rho}(\cosh(a)) \frac{\sinh((\pi-\gamma)\rho)}{\sinh(\gamma\rho)}   d\rho \, \sinh(a)  \, da, \label{equation:VertauschungIntegrationInverseLaplaceTrafor}
\end{align}
where we used Lemma \ref{lemma:LemmaInverseLaplaceLegendreQ} $(ii)$ for the last equality. Note that both integrals are finite due to \eqref{equation:SecondEstimateProductLegendreFunctions2} and \eqref{equation:EstimateProductLegendreFunctionsForSmallValuesOfZ}, respectively. Moreover the right-hand side is holomorphic in $s\in\mathbb{C}\backslash (-\infty, 0]$.

Thus, we obtain for all $s>\frac{1}{4}$:
\begin{align*}
\mathcal{L}\,&\bigg\lbrace t\mapsto \int_{\hyperball{r}{P} \cap W} K_{\mathbb{H}^2}^{\nicefrac{1}{4}}(x,x;t)-K_{W}^{\nicefrac{1}{4}}(x,x;t) dx \bigg\rbrace(s)\\
& = \int\limits_{0}^{\infty} e^{-st}\int\limits_{\hyperball{r}{P} \cap W} K_{\mathbb{H}^2}^{\nicefrac{1}{4}}(x,x;t)-K_{W}^{\nicefrac{1}{4}}(x,x;t) dx dt\\
& = \int\limits_{\hyperball{r}{P} \cap W} \int\limits_{0}^{\infty} e^{-st} \left( K_{\mathbb{H}^2}^{\nicefrac{1}{4}}(x,x;t)-K_{W}^{\nicefrac{1}{4}}(x,x;t)\right)  dt dx \\
& = \int\limits_{\hyperball{r}{P} \cap W} H^{\nicefrac{1}{4}}(x,x;s) dx,
\end{align*}
where we applied the Fubini-Tonelli theorem for the second equality (note that all functions are non-negative). Thus, we have for all $t>0$
\begin{align}
\label{equation:VertauschungIntegraleUndInverseLaplaceTrafo}
\mathcal{L}^{-1}\Bigg\lbrace s\mapsto \int\limits_{\hyperball{r}{P} \cap W} H^{\nicefrac{1}{4}}(x,x;s) dx \Bigg\rbrace (t) &=  \int\limits_{\hyperball{r}{P} \cap W} (K_{\mathbb{H}^2}^{\nicefrac{1}{4}}(x,x;t)-K_{W}^{\nicefrac{1}{4}}(x,x;t))dx \nonumber \\
& =  \int\limits_{\hyperball{r}{P} \cap W} \mathcal{L}^{-1}\Bigg\lbrace s\mapsto H^{\nicefrac{1}{4}}(x,x;s) \Bigg\rbrace (t)dx.
\end{align}
In other words, we may interchange the order of integration and $\mathcal{L}^{-1}$. Using \eqref{equation:VertauschungIntegraleUndInverseLaplaceTrafo} and \eqref{equation:VertauschungIntegrationInverseLaplaceTrafor}, we obtain for all $t>0$:
\begin{align*}
&\quad -\int\limits_{\hyperball{r}{P} \cap W} \mathcal{L}^{-1}\Bigg\lbrace s\mapsto H^{\nicefrac{1}{4}}(x,x;s) \Bigg\rbrace (t)dx \\
&=\, -\frac{1}{2}\, \mathcal{L}^{-1}\Bigg\lbrace s\mapsto \int\limits_{0}^{r} P_{ \sqrt{s} - \frac{1}{2}}(\cosh(a))  Q_{ \sqrt{s} - \frac{1}{2}}(\cosh(a))\sinh(a) da\Bigg\rbrace (t) \nonumber \\
&\quad + \frac{ \gamma}{\pi^2} \, \mathcal{L}^{-1}\Bigg\lbrace s\mapsto \int\limits_{0}^{\infty}  \int\limits_{0}^{\infty} Q_{ \sqrt{s} - \frac{1}{2}}^{-i\rho}(\cosh(a))  Q_{ \sqrt{s} - \frac{1}{2}}^{i\rho}(\cosh(a)) \frac{\sinh((\pi-\gamma)\rho)}{\sinh(\gamma\rho)}   d\rho \, \sinh(a)  \, da \Bigg\rbrace (t)\\
& \quad -\, A_{\gamma}(t;r)
\end{align*}

The above inverse Laplace transforms can be computed explicitly as we will show in Lemma \ref{lemma:InverseLaplace1}-\ref{lemma:InverseLaplaceTransformPsiFunktion} below. From Lemma \ref{lemma:InverseLaplace1} we know for all $t>0$:
\begin{align*}
&-\frac{1}{2}  \mathcal{L}^{-1}\left\lbrace \int\limits_{0}^{r} P_{ \sqrt{s} - \frac{1}{2}}(\cosh(a))  Q_{ \sqrt{s} - \frac{1}{2}}(\cosh(a)) \sinh(a) da \right\rbrace (t) \\
 = & -2\int\limits_{0}^{r} \int\limits_{0}^{\frac{\pi}{2}} K_{\mathbb{H}^2}^{\nicefrac{1}{4}}((a,\alpha),(a,-\alpha);t) \sinh(a) d\alpha da,
\end{align*}
which is equal to the second summand on the RHS of \eqref{equation:HeatTraceShift}.

For the second inverse Laplace transform, we combine Lemma \ref{lemma:IntegralProdLegendrePsiFunktion} below with Lemma \ref{lemma:InverseLaplaceTransformPsiFunktion} below to obtain
\begin{align*}
&I:=\frac{ \gamma}{\pi^2}  \mathcal{L}^{-1}\Bigg\lbrace s\mapsto \int\limits_{0}^{\infty}  \int\limits_{0}^{\infty} Q_{ \sqrt{s} - \frac{1}{2}}^{-i\rho}(\cosh(a))  Q_{ \sqrt{s} - \frac{1}{2}}^{i\rho}(\cosh(a)) \frac{\sinh((\pi-\gamma)\rho)}{\sinh(\gamma\rho)}   d\rho \, \sinh(a)  \, da \Bigg\rbrace (t)\\
& = \frac{\gamma}{\pi} \frac{1}{\sqrt{4\pi t}}  \int\limits_{0}^{\infty} e^{-\frac{u^2}{4t}} \cdot \frac{e^{\frac{u}{2}}}{e^{u}-1} \cdot \left( \int\limits_{0}^{\infty} \frac{\sinh((\pi-\gamma)\rho)}{\sinh(\gamma\rho)\sinh(\rho \pi)}\cdot \sin(\rho u) d\rho \right) du.
\end{align*}
Further,
\begin{align*}
\frac{\sinh((\pi-\gamma)\rho)}{\sinh(\gamma\rho)\sinh(\rho \pi)} &= \frac{\cosh(\rho \gamma)}{\sinh(\rho \gamma)} - \frac{\cosh(\rho \pi)}{\sinh(\rho \pi)} = \left(1 + \frac{2}{e^{2\gamma\rho}-1}\right) - \left( 1 + \frac{2}{e^{2\pi\rho}-1} \right)\\
&=2\left( \frac{1}{e^{2\gamma\rho}-1} - \frac{1}{ e^{2\pi\rho}-1} \right),
\end{align*}
and thus
\begin{align*}
I = \frac{2\gamma}{\pi} \frac{1}{\sqrt{4\pi t}}  \int\limits_{0}^{\infty} e^{-\frac{u^2}{4t}} \cdot \frac{e^{\frac{u}{2}}}{e^{u}-1} \cdot \left( \int\limits_{0}^{\infty} \left(  \frac{\sin(\rho u)}{e^{2\gamma\rho}-1} - \frac{\sin(\rho u)}{e^{2\pi\rho}-1}  \right)  d\rho \right) du.
\end{align*}
From \citep[3.911 2]{Gradshteyn}, we know that for all $u\in (0,\infty)$, $y\in\mathbb{C}$ with $\Re(y)>0$:
\begin{align*}
\int\limits_{0}^{\infty} \frac{\sin(\rho u)}{e^{ y \rho}-1}d\rho = \frac{\pi}{2y}\coth\left( \frac{\pi u}{y} \right) - \frac{1}{2u}.
\end{align*}
Hence,
\begin{align*}
I &= \frac{2\gamma}{\pi} \frac{1}{\sqrt{4\pi t}}  \int\limits_{0}^{\infty} e^{-\frac{u^2}{4t}} \cdot \frac{e^{\frac{u}{2}}}{e^{u}-1} \cdot \left(  \left( \frac{\pi}{4 \gamma}\coth\left( \frac{\pi u}{2\gamma} \right) - \frac{1}{2u} \right)   - \left( \frac{\pi}{4 \pi}\coth\left( \frac{\pi u}{2\pi} \right) - \frac{1}{2u}\right)   \right) du  \\
&= \frac{\gamma}{2 \pi} \frac{1}{\sqrt{4\pi t}}  \int\limits_{0}^{\infty} e^{-\frac{u^2}{4t}} \cdot \frac{e^{\frac{u}{2}}}{e^{u}-1} \cdot \left(  \frac{\pi}{\gamma}\coth\left( \frac{\pi}{\gamma}\cdot \frac{u}{2} \right) - \coth\left( \frac{u}{2} \right) \right) du
\end{align*}
This is exactly the third term on the RHS of \eqref{equation:HeatTraceShift}, such that the proof is completed.
\end{proof}

\begin{remark}
There also exists a Euclidean version of Theorem \ref{theorem:HeatTraceShift} published in \citep[Theorem 2]{VanDenBerg}. In their article van den Berg and Srisatkunarajah proved that formula in order to compute the heat invariants for Euclidean polygons. Besides, this formula is used by Mazzeo and Rowlett in \citep{Mazzeo} to study the so-called heat trace anomaly on Euclidean polygons, which refers to the fact that the third heat invariant is not continuous with respect to Lipschitz convergence of domains in the Euclidean plane. In the last paragraph of \citep{Mazzeo} the authors point out that an analogous formula for higher dimensional Euclidean wedges would be needed in order to generalise their results for higher dimensional polyhedra.
\end{remark}

In the remaining part of this section we will prove three lemmas which we promised and already used in the proof of Theorem \ref{theorem:HeatTraceShift}.

\begin{lemma}
\label{lemma:LemmaInverseLaplaceLegendreQ}
For all $\nu\in (-1,\infty)$ and $a>0$:
\begin{itemize}
\item[$(i)$] 
\begin{align}
\label{equation:LemmaInverseLaplaceLegendreQ}
\frac{1}{\pi} \int\limits_{0}^{\pi} Q_{\nu}\left(\cosh\left(d\left( \left(a,\frac{\alpha}{2} \right),\left( a,-\frac{\alpha}{2}\right)  \right)\right)\right) d\alpha = P_{\nu}\left( \cosh(a) \right)Q_{\nu}\left(\cosh(a)\right).
\end{align}
\item[$(ii)$]
\begin{align}
\label{equation:VerallgemeinerungIntegralausdruckSingulaereStelle}
\frac{2}{\pi}\int\limits_{0}^{\infty} \frac{\sinh(\pi \rho)}{\pi \cdot \rho}\, Q_{\nu}^{-i\rho}(\cosh(a))Q_{\nu}^{i\rho}(\cosh(a)) d\rho = P_{\nu}(\cosh(a))Q_{\nu}(\cosh(a)).
\end{align}
\end{itemize}
\end{lemma}

\begin{proof}
Let $\nu \in (-1,\infty)$ and $a\in (0,\infty)$ be given. Using \eqref{equation:HyperbolicDistancePolarCoordinates} and Corollary \ref{corollary:IntegralProduktLegendreFuerGreensFunction}, we have for all $ b \in (0,\infty)$: 
\begin{align}
\frac{1}{\pi} \int\limits_{0}^{\pi} Q_{\nu} &\left(\cosh\left(d\left( \left(a,\frac{\alpha}{2} \right),\left( b,-\frac{\alpha}{2}\right)  \right)\right)\right) d\alpha \nonumber \\
&= \frac{1}{\pi} \int\limits_{0}^{\pi} Q_{\nu}\left(\cosh(a)\cosh(b)-\sinh(a)\sinh(b)\cos(\alpha) \right) d\alpha \nonumber\\
&=\frac{2}{\pi^2} \int\limits_{0}^{\pi} \int\limits_{0}^{\infty} \cosh(\rho(\pi-\alpha)) Q_{\nu}^{-i\rho}(\cosh(b))Q_{\nu}^{i\rho}(\cosh(a))\, d\rho \,d\alpha \nonumber\\
&=\frac{2}{\pi^2}  \int\limits_{0}^{\infty} \left( \int\limits_{0}^{\pi} \cosh(\rho(\pi-\alpha)) d\alpha\right) \, Q_{\nu}^{-i\rho}(\cosh(b))Q_{\nu}^{i\rho}(\cosh(a))d\rho \nonumber\\
&=\frac{2}{\pi}  \int\limits_{0}^{\infty} \frac{\sinh(\rho \pi)}{ \pi}\cdot \frac{1}{\rho} \, Q_{\nu}^{-i\rho}(\cosh(b))Q_{\nu}^{i\rho}(\cosh(a))\, d\rho \label{equation:PunkterweiterungIntegralausdruck0}.
\end{align}
 Note also that we used the Fubini-Tonelli theorem for the third equality, which is allowed due to Theorem \ref{theorem:TheoremZentralIntegral}. More precisely,
\begin{align*}
 \frac{2}{\pi^2}  \int\limits_{0}^{\infty} \int\limits_{0}^{\pi} &\, \vert \cosh(\rho(\pi-\alpha)) \, Q_{\nu}^{-i\rho}(\cosh(b))Q_{\nu}^{i\rho}(\cosh(a))\vert\, d\alpha\, d\rho \\
&=\frac{2}{\pi^2}  \int\limits_{0}^{\infty} \left( \int\limits_{0}^{\pi} \cosh(\rho(\pi-\alpha)) d\alpha\right) \, \vert Q_{\nu}^{-i\rho}(\cosh(b))Q_{\nu}^{i\rho}(\cosh(a))\vert \,d\rho \\
&=\frac{2}{\pi}  \int\limits_{0}^{\infty} \frac{\sinh(\rho \pi)}{ \pi}\cdot \frac{1}{\rho} \, \vert Q_{\nu}^{-i\rho}(\cosh(b))Q_{\nu}^{i\rho}(\cosh(a))\vert \, d\rho <\infty.
\end{align*}
 
$(i)$. From \eqref{equation:PunkterweiterungIntegralausdruck0} and \eqref{equation:ErstesWichtigesIntegral} we obtain for all $b\in (0,\infty)$ with $a<b$:
\begin{align}
\frac{1}{\pi} \int\limits_{0}^{\pi} Q_{\nu} &\left(\cosh\left(d\left( \left(a,\frac{\alpha}{2} \right),\left( b,-\frac{\alpha}{2}\right)  \right)\right)\right) d\alpha = P_{\nu}\left( \cosh(a) \right)Q_{\nu}\left(\cosh(b)\right). \label{equation:PunkterweiterungIntegralausdruck}
\end{align}
Let $(b_n)_{n\in\mathbb{N}}\subset (a,\infty)$ be any sequence with $b_{n}>b_{n+1}$ for all $n\in\mathbb{N}$ and $\lim_{n\rightarrow\infty} b_n=a$. By continuity of $(0,\infty)\ni b\mapsto Q_{\nu}\left(\cosh(b)\right)\in\mathbb{R}$,
\begin{align}
P_{\nu}\left( \cosh(a) \right)Q_{\nu}\left(\cosh(a)\right) &= \lim\limits_{n\rightarrow\infty}P_{\nu}\left( \cosh(a) \right)Q_{\nu}\left(\cosh(b_n)\right) \nonumber \\
&=\lim\limits_{n\rightarrow\infty} \frac{1}{\pi} \int\limits_{0}^{\pi} Q_{\nu}( \cosh( d (  (a,\alpha  ), ( b_n , 0 ) ) )) d\alpha, \label{equation:PunkterweiterungIntegralausdruck2}
\end{align}
where we used \eqref{equation:PunkterweiterungIntegralausdruck} and $\cosh(d\left( \left(a,\frac{\alpha}{2} \right), \left( b_n , - \frac{\alpha}{2} \right) \right))=\cosh(d (  (a,\alpha  ), ( b_n , 0 ) ) )$ for the last equality.
Thus, it remains to show 
\begin{align*}
\lim\limits_{n\rightarrow\infty} \frac{1}{\pi}\int\limits_{0}^{\pi} Q_{\nu}( \cosh( d (  (a,\alpha  ), ( b_n , 0 ) ) )) d\alpha = \frac{1}{\pi} \int\limits_{0}^{\pi} Q_{\nu} ( \cosh( d (  (a,\alpha  ), ( a , 0 ) ) )) d\alpha.
\end{align*}
Note that for all $\alpha\in \left(0,\pi\right)$,
\begin{align*}
\varphi: [a,\infty)\ni b\mapsto \cosh( d (  (a,\alpha), ( b , 0 ) )  = \cosh(a)\cosh(b)-\sinh(a)\sinh(b)\cos(\alpha) \in (0,\infty)
\end{align*}
is strictly increasing because $\varphi''(b)=\varphi(b)>0$ for all $b>a$ and $\varphi'(a)>0$. Moreover, the function $Q_{\nu}:(1,\infty)\ni x \mapsto Q_{\nu}(x)\rightarrow \mathbb{R}$ is non-negative and strictly decreasing, which can be seen from \eqref{equation:RepresentationLegendreQIntegral}. Thus, by the monotone convergence theorem, we have 
\begin{align*}
\lim\limits_{n\rightarrow\infty}\frac{1}{\pi} \int\limits_{0}^{\pi} Q_{\nu} (\cosh\left( d ( (a,\alpha ), (b_n ,0) )\right) )\, d\alpha = \frac{1}{\pi} \int\limits_{0}^{\pi} Q_{\nu} (\cosh\left( d ( (a,\alpha ), (a,0) )\right) ) \, d\alpha.
\end{align*}
Hence, using the above equation and \eqref{equation:PunkterweiterungIntegralausdruck2}, we have
\begin{align*}
P_{\nu}\left( \cosh(a) \right)Q_{\nu}\left(\cosh(a)\right) &= \frac{1}{\pi} \int\limits_{0}^{\pi} Q_{\nu} (\cosh\left( d ( (a,\alpha ), (a,0) )\right) ) d\alpha \\
&=\frac{1}{\pi} \int\limits_{0}^{\pi} Q_{\nu}\left(\cosh\left(d\left( \left(a,\frac{\alpha}{2} \right),\left( a,-\frac{\alpha}{2}\right)  \right)\right)\right) d\alpha,
\end{align*}
which closes the proof of $(i)$.

$(ii)$. The second statement follows immediately from $(i)$ and \eqref{equation:PunkterweiterungIntegralausdruck0} with $b:=a$.
\end{proof}

We first discovered the following Laplace transform in Lemma \ref{lemma:InverseLaplace1} by performing the inverse Laplace transform with the complex inversion formula. Even though the calculations in connection with the complex inversion formula are instructive, we will compute the Laplace transform directly for the sake of a shorter proof. 

\begin{lemma}
\label{lemma:InverseLaplace1}
 For all $a,r>0$ the following Laplace transform exists and is given for all $s\in \mathcal{H}_{>\frac{1}{4}}$ by
\begin{align}
\label{equation:InverseLaplace1.1}
\mathcal{L}\bigg\lbrace t\mapsto 4 \int\limits_{0}^{r} \int\limits_{0}^{\frac{\pi}{2}}K_{\mathbb{H}^2}^{\nicefrac{1}{4}}&((a,\alpha),(a,-\alpha);t) \sinh(a)d\alpha da \bigg\rbrace (s) \nonumber \\
&= \int\limits_{0}^{r} P_{\sqrt{s}-\frac{1}{2}}(\cosh(a))Q_{\sqrt{s}-\frac{1}{2}}(\cosh(a)) \sinh(a)\, da.
\end{align}
\end{lemma} 

\begin{proof}
Let $a, r>0$ be arbitrary. Note that for all $t>0$:
\begin{align*}
4\int\limits_{0}^{\frac{\pi}{2}} &K_{\mathbb{H}^2}^{\nicefrac{1}{4}}\left( \left( a,\alpha \right),\left( a,-\alpha \right);t \right) d\alpha = 2\int\limits_{0}^{\pi} K_{\mathbb{H}^2}^{\nicefrac{1}{4}}\left( \left( a,\frac{\alpha}{2} \right),\left( a,-\frac{\alpha}{2} \right);t \right) d\alpha .
\end{align*}
Thus, by the Fubini-Tonelli theorem, \eqref{equation:GreenPlane1}, and \eqref{equation:LemmaInverseLaplaceLegendreQ} we have for all $s>\frac{1}{4}$:
\begin{align*}
\int\limits_{0}^{\infty}e^{-st} \cdot 4 \int\limits_{0}^{r} \int\limits_{0}^{\frac{\pi}{2}} &K_{\mathbb{H}^2}^{\nicefrac{1}{4}}\left( \left( a,\alpha \right),\left( a,-\alpha \right);t \right) \sinh(a)\, d\alpha\, da\, dt \\
&= \int\limits_{0}^{\infty}e^{-st} \cdot2 \int\limits_{0}^{r}\int\limits_{0}^{\pi} K_{\mathbb{H}^2}^{\nicefrac{1}{4}}\left( \left( a,\frac{\alpha}{2} \right),\left( a,-\frac{\alpha}{2} \right);t \right)\sinh(a)\, d\alpha\, da\, dt \\
&=2 \int\limits_{0}^{r} \sinh(a) \int\limits_{0}^{\pi} \int\limits_{0}^{\infty}e^{-st} K_{\mathbb{H}^2}^{\nicefrac{1}{4}}\left( \left( a,\frac{\alpha}{2} \right),\left( a,-\frac{\alpha}{2} \right);t \right) \,dt \, d\alpha \, da\\
&= \int\limits_{0}^{r} \sinh(a) \frac{1}{\pi}\int\limits_{0}^{\pi} Q_{\sqrt{s}-\frac{1}{2}}\left(  \cosh\left(d\left( \left(a,\frac{\alpha}{2} \right),\left( a,-\frac{\alpha}{2}\right)  \right)\right)\right) d\alpha\, da\\
&=\int\limits_{0}^{r} \sinh(a) P_{\sqrt{s}-\frac{1}{2}}\left( \cosh(a) \right)Q_{\sqrt{s}-\frac{1}{2}}\left(\cosh(a)\right) \, da.
\end{align*}
Note that the application of the Fubini-Tonelli theorem is justified in the second equality above, since $s>\frac{1}{4}$ and therefore all functions under the integral sign are positive. 

We have shown that the Laplace integral is convergent and equal to $\int_{0}^{r} P_{\sqrt{s}-\frac{1}{2}}\left( \cosh(a) \right) \cdot Q_{\sqrt{s}-\frac{1}{2}}\left(\cosh(a)\right) \sinh(a)\, da$ for all $s>\frac{1}{4}$. Thus, by Proposition \ref{proposition:InverseLaplaceTransform}, the Laplace transform of $t\mapsto 4\int_{0}^{r} \int_{0}^{\frac{\pi}{2}}K_{\mathbb{H}^2}^{\nicefrac{1}{4}}((a,\alpha),(a,-\alpha);t) \sinh(a) d\alpha\, da$ exists for all $s\in\mathcal{H}_{>\frac{1}{4}}$. Moreover, that Laplace transform must be equal to $\int_{0}^{r} P_{\sqrt{s}-\frac{1}{2}}(\cosh(a))Q_{\sqrt{s}-\frac{1}{2}}(\cosh(a))  \sinh(a)\, da$ for all $s\in\mathcal{H}_{>\frac{1}{4}}$, because both are holomorphic on that domain.
\end{proof}

\begin{lemma}
\label{lemma:IntegralProdLegendrePsiFunktion}
For all $s\in\mathbb{C}\backslash(-\infty,0]$, $\rho\in\mathbb{R}$:
\begin{align}
\int\limits_{1}^{\infty} Q_{ \sqrt{s} - \frac{1}{2}}^{-i\rho}(x)  Q_{ \sqrt{s} - \frac{1}{2}}^{i\rho}(x) dx =\, &\frac{\pi}{2\sinh(\rho \pi)} \cdot \frac{1}{2i\sqrt{s}}\, \cdot\nonumber \\
&\cdot \left( \psi \left( \sqrt{s}+i\rho +\frac{1}{2} \right) - \psi\left( \sqrt{s} - i\rho +\frac{1}{2} \right) \right),
\end{align}
where 
\begin{align}
\label{equation:Psi-function}
\psi(z):=\frac{\Gamma'(z)}{\Gamma(z)} 
\end{align}
is the logarithmic derivative of the gamma function, also called psi-function.
\end{lemma}

\begin{proof}
From \citep[vol. II, equation $(395)$ on p. 200]{Robin} we know that for all $\nu,\, \mu\in\mathbb{C}$ such that $\Re(\nu)>-\frac{1}{2},$ $\Re(\mu)\in (-1,1):$
\begin{align}
\label{equation:RobinZitat}
(2\nu+1)\int\limits_{1}^{\infty} Q_{ \nu}^{\mu}(x)  Q_{ \nu}^{\mu}(x) dx = \frac{\pi e^{i2\mu\pi}}{2\sin(\mu \pi)}& \cdot \frac{\Gamma(\nu+\mu+1)}{\Gamma(\nu-\mu+1)}\, \cdot \nonumber \\
&\cdot \left( \psi\left( \nu+\mu+1 \right) - \psi\left( \nu-\mu+1 \right) \right).
\end{align}
We want to remark that Robin uses a different notation in his book than we do above, in particular his psi-function is shifted by $1$ compared with \eqref{equation:Psi-function}. The relevant notation used in \citep[vol. II]{Robin} is defined at the beginning of the book on page VII.

Let $s\in\mathbb{C}\backslash(-\infty,0]$, $\rho\in\mathbb{R}$ and set $\nu:=\sqrt{s}-\frac{1}{2}$, $\mu:=i\rho$. Then \eqref{equation:RobinZitat} is equivalent to
\begin{align*}
\int\limits_{1}^{\infty} Q_{\sqrt{s}-\frac{1}{2}}^{i\rho}(x)  Q_{\sqrt{s}-\frac{1}{2}}^{i\rho}(x) dx = \frac{\pi e^{-2\rho\pi}}{2i\sinh(\rho \pi)}&\cdot \frac{\Gamma\left(\sqrt{s}+i\rho+\frac{1}{2}\right)}{\Gamma\left(\sqrt{s}-i\rho+\frac{1}{2}\right)}\cdot\frac{1}{2\sqrt{s}}\cdot\\
&\cdot\left( \psi\left( \sqrt{s}+i\rho+\frac{1}{2} \right) - \psi\left( \sqrt{s}-i\rho+\frac{1}{2} \right) \right).
\end{align*}
Thus using \eqref{equation:MagischeFormel1}, we obtain:
\begin{align*}
\int\limits_{1}^{\infty} Q_{\sqrt{s}-\frac{1}{2}}^{-i\rho}(x)  Q_{\sqrt{s}-\frac{1}{2}}^{i\rho}(x) dx &= e^{2\rho \pi}\frac{\Gamma\left( \sqrt{s} -i\rho+\frac{1}{2} \right)}{\Gamma\left( \sqrt{s} +i\rho+\frac{1}{2} \right)}\int\limits_{1}^{\infty} Q_{\sqrt{s}-\frac{1}{2}}^{i\rho}(x)  Q_{\sqrt{s}-\frac{1}{2}}^{i\rho}(x) dx \\
&=\frac{\pi}{2\sinh(\rho\pi)} \cdot \frac{1}{2i\sqrt{s}}\cdot\left( \psi\left( \sqrt{s}+i\rho+\frac{1}{2} \right) - \psi\left( \sqrt{s}-i\rho+\frac{1}{2} \right) \right).
\end{align*}
\end{proof}

It remains to compute one inverse Laplace transform. Luckily, this can be easily deduced from a known Laplace transform and using an abstract property of the Laplace transform.

\begin{lemma}
\label{lemma:InverseLaplaceTransformPsiFunktion}
Let $\rho\in\mathbb{R}$ and $t\in(0,\infty)$. Further, let 
\begin{align*}
F:\mathcal{H}_{>0}\ni s\mapsto \frac{1}{2i} \cdot \frac{1}{\sqrt{s}}\left( \psi\left( \sqrt{s}+i\rho +\frac{1}{2} \right) - \psi\left( \sqrt{s} - i\rho +\frac{1}{2} \right) \right) \in\mathbb{C}.
\end{align*}
 Then $\mathcal{L}^{-1}\left\lbrace F \right\rbrace$ exists, and
\begin{align}
\label{equation:InverseLaplaceTransformPsiFunktion}
\mathcal{L}^{-1}\left\lbrace F \right\rbrace(t) = \frac{1}{\sqrt{\pi t}}\int\limits_{0}^{\infty} e^{-\frac{u^2}{4t}} \cdot \frac{e^{\frac{u}{2}}}{e^{u}-1} \cdot \sin(\rho u)du.
\end{align}
\end{lemma}

\begin{proof}
From \citep[3.311 7]{Gradshteyn} we know for all $\nu,\,\mu\in\mathbb{C}$ with $\Re(\nu),\, \Re(\mu)>0:$
\begin{align}
\label{equation:LemmaIntegralPsiFunktion}
\int\limits_{0}^{\infty} \frac{e^{-\mu t}-e^{-\nu t}}{1-e^{-t}}dt = \psi(\nu)-\psi(\mu).
\end{align}
Let $\rho\in\mathbb{R}$ and $s\in \mathcal{H}_{>0}$. Further, set $\mu:=s-i\rho+\frac{1}{2}$ and $\nu:=s+i\rho+\frac{1}{2}$. Then \eqref{equation:LemmaIntegralPsiFunktion} gives
\begin{align*}
\psi\left( s+i\rho+\frac{1}{2} \right)-\psi\left( s-i\rho+\frac{1}{2} \right) &= \int\limits_{0}^{\infty}\frac{e^{-\left(s-i\rho+\frac{1}{2}\right) t}-e^{-\left( s+i\rho+\frac{1}{2}\right) t}}{1-e^{-t}}dt \\
&=\int\limits_{0}^{\infty}e^{-st} \cdot \frac{e^{-\frac{t}{2}}}{1-e^{-t}}  \cdot \left( e^{i\rho t}-e^{-i\rho t} \right) dt\\
&=\int\limits_{0}^{\infty}e^{-st}  \cdot \frac{e^{\frac{t}{2}}}{e^{t}-1}  \cdot 2i\sin(\rho t) dt.
\end{align*}
In other words, if $f(t):=\frac{e^{\frac{t}{2}}}{e^{t}-1}\sin(\rho t)$ for $t>0$, then
\begin{align*}
\mathcal{L}\left\lbrace f \right\rbrace(s) = \frac{1}{2i}\left(\psi\left( s+i\rho+\frac{1}{2} \right)-\psi\left( s-i\rho+\frac{1}{2} \right)\right)
\end{align*}
for all $s\in \mathcal{H}_{>0}$. Hence, it follows from  \citep[29 on p. 5]{PrudnikovInverseLaplace} that
\begin{align*}
\mathcal{L}&\left\lbrace t\mapsto \frac{1}{\sqrt{\pi t}}\int\limits_{0}^{\infty} e^{-\frac{u^2}{4t}}f(u) du \right\rbrace(s) =\frac{1}{\sqrt{s}} \cdot \frac{1}{2i}\cdot \left(\psi\left( \sqrt{s}+i\rho+\frac{1}{2} \right)-\psi\left( \sqrt{s}-i\rho+\frac{1}{2} \right)\right)
\end{align*}
for $s\in \mathcal{H}_{>0}$. Proposition \ref{proposition:InverseLaplaceTransform} now implies the statement.
\end{proof}

\section{Heat invariants for hyperbolic polygons}
\label{section:ExpansionTrace}

Let $\Omega\subset \mathbb{H}^2$ be a hyperbolic polygon. We are now ready to compute the heat invariants for $\Omega$, i.e. to compute the asymptotic expansion of the heat trace
\begin{align*}
Z_{\Omega}(t)=\int\limits_{\Omega}K_{\Omega}(x,x;t)dx.
\end{align*}
The coefficients in the asymptotic expansion will depend, in particular, on the number and the size of the angles. So let us assume $\Omega$ has $M$ angles $\gamma_1,...,\gamma_M \in (0,2\pi]$ where $M\geq 3$ is an integer.

The method we use in order to compute the asymptotic expansion of $Z_{\Omega}(t)$ is known as the principle of not feeling the boundary and was first formulated by M. Kac (see \citep{Kac}). According to this principle the heat kernels $K_{\Omega}(x,x;t)$ and $K_{\mathbb{H}^2}(x,x;t)$ have the same asymptotic expansion as $t\searrow 0$ for fixed $x\in\Omega$. In other words, the short time asymptotic behaviour of the heat kernel does not ``feel'' the presence of the boundary. Let us put this principle on a sound basis by the following lemma. We note that the lemma below is somewhat more general than it is needed in this section, but we will make use of it in full generality in subsequent sections. 

\begin{lemma}
\label{lemma:PNFBAllgemein}
Let $N$ be a two-dimensional complete Riemannian manifold whose Gaussian curvature is bounded. Let $U\subset N$ be an arbitrary domain and let $A\subset N$ be a compact subset such that $A\subset U.$ Then there exist constants $T,\,C,\,D>0$ such that
\begin{align}
\label{equation:LemmaPNFBAllgemein}
\vert K_{N} (x,y;t) - K_{U} (x,y;t) \vert \leq \frac{C}{t} \cdot e^{-\frac{D}{t}}\quad \text{ for all }\, x,y\in A,\, t\in (0,T].
\end{align}
\end{lemma}

\begin{proof}

First choose a relatively compact domain $G\subset N$ with smooth boundary $\partial G$ and such that $A\subset G\subset U$. This is always possible: Consider a so-called \emph{cutoff} function of $A$ in $U$ (see \citep[Theorem 3.5]{Grigoryan}), i.e. a smooth function $\varphi\in C^{\infty}_{c}(U)$ such that $0\leq \varphi \leq 1$  and $\varphi \equiv 1$ in a neighborhood of $A$. By Sard's theorem (see \citep[Theorem 6.10]{Lee}) there are infinitely many regular values of $\varphi$ in the interval $(0,1)$. Now choose any regular value $g\in (0,1)$ and define $G:=\varphi^{-1}(g,\infty)$.

By the minimality of the heat kernel we have for all $x,y\in G$ and $t>0$:
\begin{align*}
K_G(x,y;t)\leq K_U (x,y;t)\leq K_{N}(x,y;t),
\end{align*}
and thus
\begin{align}
\label{equation:LemmaPNFBInequality1}
0 \leq  K_{N}(x,y;t) - K_U (x,y;t) \leq K_{N}(x,y;t) - K_G(x,y;t).
\end{align}

Let $y\in A$ be fixed. The function $u:G\times (0,\infty)\rightarrow \mathbb{R},\, u(x,t):= K_{N} (x,y;t) - K_{G} (x,y;t)$ is (by definition) a smooth non-negative solution to the heat equation on $G\times (0,\infty)$. Moreover, it can be extended continuously to $\overline{G}\times [0,\infty)$ by the following procedure: If we set $u(x,t):=0$ for all $t=0$, $x\in G$ then we obtain a continuous function on $G\times [0,\infty)$ (see \citep[Exercise 9.7]{Grigoryan}). Since the boundary $\partial G$ is smooth, $K_{G} ( \cdot ,y; \cdot)$ can be extended continuously to $\overline{G} \times (0,\infty)$ by $K_{G} (x,y;t):=0$ for all $x\in\partial G$ and $t>0$ (see also the remark after Proposition \ref{proposition:heat trace bounded domain}). Thus $u(x,t)$ can be extended continuously by $u(x,t):=K_{N}(x,y;t)$ for all $x\in \partial G,\, t>0$. Alltogether, we have extended $u(x,t)$ continuously to $\overline{G}\times (0,\infty)\cup G\times \{ 0 \}$ as follows:
\begin{align*}
u(x,t) = \begin{cases}
K_{N} (x,y;t) - K_{G} (x,y;t),& \text{ if } (x,t)\in G\times (0,\infty), \\
 0,& \text{ if } (x,t)\in G\times \{ 0 \}, \\
K_{N} (x,y;t),&  \text{ if } (x,t)\in \partial G\times  (0,\infty).
\end{cases}
\end{align*}

It remains to show that $u(x,t)$ can be extended continuously to $\partial G \times \{ 0 \}$. We obtain a continuous extension when we set $u(x,0):=0$ for all $x\in\partial G$, which follows directly from the existence of constants $\tilde{T},\, C>0$ such that
\begin{align*}
0\leq u(x,t)\leq K_{N} (x,y;t) \leq \frac{C}{t}\cdot e^{-\frac{ d(x,y)^2}{16t}}\quad \text{ for all } (x,t) \in  \overline{G}\times (0,\tilde{T}].
\end{align*}
The last estimate given above is well-known. If $N$ is non-compact, then it follows from \citep[Theorem 4]{LiYauEstimate} and because of the fact that $\inf\{\, i(x)\mid x\in \overline{G} \,\}>0$, where $i(x)$ denotes the injectivity radius of $x$ (see \citep[Proposition 2.1.10]{Klingenberg}). It is also explained in \citep[remark on p. $1050$]{LiYauEstimate} how to prove the same estimate as in \citep[Theorem 4]{LiYauEstimate} if $N$ is compact. However, there is a simpler proof if $N$ is compact. If $N$ is closed, then the estimate follows from \citep[Theorem 1.1]{GrigoryanGaussian} together with the well-known asymptotic expansion for $K_N(x,x;t)$ as $t\searrow 0$, which holds uniformly in $x\in N$ (see e.g. \citep[Theorem 3.3]{Donnelly}).

Thus we can apply the parabolic maximum principle (see \citep[Theorem 8.10]{Grigoryan}) to $u(x,t)$ in order to obtain for all $x\in G$ and $t\in (0,\tilde{T}]$:
\begin{align}
\label{equation:LemmaPNFBInequality2}
u(x,t) &\leq \max\limits_{ (z,s)\in \overline{G} \times [0,t]} u(z,s) =  \max\limits_{\substack{ (z,s)\in \overline{G} \times \{ 0 \} \cup \\   \qquad \partial G \times (0,t)}} u(z,s) \nonumber \\
& = \max\limits_{\substack{ (z,s)\in \partial G \times (0,t)}}u(z,s)\nonumber \\
& \leq \max\limits_{ s\in (0,t)} \frac{C}{s}\cdot e^{-\frac{ d(\partial G ,y)^2}{16s}},
\end{align}
where $d(\partial G,y):=\inf_{z\in\partial G} d(z,y)$.

We set $D:= \frac{ d(\partial G, A)^2}{16}>0$, where $d(\partial G , A):=\inf_{y\in A} d(\partial G, y)$. Let $T\in(0,D)\cap (0,\tilde{T})$ be fixed and let $C>0$ be as above. The function $s\mapsto \frac{C}{s}\cdot e^{-\frac{ D}{s}}$ is monotonically increasing in $s\in (0,D)$. Hence it follows from \eqref{equation:LemmaPNFBInequality1} and \eqref{equation:LemmaPNFBInequality2} that for all $x,y\in A$ and $t\in (0,T]$:
\begin{align*}
0\leq K_{N}(x,y;t) - K_U (x,y;t) \leq \max\limits_{ s\in (0,t)} \frac{C}{s}\cdot e^{-\frac{ D}{s}} = \frac{C}{t}\cdot e^{-\frac{ D}{t}}
\end{align*}
\end{proof}

In order to approximate the heat kernel $K_{\Omega}(x,x;t)$ for points $x\in\Omega$ close to the boundary $\partial\Omega$ it is helpful to use the following probabilistic formula for the heat kernel. An introduction to (conditional) Wiener measures on Riemannian manifolds, written for geometers, can be found in \citep{Baer}. All other probabilistic notions we use can be found, for example, in \citep{HsuStochastik}. Again, we formulate the next lemma more general than it is needed in this section in view of applications appearing later in this thesis.

\begin{lemma}
\label{lemma:ProbabilisticFormulaHeatKernel}
Let $N$ be a two-dimensional complete Riemannian manifold whose Gaussian curvature is bounded. Let $U\subset N$ be any domain. Then we have for all $x,y\in U$ and $t>0$:
\begin{align}
\label{equation:ProbabilisticFormulaHeatKernel}
K_{U}(x,y;t)=K_{N}(x,y;t)\cdot \emph{Prob}\{\, \omega(s)\in U,\, 0\leq s \leq t \mid \omega(0)=x,\, \omega(t)=y \, \},
\end{align}
where $\omega: [0,t]\rightarrow N$ denotes a continuous curve such that $\omega(0)=x$ and $\omega(t)=y$ and $\emph{Prob}\{ ...\vert... \}$ is the normalised conditional Wiener measure. \emph{(}The measure is normalised in the sense that it is equal to the conditional Wiener measure as in \emph{\citep[Proposition 3.15]{Baer}} multiplied by $\frac{1}{K_{N}(x,y;t)}$ so that it becomes a probability measure\emph{)}.

%
\end{lemma}

\begin{proof}
It is well-known that if the Gaussian curvature is bounded from below, then the volume of any geodesic disc in $N$ increases at most exponentially. This follows by comparison with a space of constant curvature (see e.g. \citep[Lemma 35 on p. 269]{Petersen}). Thus $N$ is stochastically complete (see \citep[Theorem 11.8]{Grigoryan}), and therefore the results from \citep[Section 3.4]{Baer} are valid in our situation.

Let us introduce some notation. We denote the set of continuous paths in $N$ starting at $x\in N$ by $C_{x}([0,\infty);\,N)$. In other words, for any $x\in N$
\begin{align*}
C_{x}([0,\infty);\,N):=\{\,\omega:[0,\infty)\rightarrow N \mid \omega\,\text{is continuous and } \omega(0)=x\, \}.
\end{align*}
Similarly, for any $x,y\in N$, $t>0$ we set
\begin{align*}
C_x([0,t];\,N):&=\{\, \omega:[0,t]\rightarrow N \mid \omega\,\text{is continuous and } \omega(0)=x\,  \},\\
C_x^{y}([0,t];\,N):&=\{\, \omega:[0,t]\rightarrow N \mid \omega\,\text{is continuous and } \omega(0)=x, \, \omega(t)=y\,  \}.
\end{align*}
By a \emph{curve} we mean any element of one of the above sets. If $x\in U$ then the \emph{first exit time} from $U$ of a curve $\omega$ is defined as
\begin{align*}
\tau_U(\omega):=\inf\{\, s>0\mid \omega(s)\notin U\, \}.
\end{align*}
If the curve $\omega$ never leaves the set $U$, i.e. if $\{ s>0\mid \omega(s)\notin U\} = \emptyset$, then we set $\tau_U(\omega):=\infty$ with the convention that $t<\infty $ for all $t\in [0,\infty)$. Finally, let
\begin{align*}
1_{\{\, t< \tau_U \,\} }(\omega):=
\begin{cases}
1, \text{ if }t<\tau_U(\omega),\\
0, \text{ otherwise}.
\end{cases}
\end{align*}

We first assume that $U$ is a relatively compact domain with smooth boundary and prove \eqref{equation:ProbabilisticFormulaHeatKernel} for this case. In this special case, it is well-known that for any bounded Borel function $f:U\rightarrow \mathbb{R}$, any $x\in U$ and $t>0$: 
\begin{align}
\label{equation:ConnectionHeatSemigroupAndWienerMeasure}
\int\limits_{U}f(y)\cdot K_{U}(x,y;t)dy = \int\limits_{C_x([0,\infty);\, N)} f(\omega(t))\cdot 1_{\{\, t<\tau_U \} }(\omega) d\mathbb{P}_x(\omega),
\end{align}
where $\mathbb{P}_x$ denotes the Wiener measure on $C_x([0,\infty);\, N)$ as in \citep[Section 3.4]{Baer}.
A proof of \eqref{equation:ConnectionHeatSemigroupAndWienerMeasure} can be found in \citep[Proposition 4.1.3]{HsuStochastik}. Compare also with \citep[Theorem 8.6]{GrigoryanBrownianMotion}. 

Moreover, let $\mathbb{P}_x^{t}$ denote the Wiener measure on $C_x([0,t];\,N)$ and let $\mathbb{P}_{x,y}^{t}$ denote the conditional Wiener measure on $C_x^{y}([0,t];\,N)$ as in \citep[Section 3]{Baer}.
We have (see \citep[Remark 3.20]{Baer}) 
\begin{align}
\label{equation:RelationWienerMeasureRestrictionMap}
\mathbb{P}_x^{t} = (\text{rest}_{t})_{\ast} \mathbb{P}_x,
\end{align}
where $\text{rest}_{t}: C_x([0,\infty);\,N)\rightarrow C_x([0,t];\,N)$, $\omega\mapsto \omega_{\mid[0,t]}$, is the restriction map. When we use in the following order \eqref{equation:ConnectionHeatSemigroupAndWienerMeasure}, and \eqref{equation:RelationWienerMeasureRestrictionMap} combined with the transformation rule, and then \citep[Lemma 2.24]{Baer} we obtain for any $x\in U$ and $t>0$:
\begin{align*}
\int\limits_{U} f(y)\cdot K_{U}(x,y;t)dy &= \int\limits_{C_x([0,\infty);\, N)} f(\omega(t))\cdot 1_{\{\, t<\tau_U \} }(\omega) d\mathbb{P}_x(\omega)\\
&=\int\limits_{C_x([0,\infty);\, N)} f(\text{rest}_t(\omega)(t))\cdot 1_{\{\, t<\tau_U \} }(\text{rest}_t(\omega)) d\mathbb{P}_x(\omega)\\
&= \int\limits_{C_x([0,t];\, N)} f(\omega(t)) \cdot 1_{\{\, t<\tau_U \} }(\omega) d\mathbb{P}_x^{t}(\omega)\\
&= \int\limits_{N} \int\limits_{C_x^{y}([0,t];\, N)} f(\omega(t)) \cdot 1_{\{\, t<\tau_U \} }(\omega) d\mathbb{P}_{x,y}^{t}(\omega)dy \\
&= \int\limits_{U} f(y) \cdot  \int\limits_{C_x^{y}([0,t];\, N)}  1_{\{\, t<\tau_U \} }(\omega) d\mathbb{P}_{x,y}^{t}(\omega) dy.
\end{align*}
Note that we have used \citep[Lemma 2.24]{Baer} for the fourth equation above, where the metric measure space in Lemma $2.24$ of \citep{Baer} is taken as explained at the beginning of Section $3$ in \citep{Baer}. Recall that $f$ was an arbitrary bounded Borel function on $U$ and $x\in U$, $t>0$ were arbitrary as well. Thus for all $x\in U$ and $t>0$ we have:
\begin{align*}
K_{U}(x,y;t) &= \int\limits_{C_x^{y}([0,t];\, N)}  1_{\{\, t<\tau_U \} }(\omega) d\mathbb{P}_{x,y}^{t}(\omega)
\end{align*}
for almost all $y\in U$. We show that this equality is actually true for all $y\in U$. For this purpose, let us define
\begin{align*}
q_{U}(x,y;t):= \int\limits_{C_x^{y}([0,t];\, N)}  1_{\{\, t<\tau_U \} }(\omega) d\mathbb{P}_{x,y}^{t}(\omega)
\end{align*}
for all $x,y \in U$ and $t>0$.
One can prove exactly as in \citep[Proposition 3.1]{Sznitman}, that for all $x,y\in U$ and $t,s>0$:
\begin{align*}
q_U(x,y;t) &= q_U(y,x;t),\\
q_U(x,y;t+s) &= \int\limits_U q_U(x,z;t)q_U(z,y;s) dz.
\end{align*}
Thus we obtain for all $x,y\in U$, $t>0$:
\begin{align*}
q_U(x,y;t) &= \int\limits_U q_U \left( x,z;\frac{t}{2} \right)q_U \left( z,y;\frac{t}{2} \right) dz =\int\limits_U K_U \left( x,z;\frac{t}{2} \right)q_U \left( y,z;\frac{t}{2} \right) dz\\
&=\int\limits_U K_U \left( x,z;\frac{t}{2} \right)K_U \left( y,z;\frac{t}{2} \right) dz\\
&= K_U(x,y;t),
\end{align*}
where we have used twice that for all $x\in U$, $t>0$: $K_U(x,y;t) = q_U(x,y;t)$ for almost all $y\in U$. For the last equality we have used the semigroup identity of the heat kernel (see \citep[Theorem 7.13]{Grigoryan}). Finally we have (by definition) for all $x,y\in U$ and $t>0$:
\begin{align*}
K_{N}(x,y;t)\cdot \text{Prob}\{\, \omega(s)\in U,\, 0\leq s\leq t\mid \omega(0)=x,\, \omega(t)=y \, \}= \int\limits_{C_x^{y}([0,t];\, N)} \hspace{-0.1cm} 1_{\{\, t<\tau_U \} }(\omega) d\mathbb{P}_{x,y}^{t}(\omega).
\end{align*}
Therefore the claimed equation \eqref{equation:ProbabilisticFormulaHeatKernel} is established for relatively compact domains $U$ with smooth boundary.

If $U$ is an arbitrary domain, then we proceed as follows. Consider a sequence $\left( U_n \right)_{n\in\mathbb{N}}$ of relatively compact domains with smooth boundary, such that $U=\cup_{n=1}^{\infty}U_n$ and $\overline{U_n}\subset U_{n+1}$ for all $n\in\mathbb{N}$. For all $n\in\mathbb{N}$ we extend the domain of the heat kernel $K_{U_n}(x,y;t)$ to $U\times U\times (0,\infty)$ by setting $K_{U_n}(x,y;t):=0$ whenever $x$ or $y$ lies outside of $U_n$. On the one hand, it is well-known (see \citep{Dodziuk} or \citep[Theorem 4 on p. 188]{Chavel}) that
\begin{align*}
K_U(x,y;t)=\lim\limits_{n\rightarrow\infty }K_{U_n}(x,y;t) \quad \text{ for all } x,y\in U,\, t>0.
\end{align*}
On the other hand, from Lebesgue's monotone convergence theorem, we have for all $x,y\in U$, $t>0$:
\begin{align*}
\lim\limits_{n\rightarrow\infty}\int\limits_{C_x^{y}([0,t];\, N)}  1_{\{\, t<\tau_{U_n} \} }(\omega) d\mathbb{P}_{x,y}^{t}(\omega) = \int\limits_{C_x^{y}([0,t];\, N)}  1_{\{\, t<\tau_U \} }(\omega) d\mathbb{P}_{x,y}^{t}(\omega).
\end{align*}
When we put everything together we obtain for all $x,y\in U$ and $t>0$:
\begin{align*}
K_U(x,y;t)&=\lim\limits_{n\rightarrow\infty }K_{U_n}(x,y;t) = \lim\limits_{n\rightarrow\infty}\int\limits_{C_x^{y}([0,t];\, N)}  1_{\{\, t<\tau_{U_n} \} }(\omega) d\mathbb{P}_{x,y}^{t}(\omega) \\
&= \int\limits_{C_x^{y}([0,t];\, N)}  1_{\{\, t<\tau_U \} }(\omega) d\mathbb{P}_{x,y}^{t}(\omega)\\
&= K_{N}(x,y;t)\cdot \text{Prob}\{\, \omega(s)\in U,\, 0\leq s \leq t \mid \omega(0)=x,\, \omega(t)=y \, \}.
\end{align*}
\end{proof}

We now return to the hyperbolic polygon $\Omega$ with interior angles $\gamma_1,...,\gamma_M$. We will approximate the values $K_{\Omega}(x,x;t)$ of the heat kernel by simpler functions applying Lemma \ref{lemma:PNFBAllgemein}. The approximating functions will depend on the location of the point $x\in\Omega$. For this reason, we will decompose the polygon into three subsets, denoted by $\Omega_V,\Omega_E$ and $\Omega_I$. Roughly speaking, the set $\Omega_V$ will consist of all points which are close to a vertex; the set $\Omega_E$ will contain all points which are close to an edge but away from the vertices; and the set $\Omega_I$ will contain all points away from the boundary $\partial\Omega$. Let us define these subsets in an exact manner. 

Recall that $d(x,y)$ denotes the (hyperbolic) distance for any $x,y\in \mathbb{H}^2$ and $\hyperball{r}{P}=\lbrace\, x\in\mathbb{H}^2 \mid d(x,P)<r \,\rbrace$ for $r>0$ and $P\in\mathbb{H}^2$. For any $i\in\lbrace\, 1,...,M \, \rbrace$ let $W_i\subset \mathbb{H}^2$ be the wedge corresponding to the interior angle $\gamma_i$ as in Definition \ref{definition:Angle}. Let us denote the vertex of $W_i$ by  $P_i\in \mathbb{H}^2$ for all $i=1,...,M$. For all $r>0$ and $i=1,...,M$ we define
\begin{align*}
W_r(P_i):=\hyperball{r}{P_i} \cap W_i = \lbrace\, p\in W_i \mid d(p,P_i)< r \,\rbrace,
\end{align*}
and let
\begin{align*}
R:=\frac{1}{2}\sup\left\lbrace\, r>0 \, \middle|\, W_r(P_{j})\cap W_r(P_{k})=\emptyset, \, \forall\, j\neq k;\, \bigcup\limits_{\ell=1}^{M}W_r(P_{\ell})\subset \Omega \,\right\rbrace.
\end{align*}
The set of points close to the vertices is defined as
\begin{align}
\label{equation:DefinitionNearVertices}
\Omega_V:=\bigcup\limits_{i=1}^{M}W_R(P_i).
\end{align}
Note that the sets $W_R(P_1)$,..., $W_R(P_M)$ are pairwise disjoint. 

The set $\Omega_E$ is defined as follows. Let $\tilde{M}\in\mathbb{N}$ be the number of edges and let $E_1,...,E_{\tilde{M}}$ denote all the edges of the polygon $\Omega$ (recall Definition \ref{definition:EdgeVertex}). By definition, the edges have no common points except for vertices. For all $\tilde{\delta}>0$ and $j=1,...,\tilde{M}$ we define the set
\begin{align*}
\Omega_{E_j}(\tilde{\delta}):=\left\lbrace\, p\in\Omega \mid p\notin \Omega_V, \, d(p,E_j)<\tilde{\delta} \,\right\rbrace.
\end{align*}
We can choose $\delta=\delta(\Omega)>0$ such that $\Omega_{E_k}(\delta)\cap\Omega_{E_{\ell}}(\delta)=\emptyset$, for all $k,\, \ell\in\lbrace 1,...,\tilde{M} \rbrace$ with $k\neq \ell$. Let us fix such a constant once and for all. Now define
\begin{align}
\label{equation:DefinitionNearBoundary}
\Omega_E:=\Omega_{\partial \Omega}(\delta):=\left\lbrace\,  p\in\Omega \mid p\notin\Omega_V,\, d(p,\partial\Omega)<\delta \, \right\rbrace=\bigcup_{j=1}^{\tilde{M}}\Omega_{E_j}(\delta).
\end{align}
Also let $\Omega_{E_j}:=\Omega_{E_j}(\delta)$ for all $j\in\{1,...,\tilde{M}\}$. Thus for any point $p\in \Omega_E$ there exists a unique edge $E_j$, $j\in\{1,..,\tilde{M}\}$, which is the closest to $p$ among all edges. Note that for any $j\in\{\, 1,...,\tilde{M}\,\}$ the set $\Omega_{E_j}$ may have one or two connected components. More precisely, if $\vert E_j \vert = 2\cdot L(E_j)$, where the notation is to be understood as in Definition \ref{definition:VolumenVerallgemeinerteFlaeche}, then $\Omega_{E_j}$ has exactly two connected components, and otherwise only one. Also note that $\Omega_E$ has exactly $M$ connected components, i.e. the number of angles of $\Omega$ and the number of connected components of $\Omega_E$ is the same.

The remaining points will constitute the points away from the boundary, i.e.
\begin{align}
\label{equation:DefinitionAwayBoundary}
\Omega_I:=\Omega\backslash \left(\Omega_V\cup \Omega_E\right)=\lbrace\, p\in\Omega \mid p\notin\Omega_V, \, p\notin \Omega_E\, \rbrace.
\end{align}

Now that we have established the decomposition of the polygon, we can apply the principle of not feeling the boundary.

\begin{lemma} 
\label{lemma:PNFB}
There exist constants $T,C,D >0,$ such that for all $t\in(0,T]$, $i\in\{\, 1,...,M \,\}$ and $j\in\{\, 1,...,\tilde{M} \,\}:$
\begin{itemize}
\item[$(i)$] 
\begin{align}
\label{equation:PNFBInterior}
\vert K_{\Omega}(x,x;t)-K_{\mathbb{H}^2}(x,x;t) \vert &\leq \frac{C}{t}e^{-\frac{D}{t}} \quad \text{ for all } x\in \Omega_I. 
\end{align}

\item[$(ii)$] 
\begin{align}
\label{equation:PNFBBoundary}
\vert K_{\Omega}(x,x;t)-K_{E_j}(x,x;t) \vert &\leq \frac{C}{t}e^{-\frac{D}{t}} \quad \text{ for all } x\in \Omega_{E_j}, 
\end{align}
where $K_{E_j}(x,x;t):=K_{H_j(x)}(x,x;t)$ and $K_{H_j(x)}$ denotes the heat kernel of the half-plane $H_j:=H_j(x)\subset \mathbb{H}^2$ bounded by the line which contains the edge $E_j$ and with $x\in H_j$.
\item[$(iii)$]
\begin{align}
\label{equation:PNFBVertices}
\vert K_{\Omega}(x,x;t)-K_{\gamma_i}(x,x;t) \vert &\leq \frac{C}{t}e^{-\frac{D}{t}} \quad \text{ for all } x\in W_R(P_i), 
\end{align}
where $K_{\gamma_i}$ is the heat kernel of the wedge $W_i$.
\end{itemize}
\end{lemma}

\begin{proof}
From Lemma \ref{lemma:PNFBAllgemein} we know that there exist $T_1, C_1, D_1 > 0$ such that 
\begin{align}
\label{equation:BeweisPNFBStufe1.1}
\vert K_{\Omega}(x,x;t)-K_{\mathbb{H}^2}(x,x;t) \vert &\leq \frac{C_1}{t}e^{-\frac{D_1}{t}}\quad \text{ for all } x\in \Omega_I,\, t\in (0,T_1].
\end{align}

Next, we will use Lemma \ref{lemma:PNFBAllgemein} and Lemma \ref{lemma:ProbabilisticFormulaHeatKernel} in order to estimate the left-hand side of \eqref{equation:PNFBBoundary} and then of \eqref{equation:PNFBVertices} as in \citep[Lemma 6 and 7]{VanDenBerg}. Let $j\in\{\, 1,...,\tilde{M} \,\}$ be given, let $x\in \Omega_{E_j}$, and let $F_j$ denote the geodesic line containing the edge $E_j$. With Lemma \ref{lemma:ProbabilisticFormulaHeatKernel} we have
\begin{align*}
K_{\Omega}(x,x;t)&=K_{\mathbb{H}^2}(x,x;t)\cdot \text{Prob}\{\, \omega(s)\in \Omega,\, 0\leq s\leq t \mid \omega(0)=x=\omega(t)  \,\} \\
&\leq K_{\mathbb{H}^2}(x,x;t)\cdot \text{Prob}\{\, \omega(s)\notin E_j,\, 0\leq s\leq t \mid \omega(0)=x=\omega(t)  \,\} \\
&\leq K_{\mathbb{H}^2}(x,x;t)\cdot \text{Prob}\{\, \omega(s)\notin F_j,\, 0\leq s\leq t \mid \omega(0)=x=\omega(t)  \,\} \\
&\quad\text{ }+ K_{\mathbb{H}^2}(x,x;t)\cdot \text{Prob}\{\, \omega(s)\in  F_j \backslash E_j \text{ for some } s\in[0, t] \mid \omega(0)=x=\omega(t)  \,\} \\
&= K_{E_j}(x,x;t) + K_{\mathbb{H}^2}(x,x;t) \\
&\quad  - K_{\mathbb{H}^2}(x,x;t)\cdot \text{Prob}\{\, \omega(s)\in \mathbb{H}^2 \backslash \left( F_j\backslash E_j \right),\, 0\leq s\leq t \mid \omega(0)=x=\omega(t)  \,\} \\
&= K_{E_j}(x,x;t)+K_{\mathbb{H}^2}(x,x;t) - K_{\mathbb{H}^2 \backslash \left( F_j\backslash E_j \right)}(x,x;t).
\end{align*}
Note that we used in particular the following implications with respect to any continuous curve $\omega:[0,t]\rightarrow\mathbb{H}^2$ with $\omega(0)=x=\omega(t)$: $\omega$ stays in $\Omega \Rightarrow \omega$ does not meet $E_j \Rightarrow \omega$ does not meet $F_j$ or meets $F_j \backslash E_j$, and: $\omega$ does not meet $F_j \Leftrightarrow \omega$ stays in $H_j$.  

Thus from Lemma \ref{lemma:PNFBAllgemein} we know that there exist $\tilde{T}_2^{j}, \tilde{C}_2^{j},\tilde{D}_2^{j}>0$ such that for all $t\in(0,\tilde{T}_2^{j}]$, $x\in \Omega_{E_j}$:
\begin{align}
\label{equation:BeweisPNFB1}
K_{\Omega}(x,x;t) \leq K_{E_j}(x,x;t) + \frac{\tilde{C}_2^{j}}{t}e^{-\frac{\tilde{D}_2^{j}}{t}}.
\end{align}
Similarly, we have for all $x\in \Omega_{E_j}$:
\begin{align*}
K_{\Omega}(x,x;t) &= K_{\mathbb{H}^2}(x,x;t)\cdot \text{Prob}\{\, \omega(s)\in \Omega,\, 0\leq s\leq t \mid \omega(0)=x=\omega(t)  \,\}\\
&= K_{\mathbb{H}^2}(x,x;t)\cdot \text{Prob}\{\, \omega(s)\notin \partial\Omega,\, 0\leq s\leq t \mid \omega(0)=x=\omega(t)  \,\}\\
& \geq K_{\mathbb{H}^2}(x,x;t)\cdot \text{Prob}\{\, \omega(s)\notin F_j,\, 0\leq s\leq t \mid \omega(0)=x=\omega(t)  \,\} \\
&\quad - K_{\mathbb{H}^2}(x,x;t)\cdot \text{Prob}\{\, \omega(s)\in \partial\Omega\backslash F_j \text{ for some } s\in[0, t] \mid \omega(0)=x=\omega(t)  \,\} \\
&= K_{E_j}(x,x;t) - \left( K_{\mathbb{H}^2}(x,x;t) - K_{\mathbb{H}^2 \backslash \left( \partial\Omega\backslash F_j \right)} (x,x;t)\right).
\end{align*}
We used in particular the following implications for any continuous curve $\omega:[0,t]\rightarrow\mathbb{H}^2$ with $\omega(0)=x=\omega(t)$: $\omega$ stays in $\Omega \Leftrightarrow \omega$ does not meet $\partial\Omega$; $\omega$ does not meet $F_j \Rightarrow \omega$ does not meet $\partial\Omega$ or meets $\partial\Omega\backslash F_j$. 

Thus it follows again from Lemma \ref{lemma:PNFBAllgemein} that there exist constants $\hat{T}_2^{j}, \hat{C}_2^{j},\hat{D}_2^{j} >0$ such that for all $t\in(0,\hat{T}_2^{j}]$, $x\in \Omega_{E_j}$:
\begin{align}
\label{equation:BeweisPNFB2}
K_{\Omega}(x,x;t) \geq K_{E_j}(x,x;t) - \frac{\hat{C}_2^{j}}{t}e^{-\frac{\hat{D}_2^{j}}{t}}.
\end{align}
When we combine the estimates \eqref{equation:BeweisPNFB1} and \eqref{equation:BeweisPNFB2} and we set $T_2:=\min\cup_{j=1}^{\tilde{M}}\{ \tilde{T}_2^{j}, \hat{T}_2^{j} \}$, $C_2:=\max\cup_{j=1}^{\tilde{M}}\{ \tilde{C}_2^{j}, \hat{C}_2^{j} \}$, $D_2:=\min\cup_{j=1}^{\tilde{M}}\{ \tilde{D}_2^{j}, \hat{D}_2^{j} \}$, then we have for all $j\in\{\, 1,...,\tilde{M} \,\}$:
\begin{align}
\label{equation:BeweisPNFBStufe1.2}
\vert K_{\Omega}(x,x;t)-K_{E_j}(x,x;t) \vert &\leq \frac{C_2}{t}e^{-\frac{D_2}{t}}\quad \text{ for all } x\in \Omega_{E_j},\, t\in (0,T_2].
\end{align}

Now consider \eqref{equation:PNFBVertices} and let $i\in\{ 1,...,M\}$ be given. For any $r>0$ we set $E_i(r):=\partial \hyperball{r}{P_i} \cap W_i=\{\, p\in W_i\mid d(p,P_i)=r \,\}$. Because of $W_{2R}(P_i)\subset\Omega$, we have for all $x\in W_R(P_i)$:
\begin{align*}
K_{\Omega}(x,x;t) &\geq K_{\mathbb{H}^2}(x,x;t)\cdot \text{Prob}\{\, \omega(s)\in W_{2R}(P_i),\, 0\leq s\leq t \mid \omega(0)=x=\omega(t)  \,\}\\
&\geq K_{\mathbb{H}^2}(x,x;t)\cdot \text{Prob}\{\, \omega(s)\in W_i,\, 0\leq s\leq t \mid \omega(0)=x=\omega(t)  \,\}\\
&\quad - K_{\mathbb{H}^2}(x,x;t)\cdot \text{Prob}\{\, \omega(s)\in E_i(2R)\text{ for some } s\in[0, t] \mid \omega(0)=x=\omega(t)  \,\} \\
&= K_{\gamma_i}(x,x;t) - \left( K_{\mathbb{H}^2}(x,x;t) - K_{\mathbb{H}^2\backslash E_i(2R)}(x,x;t) \right). 
\end{align*}
We used the following implications valid for any continuous curve $\omega:[0,t]\rightarrow\mathbb{H}^2$ with $\omega(0)=x=\omega(t)$: $\omega$ stays in $\Omega \Rightarrow \omega$ stays in $W_{2R}(P_i)$; $\omega$ stays in $W_i \Rightarrow \omega$ stays in $W_{2R}(P_i)$ or meets $E_i(2R)$.
Hence, using Lemma \ref{lemma:PNFBAllgemein}, there exist $T_3^{i}, C_3^{i}, D_3^{i} >0$ such that for all $x\in W_R(P_i)$, $t\in (0,T_3^{i}]$:
\begin{align}
\label{equation:LemmaPNFBInequality3}
K_{\Omega}(x,x;t) \geq K_{\gamma_i}(x,x;t) - \frac{C_3^{i}}{t}e^{-\frac{D_3^{i}}{t}}.
\end{align}
Similarly, for all $x\in W_R(P_i)$:
\begin{align*}
K_{\Omega}(x,x;t) &\leq  K_{\mathbb{H}^2}(x,x;t)\cdot \text{Prob}\{\, \omega(s)\in W_i,\, 0\leq s\leq t \mid \omega(0)=x=\omega(t)  \,\}\\
&\quad + K_{\mathbb{H}^2}(x,x;t)\cdot \text{Prob}\{\, \omega(s)\in E_i(2R)\text{ for some } s\in[0, t] \mid \omega(0)=x=\omega(t)  \,\} \\
& = K_{\gamma_i}(x,x;t) +  \left(  K_{\mathbb{H}^2}(x,x;t) - K_{\mathbb{H}^2\backslash E_i(2R)}(x,x;t) \right).
\end{align*}
We used above the following implication which holds for any continuous curve $\omega:[0,t]\rightarrow\mathbb{H}^2$ with $\omega(0)=x=\omega(t)$: $\omega$ stays in $\Omega\Rightarrow \omega$ stays in $W_i$ or meets $E_i(2R)$. Thus we have
\begin{align}
\label{equation:LemmaPNFBInequality4}
K_{\Omega}(x,x;t) \leq K_{\gamma_i}(x,x;t) + \frac{C_3^{i}}{t}e^{-\frac{D_3^{i}}{t}}\quad \text{ for all }\, x\in W_R(P_i), \, t\in (0,T_3^{i}].
\end{align}
Let $T_3:=\min \cup_{i=1}^M\{ T_3^{i} \}$, $C_3:=\max \cup_{i=1}^M \{ C_3^{i} \}$, $D_3:=\min \cup_{i=1}^M\{ D_3^{i} \}$. From \eqref{equation:LemmaPNFBInequality3} and \eqref{equation:LemmaPNFBInequality4} it follows that for all $i\in\{\, 1,...,M \,\}$:
\begin{align}
\label{equation:BeweisPNFBStufe1.3}
\vert K_{\Omega}(x,x;t)-K_{\gamma_i}(x,x;t) \vert \leq  \frac{C_3}{t}e^{-\frac{D_3}{t}} \quad \text{ for all }\, x\in W_R(P_i), \, t\in (0,T_3].
\end{align}
Lastly, all three estimates $(i)-(iii)$ follow from \eqref{equation:BeweisPNFBStufe1.1}, \eqref{equation:BeweisPNFBStufe1.2} and \eqref{equation:BeweisPNFBStufe1.3}, when we set $T:=\min\{\, T_1,T_2,T_3 \,\}$, $C:=\max\{\, C_1,C_2,C_3 \,\}$, and $D:=\min\{\, D_1,D_2,D_3 \,\}$.
\end{proof}

Instead of applying the above lemma to the heat trace directly, we will use it to approximate the shifted heat trace
\begin{align}
\label{equation:DefinitionShiftedHeatTrace}
Z_{\Omega}^{\nicefrac{1}{4}}(t):= e^{\frac{1}{4}t}\cdot Z_{\Omega}(t) = \int\limits_{\Omega}K_{\Omega}^{\nicefrac{1}{4}}(x,x;t)dx
\end{align}
and compute its asymptotic expansion as $t\searrow 0$. Of course the two functions $Z_{\Omega}$ and $Z_{\Omega}^{\nicefrac{1}{4}}$ are closely related to each other, but the formulas for the coefficients in the asymptotic expansion of $Z_{\Omega}^{\nicefrac{1}{4}}(t)$ as $t\searrow 0$ are shorter than for $Z_{\Omega}(t)$.

Because of the preceding lemma we can approximate the shifted heat trace as follows.
\begin{corollary}
\label{corollary:PNFB}
There exist constants $T, C',D>0$ such that for all $t\in (0,T]:$
\begin{align*}
\bigg\vert Z_{\Omega}^{\nicefrac{1}{4}}(t)-\int\limits_{\Omega_I}K_{\mathbb{H}^2}^{\nicefrac{1}{4}}(x,x;t)dx-\sum\limits_{j=1}^{\tilde{M}}\int\limits_{\text{ } \Omega_{E_j}}K_{E_j}^{\nicefrac{1}{4}}(x,x;t)dx -\sum\limits_{i=1}^{M} \int\limits_{\text{ }W_R (P_i)}K_{\gamma_i}^{\nicefrac{1}{4}}(x,x;t)dx \bigg\vert \leq \frac{C'}{t}e^{-\frac{D}{t}},
\end{align*}
where $K_{\mathbb{H}^2}^{\nicefrac{1}{4}}, K_{E_j}^{\nicefrac{1}{4}}$ and $K_{\gamma_i}^{\nicefrac{1}{4}}$ denote the corresponding shifted heat kernels as in \eqref{equation:ShiftedHeat}.
\end{corollary}

\begin{proof}
We decompose the polygon into the pairwise disjoint subsets from \eqref{equation:DefinitionNearVertices}-\eqref{equation:DefinitionAwayBoundary}:
\begin{align*}
\Omega=\Omega_I\cup\Omega_E\cup\Omega_V = \Omega_{I}\cup \left( \bigcup_{j=1}^{\tilde{M}}\Omega_{E_j} \right) \cup\left( \bigcup_{i=1}^M W_R(P_i) \right).
\end{align*}
Let $T,C,D>0$ be as in Lemma \ref{lemma:PNFB} and $C':=\vert \Omega\vert e^{\frac{T}{4}} \cdot C$. Then for all $t\in(0,T]$:
\begin{align*}
&\bigg\vert Z_{\Omega}^{\nicefrac{1}{4}}(t)-\int\limits_{\Omega_I}K_{\mathbb{H}^2}^{\nicefrac{1}{4}}(x,x;t)dx-\sum\limits_{j=1}^{\tilde{M}}\int\limits_{\text{ }\Omega_{E_j}}K_{E_j}^{\nicefrac{1}{4}}(x,x;t)dx -\sum\limits_{i=1}^{M} \int\limits_{\text{ }W_R (P_i)}K_{\gamma_i}^{\nicefrac{1}{4}}(x,x;t)dx \bigg\vert = \\
&\bigg\vert \int\limits_{\Omega_I} \bigg(  K_{\Omega}^{\nicefrac{1}{4}}(x,x;t)- K_{\mathbb{H}^2}^{\nicefrac{1}{4}}(x,x;t)\bigg)dx + \sum\limits_{j=1}^{\tilde{M}}\int\limits_{\text{ }\Omega_{E_j}} \bigg( K_{\Omega}^{\nicefrac{1}{4}}(x,x;t)-K_{E_j}^{\nicefrac{1}{4}}(x,x;t)\bigg)dx\\ 
&\qquad\qquad\qquad\qquad\qquad\qquad\qquad\qquad + \sum\limits_{i=1}^{M} \int\limits_{\text{ }W_R (P_i)} \bigg( K_{\Omega}^{\nicefrac{1}{4}}(x,x;t)-K_{\gamma_i}^{\nicefrac{1}{4}}(x,x;t)\bigg) dx \bigg\vert \leq \\
& \int\limits_{\Omega_I}\bigg\vert  K_{\Omega}^{\nicefrac{1}{4}}(x,x;t)- K_{\mathbb{H}^2}^{\nicefrac{1}{4}}(x,x;t) \bigg\vert dx +\sum\limits_{j=1}^{\tilde{M}}\int\limits_{\text{ }\Omega_{E_j}} \bigg\vert  K_{\Omega}^{\nicefrac{1}{4}}(x,x;t)-K_{E_j}^{\nicefrac{1}{4}}(x,x;t)\bigg\vert dx\\
&\qquad\qquad\qquad\qquad\qquad\qquad\qquad\qquad +  \sum\limits_{i=1}^{M} \int\limits_{\text{ }W_R (P_i)}\bigg\vert K_{\Omega}^{\nicefrac{1}{4}}(x,x;t)-K_{\gamma_i}^{\nicefrac{1}{4}}(x,x;t)\bigg\vert dx \leq \\
& \left(\vert \Omega_I\vert + \sum\limits_{j=1}^{\tilde{M}} \vert \Omega_{E_j} \vert + \sum\limits_{i=1}^M\vert W_R(P_i) \vert  \right)e^{\frac{T}{4}} \frac{C}{t}e^{-\frac{D}{t}} = \frac{C'}{t}e^{-\frac{D}{t}}.
\end{align*}
\end{proof}

Corollary \ref{corollary:PNFB} shows that the shifted heat trace $Z_{\Omega}^{\nicefrac{1}{4}}(t)$ has the same asymptotic expansion for $t\searrow 0$ as the function
\begin{align}
\label{equation:ErsatzSummeAsymptotik1}
t\mapsto \int\limits_{\Omega_I}K_{\mathbb{H}^2}^{\nicefrac{1}{4}}(x,x;t)dx + \sum\limits_{j=1}^{\tilde{M}}\int\limits_{\text{ }\Omega_{E_j}}K_{E_j}^{\nicefrac{1}{4}}(x,x;t)dx + \sum\limits_{i=1}^{M} \int\limits_{\text{ }W_R (P_i)}K_{\gamma_i}^{\nicefrac{1}{4}}(x,x;t)dx.
\end{align}
For any $i\in \lbrace\, 1,...,M\, \rbrace$ an explicit formula for the function $\int_{W_R (P_i)}K_{\gamma_i}^{\nicefrac{1}{4}}(x,x;t)dx$ is given by Theorem \ref{theorem:HeatTraceShift}. Furthermore, we have for all $j = 1,...,\tilde{M}$ and $x\in \Omega_{E_j}$
\begin{align*}
K_{E_j}(x,x;t)= K_{H_j}(x,x;t) = K_{\mathbb{H}^2}(x,x;t)-K_{\mathbb{H}^2}(x,x_{E_j};t),
\end{align*}
where $x_{E_j}$ denotes the image of $x\in H_j$ under the reflection in $\partial H_j$. This follows from the following Lemma.
\begin{lemma}
\label{lemma:HeatKernelHalfPlaneProof}
Let $H\subset \mathbb{H}^2$ be a half-plane. The heat kernel of $H$ is given for all $x,y\in H$ and $t>0$ by
\begin{align*}
K_{H}(x,y;t)=K_{\mathbb{H}^2}(x,y;t)-K_{\mathbb{H}^2}(x,y_{\partial H};t),
\end{align*}
where $y_{\partial H}$ denotes the reflection of $y$ in the boundary $\partial H$.
\end{lemma}

\begin{proof}
This follows immediately, if we set $\gamma=\pi$ in \eqref{equation:GreenWedge} and then perform the inverse Laplace transform. Alternatively, the statement can also be shown directly by using the definition of $K_{\mathbb{H}^2}$ and \citep[Theorem 9.7]{Grigoryan}.
\end{proof}

 Thus \eqref{equation:ErsatzSummeAsymptotik1} is equal to
\begin{align}
\label{equation:ErsatzSummeAsymptotik2}
\int\limits_{\Omega}&K_{\mathbb{H}^2}^{\nicefrac{1}{4}}(x,x;t)dx \nonumber\\
&- \left( \sum\limits_{j=1}^{\tilde{M}} \int\limits_{\text{ }\Omega_{E_j}}K_{\mathbb{H}^2}^{\nicefrac{1}{4}}(x,x_{E_j};t)dx\right)  - M\cdot 2\int\limits_{0}^{R}\int\limits_{0}^{\frac{\pi}{2}} K_{\mathbb{H}^2}^{\nicefrac{1}{4}}((a,\alpha),(a,-\alpha);t) \sinh(a) d\alpha da \nonumber \\
&+ \sum\limits_{i=1}^{M}  \frac{\gamma_i}{2\pi}\int\limits_{0}^{\infty} \frac{e^{-\frac{u^2}{4t}}}{\sqrt{4\pi t}} \cdot \frac{e^{\frac{u}{2}}}{e^{u}-1} \cdot \left( \frac{\pi}{\gamma_i}\coth\left( \frac{\pi}{\gamma_i} \cdot \frac{u}{2} \right)  - \coth\left(\frac{u}{2}\right) \right)du  \nonumber \\
&- \sum\limits_{i=1}^{M}A_{\gamma_i}(t;R).
\end{align}

Each line of \eqref{equation:ErsatzSummeAsymptotik2} defines a function in $t\in (0,\infty)$, which we denote (including the signs) in successive order by $Z_I^{\nicefrac{1}{4}}(t)$, $Z_E^{\nicefrac{1}{4}}(t)$, $Z_V^{\nicefrac{1}{4}}(t)$, and $A(t)$, respectively. We will compute the asymptotic expansion of each function separately for $t\searrow 0$. It will turn out that $Z_I^{\nicefrac{1}{4}}(t)$ produces the contributions from the interior of the polygon, i.e. those parts which depend on the volume $\vert \Omega \vert$ of the polygon. Similarly, $Z_{E}^{\nicefrac{1}{4}}(t)$ produces the contributions from the edges of the polygon, i.e. those parts which depend on the length of the perimeter $\vert \partial \Omega \vert$. Further, $Z_V^{\nicefrac{1}{4}}(t)$ produces the contributions from the vertices of the polygon, i.e. those parts which depend on the angles of the polygon. The function $A(t)$ will not give any contribution to the asymptotic expansion as $t\searrow 0$. 

The remaining part of this section will be devoted to the asymptotic expansion of each of these functions. In the process, we will encounter the Bernoulli numbers $B_k$ and Bernoulli polynomials $B_k(x)$, which are defined by the following generating function:
\begin{align}
\label{equation:DefinitionBernoulli}
\frac{te^{xt}}{e^t-1}&=:\sum\limits_{k=0}^{\infty}B_k(x)\frac{t^k}{k!} \quad \text{for } \vert t\vert< 2\pi; \quad B_k:=B_k(0).
\end{align} 
By an easy computation one can prove the identity $B_k\left(1-x\right)=\left( -1 \right)^k B_k(x)$ for all $k\in\mathbb{N}_0$ and thus
\begin{align}
\label{equation:VanishingBernoulli}
B_{2k+1}\left(\frac{1}{2}\right)=0 \quad \text{for all } k\in\mathbb{N}_0.
\end{align}
Both Bernoulli numbers and Bernoulli polynomials will appear naturally in the asymptotic expansion of the (shifted) heat trace, or more precisely, in all the coefficients which depend explicitly on the volume $\vert \Omega \vert$ or the interior angles $\gamma_1,...,\gamma_M$ of the polygon. This resembles the situation for spherical polygons discussed in \cite{Watson}. Indeed, as explained already at the beginning of Chapter \ref{chapter:polygons}, a later comparison between the heat invariants for spherical and hyperbolic polygons will show that the corresponding parts which depend on either the volume $\vert \Omega \vert$, the perimeter $\vert \partial \Omega \vert$ or the angles of the polygon differ at most in sign.

We begin with asymptotic expansions of the first three terms in \eqref{equation:ErsatzSummeAsymptotik2}, namely the functions $Z_I^{\nicefrac{1}{4}}(t),\, Z_E^{\nicefrac{1}{4}}(t)$ and $Z_V^{\nicefrac{1}{4}} (t)$. The central tool for computing the asymptotic expansions for these functions is Watson's lemma, due to G. N. Watson.

\begin{lemma} 
\label{lemma:WatsonsLemma}
\emph{(see \cite{Bleistein}, p.103)}
Let $f:(0,\infty)\rightarrow \mathbb{R}$ be a continuous function such that there exists a constant $c>0$ with $f(x)=O(e^{cx})$ as $x\rightarrow \infty$. If $f(x)$ has an asymptotic expansion of the form
\begin{align*}
f(x) \overset{x\downarrow 0}{\sim} \sum\limits_{k=0}^{\infty} a_k x^{\rho_k -1}
\end{align*}
with $0<\rho_0<\rho_1<...\nearrow\infty$, then it follows that
\begin{align*}
\int\limits_{0}^{\infty}f(x)e^{-\frac{x}{t}}dx\text{ }\overset{t\downarrow 0}{\sim}\sum\limits_{k=0}^{\infty} a_k \cdot \Gamma(\rho_k)\cdot t^{\rho_k}.
\end{align*}
\end{lemma}
Our approach will always be the same. Using the Fubini-Tonelli theorem and substitution, we will bring those functions into a form to which Watson's lemma is applicable.

By formula \eqref{equation:HyperHeat2} for the heat kernel of the hyperbolic plane, the first term is
\begin{align}
\label{equation:DefinitionOfZ_I}
Z_I^{\nicefrac{1}{4}}(t):=\int\limits_{\Omega}K_{\mathbb{H}^2}^{\nicefrac{1}{4}}(x,x;t) dx =  \frac{\vert \Omega \vert}{(4\pi t)^{\frac{3}{2}}}\int\limits_{0}^{\infty}\frac{\sqrt{2}\cdot r}{\sqrt{\cosh r - 1}}e^{-\frac{r^2}{4t}}dr.
\end{align}
Thus we need to compute the asymptotic expansion of the integral expression on the right-hand side of \eqref{equation:DefinitionOfZ_I}.

\begin{lemma}
Let $f(r):=\frac{\sqrt{2}\cdot r}{\sqrt{\cosh r - 1}}$. Then
\begin{align}
\label{lemma:AsymptotikContributionInterior}
\int\limits_{0}^{\infty}f(r)e^{-\frac{r^2}{4t}}dr \overset{t\downarrow 0}{\sim}  \sum\limits_{k=0}^{\infty} 2\sqrt{\pi}\cdot\frac{ B_{2k}\left(\frac{1}{2}\right)}{k!}t^{k+\frac{1}{2}}.
\end{align}
\end{lemma}

\begin{proof}
First we will compute the asymptotic expansion of $f(r)$ as $r\searrow 0$. The function $f$ is obviously smooth and bounded on $(0,\infty)$ and can even be extended to a smooth function on the whole real line $\mathbb{R}$ by the following identity:
\begin{align*}
f(r)=\frac{\sqrt{2}\cdot r}{\sqrt{\cosh r -1}}=\frac{2r}{ e^{\frac{r}{2}}-e^{-\frac{r}{2}}}= 2\frac{re^{\frac{r}{2}}}{e^r-1}.
\end{align*}
Hence, by \eqref{equation:DefinitionBernoulli} we get
\begin{align*}
f(r)\overset{r\downarrow 0}{\sim} \sum\limits_{k=0}^{\infty}2\cdot\frac{B_k(\frac{1}{2})}{k!}r^k.
\end{align*} 
By \eqref{equation:VanishingBernoulli}, $B_{2k+1}\left( \frac{1}{2}\right)=0$ for all $k\in\mathbb{N}_0$ and hence we can write the expansion also as
\begin{align}
\label{equation:LemmaAsymptotikContributionInterior1}
f(r)\overset{r\downarrow 0}{\sim} \sum\limits_{k=0}^{\infty}2\cdot\frac{B_{2k}(\frac{1}{2})}{(2k)!}r^{2k}.
\end{align}
We will now use Watson's lemma  to get the expansion we claimed. We first substite $r$ by $2\sqrt{r}$ and get
\begin{align*}
\int\limits_0^{\infty}f(r)e^{-\frac{r^2}{4t}}dr = \int_0^{\infty}\frac{f\left(2\sqrt{r}\right)}{\sqrt{r}}e^{-\frac{r}{t}}dr=\int\limits_{0}^{\infty} \tilde{f}(r)e^{-\frac{r}{t}}dr,
\end{align*}
where $\tilde{f}(r):=\frac{f\left(2\sqrt{r}\right)}{\sqrt{r}}$. From \eqref{equation:LemmaAsymptotikContributionInterior1} it follows that
\begin{align*}
\tilde{f}(r)\overset{r\downarrow 0}{\sim}\sum\limits_{k=0}^{\infty} 2^{2k+1} \cdot \frac{ B_{2k}\left( \frac{1}{2} \right)}{(2k)!}\cdot r^{\left( k+\frac{1}{2}\right)-1}.
\end{align*}
Hence by Watson's lemma we get
\begin{align}
\label{equation:LemmaAsymptotikContributionInterior2}
\int\limits_{0}^{\infty} \tilde{f}(r)e^{-\frac{r}{t}}dr \overset{t\downarrow 0}{\sim} \sum\limits_{k=0}^{\infty} 2^{2k+1} \cdot \frac{ B_{2k}\left( \frac{1}{2} \right)}{(2k)!}\cdot\Gamma\left(k+\frac{1}{2}\right)\cdot t^{ k+\frac{1}{2}}.
\end{align}
From \citep[8.339 2]{Gradshteyn}, we know that for all $k\in\mathbb{N}_0$
\begin{align}
\label{equation:GammaWert}
\Gamma\left( k+\frac{1}{2}\right)=\frac{\left(2k \right)!}{2^{2k}\cdot k!}\cdot \sqrt{\pi}.
\end{align}
When we combine \eqref{equation:LemmaAsymptotikContributionInterior2} with \eqref{equation:GammaWert}, then the claimed asymptotic expansion \eqref{lemma:AsymptotikContributionInterior} follows.
\end{proof}

One easily deduces the asymptotic expansion of $Z_I(t):=e^{-\frac{t}{4}}\cdot Z_{I}^{\nicefrac{1}{4}}(t)$ as $t\searrow 0$ from the preceding lemma and \eqref{equation:DefinitionOfZ_I}:

\begin{corollary} \hfill
\label{corollary:AsymptotischeEntwicklungShiftedZ_I}
\begin{itemize}
\item[$(i)$]
\begin{align}
\label{equation:AsymptotischeEntwicklungShiftedZ_I}
Z_I^{\nicefrac{1}{4}}(t)\overset{t\downarrow 0}{\sim}\frac{\vert\Omega\vert}{4\pi t}\sum\limits_{k=0}^{\infty}\frac{ B_{2k}\left( \frac{1}{2} \right) }{k!} t^k.
\end{align}
\item[$(ii)$]
\begin{align}
\label{equation:AsymptotischeEntwicklungZ_I}
Z_I(t)\overset{t\downarrow 0}{\sim}\frac{\vert\Omega\vert}{4\pi t}\sum\limits_{k=0}^{\infty} \tilde{i}_k t^k,
\end{align}
where $\tilde{i}_k:=\frac{1}{k!}\sum\limits_{\ell=0}^k\binom{k}{\ell}\left(-\frac{1}{4}\right)^{k-\ell}B_{2\ell}\left(\frac{1}{2}\right).$
\end{itemize}
\end{corollary}

In order to compare the above coefficients with the corresponding coefficients for spherical polygons given in \cite{Watson}, we rewrite \eqref{equation:AsymptotischeEntwicklungZ_I} appropriately. Because of $\tilde{i}_0=B_0\left(\frac{1}{2}\right)=1$, we can alternatively write the asymptotic expansion as
\begin{align}
Z_I(t)\overset{t\downarrow 0}{\sim} \frac{\vert\Omega\vert}{4\pi t} + \sum\limits_{k=0}^{\infty} i_k^{\mathbb{H}}\cdot  t^k,
\end{align}
where 
\begin{align}
i_k^{\mathbb{H}} := \frac{\vert \Omega \vert}{4\pi}\cdot\tilde{i}_{k+1}= \frac{\vert \Omega\vert}{4\pi}\cdot \frac{1}{(k+1)!}\sum\limits_{\ell=0}^{k+1}\binom{k+1}{\ell}\left(-\frac{1}{4}\right)^{k+1-\ell}B_{2\ell}\left(\frac{1}{2}\right).
\end{align}

 If we compare the coefficients $i_k^{\mathbb{H}}$ with the corresponding coefficients for spherical polygons given in \citep[equation (28) of Corollary 3]{Watson} - which we denote unlike Watson as $i_k^{\mathbb{S}}$ - we see the relation $i_k^{\mathbb{H}}=\left( -1 \right)^{k+1}i_k^{\mathbb{S}}$ for all $ k\in\mathbb{N}_0$. (Note that in \citep[equation (28)]{Watson} the constant $ \frac{\vert \Omega\vert}{4\pi}$ is actually completely missing because of a typo). As we mentioned before, that relation between $i_k^{\mathbb{H}}$ and $i_k^{\mathbb{S}}$ jibes with general facts from \citep{GilkeyBranson}. It follows from \citep{GilkeyBranson} that $i_k^{\mathbb{H}}$ and $i_k^{\mathbb{S}}$ are given by a formula of the form $\vert \Omega \vert \cdot c_k \cdot \kappa^{k+1}$, where $c_k\in\mathbb{R}$ is some universal constant, $\kappa$ denotes the Gaussian curvature of $\mathbb{H}^2$ and $\mathbb{S}^2$, respectively, and $\Omega$ represents a hyperbolic or spherical polygon, as appropriate. Note that $\mathbb{H}^2$ and $\mathbb{S}^2$ have constant Gaussian curvature equal to $-1$ and $1$, respectively. In particular, the covariant derivatives of the Riemannian curvature tensor vanish and hence all curvature invariants vanish as well, except for powers of the Gaussian (or sectional) curvature.
 
We now compute the asymptotic expansion of the second term in \eqref{equation:ErsatzSummeAsymptotik2}, namely the function
\begin{align}
\label{equation:DefinitionOfZ_E(t)}
Z_E^{\nicefrac{1}{4}}(t) := -  & \left( \sum\limits_{j=1}^{\tilde{M}}  \int\limits_{\text{ }\Omega_{E_j}}K_{\mathbb{H}^2}^{\nicefrac{1}{4}}(x,x_{E_j};t)dx \right) \nonumber \\
& \qquad\quad\qquad\quad - M\cdot 2\int\limits_{0}^{R}\int\limits_{0}^{\frac{\pi}{2}} K_{\mathbb{H}^2}^{\nicefrac{1}{4}}((a,\alpha),(a,-\alpha);t) \sinh(a) d\alpha da .
\end{align}
Recall that the sets $\Omega_{E_j}$, $j=1,..., \tilde{M}$, are pairwise disjoint, and $\Omega_E=\cup_{j=1}^{\tilde{M}}\Omega_{E_j}$ has $M$ connected components. Moreover, $\Omega_{E_j}$ consists of two connected components if and only if $\vert E_j \vert = 2\cdot L(E_j)$, and only one connected component if and only if $\vert E_j \vert = L(E_j)$ (see Definition \ref{definition:VolumenVerallgemeinerteFlaeche}). 

Let $C$ be any connected component of $\Omega_E$. Let $\ell \in\{ 1,...,\tilde{M} \}$ be the unique index such that $C\subset \Omega_{E_{\ell}}$. We will compute the asymptotic expansion of 
\begin{align*}
I_{C}^{\nicefrac{1}{4}}(t):&=\int\limits_{\text{ }C}K_{\mathbb{H}^2}^{\nicefrac{1}{4}}(x,x_{E_{\ell}};t)dx + 2\int\limits_{0}^{R}\int\limits_{0}^{\frac{\pi}{2}} K_{\mathbb{H}^2}^{\nicefrac{1}{4}}((a,\alpha),(a,-\alpha);t) \sinh(a) d\alpha da \\
&=\int\limits_{Q_R^{1}\cup C \cup Q_{R}^2}K_{\mathbb{H}^2}^{\nicefrac{1}{4}}(x,x_{E_{\ell}};t)dx,
\end{align*}
where $Q_R^1, Q_R^2$ denote the quarter circles shown in Fig. \ref{Skizze:RandbeitragEcke1}. More precisely, let $y\in C$ be fixed and let $\tilde{P}, \hat{P}\in\mathbb{H}^2$ be the endpoints of the edge $E_{\ell}$. Then $Q_R^1$ and $Q_R^2$ denote the quarter circles with radius $R$ and center $\tilde{P}$ and $\hat{P}$, respectively, bounded on one side by $E_{\ell}$ and such that $Q_R^1 \cup C \cup Q_R^2\subset H_{\ell}(y)$.

\begin{figure} [ht] 
 \centering
\begin{tikzpicture}


\node[below] at (1,0) {$E_{\ell}$};
\node[above] at (1,0) {$C$};
\node[above] at (0.6,1.2) {$\Omega$};


\fill[fill=black!10!white] (3,0) -- (1.5,0) arc (180:90:15mm) -- (3,0); 
\fill[fill=black!10!white] (-1,0) -- (0.5,0) arc (0:90:15mm) -- (-1,0);  
\draw[dashed] (0.5,0) arc (0:90:15mm);
\draw[dashed] (3,0) -- (3,1.5);
\draw[dashed] (-1,0) -- (-1,1.5);
\draw[dashed] (1.5,0) arc (180:90:15mm); 

\node[circle,fill,inner sep=1pt] at (3,0) {};
\node[circle,fill,inner sep=1pt] at (-1,0) {};

\node[right] at (-0.8,0.3) {\tiny $Q_R^{1}$};
\node[above] at (2.1,0) {\tiny $Q_R^{2}$};

\node[below] at (9.5,0) {$E_{\ell}$};
\node[above] at (9.5,0) {$C$};
\node[above] at (9.3,1.2) {$\Omega$};

\fill[fill=black!10!white] (7.5,0) -- (9,0) arc (0:90:15mm) -- (7.5,0);  
\fill[fill=black!10!white] (11.5,0) -- (10,0) arc (180:90:15mm) -- (11.5,0); 
\draw[dashed] (10,0) arc (180:90:15mm); 
\draw[dashed] (11.5,0) -- (11.5,1.5);
\draw[dashed] (9,0) arc (0:90:15mm);
\draw[dashed] (7.5,0) -- (7.5,1.5);

\node[circle,fill,inner sep=1pt] at (7.5,0) {};
\node[circle,fill,inner sep=1pt] at (11.5,0) {};

\node[above] at (10.67,0) {\tiny $Q_R^2$};
\node[above] at (8.1,0) {\tiny $Q_R^1$};


\draw[dashed] (0.36,0.634) -- (1.64,0.634);
\draw[dashed] (8.860,0.634) -- (10.14,0.634);

\draw [line width=0.25mm]  (1.6,1.9)   -- (3,0) -- (-1,0) -- (-1.5,1.7) ; 
\draw [line width=0.25mm]  (10.4,1.5)    -- (11.5,0) -- (7.5,0) -- (6.5,-1); 

\end{tikzpicture}
\caption{Quarter circles}\label{Skizze:RandbeitragEcke1}
\end{figure}
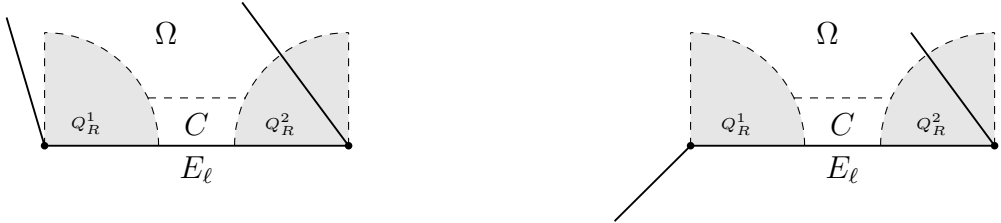

We decompose both quarter circles such that
\begin{align*}
Q_R^1&=S_1 \cup \left( Q_R^1\backslash S_1 \right), \\
Q_R^2&=S_2 \cup \left( Q_R^2\backslash S_2 \right),
\end{align*}
where
\begin{align*}
S_1:&=\{ \, x\in Q_R^1 \mid d(x,E_{\ell})\leq \delta \,\}, \\
S_2:&=\{ \, x\in Q_R^2 \mid d(x,E_{\ell})\leq \delta \,\}.
\end{align*}
Thus
\begin{align*}
Q_R^1\cup C \cup Q_{R}^2 = \underbrace{\big( S_1\cup C \cup S_2 \big)}_{=:S} \bigcup \underbrace{\big( \left( Q_R^1 \backslash S_1 \right) \cup \left( Q_R^2\backslash S_2 \right) \big) }_{=:\Lambda},
\end{align*}
and, since $S$ and $\Lambda$ are disjoint,
\begin{align*}
I_{C}^{\nicefrac{1}{4}}(t) = \int\limits_{\Lambda}K_{\mathbb{H}^2}^{\nicefrac{1}{4}}(x,x_{E_{\ell}};t)dx + \int\limits_{S}K_{\mathbb{H}^2}^{\nicefrac{1}{4}}(x,x_{E_{\ell}};t)dx.
\end{align*}
The first integral does not give any contribution to the asymptotic expansion, because of the estimate \eqref{equation:EstimateHeatKernelPlane} and since
\begin{align*}
d\left(x, x_{E_{\ell}}\right)\geq 2\delta \quad \text{ for all } x\in \Lambda.
\end{align*}
When we choose Fermi coordinates $\left( \rho, \sigma \right)$ with base line $E_{\ell}$, where $\rho$ denotes the distance to $E_{\ell}$, we can parametrise the set $S$ as follows:
\begin{align*}
S = \{ \, \left( \rho, \sigma \right) \mid 0\leq \sigma\leq L(E_{\ell}) ;\text{ } 0\leq \rho \leq \delta \, \},
\end{align*}
where $L(E_{\ell})$ denotes the length of the edge $E_{\ell}$. Since $d\left(x, x_{E_{\ell}}\right)=2\rho$ for any $x=\left( \rho, \sigma \right)$, and the Riemannian metric $g$ is given in Fermi coordinates as $g=d\rho^2+\cosh^2(\rho)d\sigma^2$ (see e.g. \citep[formula $(1.1.9)$]{Buser}), we get with \eqref{equation:HyperHeat2}:
\begin{align}
\label{equation:ContributionEdgeIntegralS}
\int\limits_{S}K_{\mathbb{H}^2}^{\nicefrac{1}{4}}(x,x_{E_{\ell}};t)dx= \frac{L(E_{\ell})}{(4\pi t)^{\frac{3}{2}}} \int\limits_{0}^{\delta}  \int\limits_{2\rho}^{\infty}  \frac{\sqrt{2}\cdot r\cosh(\rho)}{\sqrt{\cosh r -\cosh (2\rho)}}\cdot e^{-\frac{r^2}{4t}}\, dr \, d\rho.
\end{align}
Again, we need to determine the asymptotic expansion of the integral expression on the right-hand side of \eqref{equation:ContributionEdgeIntegralS}.
\begin{lemma}
Let $g(r,\rho):=\frac{\sqrt{2}\cdot r\cosh(\rho)}{\sqrt{\cosh r -\cosh (2\rho)}}$. Then
\begin{align}
\int\limits_{0}^{\delta}  \int\limits_{2\rho}^{\infty} \text{ } g(r,\rho)\cdot e^{-\frac{r^2}{4t}} \, dr \, d\rho \text{ } \overset{t\downarrow 0}{\sim} \pi \sum\limits_{k=0}^{\infty} \delta_{0k} t^{k+1},
\end{align}
where $\delta_{0k}$ is the Kronecker delta, i.e.
\begin{align*}
\delta_{0k}:=
\begin{cases}
1,\, \text{if } k=0,\\
0,\, \text{if } k\geq 1.
\end{cases}
\end{align*}
\end{lemma}

\begin{proof}
By using first the Fubini-Tonelli theorem and then substituting $r$ by $2\sqrt{r}$, we get
\begin{align}
\label{equation:BeweisUmformungG(r)Integral}
\int\limits_{0}^{\delta}  \int\limits_{2\rho}^{\infty} g(r,\rho)\cdot e^{-\frac{r^2}{4t}} \, dr \, d\rho &= \int\limits_{0}^{\infty}  \int\limits_{0}^{\min\lbrace \delta,\frac{r}{2} \rbrace} g(r,\rho)\cdot e^{-\frac{r^2}{4t}} \, d\rho \, dr \nonumber \\
&=\int\limits_{0}^{\infty} \left( \int\limits_{0}^{\min\lbrace \delta,\sqrt{r} \rbrace}  \frac{g(2\sqrt{r},\rho)}{\sqrt{r}} \, d\rho \right) \cdot e^{-\frac{r}{t}} dr \nonumber\\
&=\int\limits_{0}^{\infty}  G(r) \cdot e^{-\frac{r}{t}} \, dr,
\end{align}
where
\begin{align*}
G(r):= \int\limits_{0}^{\min\lbrace \delta,\sqrt{r} \rbrace} \frac{g(2\sqrt{r},\rho)}{\sqrt{r}} \, d\rho.
\end{align*}
In order to use Watson's lemma we need to determine the asymptotic expansion of $G(r)$ as $r\searrow 0$. We claim that for all $r\in (0,\delta^2)$ the function $G(r)$ is constant and equal to $\pi$. By using the identity $\cosh(2x)=2\sinh^2(x)+1$ we get for all $r\in\left(0,\delta^2 \right)$:
\begin{align*}
G(r)&=\int\limits_{0}^{\min\lbrace \delta,\sqrt{r} \rbrace} \frac{g(2\sqrt{r},\rho)}{\sqrt{r}} \, d\rho = 2\cdot \int\limits_{0}^{\sqrt{r}} \frac{\cosh(\rho)}{\sqrt{\sinh^2(\sqrt{r})-\sinh^2\left(\rho \right)}} \, d\rho  \\
&= 2\int\limits_{0}^{\sinh(\sqrt{r})}\frac{1}{\sqrt{\sinh^2(\sqrt{r})-\rho^2}} \, d\rho = 2\int\limits_{0}^{1} \frac{1}{\sqrt{1-\rho^2}} \, d\rho = \pi.
\end{align*}
Thus we have shown, in particular, that 
\begin{align*}
G(r) \overset{r\downarrow 0}{\sim} \sum\limits_{k=0}^{\infty} \left( \pi \delta_{0k}\right) r^{(k+1)-1}.
\end{align*}
 By Watson's lemma we immediately get 
\begin{align}
\label{equation:BeweisG(r)Asymptotik}
\int\limits_{0}^{\infty} G(r) e^{-\frac{r}{t}}dr \text{ } \overset{t\downarrow 0}{\sim} \pi \sum\limits_{k=0}^{\infty}  \delta_{0k} \cdot\Gamma(k+1)\cdot t^{k+1} = \pi\cdot t + 0\cdot t^2 + 0\cdot t^3+....
\end{align}
Thus the lemma follows from \eqref{equation:BeweisG(r)Asymptotik} and \eqref{equation:BeweisUmformungG(r)Integral}.
\end{proof}

\begin{corollary}
\label{corollary:AsymptoticExpansionBoundary}
\begin{align}
\label{equation:CorollaryAsymptoticExpansionBoundary}
\int\limits_{S}K_{\mathbb{H}^2}^{\nicefrac{1}{4}}(x,x_{E_i};t)dx \overset{t\downarrow 0}{\sim} \frac{L(E_{\ell})}{8\sqrt{\pi t}}\cdot \sum\limits_{k=0}^{\infty}\delta_{0k}\cdot t^k,
\end{align}
where $\delta_{0k}$ is the Kronecker delta.
\end{corollary}

From the discussion above, we know that the right-hand side of \eqref{equation:CorollaryAsymptoticExpansionBoundary} also gives the asymptotic expansion of $I_{C}^{\nicefrac{1}{4}}(t)$ as $t\searrow 0$. This in turn gives us the asymptotic expansion of $Z_E^{\nicefrac{1}{4}}(t)$ as follows. By definition, if $C_1,...,C_M$ are all the connected components of $\Omega_E$, then we have
\begin{align*}
Z_E^{\nicefrac{1}{4}}(t)=-\sum\limits_{i=1}^{M}I_{C_i}^{\nicefrac{1}{4}}(t).
\end{align*}
Thus, from Corollary \ref{corollary:AsymptoticExpansionBoundary} we obtain the following asymptotic expansions.

\begin{corollary} \hfill
\label{corollary:AsymptoticExpansionZ_E}
\begin{itemize}
\item[$(i)$]
\begin{align}
\label{equation:AsymptoticExpansionShiftZ_E}
Z_E^{\nicefrac{1}{4}}(t)\overset{t\downarrow 0}{\sim} -\frac{\vert \partial\Omega \vert}{8\sqrt{\pi t}}\cdot \sum\limits_{k=0}^{\infty}\delta_{0k}\cdot t^k,
\end{align}
where $\vert \partial \Omega \vert = \sum\limits_{j=1}^{\tilde{M}}\vert E_j \vert$, and $\vert E_{j} \vert$ is defined as in Definition \emph{\ref{definition:VolumenVerallgemeinerteFlaeche}}.  
\item[$(ii)$]
\begin{align}
\label{equation:AsymptoticExpansionZ_E1}
Z_E(t)\overset{t\downarrow 0}{\sim} -\frac{\vert \partial\Omega \vert}{8\sqrt{\pi t}}\cdot \sum\limits_{k=0}^{\infty}\tilde{b}_k\cdot t^k,
\end{align}
where $Z_{E}(t):=e^{-\frac{t}{4}}\cdot Z_{E}^{\nicefrac{1}{4}}(t)$ and $\tilde{b}_k:=\left(-\frac{1}{4}\right)^k\frac{1}{k!}$.
\end{itemize}
\end{corollary}

Let us rewrite the asymptotic expansion of $Z_E(t)$, such that its coefficients can better be compared with their spherical counterparts. Since $\tilde{b}_0=1$ we can alternatively write the asymptotic expansion \eqref{equation:AsymptoticExpansionZ_E1} as
\begin{align}
Z_E(t) \overset{t\downarrow 0}{\sim} - \frac{\vert \partial\Omega \vert}{8\sqrt{\pi t}} + \sum\limits_{k=0}^{\infty} b_k^{\mathbb{H}}\cdot t^{k+\frac{1}{2}},
\end{align}
where
\begin{align}
b_k^{\mathbb{H}}:= - \frac{\vert \partial\Omega \vert}{8\sqrt{\pi}} \cdot \tilde{b}_{k+1} = \vert \partial\Omega \vert \frac{(-1)^{k}}{4^{k+2}\cdot \sqrt{\pi}\cdot 2\cdot (k+1)!} \quad \text{ for all } k\in\mathbb{N}_0.
\end{align}

When we compare the coefficients $b_k^{\mathbb{H}}$ with the corresponding spherical coefficients $b_k^{\mathbb{S}}$ from \citep[formula (27)]{Watson}, we see that $b_k^{\mathbb{H}}=(-1)^{k+1}b_k^{\mathbb{S}}$. Similarly as for the coefficients contributed from the interior of a polygon, that relation between $b_k^{\mathbb{H}}$ and $b_k^{\mathbb{S}}$ holds because of the following fact. The constants $b_k^{\mathbb{H}}$ and $b_k^{\mathbb{S}}$ are given by a formula of the form $\vert \partial\Omega \vert\cdot d_k\cdot \kappa^{k+1}$, where $d_k\in\mathbb{R}$ is some universal constant and $\kappa$ denotes the Gaussian curvature of $\mathbb{H}^2$ and $\mathbb{S}^2$, respectively (see \citep{GilkeyBranson}). Note that the geodesic curvature of any edge of $\Omega$ vanishes so that the boundary does not contribute any curvature invariants.

Consider now the third term in \eqref{equation:ErsatzSummeAsymptotik2}, i.e.
\begin{align*}
Z_V^{\nicefrac{1}{4}}(t) := \sum\limits_{i=1}^{M}  \frac{\gamma_i}{2\pi}\int\limits_{0}^{\infty} \frac{e^{-\frac{u^2}{4t}}}{\sqrt{4\pi t}} \cdot \frac{e^{\frac{u}{2}}}{e^{u}-1} \cdot \left( \frac{\pi}{\gamma_i}\coth\left( \frac{\pi}{\gamma_i} \cdot \frac{u}{2} \right) - \coth\left(\frac{u}{2}\right) \right)du .
\end{align*}
Let $i\in \{\, \, 1,..,M \,\}$ be arbitrary and consider
\begin{align*}
I_{\gamma_i}^{\nicefrac{1}{4}}(t):=\frac{\gamma_i}{2\pi}\frac{1}{\sqrt{4\pi t}}\int\limits_{0}^{\infty} e^{-\frac{u^2}{4t}} \cdot  \frac{e^{\frac{u}{2}}}{e^{u}-1} \cdot \left( \frac{\pi}{\gamma_i}\coth\left(  \frac{\pi}{\gamma_i} \cdot \frac{u}{2} \right) - \coth\left(\frac{u}{2}\right) \right)du.
\end{align*}

We first compute the asymptotic expansion of the above integral expression with Watson's lemma.

\begin{lemma} Let $q(u):=\frac{e^{\frac{u}{2}}}{e^{u}-1} \left(\frac{\pi}{\gamma_i}\coth\left(  \frac{\pi}{\gamma_i} \cdot \frac{u}{2} \right) - \coth\left(\frac{u}{2}\right)  \right) $. Then
\begin{align}
\int\limits_{0}^{\infty} e^{-\frac{u^2}{4t}} q(u)du \overset{t\downarrow 0}{\sim} \sum\limits_{k=0}^{\infty} a_k (\gamma_i )\cdot  t^{k+\frac{1}{2}},
\end{align}
where 
\begin{align}
a_k (\gamma_i ):=  \sum\limits_{\ell=1}^{k+1} \binom{2k+2}{2\ell} \frac{B_{2k-2\ell+2}\left(\frac{1}{2}\right)\cdot B_{2\ell}\cdot \sqrt{\pi}}{(k+1)!(2k+1)}  \left(\left(\frac{\pi}{\gamma_i}\right)^{2\ell}-1\right).
\end{align}
\end{lemma}

\begin{proof}
Substituting $u$ by $2\sqrt{u}$, we obtain
\begin{align*}
\int\limits_{0}^{\infty} e^{-\frac{u^2}{4t}} q(u) du = \int\limits_{0}^{\infty} e^{-\frac{u}{t}} \underbrace{\frac{q\left(2\sqrt{u}\right)}{\sqrt{u}}}_{=:q_1(u)} du.
\end{align*}
Because of Watson's lemma it remains to compute the asymptotic expansion of $q_1(u)$ as $u\searrow 0$. On the other hand, it is more convenient to compute the asymptotic expansion of $q_2(u):=q_1(u^2)$, since then the roots will not appear during the calculations. We have:
\begin{align}
\label{equation:DarstellungVonq_2}
q_2(u)= q_1(u^2) = \frac{q\left(2u \right)}{u} = \frac{1}{u^2} \cdot \frac{ u \cdot e^{u}}{e^{2 u}-1} \cdot \left( \frac{\pi}{\gamma_i}\coth\left( \frac{\pi}{\gamma_i} \cdot u \right) - \coth\left( u \right) \right).
\end{align}
As we see, $q_2$ is a product of simple functions whose asymptotic expansions are easily computed as follows.

By definition of the Bernoulli polynomials, we know for all $u\in\mathbb{R}$ with $\vert u \vert< \pi$:
\begin{align}
\label{equation:Vereinfacheq2}
\frac{ u \cdot e^{u}}{e^{2 u}-1} = \frac{1}{2} \frac{ 2u \cdot e^{u}}{e^{2 u}-1} = \frac{1}{2} \sum\limits_{k=0}^{\infty}B_k\left( \frac{1}{2} \right)\frac{(2u)^{k}}{k!} = \frac{1}{2} \sum\limits_{k=0}^{\infty}B_{2k}\left( \frac{1}{2} \right)4^k \frac{u^{2k}}{(2k)!},
\end{align}
where we have used \eqref{equation:VanishingBernoulli} for the last equality. 

Let us investigate the last term in \eqref{equation:DarstellungVonq_2}. Observe that $\coth(x)=1+\frac{2}{e^{2x}-1}$ for all $x\in \mathbb{R}\backslash \{ 0 \}$. Further, the odd Bernoulli numbers are given by $B_{1}=-\frac{1}{2}$ and $B_{2k+1}=0$ for all $k\in\mathbb{N}$. Hence, for all $u\in\mathbb{R}\backslash \{ 0 \}$ with $\vert u  \vert < \min\{\, \pi, \gamma_i \,\}$:
\begin{align}
\label{equation:Vereinfacheq2.1}
\frac{\pi}{\gamma_i}\coth\left(  \frac{\pi}{\gamma_i} \cdot u \right) -  \coth\left( u \right) &=  \frac{\pi}{\gamma_i}-1  + \frac{1}{u}\left( \frac{2 \frac{\pi}{\gamma_i} u }{e^{2 \frac{\pi}{\gamma_i} u }-1} - \frac{2u}{e^{2u}-1} \right) \nonumber \\
 &=  \frac{\pi}{\gamma_i} - 1  + \frac{1}{u}\left( \sum\limits_{k=0}^{\infty} B_k \frac{\left(2\frac{\pi}{\gamma_i} u\right)^k}{k!} - \sum\limits_{k=0}^{\infty} B_k \frac{(2u)^k}{k!} \right) \nonumber \\
 &=  \frac{\pi}{\gamma_i} -1  + \frac{1}{u} \sum\limits_{k=1}^{\infty}\left(\left(\frac{\pi}{\gamma_i}\right)^k - 1 \right)2^k B_k \frac{u^k}{k!}\nonumber\\
 &=\frac{1}{u}\sum\limits_{k=2}^{\infty}\left(\left(\frac{\pi}{\gamma_i}\right)^k -1\right)2^k B_k \frac{u^{k}}{k!} \nonumber \\
 &=\frac{1}{u} \sum\limits_{k=1}^{\infty}\left(\left(\frac{\pi}{\gamma_i}\right)^{2k}-1\right)2^{2k} B_{2k} \frac{u^{2k}}{(2k)!} \nonumber\\
 &= 2\cdot \sum\limits_{k=0}^{\infty}\left(\left(\frac{\pi}{\gamma_i}\right)^{2k+2}-1\right)2^{2k+1} B_{2k+2} \frac{u^{2k+1}}{(2k+2)!}.
\end{align}
Thus we get from \eqref{equation:DarstellungVonq_2}, \eqref{equation:Vereinfacheq2} and \eqref{equation:Vereinfacheq2.1}, for all $0<  u  < \min\{\, \pi, \gamma_i \,\}$:
\begin{align*}
q_2(u) &= \frac{1}{u} \sum\limits_{k=0}^{\infty}B_{2k}\left( \frac{1}{2} \right)4^k \frac{u^{2k}}{(2k)!}  \cdot  \sum\limits_{\ell=0}^{\infty}\left(\left(\frac{\pi}{\gamma_i}\right)^{2\ell+2}-1\right)2^{2\ell+1} B_{2\ell+2} \frac{u^{2\ell}}{(2\ell+2)!}  \\
& = \frac{1}{u} \sum\limits_{k=0}^{\infty} \sum\limits_{\ell=0}^{k}  B_{2(k-\ell)}\left(\frac{1}{2}\right)4^{k-\ell}\frac{u^{2(k-\ell)}}{\left(2(k-\ell)\right)!}  \left(\left(\frac{\pi}{\gamma_i}\right)^{2\ell+2}-1\right)2^{2\ell+1} B_{2\ell+2} \frac{u^{2\ell}}{(2\ell+2)!} \\
&=\frac{1}{u} \sum\limits_{k=0}^{\infty} \sum\limits_{\ell=0}^{k} \binom{2k+2}{2\ell+2} \cdot B_{2(k-\ell)}\left(\frac{1}{2}\right)\cdot B_{2\ell+2} \cdot 4^{k}\cdot 2 \cdot  \left(\left(\frac{\pi}{\gamma_i}\right)^{2\ell+2}-1\right)   \frac{u^{2k}}{(2k+2)!} \\
&= \frac{1}{u} \sum\limits_{k=0}^{\infty}  \sum\limits_{\ell=1}^{k+1} \binom{2k+2}{2\ell} \cdot B_{2k-2\ell+2}\left(\frac{1}{2}\right)\cdot B_{2\ell} \cdot 2 \cdot  \left(\left(\frac{\pi}{\gamma_i}\right)^{2\ell}-1\right)  4^{k}  \frac{u^{2k}}{(2k+2)!}.
\end{align*}
Thus for all $ 0 < u < \min \{\, \pi^2, \gamma_{i}^2 \, \}$:
\begin{align*}
q_1(u)&=q_2(\sqrt{u})\\
&=\sum\limits_{k=0}^{\infty}  \sum\limits_{\ell=1}^{k+1} \binom{2k+2}{2\ell} \cdot B_{2k-2\ell+2}\left(\frac{1}{2}\right)\cdot B_{2\ell} \cdot 2 \cdot  \left(\left(\frac{\pi}{\gamma_i}\right)^{2\ell}-1\right) \cdot 4^{k}  \frac{u^{\left( k+\frac{1}{2} \right)-1}}{(2k+2)!}.
\end{align*}
This establishes the asymptotic expansion of $q_1(u)$ as $u\searrow 0$. Applying Watson's lemma we obtain
\begin{align*}
\int\limits_{0}^{\infty} e^{-\frac{u}{t}} q_1(u) du \overset{t\downarrow 0}{\sim} &\sum\limits_{k=0}^{\infty} \Gamma\left( k+\frac{1}{2} \right) \cdot\\
& \cdot  \sum\limits_{\ell=1}^{k+1} \binom{2k+2}{2\ell} \cdot B_{2k-2\ell+2}\left(\frac{1}{2}\right)\cdot B_{2\ell} \cdot 2 \cdot \left(\left(\frac{\pi}{\gamma_i}\right)^{2\ell}-1\right) \cdot 4^{k}  \frac{t^{ k+\frac{1}{2} }}{(2k+2)!}.
\end{align*}
The coefficient for $k=0$ in the above asymptotic expansion equals by $\Gamma\left( \frac{1}{2} \right)=\sqrt{\pi}$:
\begin{align*}
\sqrt{\pi}B_{0}\left( \frac{1}{2} \right)\cdot B_2 \cdot \left(  \left( \frac{\pi}{\gamma_i}\right)^2 - 1 \right) = a_0 (\gamma_i ).
\end{align*}
For the other coefficients, we use the doubling formula (see e.g. \citep[8335 1]{Gradshteyn} or \citep[formula (3.5)]{Temme}) 
\begin{align}
\label{equation:DoublingFormula}
\Gamma(z)\Gamma\left(z+\frac{1}{2}\right)=2^{1-2z}\sqrt{\pi}\Gamma\left(2z\right).
\end{align}
In particular this implies for all $k\in\mathbb{N}$:
\begin{align*}
\Gamma\left( k+\frac{1}{2} \right) = \frac{2}{2^{2k}}\sqrt{\pi}\frac{\Gamma(2k)}{\Gamma(k)} = \frac{2}{2^{2k}}\sqrt{\pi}\frac{(2k-1)!}{(k-1)!}.
\end{align*}
Hence, using $\frac{(2k-1)!}{(k-1)!}\cdot \frac{1}{(2k+2)!} = \frac{1}{4}\cdot \frac{1}{(k+1)!\cdot (2k+1)}$:
\begin{align*}
\int\limits_{0}^{\infty} e^{-\frac{u}{t}} q_1(u) du \overset{t\downarrow 0}{\sim} \sum\limits_{k=0}^{\infty} \underbrace{  \sum\limits_{\ell=1}^{k+1} \binom{2k+2}{2\ell} \frac{B_{2k-2\ell+2}\left(\frac{1}{2}\right)\cdot B_{2\ell}\cdot \sqrt{\pi}}{(k+1)!(2k+1)}  \left(\left(\frac{\pi}{\gamma_i}\right)^{2\ell}-1\right) }_{=a_k ( \gamma_i )} \cdot t^{k+\frac{1}{2}}.
\end{align*}
\end{proof}

\begin{corollary}
\label{corollary:ContributionOneVertexShift}
\begin{align}
I_{\gamma_i}^{\nicefrac{1}{4}}(t) \overset{t\downarrow 0}{\sim} \sum\limits_{k=0}^{\infty} c_{k}^{\mathbb{H}}(\gamma_i) \cdot t^k,
\end{align}
where 
\begin{align}
\label{equation:DefinitionC_k(Gamma)}
c_k^{\mathbb{H}}(\gamma_i):&= \frac{\gamma_i}{4\pi \sqrt{\pi}}\cdot a_k(\gamma_i) = \frac{\gamma_i}{\pi}\sum\limits_{\ell=1}^{k+1} \binom{2k+2}{2\ell} \frac{B_{2k-2\ell+2}\left(\frac{1}{2}\right)\cdot  B_{2\ell}}{4\cdot (k+1)!(2k+1)}  \left( \left(\frac{\pi}{\gamma_i}\right)^{2\ell}-1\right) \nonumber \\
&= \sum\limits_{\ell=1}^{k+1} \binom{2k+2}{2\ell} \frac{B_{2k-2\ell+2}\left(\frac{1}{2}\right)\cdot  B_{2\ell}}{4\cdot (k+1)!(2k+1)}  \cdot \frac{ \pi^{2\ell}-\gamma_{i}^{2\ell}}{\pi\cdot \gamma_i^{2\ell-1}} .
\end{align}
\end{corollary}

Since 
\begin{align*}
Z_V^{\nicefrac{1}{4}}(t)=\sum\limits_{i=1}^{M}I_{\gamma_i}^{\nicefrac{1}{4}}(t),
\end{align*}
we obtain the following corollary.

\begin{corollary} \hfill
\label{corollary:EckenbeitragWinkelPolygon}
\begin{itemize}
\item[$(i)$]
\begin{align}
Z_V^{\nicefrac{1}{4}}(t) \overset{t\downarrow 0}{\sim} \sum\limits_{k=0}^{\infty} c_k^{\mathbb{H}} \cdot t^k,
\end{align}
where $c_k^{\mathbb{H}}:= \sum\limits_{i=1}^{M} c_k^{\mathbb{H}}(\gamma_i)$.
\item[$(ii)$]
\begin{align}
Z_V (t) \overset{t\downarrow 0}{\sim}\sum\limits_{k=0}^{\infty} \nu_k^{\mathbb{H}}\cdot t^k,
\end{align}
where $Z_V (t):=e^{-\frac{t}{4}}\cdot Z_V^{\nicefrac{1}{4}}(t)$ and $\nu_k^{\mathbb{H}}:= \sum\limits_{\ell=0}^{k}\frac{1}{(k-\ell)!}\left(-\frac{1}{4}\right)^{k-\ell} c_{\ell}^{\mathbb{H}}$.
\end{itemize}
\end{corollary}

\begin{remark}
Observe that when we compare our coefficients $c_k^{\mathbb{H}}(\gamma_i)$ given in \eqref{equation:DefinitionC_k(Gamma)} with the corresponding spherical coefficients $c_k^{\mathbb{S}}(\gamma_i)$ (see \citep[formula $(22)$]{Watson}), then we have $c_k^{\mathbb{H}}(\gamma_i)=(-1)^{k}c_k^{\mathbb{S}}(\gamma_i)$.
Hence, when we compare $\nu_k^{\mathbb{H}}$ with its spherical counterpart $\nu_k^{\mathbb{\mathbb{S}}}$ in \citep[formula (29)]{Watson}, we still have the relation $\nu_k^{\mathbb{H}} = (-1)^{k}\nu_k^{\mathbb{S}}$. As we will see later, there is a deeper reason why the latter relation must hold if $\gamma_i = \frac{\pi}{k}$ for some $k\in\mathbb{N}$.    
\end{remark}

Finally, let us show that 
\begin{align}
A(t):=\sum\limits_{i=1}^{M} A_{\gamma_i}(t;R)
\end{align}
produces no contribution to the asymptotic expansion as $t\searrow 0$.

\begin{lemma}
\label{lemma:A(t)AsymptotikNull}
Let $i\in\lbrace\, 1,..., M\, \rbrace$ and $A_{\gamma_i}(t;R)$ be the function defined in \eqref{equation:DefinitionFunktionA_gamma}. Then for all $m\in\mathbb{N}:$
\begin{align}
A_{\gamma_i}(t;R) = o\left( t^m \right),\, \text{ as }t\searrow 0.
\end{align}
\end{lemma}

\begin{proof}
Recall $A_{\gamma_i}(t;R) = \mathcal{L}^{-1}\left\lbrace F \right\rbrace(t)$ with
\begin{align*}
F(s):=  \int\limits_{R}^{\infty}  \int\limits_{0}^{\infty} \frac{\gamma_i}{\pi^2}Q_{ \sqrt{s} - \frac{1}{2}}^{-i\rho}(\cosh(a))  Q_{ \sqrt{s} - \frac{1}{2}}^{i\rho}(\cosh(a))\frac{\sinh(\rho (\pi-\gamma_i))}{\sinh(\rho \gamma_i)} \sinh(a)\, d\rho \, da
\end{align*}
for $s\in\mathcal{H}_{>\frac{1}{4}}$. Let us first find an appropriate upper bound for the product of the associated Legendre functions under the integral sign. From \eqref{equation:AbschaetzungProduktLegendre1} we have for all $a>0$, $s\in \mathcal{H}_{>\frac{1}{4}}$ and $\rho\geq 0$:
\begin{align*}
\vert Q_{ \sqrt{s} - \frac{1}{2}}^{-i\rho}(\cosh(a))  Q_{ \sqrt{s} - \frac{1}{2}}^{i\rho}(\cosh(a)) \vert \leq \,&\frac{1}{4^{\Re(\sqrt{s}) + \frac{1}{2}}} \cdot \bigg\vert \frac{\Gamma(\sqrt{s}+\frac{1}{2} + i\rho)\cdot \Gamma(\sqrt{s}+\frac{1}{2} - i\rho)}{\Gamma(\sqrt{s}+\frac{1}{2})^2} \bigg\vert \, \cdot \\
& \, \cdot \frac{1}{\cosh(a)-1} \cdot \left( \int\limits_{0}^{\pi} \left( \frac{1-\cos^2(t)}{\cosh(a)+\cos(t)}\right)^{\Re(\sqrt{s})} dt \right)^{2}.
\end{align*}
The same argument as in \eqref{equation:AbschaetzungUnterIntegralFein} shows $0<\frac{1-\cos^2(t)}{\cosh(a)+\cos(t)}\leq 2 e^{-a}$ for all $t\in (0,\pi)$ and thus
\begin{align}
\vert Q_{ \sqrt{s} - \frac{1}{2}}^{-i\rho}(\cosh(a))  Q_{ \sqrt{s} - \frac{1}{2}}^{i\rho}(\cosh(a)) \vert \leq \,&\frac{\pi^2 e^{-2a\Re(\sqrt{s})}}{2(\cosh(a)-1)}\cdot  \nonumber \\
& \cdot \bigg\vert\frac{\Gamma(\sqrt{s}+\frac{1}{2} + i\rho)\cdot \Gamma(\sqrt{s}+\frac{1}{2} - i\rho)}{\Gamma(\sqrt{s}+\frac{1}{2})^2} \bigg\vert
\end{align}
for all $a>0$, $s\in \mathcal{H}_{>\frac{1}{4}}$ and $\rho\geq 0$.

We will estimate the above quotient of gamma functions. Note that, 
\begin{align*}
\bigg\vert\frac{\Gamma(\sqrt{s}+\frac{1}{2} + i\rho)\cdot \Gamma(\sqrt{s}+\frac{1}{2} - i\rho)}{\Gamma(\sqrt{s}+\frac{1}{2})^2} \bigg\vert = \bigg\vert B\left(\sqrt{s}+\frac{1}{2}+i\rho, \sqrt{s}+\frac{1}{2}-i\rho\right) \cdot  \frac{\Gamma(2(\sqrt{s}+\frac{1}{2}))}{\Gamma(\sqrt{s}+\frac{1}{2})^2} \bigg\vert,
\end{align*}
where $B$ denotes the beta function (see \citep[formula (5) on p. 9]{Erdelyi}). From  \citep[formula (24) on p. 11]{Erdelyi} and because $\Re(\sqrt{s})-\frac{1}{2}>0$ for all $s\in\mathcal{H}_{>\frac{1}{4}}$, we have
\begin{align*}
\bigg\vert B\left(\sqrt{s}+\frac{1}{2}+i\rho, \sqrt{s}+\frac{1}{2}-i\rho\right)\bigg\vert \, &\leq \int\limits_{0}^{\infty} e^{- \left(\Re(\sqrt{s})+\frac{1}{2}\right)t} \left( \frac{e^{t}-1}{e^{t}} \right)^{ \Re(\sqrt{s})-\frac{1}{2} }dt \\
&=\int\limits_{0}^{\infty} e^{- \Re(\sqrt{s}) (2t-\ln(e^{t}-1))} \frac{1}{\sqrt{e^{t}-1}} \, dt.
\end{align*}
Observe that the function $(0,\infty) \ni t\mapsto 2t-\ln(e^{t}-1)\in\mathbb{R}$ has a global minimum at $t=\ln(2)$ where the minimal value is $\ln(4)$. Thus, by Laplace's method (see e.g. \citep[Theorem 7.1]{Olver} or \citep[formula (5.1.21)]{Bleistein}), there exists a constant $C_1 >0$ such that for all $\rho\geq 0$ and $s\in \mathcal{H}_{>\frac{1}{4}}$:
\begin{align*}
\bigg\vert B\left(\sqrt{s}+\frac{1}{2}+i\rho, \sqrt{s}+\frac{1}{2}-i\rho\right)\bigg\vert \leq C_1 \cdot \frac{e^{-\Re(\sqrt{s})\ln(4)}}{\Re(\sqrt{s})^{\frac{1}{2}}}.
\end{align*}
Further, using the doubling formula \eqref{equation:DoublingFormula} and then the asymptotic estimate \eqref{equation:GammaQuotAsymp}, there exists some constant $C_2>0$ such that for all $s\in\mathcal{H}_{>\frac{1}{4}}$: 
\begin{align*}
\bigg\vert \frac{\Gamma(2(\sqrt{s}+\frac{1}{2}))}{\Gamma(\sqrt{s}+\frac{1}{2})^2} \bigg\vert = \bigg\vert \frac{4^{\Re(\sqrt{s})}}{\sqrt{\pi}} \cdot \frac{\Gamma(\sqrt{s}+1)}{\Gamma(\sqrt{s}+\frac{1}{2})} \bigg\vert \leq C_2\cdot 4^{\Re(\sqrt{s})}\cdot ( \Re(\sqrt{s}) )^{\frac{1}{2}}.
\end{align*}
Hence, there exists $C>0$ such that for all $\rho\geq 0$, $a>0$, and $s\in\mathcal{H}_{>\frac{1}{4}}$: 
\begin{align}
\label{equation:BeweisA(t)AsymptoticProductLegendre}
\vert Q_{ \sqrt{s} - \frac{1}{2}}^{-i\rho}(\cosh(a))  Q_{ \sqrt{s} - \frac{1}{2}}^{i\rho}(\cosh(a)) \vert \leq \frac{C}{\cosh(a)-1} \cdot e^{-2a\Re(\sqrt{s})}.
\end{align}

We consider two cases. Case $1$: $\gamma_i > \frac{\pi}{2}$:

From the converse to Watson's lemma (see \citep[proof on p. 31]{Wong}) it suffices to show, for each $m\in\mathbb{N}$, that $F(s)=o\left( \frac{1}{s^{m}} \right)$ as $\vert s\vert \rightarrow \infty$ with $s\in\mathcal{H}_{>\frac{1}{4}}$.

By assumption, $\vert \pi-\gamma_i \vert<\gamma_i$ and thus, by \citep[3.981 5]{Gradshteyn}, 
\begin{align}
\int\limits_{0}^{\infty} \frac{\sinh(\rho \vert \pi-\gamma_i \vert)}{\sinh(\rho \gamma_i)} d\rho  = \frac{\pi}{2\gamma_i} \cdot \frac{\sin\left( \pi\frac{\vert \pi-\gamma_i \vert}{\gamma_i} \right)}{1+\cos\left( \pi\frac{\vert \pi-\gamma_i \vert}{\gamma_i} \right)}.
\end{align}

With $C>0$ as in the upper bound \eqref{equation:BeweisA(t)AsymptoticProductLegendre}, we have for all $s\in\mathcal{H}_{>\frac{1}{4}}$, noting that $2\cdot \Re(\sqrt{s})\geq \sqrt{ \vert s \vert}$ for these $s$:
\begin{align*}
\left\vert F(s) \right\vert &\leq \frac{\gamma_i}{\pi^2}  \int\limits_{R}^{\infty} \int\limits_{0}^{\infty}  \left\vert Q_{ \sqrt{s} - \frac{1}{2}}^{-i\rho}(\cosh(a))  Q_{ \sqrt{s} - \frac{1}{2}}^{i\rho}(\cosh(a)) \right\vert \sinh(a)  \cdot \frac{\sinh(\rho \vert \pi-\gamma_i \vert)}{\sinh(\rho \gamma_i)} \, d\rho\, da\\
&\leq \frac{C\cdot \gamma_i \sinh(R)}{\pi^2 (\cosh(R)-1)} \int\limits_{R}^{\infty}  e^{-2a\Re(\sqrt{s})} \int\limits_{0}^{\infty} \frac{\sinh(\rho \vert \pi-\gamma_i \vert)}{\sinh(\rho \gamma_i)}\, d\rho\, da   \\
&=\frac{C\cdot \sinh(R)}{2 \pi (\cosh(R)-1)}\cdot \frac{e^{-2 R\, \Re(\sqrt{s})}}{2\cdot \Re(\sqrt{s})} \cdot \frac{\sin\left( \pi\frac{\vert \pi-\gamma_i \vert}{\gamma_i} \right)}{1+\cos\left( \pi\frac{\vert \pi-\gamma_i \vert}{\gamma_i} \right)} \\
&\leq \tilde{C} \cdot \frac{e^{- R\, \sqrt{\vert s \vert}}}{\sqrt{\vert s \vert}} ,
\end{align*}
where $\tilde{C}:=\frac{C\cdot \sinh(R)}{2\pi (\cosh(R)-1)}\cdot \frac{\sin\left( \pi\frac{ \vert \pi-\gamma_i \vert}{\gamma_i} \right)}{1+\cos\left( \pi\frac{\vert \pi-\gamma_i \vert}{\gamma_i} \right)}$. Since $R>0$, we have for all $m\in\mathbb{N}$: $F(s)=o\left( \frac{1}{s^m} \right)$ as $\vert s \vert\rightarrow\infty$ with $s\in\mathcal{H}_{>\frac{1}{4}}$.

Case $2$: $\gamma_i \in \left( 0,\frac{\pi}{2} \right]$:

In this case the above argument will not work, since now $\pi-\gamma_i\geq \gamma_i$ and thus the integral $\int_{0}^{\infty} \frac{\sinh(\rho (\pi-\gamma_i))}{\sinh(\rho \gamma_i)} d\rho$ is not convergent. Instead, we use that for all $N \in\mathbb{N}$ and $x,y\in\mathbb{R}$:
\begin{align*}
\sinh(x-y)=2\sum\limits_{n=1}^{N}\cosh(x-2ny)\cdot \sinh(y) + \sinh(x-(2N+1)y),
\end{align*}
which one proves easily using the definition of $\cosh$ and $\sinh$. In particular, for all $\rho>0$:
\begin{align}
\label{equation:AbgedrehteFormelZurReduktionDerFaelle}
\frac{\sinh(\rho (\pi-\gamma_i))}{\sinh(\rho \gamma_i)}=2\cdot \sum\limits_{n=1}^{N}\cosh\left( \rho\left( \pi-2n\gamma_i \right) \right) + \frac{\sinh(\rho (\pi-2N\gamma_i - \gamma_i))}{\sinh(\rho \gamma_i)}.
\end{align}
In the special case $\gamma_i=\frac{\pi}{2N}$, this becomes equal to $2\cdot \sum_{n=1}^{N-1}\cosh\left( \rho \left( \pi-2n\gamma_i \right) \right) +1$. (These formulas were also used in \citep[proof of Theorem 2]{VanDenBerg} to estimate the error term.)

Suppose first that there exists some $N\in\mathbb{N}$ with $\gamma_i=\frac{\pi}{2N}$. Then for all $s\in\mathcal{H}_{>\frac{1}{4}}$:
\begin{align}
 F(s) &=  \int\limits_{R}^{\infty}  \int\limits_{0}^{\infty} \frac{\gamma_i}{\pi^2}\, Q_{ \sqrt{s} - \frac{1}{2}}^{-i\rho}(\cosh(a))  Q_{ \sqrt{s} - \frac{1}{2}}^{i\rho}(\cosh(a))\frac{\sinh(\rho (\pi-\gamma_i))}{\sinh(\rho \gamma_i)} \sinh(a) \, d\rho\,  da \nonumber \\
&=2\sum\limits_{n=1}^{N-1} \int\limits_{R}^{\infty}  \int\limits_{0}^{\infty} \frac{\gamma_i}{\pi^2}\, Q_{ \sqrt{s} - \frac{1}{2}}^{-i\rho}(\cosh(a))\cdot \nonumber \\
&\qquad\qquad\qquad\qquad\quad \cdot Q_{ \sqrt{s} - \frac{1}{2}}^{i\rho}(\cosh(a))\cosh\left( \rho \left( \pi-2n\gamma_i \right) \right) \sinh(a)\, d\rho\,  da \nonumber \\
&\qquad +  \int\limits_{R}^{\infty}  \int\limits_{0}^{\infty} \frac{\gamma_i}{\pi^2}Q_{ \sqrt{s} - \frac{1}{2}}^{-i\rho}(\cosh(a))  Q_{ \sqrt{s} - \frac{1}{2}}^{i\rho}(\cosh(a)) \sinh(a) \, d\rho\,  da \nonumber\\
&= 2\gamma_i \sum\limits_{n=1}^{N-1}  \int\limits_{R}^{\infty}  G_{\mathbb{H}^2}^{\nicefrac{1}{4}}\left( (a,n\gamma_i), (a,- n\gamma_i);s  \right) \sinh(a)\, da \nonumber\\
&\qquad + \gamma_i \int\limits_{R}^{\infty} G_{\mathbb{H}^2}^{\nicefrac{1}{4}}\left( \left( a,\frac{\pi}{2} \right), \left( a,- \frac{\pi}{2} \right);s  \right)  \sinh(a) \,   da \nonumber\\
&=\frac{\gamma_i}{\pi} \sum\limits_{n=1}^{N-1}  \int\limits_{R}^{\infty}  Q_{\sqrt{s}-\frac{1}{2}}\left( \cosh\left(d\left(\left( a,n\gamma_i \right), \left( a,- n\gamma_i \right) \right)\right)\right)  \sinh(a) \, da \nonumber \\
&\qquad + \frac{\gamma_i}{2\pi} \int\limits_{R}^{\infty} Q_{\sqrt{s}-\frac{1}{2}}\left( \cosh \left( d\left(\left( a,\frac{\pi}{2} \right), \left( a,- \frac{\pi}{2} \right)\right) \right)\right) \sinh(a) \,   da, \label{equation:BeweisRestgliedKeinBeitrag}
\end{align}
where we used \eqref{equation:GreenPlane2} for the third equality, and \eqref{equation:GreenPlane1} for the last one.

Using \eqref{equation:RepresentationLegendreQIntegral} and \eqref{equation:AbschaetzungUnterIntegralFein}, for all $s\in\mathcal{H}_{>\frac{1}{4}}$ and $(x,y)\in\offdiag(\mathbb{H}^2)$:
\begin{align*}
\vert Q_{\sqrt{s}-\frac{1}{2}}\left( \cosh d\left(\left(x, y\right) \right) \right) \vert \leq \frac{1}{\cosh(d(x,y))-1}\cdot e^{-d(x,y)\left(\Re(\sqrt{s})-\frac{1}{2}\right)}.
\end{align*}
Thus, because of \eqref{equation:BeweisRestgliedKeinBeitrag} and $\Re(\sqrt{s})-\frac{1}{2}>0$ for all $s\in\mathcal{H}_{>\frac{1}{4}}$, there exists some constant $\varepsilon >0$ such that
\begin{align*}
\vert F(s) \vert &\leq N\cdot \frac{\gamma_i}{\pi}\, e^{-\varepsilon \left(\Re(\sqrt{s})-\frac{1}{2}\right)} \int\limits_{R}^{\infty} \frac{\sinh(a)}{\cosh(a)^2-\sinh^2(a)\cos(2\gamma_i)-1}\, da\\
&=\frac{e^{-\varepsilon \left(\Re(\sqrt{s})-\frac{1}{2}\right)}}{2(1-\cos(2\gamma_i))}\int\limits_{R}^{\infty} \frac{1}{\sinh(a)} \, da \\
&< \frac{e^{-\varepsilon \left(\Re(\sqrt{s})-\frac{1}{2}\right)}}{2(1-\cos(2\gamma_i))}\int\limits_{R}^{\infty} \frac{6}{a^3}\, da \\
&= \frac{3\cdot e^{\frac{\epsilon}{2}}}{2(1-\cos(2\gamma_i))R^2} \cdot e^{-\varepsilon \Re(\sqrt{s})}\\
&\leq \frac{3\cdot e^{\frac{\epsilon}{2}}}{2(1-\cos(2\gamma_i))R^2} \cdot e^{-\frac{\varepsilon}{2} \sqrt{\vert s \vert}},
\end{align*}
where, for the last equality, we used again $2\Re(\sqrt{s})\geq \sqrt{\vert s \vert}$ for all $s\in\mathcal{H}_{>\frac{1}{4}}$.


Because $\varepsilon>0$, we have $F(s)=o\left(\frac{1}{s^m}\right)$ as $\vert s \vert\rightarrow\infty$ with $s\in\mathcal{H}_{>\frac{1}{4}}$ for all $m\in\mathbb{N}$.

Suppose now there exists $N\in\mathbb{N}$ such that  $\gamma_i\in\left( \frac{\pi}{2(N+1)}, \frac{\pi}{2N} \right)$. Then one shows again $F(s)=o\left(\frac{1}{s^m}\right)$ as $\vert s \vert\rightarrow\infty$ with $s\in\mathcal{H}_{>\frac{1}{4}}$ for all $m\in\mathbb{N}$. This follows along the lines of the previous two cases, which are discussed above in detail. More precisely, using \eqref{equation:AbgedrehteFormelZurReduktionDerFaelle}, we have
\begin{align*}
F(s)= &\,2\cdot \sum\limits_{n=1}^{N} \int\limits_{R}^{\infty}  \int\limits_{0}^{\infty} \frac{\gamma_i}{\pi^2}\, Q_{ \sqrt{s} - \frac{1}{2}}^{-i\rho}(\cosh(a))  Q_{ \sqrt{s} - \frac{1}{2}}^{i\rho}(\cosh(a))\cosh\left( \rho\left( \pi-2n\gamma_i \right) \right) \sinh(a) \, d\rho\,  da \\
&+ \int\limits_{R}^{\infty}  \int\limits_{0}^{\infty} \frac{\gamma_i}{\pi^2}\, Q_{ \sqrt{s} - \frac{1}{2}}^{-i\rho}(\cosh(a))  Q_{ \sqrt{s} - \frac{1}{2}}^{i\rho}(\cosh(a))\frac{\sinh(\rho (\pi-2N\gamma_i - \gamma_i))}{\sinh(\rho \gamma_i)} \sinh(a) \, d\rho\,  da.
\end{align*}
Now, the first sum can be estimated as we did above. The second summand can be estimated as in case 1, since $\vert \pi-2N\gamma_i - \gamma_i \vert< \gamma_i$.

Finally, as mentioned above, the statement of the lemma follows from the converse to Watson's lemma.
\end{proof}

Let us summarise the discussion of this section into a theorem.

\begin{theorem}
\label{theorem:AsymptoticExpansionHeatTraceHyperbolPolygon}
Let $\Omega\subset\mathbb{H}^2$ be a hyperbolic polygon with $M\geq 3$ angles and let $\gamma_1,...,\gamma_M\in (0,2\pi]$ denote the angles of the polygon. Then
\begin{align}
\label{equation:AsymptoticExpansionHeatTraceHyperbolPolygon}
Z_{\Omega}(t)  \overset{t\downarrow 0}{\sim} \frac{ \vert \Omega \vert}{4\pi t} - \frac{\vert \partial\Omega \vert}{8\sqrt{\pi t}} + \sum\limits_{k=0}^{\infty}\left( i_k^{\mathbb{H}} + b_k^{\mathbb{H}}\cdot t^{\frac{1}{2}} + \nu_k^{\mathbb{H}} \right)t^k,
\end{align}
where the coefficients are given as follows:
\begin{align*}
i_k^{\mathbb{H}} &= \frac{\vert \Omega\vert}{4\pi}\frac{1}{(k+1)!}\sum\limits_{\ell=0}^{k+1}\binom{k+1}{\ell}\left(-\frac{1}{4}\right)^{k+1-\ell}B_{2\ell}\left(\frac{1}{2}\right), \\
b_k^{\mathbb{H}} &= \vert \partial\Omega \vert \frac{(-1)^{k}}{4^{k+2}\cdot \sqrt{\pi}\cdot 2\cdot (k+1)!}, \\
\nu_k^{\mathbb{H}} &= \sum\limits_{\ell=0}^{k}\frac{1}{(k-\ell)!}\left(-\frac{1}{4}\right)^{k-\ell} c_{\ell}^{\mathbb{H}}, \\
\text{ where }\,\,\,\, c_{\ell}^{\mathbb{H}} &= \sum\limits_{i=1}^{M} c_{\ell}^{\mathbb{H}}(\gamma_i) \\
\text{ with }\,\,\,\, c_{\ell}^{\mathbb{H}}(\gamma_i)&=  \sum\limits_{\nu=1}^{\ell +1} \binom{2\ell +2}{2\nu} \frac{B_{2\ell -2\nu+2}\left(\frac{1}{2}\right)\cdot  B_{2\nu}}{4\cdot (\ell +1)!(2\ell +1)} \cdot  \frac{ \pi^{2\nu}-\gamma_{i}^{2\nu}}{\pi\cdot \gamma_i^{2\nu-1}}.
\end{align*}
\end{theorem}

\section{Consequences for polygons}
\label{section:ConsequencesPolygons}

In this section we want to put Theorem \ref{theorem:AsymptoticExpansionHeatTraceHyperbolPolygon} into a more general context and then derive some consequences from the heat invariants.

\begin{definition}
\label{definition:InjektivityRadius}
Suppose $N$ is an arbitrary two-dimensional complete Riemannian manifold. For all $P\in N$ we denote the injectivity radius of $P$ by $i(P)$. Thus we have $i(P)\in (0,\infty]$. Furthermore, for all $P\in N$ and $\rho\in (0,i(P)]$ we define $B_{\rho}(P):=\{\, x\in N \mid d(x,P)<\rho \,\}$, where $d(x,P)$ denotes the distance between $x$ and $P$. In other words, the set $B_{\rho}(P)$ denotes the geodesic disc with radius $\rho$ and center $P$.
\end{definition}

\begin{definition}
\label{definition:PolygonGeneral}
A \emph{polygon} is any relatively compact domain $\Omega\subset N$ whose boundary is a union of finitely many geodesic segments, where $N$ is a two-dimensional complete Riemannian manifold. If, in addition, $N$ has constant curvature $\kappa\in\mathbb{R}$, then $\Omega$ is called a \emph{polygon of constant curvature} $\kappa$. 

The terms \emph{smooth boundary point}, \emph{vertex} and \emph{edge} of a polygon are defined analogously as in Definition \ref{definition:EdgeVertex}.
\end{definition}

Note that according to Definition \ref{definition:PolygonGeneral} a polygon may have no vertices at all. For example, a domain in the cylinder $\mathbb{S}^1\times \mathbb{R}\subset\mathbb{R}^3$ bounded by two closed geodesics is a polygon without vertices. Therefore, we regard each smooth and closed boundary component of a polygon also as an edge of the polygon. Moreover, we allow $\partial\Omega = \emptyset$ as well, namely if $N$ is compact and $\Omega = N$.

\begin{definition}
\label{definition:SectorWedge}
Let $N$ be an arbitrary two-dimensional complete Riemannian manifold. A domain $V\subset N$ is called a \emph{circle sector} if there exist $P\in N$, $\rho\in (0, i(P)]$, $\gamma \in (0,2\pi]$ and polar coordinates $(a,\alpha)$ on $B_{\rho}(P)$ (see \citep[Definition $5.3.1$]{KlingenbergSurface}), such that we have 
\begin{align*}
V = \{\, (a,\alpha) \mid 0< a < \rho,\, 0<\alpha <\gamma \,\}.
\end{align*}
We call $P$ a \emph{vertex} of $V$ and we will refer to $V$ as a circle sector at $P$ with \emph{angle} $\gamma$ and \emph{radius} $\rho$. Moreover, a circle sector at $P$ with radius $\rho = i(P)$ is also called a \emph{wedge}.
\end{definition}

Note that a circle sector may have several vertices (e.g. a spherical wedge has two vertices), but a circle sector at $P$ has a unique radius and angle at $P$.

 The \emph{angles} of a polygon are defined analogously as in Definition \ref{definition:Angle}, where we had defined the angles of hyperbolic polygons. Note that if $\Omega$ is a polygon with non-empty boundary $\partial\Omega$ and $P\in\partial\Omega$ is fixed, then there exists some $\rho\in (0, i(P)]$ such that $\Omega \cap B_{\rho}(P)$ is a disjoint union of circle sectors at $P$ with radius $\rho$.

In the following we aim to generalise Theorem \ref{theorem:AsymptoticExpansionHeatTraceHyperbolPolygon} to polygons of constant curvature. For that purpose, we need the following local version of Lemma \ref{lemma:PNFB}.

\begin{lemma}
\label{lemma:PNFBlokal}
Let $N$ be a two-dimensional complete Riemannian manifold whose Gaussian curvature is bounded, and let $\Omega\subset N$ be a polygon.
\begin{itemize}
\item[$(i)$] Let $P\in \Omega$ be fixed and let $\rho >0$ be such that $B_{2 \rho}(P)\subset \Omega$. There exist constants $T_1, C_1, D_1>0$ such that for all $x\in B_{\rho}(P)$ and $t\in (0,T_1]:$
\begin{align}
\label{equation:PNFBlokalInterior}
 0 \leq K_{\Omega}(x,x;t) - K_{B_{2 \rho}(P)} (x,x;t) \leq \frac{C_1}{t}e^{-\frac{D_1}{t}}.
\end{align} 
\item[$(ii)$] Let $Q\in\partial\Omega$ be fixed and let $r>0$ be such that $\Omega\cap B_{2r}(Q)$ is a disjoint union of circle sectors at $Q$ with radius $2r$. Suppose $W_{2r}(Q)$ denotes one of those circle sectors with an angle $\gamma\in (0,2\pi]$. Let $W_{r}(Q)$ be the circle sector at $Q$ with radius $r$ and angle $\gamma$, which is contained in $W_{2r}(Q)$. Then there exist constants $T_2, C_2, D_2>0$ such that for all $x\in W_{r}(Q)$ and $t\in (0,T_2]:$
\begin{align}
\label{equation:PNFBlokalBoundary}
0 \leq K_{\Omega}(x,x;t) - K_{W_{2 r}(Q)} (x,x;t) \leq \frac{C_2}{t}e^{-\frac{D_2}{t}}. 
\end{align}
\item[$(iii)$] Let $Z\in N$ be arbitrary. Suppose we have given for all $i\in\{1,2,3\}$ a circle sector $W_{r_i}(Z)$ at $Z$ with angle $\theta\in (0,2\pi]$ and radius $r_i\in (0,i(Z)]$ such that $W_{r_1}(Z)\subset W_{r_2}(Z)\subset W_{r_3}(Z)$. Then there exist constants $T_3, C_3, D_3>0$ such that for all $x\in W_{r_1}(Z)$ and $t\in (0,T_3]:$
\begin{align}
\label{equation:PNFBlokalSectors}
0 \leq K_{W_{r_3}(Z)}(x,x;t) - K_{W_{r_2}(Z)} (x,x;t) \leq \frac{C_3}{t}e^{-\frac{D_3}{t}}.
\end{align}
\end{itemize}
\end{lemma}

\begin{proof}
First, let us consider the estimates in \eqref{equation:PNFBlokalInterior}. Since $B_{2 \rho}(P)\subset \Omega$, it follows from the minimality property of the heat kernel that $0\leq K_{\Omega}(x,x;t) - K_{B_{2 \rho}(P)} (x,x;t)$ for all $x\in  B_{\rho}(P)$ and $t>0$. Further, because of Lemma \ref{lemma:PNFBAllgemein} there exist constants $T_1, \tilde{C}_1, D_1>0$ such that for all $x\in B_{\rho}(P)$  and $t\in (0, T_1]$:
\begin{align*}
\vert K_N (x,x;t) - K_{\Omega}(x,x;t) \vert \leq \frac{\tilde{C}_1}{t}e^{-\frac{D_1}{t}}, \,\, \text{ and } \,\, \vert K_N (x,x;t) - K_{B_{2 \rho}(P)}(x,x;t) \vert \leq \frac{\tilde{C}_1}{t}e^{-\frac{D_1}{t}}.
\end{align*}
With $C_1:=2\cdot \tilde{C}_1$ and $D_1,T_1>0$ as above, we have for all $x\in B_{\rho}(P)$  and $t\in (0, T_1]$:
\begin{align*}
K_{\Omega}(x,x;t) - K_{B_{2 \rho}(P)}(x,x;t) &\leq \vert K_{\Omega}(x,x;t) - K_N (x,x;t)\vert + \vert K_{N}(x,x;t) - K_{B_{2 \rho}(P)} (x,x;t)\vert \\
&\leq \frac{C_1}{t}e^{-\frac{D_1}{t}}.
\end{align*}
This closes the proof of part $(i)$. Now consider \eqref{equation:PNFBlokalBoundary}.

As before, because of $W_{2 r}(Q)\subset \Omega$ it follows from the minimality of the heat kernel that $0\leq K_{\Omega}(x,x;t) - K_{W_{2 r}(Q)} (x,x;t)$ for all $x\in  W_{r}(Q)$ and $t>0$. The upper estimate follows as in the proof of Lemma \ref{lemma:PNFB}. More precisely, let $E(2r):=\{\, p\in \overline{W_{2 r}(Q)} \mid d(p,Q) = 2r \,\}$. Using Lemma \ref{lemma:ProbabilisticFormulaHeatKernel}, we have for all $x\in  W_{r}(Q)$ and $t>0$:
\begin{align*}
K_{\Omega}(x,x;t) &\leq  K_{N}(x,x;t)\cdot \text{Prob}\{\, \omega(s)\in W_{2r}(Q),\, 0\leq s\leq t \mid \omega(0)=x=\omega(t)  \,\}\\
&\quad + K_{N}(x,x;t)\cdot \text{Prob}\{\, \omega(s)\in E(2r)\text{ for some } s\in[0, t] \mid \omega(0)=x=\omega(t)  \,\} \\
& = K_{W_{2r}(Q)}(x,x;t) +  \left(  K_{N}(x,x;t) - K_{N\backslash E_i(2R)}(x,x;t) \right).
\end{align*}
Thus by Lemma \ref{lemma:PNFBAllgemein} there exist constants $T_2, C_2, D_2>0$ such that \eqref{equation:PNFBlokalBoundary} holds.

The proof of part $(iii)$ is analogous as for $(ii)$.
\end{proof}

As we have already mentioned, the heat invariants for Euclidean polygons (\citep{VanDenBerg}) and spherical polygons (\citep{Watson}) are known. They are given by \eqref{equation:EuclideanAsymp} and \eqref{equation:SphereAsymp}, respectively. The following corollary unifies these results with our Theorem \ref{theorem:AsymptoticExpansionHeatTraceHyperbolPolygon}.

\begin{corollary}
\label{corollary:HeatAsymptoticConstantCurvaturePolygon}
Let $N$ be a two-dimensional complete Riemannian manifold of constant curvature $\kappa\in\mathbb{R}$, and let $\Omega\subset N$ be a polygon. Suppose that $\Omega$ has $M\in\mathbb{N}_0$ angles, which are denoted by $\gamma_1,...,\gamma_M\in (0,2\pi]$. As usual, we denote the heat trace of $\Omega$ by $Z_{\Omega}$. Then
\begin{align}
\label{equation:HeatAsymptoticConstantCurvaturePolygon}
Z_{\Omega}(t)  \overset{t\downarrow 0}{\sim} I_{\kappa} + B_{\kappa} + \sum\limits_{i=1}^{M} V_{\kappa}(\gamma_i),
\end{align}
where
\begin{align*}
I_{\kappa}:= \frac{\vert \Omega \vert}{4\pi t} \sum\limits_{\nu=0}^{\infty} f_{\nu} \cdot \kappa^{\nu} \cdot t^{\nu}; \quad B_{\kappa}:=\frac{\vert \partial\Omega \vert}{8\sqrt{\pi t}} \sum\limits_{\nu =0}^{\infty} r_{\nu} \cdot \kappa^{\nu} \cdot t^{\nu}; \quad V_{\kappa}(\gamma_i):= \sum\limits_{\nu =0}^{\infty} e_{\nu}(\gamma_i) \cdot \kappa^{\nu } \cdot t^{\nu};
\end{align*}
and $f_{\nu}, r_{\nu}$ and $e_{\nu}(\gamma_i)$ are the following universal coefficients for all $\nu \in \mathbb{N}_0:$
\begin{align}
f_{\nu}:&=  \frac{1}{\nu ! \cdot 4^{\nu}} \sum\limits_{\ell=0}^{\nu} \binom{\nu}{\ell} \left( - 4 \right)^{\ell}  B_{2\ell}\left( \frac{1}{2} \right), \label{equation:CoefficientsAreaHeatAsymptoticConstantCurvaturePolygon}\\
 r_{\nu} :&=  -\frac{1}{4^{\nu} \nu!}, \label{equation:CoefficientsBoundaryHeatAsymptoticConstantCurvaturePolygon} \\
e_{\nu}(\gamma_i):&= \frac{1}{4^{\nu}} \sum\limits_{\ell=0}^{\nu} \frac{\left( - 4 \right)^{\ell}}{(\nu-\ell)!} \sum\limits_{j=1}^{\ell+1}  \binom{2\ell +2}{2 j} \frac{B_{2\ell-2j+2}\left(\frac{1}{2}\right)\cdot  B_{2 j}}{4\cdot (\ell+1)!(2\ell+1)}  \cdot \frac{ \pi^{2 j}-\gamma_{i}^{2 j}}{\pi\cdot \gamma_i^{2 j-1}}  \label{equation:CoefficientsVerticesHeatAsymptoticConstantCurvaturePolygon}.
\end{align}
\emph{(}Note that $\vert \partial\Omega \vert$ is defined analogously as in Definition \emph{\ref{definition:VolumenVerallgemeinerteFlaeche}} and $e_{\nu}(\pi)=0$ for all $\nu\in\mathbb{N}_0$. Further, $0^0$ is understood to be $1$, which occurs if $\kappa,\nu=0$.\emph{)}
\end{corollary}

\begin{proof}
We proceed in several steps. Let us start with some comments on the cases we already know. If $N=\mathbb{H}^2$ we have $\kappa=-1$ and the asymptotic expansion \eqref{equation:HeatAsymptoticConstantCurvaturePolygon} coincides with the asymptotic expansion given in \eqref{equation:AsymptoticExpansionHeatTraceHyperbolPolygon}. Note that
\begin{align*}
f_{\nu} =   \frac{4\pi}{\vert \Omega \vert}  (-1)^{\nu} \cdot i_{\nu-1}^{\mathbb{H}},\quad r_{\nu} =  \frac{8\sqrt{\pi}}{\vert \partial\Omega \vert} (-1)^{\nu} \cdot b_{\nu-1}^{\mathbb{H}},\quad e_{\nu}(\gamma_i) = (-1)^{\nu} \sum\limits_{\ell=0}^{\nu} \frac{1}{(\nu-\ell)!} \left( - \frac{1}{4} \right)^{\nu-\ell} c_{\ell}^{\mathbb{H}}(\gamma_i),
\end{align*}
where the coefficients on the right-hand sides are defined as in Theorem \ref{theorem:AsymptoticExpansionHeatTraceHyperbolPolygon} (compare also with the formulas \eqref{equation:AsymptotischeEntwicklungZ_I} and \eqref{equation:AsymptoticExpansionZ_E1}).

If $N=\mathbb{S}^2$, then $\kappa=1$ and \eqref{equation:HeatAsymptoticConstantCurvaturePolygon} reduces to the asymptotic expansion computed in \citep{Watson} (see \eqref{equation:SphereAsymp} and our comment on the coefficients thereafter). And if $M=\mathbb{R}^2$, then $\kappa=0$ and one easily shows that \eqref{equation:HeatAsymptoticConstantCurvaturePolygon} is equal to the first three terms of  \eqref{equation:EuclideanAsymp}.

Next, we generalise these results one step further by scaling the standard metric on the hyperbolic plane and the unit sphere. Let $g_{\mathbb{H}^2}$ denote the standard metric on the hyperbolic plane $\mathbb{H}^2$ and let $c\in (0,\infty)$ be a constant. Consider the space $\left( \mathbb{H}^2, c\,   g_{\mathbb{H}^2} \right)$. The curvature of $\mathbb{H}^2$ with respect to the scaled metric is constant and equal to $\kappa = -\frac{1}{c}$. Obviously, a domain $\Omega\subset \mathbb{H}^2$ is a polygon with respect to $g_{\mathbb{H}^2}$ and has eigenvalues $0<\lambda_1<\lambda_2\leq \lambda_3\leq...$ if and only if it is a polygon with respect to $c\,  g_{\mathbb{H}^2}$ with eigenvalues $0<\frac{\lambda_1}{c }<\frac{\lambda_2}{c }\leq \frac{\lambda_3}{c }\leq...$. Hence, 
\begin{align*}
Z_{\left( \Omega,c\, g_{\mathbb{H}^2}\right)}\left( t \right) = Z_{\left( \Omega, g_{\mathbb{H}^2}\right)}\left( \frac{t}{c } \right),
\end{align*}
and the claim follows easily for polygons in the space $\left( \mathbb{H}^2, c\,  g_{\mathbb{H}^2} \right)$. The same argument can be applied to $\mathbb{S}^2$ as well.

Lastly, suppose that $N$ is an arbitrary two-dimensional complete Riemannian manifold of constant curvature $\kappa\in\mathbb{R}$. Since the metric can be scaled appropriately as above, we assume without loss of generality that $\kappa\in \{ -1,0,1 \}$.  Then $N$ is locally isometric to $M_{\kappa}$, where
\begin{align*}
M_{\kappa}:= \begin{cases}
\left( \mathbb{H}^2,  g_{\mathbb{H}^2}\right),& \text{if }\kappa = -1, \\
\left( \mathbb{R}^2,  g_{\mathbb{R}^2}\right),& \text{if }\kappa = 0, \\
\left( \mathbb{S}^2,  g_{\mathbb{S}^2}\right),& \text{if }\kappa = 1, \\
\end{cases}
\end{align*}
and where $g_{\mathbb{R}^2}$ and $g_{\mathbb{S}^2}$ denote the standard metric on $\mathbb{R}^2$ and $\mathbb{S}^2$, respectively.
We construct a certain finite open cover of $\Omega$. Let us choose around any $Q\in \overline{\Omega}$ a relatively compact geodesic disc  $\genball{r_Q}{Q}$ with $r_Q\in (0,i(Q)]$ as follows: 
\begin{itemize}
\item[$(i)$] If $Q$ is a vertex, we let $r_Q:=2R$, where $R>0$ is such that the discs of radius $2R$ around all vertices are mutually disjoint and can be isometrically embedded into $M_{\kappa}$, and such that $B_{2R}(Q)\cap \Omega$ is a disjoint union of finitely many circle sectors at $Q$ with radius $2R$.
\item[$(ii)$] If $Q$ lies on an edge $E$, but is not a vertex, then choose $r_Q$ such that $\genball{r_Q}{Q}$ contains no vertices, no points from $\partial\Omega \backslash E$, and such that it can also be embedded isometrically into $M_{\kappa}$.
\item[$(iii)$] For any $Q\in\Omega$ we choose $r_Q$ such that $\genball{r_Q}{Q}$ is completely contained in $\Omega$ and can be isometrically embedded into $M_{\kappa}$.
\end{itemize}
Consider the collection of all those discs. Obviously, when we replace any disc of this collection by the geodesic disc with the same center and half of the radius, we obtain an open cover of $\overline{\Omega}$. Let $\genball{R_1}{Q_1},...,\genball{R_n}{Q_n}\subset N$ be a finite subcover of it, where $n\in\mathbb{N}$ is fixed. By construction, any $B_{2\cdot R_{\ell}}(Q_{\ell})$ with $\ell\in\lbrace 1,...,n \rbrace$ can be isometrically embedded into $M_{\kappa}$. Henceforth, we identify these discs with their images in $M_{\kappa}$ through such isometric embeddings. 

Roughly speaking, the idea is now to decompose $\Omega$ into small pieces using the finite cover above. Then we apply Lemma \ref{lemma:PNFBlokal} to $N$ and each piece of $\Omega$. By this process the full  asymptotic expansion of $Z_{\Omega}(t)$ splits into several parts such that each part is attached to one of the pieces of $\Omega$. Then we translate everything to $M_{\kappa}$ by the isometric embeddings and apply Lemma \ref{lemma:PNFBlokal} to $M_{\kappa}$.

We have three kinds of geodesic discs in our finite covering, which lead to the following decomposition of $\Omega$. Let $\eta\in\mathbb{N}_0$ and $\{\ell_1,...,\ell_{\eta}\}\subset \{1,...,n\}$ be such that $\{ B_{R}(Q_{\ell_i}) \}_{i=1}^{\eta}$ is the collection of all discs in the finite cover induced by $(i)$. By definition, $\Omega\cap \left( \cup_{i=1}^{\eta} B_{R}(Q_{\ell_i}) \right)$ consists of a disjoint union of $M$ circle sectors, each contained in exactly one angle of the polygon. Let $\lbrace W_{R} (P_j)\rbrace_{j=1}^{M}$ be all those circle sectors, where $P_j$ denotes the vertex of $\Omega$ corresponding to the angle $\gamma_j$. Similarly, let $\{ \genball{R_{m_i}}{Q_{m_i}} \}_{i=1}^{\tau}$ with $\tau\in\mathbb{N}_{0}$ and $\{m_1,...,m_{\tau}\}\subset \{1,...,n\}$ be all those discs of the finite cover induced by $(ii)$. For each $i\in\{1,...,\tau\}$ the set $\Omega\cap \genball{R_{m_i}}{Q_{m_i}}$ is either a half-disc or a disjoint union of two half-discs. Let $\lbrace B_{e_{j}}\rbrace_{j=1}^{k_1}$ be the collection of all the half-discs which are obtained that way. Finally, let $\lbrace B_{i_{j}}\rbrace_{j=1}^{k_2}$ be all the discs in the finite cover coming from $(iii)$. Then
\begin{align*}
\Omega =  V \cupdot E \cupdot I,
\end{align*}
where 
\begin{align*}
V:=\bigcupdot_{j=1}^{M} W_{R}(P_j),\, E:=\bigcupdot_{j=1}^{k_1} B_{e_{j}}\backslash \left( V\cup \left(\cup_{\ell = 1}^{j-1} B_{e_{\ell}}\right)\right),\, I:=\bigcupdot_{j=1}^{k_2}B_{i_{j}}\backslash \left( V\cup E \cup \left( \cup_{\ell=1}^{j-1}B_{i_{\ell}}\right) \right).
\end{align*}

Let us show how to do the rest of the program. Let $W_{R}(P_j)$ be arbitrary with $j\in \{ 1,...,M \}$. Recall that the circle sector $W_{R}(P_j)$ has angle $\gamma_j$. As in Lemma \ref{lemma:PNFBlokal}, let $W_{2R}(P_j)$ denote the circle sector at $P_j$ with radius $2R$ and angle $\gamma_j$ such that $W_{R}(P_j)\subset W_{2R}(P_j)$. It follows from \eqref{equation:PNFBlokalBoundary} that the function 
\begin{align*}
t\mapsto \int_{W_R(P_j)}K_{\Omega}(x,x;t)dx
\end{align*}
has the same asymptotic expansion as the function $t\mapsto \int_{W_{R}(P_j)}K_{W_{2R}(P_j)}(x,x;t)dx$ as $t\searrow 0$. Note that the domains $W_{R}(P_j)$ and $W_{2R}(P_j)$ are identified with their isometric images in $M_{\kappa}$ and the function $K_{W_{2R}(P_j)}$ is identified with the heat kernel of that isometric image $W_{2R}(P_j)\subset M_{\kappa}$. When we apply \eqref{equation:PNFBlokalSectors} to $M_{\kappa}$, we see that the function $t\mapsto \int_{W_{R}(P_j)}K_{W_{2R}(P_j)}(x,x;t)dx$ has the same asymptotic expansion as the function
\begin{align*}
t\mapsto \int_{W_R(P_j)}K_{W_{\gamma_j}(P_j)}(x,x;t)dx
\end{align*}
as $t\searrow 0$, where $W_{\gamma_j}(P_j)\subset M_{\kappa}$ denotes the wedge at $P_j$ with angle $\gamma_j$ such that $W_R(P_j)\subset W_{\gamma_j}(P_j)$. Note that for $\kappa = -1$ the asymptotic expansion of that function follows from Section \ref{section:ExpansionTrace}. Similarly, if $\kappa=0$ and $\kappa=1$ its asymptotic expansion follows from \citep{VanDenBerg} and \citep{Watson}, respectively. Hence we know the asymptotic expansion of $t\mapsto \int_{W_R(P_j)}K_{\Omega}(x,x;t)dx$ as $t\searrow 0$.

The same kind of argument can also be applied to the other pieces of the decomposition, i.e. any of those disjoint subsets into which $E$ and $I$ are decomposed above. Note that for the subsets of $I$ one needs to apply the above argument with \eqref{equation:PNFBlokalInterior} instead of \eqref{equation:PNFBlokalBoundary}, and \eqref{equation:LemmaPNFBAllgemein} instead of \eqref{equation:PNFBlokalSectors}. From that analysis the statement follows.
\end{proof}

This corollary shows how the heat invariants for Euclidean, spherical and hyperbolic polygons are linked to each other. Further, it follows from Corollary \ref{corollary:HeatAsymptoticConstantCurvaturePolygon} that polygons of constant zero curvature have at most three nonvanishing heat invariants. This was previously well-known for Euclidean polygons (\citep{VanDenBerg}).

Another observation of Corollary \ref{corollary:HeatAsymptoticConstantCurvaturePolygon} is the following: Suppose a polygon of constant curvature $\kappa\in\mathbb{R}$ is given. Then an angle $\gamma\in (0,2\pi]$ of the polygon contributes the terms $V_{\kappa}(\gamma) = \sum_{\nu =0}^{\infty} e_{\nu}(\gamma) \cdot \kappa^{\nu } \cdot t^{\nu}$ to its heat trace asymptotic expansion. In other words, the contributions are given as polynomials in the Gaussian curvature with universal coefficients $e_{\nu}(\gamma)\in\mathbb{R}$ for all $\nu\in\mathbb{N}_{0}$. Another reason for this phenomenon will be provided through orbifold theory in Section \ref{section:ApplicationsOrbifolds} for angles $\gamma=\frac{\pi}{k}$, $k\in\mathbb{N}_{k\geq 2}$. 

We think that it is an interesting problem to investigate the contribution of an angle of a geodesic polygon to the heat invariants if the Gaussian curvature is arbitrary. We expect that in the general case the coefficients of an angle contribution resemble the case of constant curvature. That is, we conjecture that they are always given as polynomials in the Gaussian curvature and its covariant derivatives at the vertex.
 
In the remaining part of this section we draw some conclusions from the heat invariants for polygons of constant curvature. Obviously, the heat trace only depends on the spectrum of the Dirichlet Laplacian through \eqref{equation:HeatTraceBoundedDomain}. Thus, all heat invariants are determined by the spectrum. We may ask: Which geometric properties of a polygon are determined by its heat invariants and therefore  also are spectral invariants? The next corollary provides a first answer.

 Recall that for any polygon $\Omega$ of constant curvature $\kappa$, the Gau{\ss}-Bonnet theorem states:
\begin{align}
\label{equation:GaussBonnet}
\sum\limits_{i=1}^M \gamma_i = \vert \Omega \vert\cdot \kappa + M\cdot \pi - 2\pi\chi(\Omega),
\end{align}
where $M\in\mathbb{N}_{0}$ denotes the number of angles and $\gamma_1,...,\gamma_M$ are the angles of $\Omega$ (see \citep[Theorem V.2.7]{ChavelRiem}). Note that \eqref{equation:GaussBonnet} holds even in the nonorientable case and for polygons of constant curvature in the sense Definition \ref{definition:PolygonGeneral}. We will use this formula repeatedly in the sequel.

\begin{corollary}
\label{corollary:SpectralInvariantsVolumePerimeterCurvature}
Let $\Omega$ be a polygon of constant curvature. Then the volume $\vert \Omega \vert$, the perimeter $\vert \partial\Omega \vert$ and the curvature of the polygon are spectral invariants of $\Omega$. 
\end{corollary}

\begin{proof}
The heat invariants given by the coefficients of $\frac{1}{t}$ and $\frac{1}{\sqrt{t}}$ determine the volume and the perimeter of the polygon, respectively. If $\vert \partial\Omega \vert \neq 0$, then the curvature of the polygon is determined (for example) by the heat invariant corresponding to $\sqrt{t}$. If $\vert \partial\Omega \vert = 0$, then $\Omega$ has no boundary and vertices and thus the right-hand side of \eqref{equation:HeatAsymptoticConstantCurvaturePolygon} reduces to $I_{\kappa}$. Hence, the curvature can be gleaned from the heat invariant corresponding to $t^{0}$.
\end{proof}

\begin{definition}
If the boundary of a polygon is a disjoint union of piecewise geodesic simple closed curves, then we call it a \emph{simple polygon}.
\end{definition}

Note that for simple polygons the number of angles, edges, and vertices is the same.

\begin{theorem}
\label{theorem:SpectralInvariantsAnglesEulerCharacteristic}
Let $\Omega$ be a polygon of nonzero constant curvature. Then the number of angles which are not equal to $\pi$ is a spectral invariant. Moreover, the multiset consisting of all angles of $\Omega$ which are not equal to $\pi$ is a spectral invariant as well. Furthermore, the Euler characteristic $\chi(\Omega)$ of the polygon is a spectral invariant.

In particular, if $\Omega$ is a simple polygon with nonzero constant curvature, then the number of vertices and the multiset of all angles are spectral invariants.
\end{theorem}

\begin{proof}
Because the Gaussian curvature is not zero, there are infinitely many nonvanishing heat invariants. Let $\kappa\in\mathbb{R}\backslash\{ 0 \}$ denote the curvature of $\Omega$ and let $\gamma_1,...,\gamma_M$ be all angles of $\Omega$ which are not equal to $\pi$, where $M\in\mathbb{N}_0$. Then the coefficient corresponding to $t^{\nu}$, $\nu\in\mathbb{N}_0$, is given by
\begin{align*}
\frac{\vert \Omega \vert}{4\pi}f_{\nu+1}\cdot \kappa^{\nu+1} + \sum\limits_{i=1}^{M} e_{\nu}(\gamma_i) \kappa^{\nu}.
\end{align*}
Note that angles which are equal to $\pi$ can be ignored since $e_{\nu}(\pi)=0$ for all $\nu\in\mathbb{N}_0$.
We know from Corollary \ref{corollary:SpectralInvariantsVolumePerimeterCurvature} that the volume and the curvature of the polygon are determined by the spectrum. Therefore, the spectrum also determines the value of the sum $\sum_{i=1}^{M} e_{\nu}(\gamma_i)$ for all $\nu\in\mathbb{N}_0$. Note that for each $\nu\in\mathbb{N}_0$, the coefficient of the highest power term
\begin{align*}
W_{\nu}:=\sum\limits_{i=1}^{M} \frac{\pi^{2\nu+2} - \gamma_i^{2\nu+2}}{\pi \gamma_i^{2\nu+1}}
\end{align*}
in $\sum_{i=1}^{M} e_{\nu}(\gamma_i)$ corresponds to $\ell = \nu,\, j=\nu + 1$ in \eqref{equation:CoefficientsVerticesHeatAsymptoticConstantCurvaturePolygon} and equals $(-1)^{\nu}\cdot \frac{B_{2\nu}}{4(\nu+1)!(2\nu+1)}$. Since the Bernoulli numbers with even index never vanish, i.e. $B_{2\nu}\neq 0$ for all $\nu\in\mathbb{N}_0$ (see e.g. \citep[p. 23]{Noerlund}), we conclude by induction that the spectrum determines the sequence $(W_{\nu})_{\nu\in\mathbb{N}_0}$.

Thus we also obtain the following spectral invariants:
\begin{align*}
W_{\nu,1}:=\frac{1}{\pi^{2\nu+1}}\left( W_{\nu+1} - W_{\nu}\right) = \sum\limits_{i=1}^{M}\left( \left( \frac{\pi}{\gamma_i}\right)^2 - 1 \right)\left( \frac{1}{\gamma_i} \right)^{2\nu+1}\, \text{ for }  \nu\in\mathbb{N}_{0}.
\end{align*}
We claim that the smallest angle can be deduced from the sequence $(W_{\nu,1})_{\nu\in\mathbb{N}_0}$.

Since $\frac{\pi^2}{\gamma_i^2} - 1\neq 0$ by assumption, we have $\lim_{\nu\rightarrow \infty} \gamma^{2\nu+1} \cdot W_{\nu,1} = 0$ for all $\gamma>0$ if and only if $M=0$. Thus the spectrum determines whether the polygon has vertices or not. 

If $M\geq 1$ then for $\gamma=\min_i  \gamma_i =:\theta_1$ the above limit is not zero. More precisely, observe that there exists some $n_1 \in\mathbb{N}$ such that for all $\gamma \in (0, \theta_1]$ we have
\begin{align*}
\lim\limits_{\nu\rightarrow\infty} \gamma^{2\nu+1}\cdot W_{\nu,1} = 
\begin{cases}
0,\, &\text{ if } \gamma<\theta_1, \\
\left( \left( \frac{\pi}{\theta_1}\right)^2 - 1 \right)\cdot n_1,\, &\text{ if } \gamma = \theta_1.
\end{cases}
\end{align*}
Thus we have
\begin{align*}
\theta_1 = \inf\left\{\, \gamma>0 \, \middle| \, \gamma^{2\nu+1}\cdot W_{\nu, 1} \text{ is convergent as } \nu\rightarrow\infty \text{ with nonzero limit } \right\}.
\end{align*}
So the magnitude of the smallest angle $\theta_1$ is a spectral invariant. 

When we do the same argument as above with $W_{\nu,1}$ replaced by 
\begin{align*}
W_{\nu,2}:=W_{\nu, 1} -  \left( \left( \frac{\pi}{\theta_1}\right)^2 - 1 \right)\left( \frac{1}{\theta_1} \right)^{2\nu+1}
\end{align*}
we obtain the value of $\theta_2:=\min \left( \{\gamma_1,...,\gamma_M\}\backslash\{\theta_1\}\right)$. Repeating this argument in the above manner, we successively obtain $M$ values $\theta_1\leq\theta_2\leq ... \leq \theta_M$. This process will eventually stop when $W_{\nu,M+1}=0$ for all $\nu\in\mathbb{N}_{0}$. Thus, we obtain the value of $M$ as well as the multiset of angles $\{ \gamma_1,...,\gamma_M \} = \{ \theta_1,...,\theta_M \}$.

By \eqref{equation:GaussBonnet}, the Euler characteristic $\chi(\Omega)$ is now determined by the spectrum as well.

If $\Omega$ is a simple polygon, then there are no angles equal to $\pi$ and the number of angles is equal to the number of vertices. Thus the number of vertices of $\Omega$ as well as the multiset of all angles of the polygon are spectral invariants.
\end{proof}

\begin{corollary}
\label{corollary:SpectralInvariantsFromHeatInvariants}
Let $\Omega$ be a polygon of constant zero curvature. Then the spectrum, the Euler characteristic and the number $M$ of all angles of the polygon determines $\sum_{i=1}^{M}\frac{1}{\gamma_i}$, where $\gamma_1,...,\gamma_M\in (0,2\pi]$ denote the angles of the polygon.
\end{corollary}

\begin{proof}
The heat invariant corresponding to $t^{0}$ is given by 
\begin{align*}
\frac{\vert \Omega \vert}{4\pi}f_1\cdot \kappa + \sum\limits_{i=1}^{M} e_{0}(\gamma_i) =  \frac{1}{24 \pi}\left(  \sum\limits_{i=1}^{M} \frac{\pi^2 }{\gamma_i} - \sum\limits_{i=1}^{M} \gamma_i  \right).
\end{align*}
By the Gau{\ss}-Bonnet theorem we have $\sum_{i=1}^{M} \gamma_i = M \pi - 2\pi \chi(\Omega)$. Hence the sum $\sum_{i=1}^{M}\frac{1}{\gamma_i}$ is determined by the spectrum together with $M$ and $\chi(\Omega)$.
\end{proof}

Note that the above Corollary is well-known for simple Euclidean polygons (see e.g. \cite{GrieserTriangle}).

As it is usual, we call two polygons \emph{isospectral} if the spectrum of their Dirichlet Laplacians are equal and all corresponding eigenvalues have the same multiplicities.
By Corollary \ref{corollary:SpectralInvariantsVolumePerimeterCurvature} two polygons with different area, perimeter or curvature can never be isospectral. Or, positively formulated, two isospectral polygons do always have the same area, perimeter and curvature. But what else can be said about isospectral polygons? On the one hand, it is well-known that there exist isospectral polygons which are not isometric. Therefore it is not always possible to deduce the whole geometry of a polygon from its spectrum. On the other hand there are also some positive results known for simple Euclidean polygons. Let us review some of those results.

Obviously, if a polygon has zero curvature, then the heat invariants do not provide much information about the geometry of the polygon. They do not tell us how many vertices a simple polygon has, in contrast to the heat invariants for simple polygons with nonzero curvature (see Theorem \ref{theorem:SpectralInvariantsAnglesEulerCharacteristic} above). However, within special classes of polygons we can deduce nevertheless enough geometric information from the heat invariants to distinguish polygons by their spectra. For example, it is known that the spectrum of a Euclidean triangle determines the triangle in the following sense: If we know that $\Omega$ is a Euclidean polygon with $3$ angles, then the spectrum determines the polygon $\Omega$ up to isometry. This was first proven by C. Durso in her Ph.D. thesis \citep{Durso} using heat invariants and spectral invariants derived by other methods. Later on, a much shorter proof was given by D. Grieser and S. Maronna in \citep{GrieserTriangle}. Grieser and Maronna in fact prove that one can deduce from the heat invariants the values of all three angles of a triangle. 

Similar results are given in the recent article \citep{RowlettLu}. Z. Lu and J. Rowlett prove that one can deduce the values of the angles of a (Euclidean) parallelogram from its heat invariants and thus they obtain the following result: If we know a priori that $\Omega$ is a parallelogram, then the spectrum determines $\Omega$ up to isometry. They also prove in \citep{RowlettLu} that the regular $n$-gon uniquely maximises the isoperimetric ratio $\frac{\vert \Omega \vert}{ \vert \partial\Omega\vert^2}$ among all $n$-gons and thus they conclude: If an $n$-gon is isospectral to a regular $n$-gon, then they are isometric. There are a couple more of interesting results in \citep{RowlettLu}, but those are established using not only the heat invariants but other methods as well.

The following two corollaries add some simple and new observations.

\begin{remark}
If $\Omega\subset N$ is a simply connected polygon with constant curvature $\kappa = 0$, then $\Omega$ is isometric to a Euclidean polygon in $\mathbb{R}^2$. In fact, $\Omega$ is isometric to each component of its preimage under the universal covering $\mathbb{R}^2\rightarrow N$. In particular, $M\geq 3$ and $\sum_{i=1}^{M}\gamma_i = (M-2)\pi$ by elementary geometry, where $\gamma_1,...,\gamma_M$ denote all angles of $\Omega$.
\end{remark}

\begin{corollary}
\label{corollary:SpektralinvariantenFlachePolygone}
Let $\Omega$ be a simply connected Euclidean polygon. Let $M\geq 3$ denote the number of all angles. Then:
\begin{itemize}
\item[$(i)$] The spectrum together with $M$ determines whether $\Omega$ is equiangular or not.
\item[$(ii)$] Let $c_0:=\sum_{i=1}^{M}\frac{\pi^2 - \gamma_i^2}{24\pi\gamma_i}$ be the heat invariant corresponding to the power $t^0$. Then $c_0 > \frac{1}{6}$. Furthermore, $M\geq 1 + \frac{6\cdot c_0}{6 c_0 - 1}$, which gives a nontrivial bound if $c_0\in (\frac{1}{6}, \frac{1}{3})$.
\end{itemize}
\end{corollary}

\begin{proof}
$(i)$. Since $\Omega$ is simply connected, we have $\chi(\Omega)=1$. Thus, from Corollary \ref{corollary:SpectralInvariantsFromHeatInvariants} the spectrum together with $M$ determines $\sum_{i=1}^{M}\frac{1}{\gamma_i}$. Applying the arithmetic-harmonic mean inequality, we have
\begin{align*}
\frac{M}{\sum_{i=1}^{M}\frac{1}{\gamma_i}} \leq \frac{\sum_{i=1}^{M}\gamma_i}{M} = \frac{(M-2)\pi}{M},
\end{align*}
with equality if and only if $\gamma_1 = \gamma_2 = ... = \gamma_M$. Thus all angles of the polygon $\Omega$ are equal if and only if $\frac{M^2}{(M-2)\pi} = \sum_{i=1}^{M}\frac{1}{\gamma_i}$.

$(ii)$. Again by $\sum_{i=1}^{M}\gamma_i = (M-2)\pi$ and the arithmetic-harmonic mean inequality we have
\begin{align*}
c_0 &= \frac{\pi}{24} \left( \sum\limits_{i=1}^{M} \frac{1}{\gamma_i} \right) - \frac{M-2}{24} \\
& \geq \frac{1}{24}\cdot \frac{M^2}{(M-2)}  - \frac{M-2}{24}.
\end{align*}
This is equivalent to
\begin{align*}
 12 c_0 - 1  \leq M\left( 6c_0 - 1 \right).
\end{align*}
If $c_0<\frac{1}{6}$, then it follows that $M\leq 1 + \frac{6c_0}{6c_0 - 1 }<2$. This is impossible by $M\geq 3$. Similarly, if $c_0 = \frac{1}{6}$, then we obtain $1\leq 0$ which is a contradiction. Thus it follows that $c_0 > \frac{1}{6}$.

By $c_0>\frac{1}{6}$, we now obtain $M\geq 1 + \frac{6 c_0}{6 c_0 - 1}$, which completes the proof of $(ii)$. This lower bound is useless for $c_0>\frac{1}{3}$, since the function $f(x):=1 + \frac{6 x}{6 x - 1}$ is strictly decreasing in $(\frac{1}{6}, \infty)$ and $f(\frac{1}{3}) = 3$.
\end{proof}

\begin{corollary} \quad
\label{corollary:SmoothDomainsSimplyConnectedPolygonalDomainsNeverIsospectral}
\begin{itemize}
\item[$(i)$] A simply connected Euclidean polygon can never be isospectral to a two-dimensional compact manifold with nonempty smooth boundary.
\item[$(ii)$] A polygon with zero curvature can not be isospectral to a smooth domain in the Euclidean plane.
\end{itemize}
\end{corollary}

\begin{proof}
$(i)$. Note that for a compact two-dimensional manifold $D$ with nonempty smooth boundary the heat invariant corresponding to the power $t^0$ is given by $\frac{1}{6}\chi(D)$, where $\chi(D)$ is the Euler characteristic of $D$. We have $\frac{1}{6}\chi(D) \leq \frac{1}{6}(2-b)\leq \frac{1}{6}$, where $b\in\mathbb{N}$ denotes the number of boundary components of $D$. Using Corollary \ref{corollary:SpektralinvariantenFlachePolygone} $(ii)$, the statement follows.

$(ii)$. Note that for a smooth domain in the Euclidean plane the heat invariant corresponding to the power $t^{\frac{1}{2}}$ is always positive (see the coefficient $c_3$ in \citep{Smith}). But the corresponding heat invariant for a polygon with zero curvature is always $0$. 
\end{proof}

\begin{remark}
Statement $(i)$ of the above corollary was proven independently in the recent article \citep{RowlettLuCorners}. In our case, Corollary \ref{corollary:SmoothDomainsSimplyConnectedPolygonalDomainsNeverIsospectral} appeared as a by-product from our desire to estimate the number of angles of a simply connected Euclidean polygon (Corollary \ref{corollary:SpektralinvariantenFlachePolygone} $(ii)$). Z. Lu and J. Rowlett instead investigate in \citep{RowlettLuCorners} simply connected planar domains with piecewise smooth Lipschitz boundary.
\end{remark}

We have seen that, if a Euclidean polygon is simply connected, Corollary \ref{corollary:SpektralinvariantenFlachePolygone} $(ii)$ in some cases gives a lower bound for the number of angles. In general, however, the heat invariants do not suffice to detect the exact number of angles. We conjecture that Theorem \ref{theorem:SpectralInvariantsAnglesEulerCharacteristic} also holds for polygons of constant zero curvature. That is, we conjecture in particular that the number of vertices as well as the multiset of all angles of a simple polygon with zero curvature is always determined by its spectrum. But we do not know how to prove this conjecture. As we mentioned above, there are lots of known pairs of isospectral and nonisometric simple Euclidean polygons (see e.g. \citep{GordonWebbWolpert}, \citep{BuserConwayDoyle}, \citep{Chapman}). All of these isospectral pairs indeed have the same number of vertices and the same multiset of angles.

Let us also include polygons with nonzero curvature into our considerations. Theorem \ref{theorem:SpectralInvariantsAnglesEulerCharacteristic} shows that pairs of isospectral simple hyperbolic polygons as well as pairs of simple spherical polygons must have the same number of vertices and multiset of angles. However, as for Euclidean polygons, it is known that the spectrum does not always determine all of the geometry. There exist pairs of hyperbolic and spherical polygons which are isospectral and nonisometric (see \citep{GordonWebbHyperbolic} for such pairs of  isospectral polygons). Nevertheless, we can deduce some information about the geometry from the spectrum.

\begin{corollary}
\label{corollary:PolygonNotIsospectralSmoothDomain}
Let $\Omega$ be a polygon with at least one angle not equal to $\pi$ and of constant curvature $\kappa\in\mathbb{R}$. Let $c_0$ denote the heat invariant for $\Omega$ corresponding to $t^{0}$. Further, let $D$ be a two-dimensional Riemannian manifold with smooth boundary.  Then:
\begin{itemize}
\item[$(i)$] $c_0>\frac{1}{6}\chi(\Omega)$.
\item[$(ii)$] If $\Omega$ is isospectral to $D$, then $\chi(\Omega)<\chi(D)$.
\end{itemize}


\end{corollary}

\begin{proof}
Let $\gamma_1,...,\gamma_M\in (0,2\pi]$ denote all angles of $\Omega$, where $M\in\mathbb{N}$. From Corollary \ref{corollary:HeatAsymptoticConstantCurvaturePolygon}, we know that the heat invariant for $\Omega$ corresponding to $t^{0}$ is
\begin{align*}
c_{0}:=\frac{\vert \Omega \vert}{4\pi}f_1 \kappa + \sum\limits_{i=1}^{M} e_{0}(\gamma_i).
\end{align*}
The Bernoulli polynomial $B_2(x)$ is given by $B_2(x)=x^2-x+\frac{1}{6}$ (see \citep[$24.2(iv)$ Tables]{NIST}) and thus $B_2=\frac{1}{6}$ and $B_2(\frac{1}{2})=-\frac{1}{12}$. Hence, $f_1=\frac{1}{3}$ and 
\begin{align*}
c_{0}  = \frac{\vert \Omega \vert}{12\pi} \kappa + \frac{\pi}{24} \left( \sum\limits_{i=1}^{M}\frac{1}{\gamma_i} \right) - \frac{1}{24\pi} \sum\limits_{i=1}^{M}\gamma_i. 
\end{align*}
By \eqref{equation:GaussBonnet} and the arithmetic-geometric mean inequality we have
\begin{align*}
c_0 &= \frac{1}{6}\chi(\Omega) +\frac{1}{24}\sum\limits_{i=1}^{M}\left( \frac{\pi}{\gamma_i} + \frac{\gamma_i}{\pi} \right) -\frac{M}{12} \\
&\geq  \frac{1}{6}\chi(\Omega),
\end{align*}
with equality if and only if $\gamma_i = \pi$ for all $i=1,...,M$. This proves both statements of the corollary. For $(ii)$ note that, because $D$ is a smooth domain, the heat invariant for $D$ corresponding to $t^{0}$ is just $\frac{1}{6}\chi(D)$.
\end{proof}

Corollary \ref{corollary:PolygonNotIsospectralSmoothDomain} shows, in particular, that if a polygon with at least one vertex is isospectral to a smoothly bounded compact domain, then they can never be homeomorphic. It is suprising that, to the best of our knowledge, this result was never published for domains in the Euclidean plane.


Hyperbolic and spherical triangles are determined up to isometry by their spectrum within the class of all polygons.

\begin{corollary}
\label{corollary:IsospectralTriangle} 
Let $\Omega$ be any polygon of constant curvature. Then:
\begin{itemize}
\item[$(i)$] If $\Omega$ is isospectral to a hyperbolic triangle $T_{\mathbb{H}}$, then $\Omega$ and $T_{\mathbb{H}}$ are isometric.
\item[$(ii)$] If $\Omega$ is isospectral to a spherical triangle $T_{\mathbb{S}}$, then $\Omega$ and $T_{\mathbb{S}}$ are isometric.
\end{itemize}
\end{corollary}

\begin{proof}
Both statements follow immediately from Corollary \ref{corollary:SpectralInvariantsVolumePerimeterCurvature} and Theorem \ref{theorem:SpectralInvariantsAnglesEulerCharacteristic}. Just note that hyperbolic and spherical triangles are uniquely determined by their angles.
\end{proof}

Note that in Corollary \ref{corollary:IsospectralTriangle} we do not need to assume a priori that $\Omega$ is a triangle, unlike in the analogous result in \citep{GrieserTriangle} for the Euclidean plane.

\begin{corollary}
\label{corollary:ExistencePolygonsSpectrallyDetermined}
Let $\Omega$ be any hyperbolic or spherical polygon.
\begin{itemize}
\item[$(i)$] The spectrum determines whether $\Omega$ is convex or not.
\item[$(ii)$] If $\Omega$ is isospectral to a regular hyperbolic polygon $R_{\mathbb{H}}$, then $\Omega$ and $R_{\mathbb{H}}$ are isometric.
\item[$(iii)$] If $\Omega$ is isospectral to a regular spherical polygon $R_{\mathbb{S}}$, then $\Omega$ and $R_{\mathbb{S}}$ are isometric.
\end{itemize}
\end{corollary}

\begin{proof}
$(i)$. As in the Euclidean plane, a polygon is convex if and only if the values of all its angles are contained in $(0,\pi)$. Thus the claim follows from Theorem \ref{theorem:SpectralInvariantsAnglesEulerCharacteristic}.


$(ii)$ and $(iii)$. It follows from \citep{Porti} that there is a unique polygon which minimises the perimeter among all equiangular polygons, namely the polygon with an inscribed circle. Since regular polygons do have an inscribed circle, they minimise the perimeter within the class of all equiangular polygons of fixed angle.
\end{proof}

\section{Method of images and finite trigonometric sums}
\label{section:ApplicationPolygons}

As in Section \ref{section:ExpansionTrace}, let $\Omega$ be a hyperbolic polygon with angles $\gamma_1,...,\gamma_M$, where $M\geq 3$ is an integer. In addition we assume that for all $i\in\{\, 1,...,M \,\}$ there exists some integer $k(i)\geq 2$ such that $\gamma_i=\frac{\pi}{k(i)}$. Let $W_{\gamma_i}$ be the wedge corresponding to $\gamma_i$ and let $P_i$ denote the vertex of $W_{\gamma_i}$ for all $i=1,...,M$. 

In the following we will compute the asymptotic expansion of the heat trace $Z_{\Omega}(t)$ as $t\searrow 0$. Of course, this problem is only a special case of Theorem \ref{theorem:AsymptoticExpansionHeatTraceHyperbolPolygon} such that we already know the answer. However, as we will show, one can approach the problem with much more elementary means in the present situation. The reason is that the heat kernel for a wedge with angle $\frac{\pi}{k}$, $k\in\mathbb{N}_{\geq 2}$, has a more elementary expression than in Corollary \ref{corollary:FormelHeatKernelWedgeShift}, which is easily deduced from Sommerfeld's method of images. We will use this elementary formula as a substitute for Corollary \ref{corollary:FormelHeatKernelWedgeShift} and compute the heat invariants again by Kac's principle of not feeling the boundary. The contributions from the vertices of the polygon will then appear as finite trigonometric sums. By comparison with the formulas of Section \ref{section:ExpansionTrace} we obtain explicit evaluations for those trigonometric sums.

We begin with the construction for the heat kernel using Sommerfeld's method of images. Let $W\subset\mathbb{H}^2$ be a hyperbolic wedge with vertex $P$ and interior angle $\gamma = \frac{\pi}{k}$, where $k\geq 2$ is an integer. Let us choose $2k$ different rays $s_1,...,s_{2k}$ originating at $P$ such that any two adjacent rays $s_j, s_{j+1}$ (counting indices mod $2k$) form a hyperbolic wedge of angle $\gamma$ and $s_1,\, s_2$ enclose the wedge $W$ we started with. Thus we get a decomposition of $\mathbb{H}^2$ into $2k$ congruent wedges (see Fig. \ref{Skizze:2kStrahlenInP}).

\begin{figure} [ht] 
 \centering
\begin{tikzpicture}


\fill[fill=black!10!white] (0,0) -- (2,0) arc (0:60:2) -- (0,0);


\draw[dashed] (2,0) arc (0:360:2);
\draw (0,0) -- (2,0);

\coordinate (target1) at ({2*cos(60)}, {2*sin(60)});
\draw (0,0) -- (target1);

\coordinate (target2) at ({2*cos(120)}, {2*sin(120)});
\draw (0,0) -- (target2);

\coordinate (target3) at ({2*cos(180)}, {2*sin(180)});
\draw (0,0) -- (target3);

\coordinate (target4) at ({2*cos(240)}, {2*sin(240)});
\draw (0,0) -- (target4);

\coordinate (target5) at ({2*cos(300)}, {2*sin(300)});
\draw (0,0) -- (target5);


\node[above] at (1.6,0) {\small $W$};
\node[above] at (0,0) {\footnotesize $P$};
\node[right] at (2,0) {\footnotesize $s_1$};
\node[above right] at (target1) {\footnotesize $s_2$};
\node[above left] at (target2) {\footnotesize $s_3$};
\node[left] at (target3) {\footnotesize $s_4$};
\node[below left] at (target4) {\footnotesize $s_5$};
\node[below right] at (target5) {\footnotesize $s_6$};


\draw (0.5,0) arc (0:60:0.5);
\node[above ] at (0.3,-0.1) {\tiny $\gamma$};


\end{tikzpicture}
\caption{ Rays for $k=3$ } \label{Skizze:2kStrahlenInP}
\end{figure}
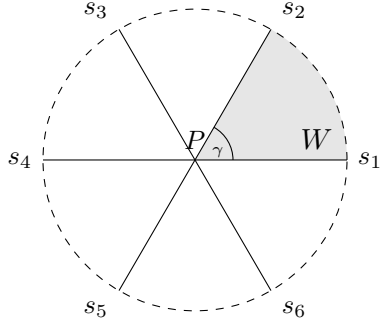

Next, for any given point $y\in W$ we construct image points by means of successive reflections in those rays. More precisely, for any fixed point $y_1:=y\in W$ we first reflect this point in the ray $s_2$ and denote the image point by $y_2$. Secondly, we reflect $y_2$ in the ray $s_3$ and get another point $y_3$, etc. By this process we obtain $2k$ different points $y_1,...,y_{2k}$ each of which lies in exactly one of the wedges described above (see Fig. \ref{Skizze:Spiegelpunkte}). Note that if we reflect the point $y_{2k}$ in the ray $s_1$ the image point coincides with the initial point $y_1$. 

\begin{figure} [ht] 
 \centering
\begin{tikzpicture}


\fill[fill=black!10!white] (0,0) -- (2,0) arc (0:60:2) -- (0,0);


\draw[dashed] (2,0) arc (0:360:2);
\draw (0,0) -- (2,0);

\coordinate (target1) at ({2*cos(60)}, {2*sin(60)});
\draw (0,0) -- (target1);

\coordinate (target2) at ({2*cos(120)}, {2*sin(120)});
\draw (0,0) -- (target2);

\coordinate (target3) at ({2*cos(180)}, {2*sin(180)});
\draw (0,0) -- (target3);

\coordinate (target4) at ({2*cos(240)}, {2*sin(240)});
\draw (0,0) -- (target4);

\coordinate (target5) at ({2*cos(300)}, {2*sin(300)});
\draw (0,0) -- (target5);


\node[above] at (1.6,0) {\small $W$};
\node[above] at (0,0) {\footnotesize $P$};
\node[right] at (2,0) {\footnotesize $s_1$};
\node[above right] at (target1) {\footnotesize $s_2$};
\node[above left] at (target2) {\footnotesize $s_3$};
\node[left] at (target3) {\footnotesize $s_4$};
\node[below left] at (target4) {\footnotesize $s_5$};
\node[below right] at (target5) {\footnotesize $s_6$};


\draw node[circle, draw, minimum size=3mm, inner sep=0pt] at (1,1) {$+$}; 

\coordinate (y2) at ({1.4142*cos(75)}, {1.4142*sin(75)});
\draw node[circle, draw, minimum size=3mm, inner sep=0pt] at (y2) {$-$};

\coordinate (y3) at ({1.4142*cos(165)}, {1.4142*sin(165)});
\draw node[circle, draw, minimum size=3mm, inner sep=0pt] at (y3) {$+$}; 

\coordinate (y4) at ({1.4142*cos(195)}, {1.4142*sin(195)});
\draw node[circle, draw, minimum size=3mm, inner sep=0pt] at (y4) {$-$}; 

\coordinate (y5) at ({1.4142*cos(285)}, {1.4142*sin(285)});
\draw node[circle, draw, minimum size=3mm, inner sep=0pt] at (y5) {$+$}; 

\coordinate (y6) at ({1.4142*cos(315)}, {1.4142*sin(315)});
\draw node[circle, draw, minimum size=3mm, inner sep=0pt] at (y6) {$-$};


\node[right] at (1.1,1) {\footnotesize $y_1$};
\node[above left] at (y2) {\footnotesize $y_2$};
\node[above right] at (y3) {\footnotesize $y_3$};
\node[below right] at (y4) {\footnotesize $y_4$};
\node[below left] at (y5) {\footnotesize $y_5$};
\node[above right] at (y6) {\footnotesize $y_6$};


\draw (0.5,0) arc (0:60:0.5);
\node[above ] at (0.3,-0.1) {\tiny $\gamma$};


\end{tikzpicture}
\caption{Mirror images for $k=3$}\label{Skizze:Spiegelpunkte}
\end{figure}
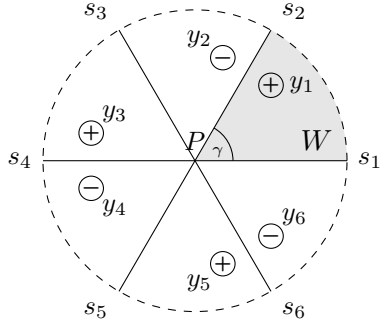

When we introduce polar coordinates $(\rho, \theta)$ with base point $P$, where $\rho\in(0,\infty)$ is the radius and $\theta\in [0,2\pi)$ is the angle measured with respect to $s_1$, then we can parametrise the wedge as $W=\{\,(\rho,\theta)\mid 0<\rho<\infty,\, 0<\theta<\gamma \,\}$. Further, for any initial point $y=(\rho,\theta)\in W$, the coordinates of the image points $y_{j}$, $j=2,3,..,2k$, are given by $y_{j}=\left( \rho, \theta_{j} \right)$ with
\begin{align*}
\theta_{j}:=\begin{cases}
(j-1)\cdot \gamma+\theta, & \text{if $j$ is odd}\\
j\cdot \gamma - \theta, & \text{if $j$ is even.}
\end{cases}
\end{align*}

The heat kernel $K_{\gamma}$ for the wedge $W$ can now be written as in the following lemma.

\begin{lemma}
\label{lemma:FormelHeatKernelWedgeSpiegelung}
\begin{align*}
K_{\gamma}(x,y;t)=\sum\limits_{j=1}^{2k}\left(-1 \right)^{j+1} K_{\mathbb{H}^2}(x,y_j;t)  \text{ for all } x,y\in W\,\text{ and } t>0,
\end{align*} 
where $K_{\mathbb{H}^2}$ denotes, as usual, the heat kernel of the hyperbolic plane.
\end{lemma}

\begin{proof}
Let $\I_j$ be the isometry of $\mathbb{H}^2$ which corresponds to $y_{j}$, i.e. $\I_1:\mathbb{H}^2\rightarrow\mathbb{H}^2$ denotes the identity map, and, for all $j=2,3,...,2k$ the isometry $\I_j:\mathbb{H}^2  \rightarrow\mathbb{H}^2$ is represented in the above polar coordinates by
\begin{align*}
\I_j(\rho,\theta):=\begin{cases}
(\rho, (j-1)\cdot \gamma+\theta), & \text{if $j$ is odd}\\
(\rho, j\cdot \gamma - \theta), & \text{if $j$ is even.}
\end{cases}
\end{align*}
We define $\tilde{K}_{W}:\mathbb{H}^2\times \mathbb{H}^2\times (0,\infty)\ni (x,y,t)\mapsto \tilde{K}_{W}(x,y;t):=\sum_{j=1}^{2k}\left(-1 \right)^{j+1} K_{\mathbb{H}^2}(x,\I_jy;t)\in\mathbb{R}$. Note that $I_j y=y_j$ for all $y\in W$ and $j=1,2,...,2k$. 

%

First, we prove that for any $y\in W$ the function $W\times (0,\infty)\ni (x,t)\mapsto \tilde{K}_{W}(x,y;t)\in\mathbb{R}$ is a fundamental solution to the heat equation at $y$ (recall Definition \ref{definition:FundamentalSolution}): Obviously, if $y\in W$ is given, then  $(\partial_t+\Delta )\tilde{K}_{W}(\cdot,y;\cdot)\equiv 0$ on $W\times (0,\infty)$, because each function $K_{\mathbb{H}^2}(\cdot, y_j;\cdot)$ is a solution to the heat equation. Moreover, let $f\in C_{c}^{\infty}(W)$ be arbitrary. We extend that function by $f(x):=0$ for all $x\in \mathbb{H}^2\backslash{W}$ (and we denote that extended function again by $f$) such that $f$ also belongs to $C_{c}^{\infty}(\mathbb{H}^2)$. Now, for all $y\in W$
\begin{align*}
\int\limits_{W} \tilde{K}_{W}(x,y;t) f(x) dx = \sum\limits_{j=1}^{2k}\left(-1 \right)^{j+1}  \int\limits_{\mathbb{H}^2} K_{\mathbb{H}^2}(x,y_j;t) f(x) dx.
\end{align*}
By definition of $K_{\mathbb{H}^2}$, we have $\lim_{t\searrow 0}\int_{\mathbb{H}^2} K_{\mathbb{H}^2}(x,y_j;t) f(x) dx =f(y_j)$. Finally, because $f(y_j)=0$ for all $j=2,3,...,2k$ we obtain $\lim_{t\searrow 0}\int_{W} \tilde{K}_{W}(x,y;t) f(x) dx = f(y)$.

Second, we will prove that $\tilde{K}_{W}$ is non-negative on $W\times W\times(0,\infty)$, i.e. $\tilde{K}(x,y;t)\geq 0$ for all $x,y\in W$ and $t>0$. Let $\varphi\in C_{c}^{\infty}(W)$ with $\varphi\geq 0$ be arbitrary and extend, as before, that function by $\varphi(y):=0$ for all $y\in \mathbb{H}^2\backslash  W  $. Consider the function $u:\mathbb{H}^2\times [0,\infty)\rightarrow \mathbb{R}$, defined as
\begin{align*}
u(x,t):=\begin{cases}
\int_{\mathbb{H}^2} \tilde{K}_{W}(x,y;t) \cdot \varphi(y) \,dy & \text{ for all } x\in \mathbb{H}^2 \text{ and }t>0, \\
\varphi(x) & \text{ for all } x\in \mathbb{H}^2 \text{ and }t=0.
\end{cases}
\end{align*}
Note that for all $x\in\mathbb{H}^2$ and $t>0$:
\begin{align*}
u(x,t) = \sum\limits_{j=1}^{2k}\left(-1 \right)^{j+1}  \int\limits_{\mathbb{H}^2} K_{\mathbb{H}^2}(x,y;t) \varphi(\I_j^{-1}y)\, dy,
\end{align*}
where $\I_j^{-1}$ denotes the inverse to $\I_j$.
Hence, $u_{\vert \mathbb{H}^2\times (0,\infty)}$ is a smooth solution to the heat equation, which follows from the definition of $K_{\mathbb{H}^2}$ (see \citep[formula (7.49)]{Grigoryan}) and \citep[Theorem 7.10]{Grigoryan}. Moreover, $u$ is continuous (see \citep[Remark 7.17]{Grigoryan}). 

We want to apply the parabolic minimum principle to show that $u_{\vert W \times (0,\infty)}$ is non-negative. For that purpose, note that $\tilde{K}_{W}(x,y;t)=0$ for all $x\in \partial W$, $y\in W$ and $t\in (0,\infty)$. That follows from the definition of $y_j$ and because $K_{\mathbb{H}^2}(x,y;t)=K_{\mathbb{H}^2}(\tilde{x},\tilde{y};t)$ for all $x,y,\tilde{x},\tilde{y}\in\mathbb{H}^2$ with $d(x,y)=d(\tilde{x},\tilde{y})$. Thus, we also have $u(x,t)=0$ for all $x\in \partial W$ and $t>0$. Moreover,  $u(x,0)=\varphi(x)\geq 0$ for all $x\in\overline{W}$. Further, for any compact set $K\subset \mathbb{H}^2$, $T>0$, and any sequence $(x_{n})_{n\in\mathbb{N}}\subset \mathbb{H}^2$ with $\lim_{n\rightarrow \infty}d(x_n,P)=\infty$, we have by the estimate \eqref{equation:EstimateHeatKernelPlane}:
\begin{align}
\label{equation:UniformEstimateInfinityHeatKernel}
\lim\limits_{n\rightarrow\infty}\sup_{(y,t)\in K\times (0,T]} \vert K_{\mathbb{H}^2}(x_n,y;t) \vert = 0,
\end{align}
and consequently
\begin{align*}
\lim\limits_{n\rightarrow\infty} \sup_{t\in (0,T]} \vert u(x_n,t)\vert = 0.
\end{align*}
Now, using the parabolic minimum principle (see \citep[Theorem 8.10]{Grigoryan}) we conclude $u(x,t)\geq 0$ for all $x\in W$ and $t>0$. More precisely, suppose that there exists some point $(x_{\ast},t_{\ast})\in W\times (0, T)$, for some $T>0$ such that $u(x_{\ast},t_{\ast})<0$. Choose some $R>0$ such that $u(x_{\ast},t_{\ast}) < u(x,t)$ for all $t\in (0,T]$ and $x\in W$ with $d(x,P)\geq R$. But then, the minimum of $u$ on the set $\overline{ W\cap B_R(P)}\times [0,T]$ is not attained on its parabolic boundary, which is a contradiction to the parabolic minimum principle. Finally, because $\tilde{K}_W$ is continuous and $\varphi:W\rightarrow \mathbb{R}$ was an arbitrary non-negative and compactly supported smooth function, we have for any $x\in W$: $\tilde{K}_{W}(x,y;t)\geq 0$ for all $y\in W$ and $t>0$. Using the symmetry $\tilde{K}_{W}(x,y;t)=\tilde{K}_{W}(y,x;t)$ we indeed have $\tilde{K}_{W}(x,y;t)\geq 0$ for all $x,y\in W$ and $t>0$.

The statement now follows from \citep[Theorem 9.7]{Grigoryan} and the estimates \eqref{equation:UniformEstimateInfinityHeatKernel}, \eqref{equation:EstimateHeatKernelPlane}.
\end{proof}

We decompose the polygon $\Omega$ into the same subsets as in Section \ref{section:ExpansionTrace}. Note that $\Omega$ is a \emph{simple} hyperbolic polygon because of the assumptions on its angles and thus $\Omega$ has as many edges as angles (i.e. $\tilde{M} = M$). Further, we can assume without loss of generality that the edge $E_i$ is equal to the geodesic segment between the vertices $P_i$ and $P_{i+1}$ for all $i=1,...,M$ with $P_{M+1}:=P_1$. Recall that due to the principle of not feeling the boundary (see Corollary \ref{corollary:PNFB} and \eqref{equation:ErsatzSummeAsymptotik1}), $Z_{\Omega}^{\nicefrac{1}{4}}(t)$ has the same asymptotic expansion as the following sum of functions:
\begin{align}
\label{equation:AsymptotikErsatzSummeShiftedHeatTrace}
\int\limits_{\Omega_I}K_{\mathbb{H}^2}^{\nicefrac{1}{4}}(x,x;t)dx + \sum\limits_{i=1}^{M}\int\limits_{\text{ }\Omega_{E_i}}K_{E_i}^{\nicefrac{1}{4}}(x,x;t)dx + \sum\limits_{i=1}^{M} \int\limits_{\text{ }W_R(P_i)}K_{\gamma_i}^{\nicefrac{1}{4}}(x,x;t)dx,
\end{align}
where we use the same notation as in Section \ref{section:ExpansionTrace}; in particular,
\begin{align*}
K_{E_i}(x,x;t) = K_{\mathbb{H}^2}(x,x;t) - K_{\mathbb{H}^2}(x,x_{E_i};t).
\end{align*}

Let us insert the above formulas for $K_{\gamma_i}$ and $K_{E_i}$ into \eqref{equation:AsymptotikErsatzSummeShiftedHeatTrace}. Since we are now confronted with several wedges at the same time, the notation for the image points $y_j$ in the definition of $K_{\gamma}$ can become ambiguous sometimes. In such a case we will write $y_j(\gamma_i)$ instead of $y_j$ to indicate from which wedge the point $y_j$ originates from, i.e. we denote the heat kernel of $W_{\gamma_i}$ as $K_{\gamma_i}(x,y;t)=\sum_{j=1}^{2k(i)} (-1)^{j+1}K_{\mathbb{H}^2}(x,y_j(\gamma_i);t)$. 
The sum \eqref{equation:AsymptotikErsatzSummeShiftedHeatTrace} is then equal to
\begin{align*}
\int\limits_{\Omega}K_{\mathbb{H}^2}^{\nicefrac{1}{4}}(x,x;t)dx -  \sum\limits_{i=1}^{M} \Big( \int\limits_{\text{ }\Omega_{E_i}}K_{\mathbb{H}^2}^{\nicefrac{1}{4}}(x,x_{E_i};t)dx &+ \sum\limits_{j=1}^{k(i)} \int\limits_{\text{ }W_R(P_i)} K_{\mathbb{H}^2}^{\nicefrac{1}{4}}\left( x,x_{2j}(\gamma_i);t \right)dx \Big) +  \\
&+ \sum\limits_{i=1}^{M} \sum\limits_{j=1}^{k(i)-1} \int\limits_{\text{ }W_R(P_i)} K_{\mathbb{H}^2}^{\nicefrac{1}{4}}\left( x,x_{2j+1}(\gamma_i);t \right)dx.
\end{align*}
Thus we need to compute the asymptotic expansion of the following functions:
\begin{align*}
Z_I^{\nicefrac{1}{4}}(t)&=\int\limits_{\Omega}K_{\mathbb{H}^2}^{\nicefrac{1}{4}}(x,x;t)dx, \\
\tilde{Z}_{E}^{\nicefrac{1}{4}}(t):&= -\sum\limits_{i=1}^{M}\Big( \int\limits_{\text{ }\Omega_{E_i}}K_{\mathbb{H}^2}^{\nicefrac{1}{4}}(x,x_{E_i};t)dx  + \sum\limits_{j=1}^{k(i)} \int\limits_{\text{ }W_R(P_i)} K_{\mathbb{H}^2}^{\nicefrac{1}{4}}\left( x,x_{2j}(\gamma_i);t \right)dx \Big), \\
\tilde{Z}_V^{\nicefrac{1}{4}}(t):&= \sum\limits_{i=1}^{M}  \sum\limits_{j=1}^{k(i)-1}  \int\limits_{\text{ }W_R(P_i)} K_{\mathbb{H}^2}^{\nicefrac{1}{4}}\left( x,x_{2j+1}(\gamma_i);t \right)dx.
\end{align*}

\begin{remark}
In Section \ref{section:ExpansionTrace} we also worked with three similar functions denoted there by $Z_I^{\nicefrac{1}{4}}, Z_{E}^{\nicefrac{1}{4}}$ and $Z_V^{\nicefrac{1}{4}}$. The definition for $Z_I^{\nicefrac{1}{4}}$, given in \eqref{equation:DefinitionOfZ_I}, was exactly the same as above and, as we will see in Corollary \ref{corollary:HeatKernelsHalfSpaceTheSame}, $\tilde{Z}_{E}^{\nicefrac{1}{4}} = Z_{E}^{\nicefrac{1}{4}}$. But, strictly speaking, the functions $\tilde{Z}_V^{\nicefrac{1}{4}}$ and $Z_V^{\nicefrac{1}{4}}$ are not equal to each other, even though they have the same asymptotic expansion as $t\searrow 0$.
\end{remark}

The asymptotic expansion of $Z_I^{\nicefrac{1}{4}}(t)$ was already computed in Corollary $\ref{corollary:AsymptotischeEntwicklungShiftedZ_I}$ $(i)$.

\begin{corollary}
\label{corollary:HeatKernelsHalfSpaceTheSame}
We have $\tilde{Z}_{E}^{\nicefrac{1}{4}} = Z_E^{\nicefrac{1}{4}}$, where $Z_E^{\nicefrac{1}{4}}$ is defined as in $\eqref{equation:DefinitionOfZ_E(t)}$. In particular, by Corollary $\ref{corollary:AsymptoticExpansionZ_E}$ $(i):$ 
\begin{align}
Z_{E}^{\nicefrac{1}{4}}(t)\overset{t\downarrow 0}{\sim} -\frac{\vert \partial\Omega \vert}{8\sqrt{\pi t}}\cdot \sum\limits_{k=0}^{\infty}\delta_{0k}\cdot t^k.
\end{align}
\end{corollary}

\begin{proof}
 For any $t>0$, the function $K_{\mathbb{H}^2}^{\nicefrac{1}{4}}\left( x,y;t \right)$ depends only on the distance $d(x,y)$ between $x$ and $y$. In other words, $K_{\mathbb{H}^2}^{\nicefrac{1}{4}}\left( x,y;t \right)= K_{\mathbb{H}^2}^{\nicefrac{1}{4}}\left( \phi(x),\phi(y);t \right)$ for any isometry $\phi:\mathbb{H}^2\rightarrow\mathbb{H}^2$. In particular, the heat kernel is symmetric in $x,y$. Fix some $i\in\{1,...,M\}$ and let $S_j:\mathbb{H}^2\rightarrow\mathbb{H}^2$ denote the reflection in the line determined by the ray $s_j$, where the rays $s_1,...,s_{2k(i)}$ are defined as above (compare Fig. \ref{Skizze:Spiegelpunkte}) now with respect to the wedge $W_{\gamma_i}$. (In particular $s_1$ contains $E_i$.) Let $D_{\alpha}$, $\alpha\in\mathbb{R}$, denote the rotation around $P_i$ by angle $\alpha$ in the positive direction. Then $S_{j+1} = D^{2j}_{\gamma_i}\circ S_1$ and $S_{j+1} = D_{\gamma_i}^j\circ S_1\circ D_{\gamma_i}^{-j}$. In particular, $S_{j+1}x = x_{2j}=D_{\gamma_i}^{2j}S_1 x$; also, note that $S_1(W_R(P_i))=D_{\gamma_i}^{-1}(W_R(P_i))$. Therefore
\begin{align*}
\int\limits_{W_R(P_i)} K_{\mathbb{H}^2}^{\nicefrac{1}{4}}\left( x,x_{2j};t \right) dx &= \int\limits_{S_{j+1}(W_R(P_i))} K_{\mathbb{H}^2}^{\nicefrac{1}{4}}\left( x, S_{j+1}x ;t \right) dx\\
&= \int\limits_{S_{j+1}(W_R(P_i))} K_{\mathbb{H}^2}^{\nicefrac{1}{4}}\left( D_{\gamma_i}^{-j} x, S_1 D_{\gamma_i}^{-j}x ;t \right) dx\\
& =\int\limits_{D_{\gamma_i}^{-j}S_{j+1}(W_R(P_i))} K_{\mathbb{H}^2}^{\nicefrac{1}{4}}\left( x, S_1 x ;t \right)\\
&=\int\limits_{D_{\gamma_i}^{j-1}( W_R(P_i))} K_{\mathbb{H}^2}^{\nicefrac{1}{4}}\left( x, S_1 x ;t \right).
\end{align*}
The union of $D_{\gamma_i}^{j-1}( W_R(P_i))$ over $j=1,...,k(i)$ is equal to a half-disc and can be parametrised in polar coordinates as
\begin{align*}
\bigcup_{j=1}^{k(i)}D_{\gamma_i}^{j-1} (W_R(P_i)) = \{\, \left( a,\alpha \right) \mid 0<a<R,\, 0<\alpha<\pi\,\}.
\end{align*}
Hence,
\begin{align*}
 \sum\limits_{j=1}^{k(i)} \int\limits_{\text{ }W_R(P_i)} K_{\mathbb{H}^2}^{\nicefrac{1}{4}}\left( x,x_{2j}(\gamma_i);t \right)dx &=   \int\limits_{\bigcup_{j=1}^{k(i)}D_{\gamma_i}^{j-1}(W_R(P_i))} K_{\mathbb{H}^2}^{\nicefrac{1}{4}}\left( x,x_{E_i}(\gamma_i);t \right)dx \\
 & = 2\int\limits_{0}^{R}\int\limits_{0}^{\frac{\pi}{2}} K_{\mathbb{H}^2}^{\nicefrac{1}{4}}\left( (a,\alpha), (a, -\alpha);t \right) \sinh(a) d\alpha da,
\end{align*}
and
\begin{align*}
\tilde{Z}_{E}^{\nicefrac{1}{4}}(t)=-\sum\limits_{i=1}^{M}\Bigg\{ \int\limits_{\text{ }\Omega_{E_i}}K_{\mathbb{H}^2}^{\nicefrac{1}{4}}(x,x_{E_i};t)dx  +  2\int\limits_{0}^{R}\int\limits_{0}^{\frac{\pi}{2}} K_{\mathbb{H}^2}^{\nicefrac{1}{4}}\left( (a,\alpha), (a, -\alpha);t \right) \sinh(a) d\alpha da \Bigg\}.
\end{align*}
Therefore, the function $\tilde{Z}_{E}^{\nicefrac{1}{4}}$ is exactly the same function as $Z_{E}^{\nicefrac{1}{4}}$ from \eqref{equation:DefinitionOfZ_E(t)} of Section \ref{section:ExpansionTrace}. Thus the asymptotic expansion of $Z_{E}^{\nicefrac{1}{4}}(t)$ as $t\searrow 0$ is given by Corollary \ref{corollary:AsymptoticExpansionZ_E} $(i)$.
\end{proof}

It remains to determine the asymptotic expansion of $\tilde{Z}_V^{\nicefrac{1}{4}}(t)$. For all $i\in\lbrace 1,...,M \rbrace$ we define
\begin{align}
\tilde{I}_{\gamma_i}^{\nicefrac{1}{4}}(t):=\sum\limits_{j=1}^{k(i)-1}\int\limits_{W_R(P_i)} K_{\mathbb{H}^2}^{\nicefrac{1}{4}}(x,x_{2j+1}(\gamma_i);t)dx,
\end{align}
such that 
\begin{align}
\tilde{Z}_V^{\nicefrac{1}{4}}(t) = \sum\limits_{i=1}^{M} \tilde{I}_{\gamma_i}^{\nicefrac{1}{4}}(t).
\end{align}
The discussion so far shows that $\tilde{I}_{\gamma_i}^{\nicefrac{1}{4}}(t)$ will provide the contributions from the angle $\gamma_i$ to the heat invariants for the polygon $\Omega$. So let $i\in\lbrace 1,...,M \rbrace$ be fixed in the following discussion. In order to determine the asymptotic expansion of $\tilde{I}_{\gamma_i}^{\nicefrac{1}{4}}(t)$ we define, for each $j\in\lbrace  1,...,k(i)-1\rbrace$, the functions
\begin{align}
\label{equation:DefinitionI_{ij}}
I_{ij}(t):=\int\limits_{W_R(P_i)} K_{\mathbb{H}^2}^{\nicefrac{1}{4}}(x,x_{2j+1}(\gamma_i);t)dx.
\end{align}
Then
\begin{align}
\label{equation:DarstellungI_gammaSumme}
\tilde{I}_{\gamma_i}^{\nicefrac{1}{4}}(t) =  \sum\limits_{j=1}^{k(i)-1} I_{ij}(t).
\end{align}
Let us fix also the parameter $j\in\lbrace  1,...,k(i)-1\rbrace$ and first compute the asymptotic expansion of $I_{ij}(t)$ as $t\searrow 0$.

When we choose polar coordinates $(a ,\alpha)$ with base point $P_i$, the domain of integration in \eqref{equation:DefinitionI_{ij}} can be parametrised as
\begin{align}
W_R(P_i)=\Big\lbrace{ (a,\alpha) \mid 0\leq a\leq R,\, 0\leq \alpha\leq \frac{\pi}{k(i)} \Big\rbrace}.
\end{align}
Recall that the volume form in polar coordinates is given by $dx =\sinh(a)da d\alpha$ (see Section \ref{section:Green}) and the distance of any two points $(a,\alpha),\, (a,\alpha')\in\mathbb{H}^2$ with the same radius $a\in\left( 0,\infty \right)$ and angles $\alpha,\,\alpha'\in (0,2\pi]$ is (recall \eqref{equation:HyperbolicDistancePolarCoordinates}):
\begin{align*}
\cosh\left( d((a,\alpha), (a,\alpha')) \right) &= \cosh^2(a)-\sinh^2(a)\cos(\alpha'-\alpha)\\
&=\cosh^2(a)\cdot\left(1- \cos(\alpha'-\alpha)\right)+\cos(\alpha'-\alpha).
\end{align*}
Hence, by \eqref{equation:HyperHeat2} and the fact that the angles of $x$ and $x_{2j+1}(\gamma_i)$ differ by $2j\gamma_i$ for each $x\in W_{R}(P_i)$, we get the following formula for $I_{ij}(t)$ in polar coordinates: 
\begin{align}
\label{equation:FormelI_{ij}PolarCoordinates}
I_{ij}(t)&=\int\limits_{W_R(P_i)} K_{\mathbb{H}^2}^{\nicefrac{1}{4}}(x,x_{2j+1}(\gamma_i);t)dx \nonumber \\
&=  \frac{1}{4\sqrt{\pi} t^{\frac{3}{2}}} \frac{1}{k(i)} \int\limits_{0}^{R}  \int\limits_{\arcosh\left( \cosh^2(a)\left(1 -\cos\left( 2j\frac{\pi}{k(i)}\right)\right)+ \cos\left( 2j\frac{\pi}{k(i)}\right)\right)}^{\infty} ... \nonumber  \\
&  \qquad ... \frac{1}{\sqrt{2}}\frac{\sinh(a)\cdot \rho\cdot e^{-\frac{\rho^2}{4t}}}{\sqrt{\cosh(\rho)-\cosh^2(a)\left(1-\cos\left( 2j\frac{\pi}{k(i)} \right)\right)- \cos\left( 2j\frac{\pi}{k(i)}\right)}} \, d\rho \, da. 
\end{align}

Before we investigate $I_{ij}(t)$ further, let us discuss the following coefficients which will appear in the asymptotic expansion of $I_{ij}(t)$. For all $k, \, \eta \in\mathbb{N}_0$ we define the coefficient $d_{k}(\eta)$ via the generating function:
\begin{align}
\label{equation:DefinitionCoefficientsD_k(n)}
\sinh^k\left( r\right)=\sum\limits_{\eta=0}^{\infty} d_k \left( \eta \right)\cdot r^{\eta}\, \text{ for all }\, r\in\mathbb{R}.
\end{align}
Note that by definition $d_0(0)=1$, $d_0(\eta)=0$ for all $\eta\geq 1$ and $d_k(0)=0$ for all $k\geq 1$. For $k\geq 1,\, \eta\geq 1$ we can compute the coefficients $d_k(\eta)$ by the formulas
\begin{align}
d_{k}(\eta) &= \sum\limits_{\substack{\ell_1,...,\ell_{k}\in\mathbb{N}\\ \text{ are odd, s.t. } \\ \sum_{\tau=1}^{k}\ell_{\tau}=\eta}} \prod\limits_{j=1}^{k} \frac{1}{\ell_j !}  \label{equation:FormulasCoefficientsD_k(n)1}\\
&= \frac{1}{2^{k}\cdot \eta!}\sum\limits_{p=0}^{k}\binom{k}{p} (-1)^{p} \left(k-2p\right)^{\eta}. \label{equation:FormulasCoefficientsD_k(n)2}
\end{align}

The first equality \eqref{equation:FormulasCoefficientsD_k(n)1} is clear by $\sinh(r)=r+\frac{r^3}{3!} + \frac{r^5}{5!}+...$, and
formula \eqref{equation:FormulasCoefficientsD_k(n)2} is true since
\begin{align}
\label{equation:BeweisFormulasCoefficientsD_k(n)2}
\sinh^{k}(r)&=\left( \frac{e^{r}-e^{-r}}{2} \right)^{k}=\frac{1}{2^{k}}\sum\limits_{p=0}^{k}\binom{k}{p} (-1)^{p}e^{\left(k-2p\right)r}\nonumber \\
&=\frac{1}{2^{k}}\sum\limits_{p=0}^{k}\binom{k}{p} (-1)^{p} \sum\limits_{\eta=0}^{\infty} \frac{\left(k-2p\right)}{\eta!}^{\eta} r^{\eta} \nonumber\\
&= \sum\limits_{\eta=0}^{\infty}  \frac{1}{2^{k}\cdot \eta!}\sum\limits_{p=0}^{k}\binom{k}{p} (-1)^{p} \left(k-2p\right)^{\eta} r^{\eta}.
\end{align}
(Compare \eqref{equation:DefinitionCoefficientsD_k(n)} with \eqref{equation:BeweisFormulasCoefficientsD_k(n)2}.)

Moreover,
\begin{align}
\label{equation:CoefficientsD_kVanish}
d_k(\eta)=0 \, \text{ for all }\, k>\eta
\end{align}
because $r\mapsto \sinh^k(r)$ has a zero of order $k$ in $0$.

Let us determine the asymptotic expansion of the integral expression in \eqref{equation:FormelI_{ij}PolarCoordinates}. 
\begin{lemma}
\label{lemma:AsymptoticExpansionQWatson}
Let
\begin{align*}
\psi_{ij}(a):=\arcosh\left( \cosh^2(a)\left(1 -\cos\left( 2j\frac{\pi}{k(i)}\right)\right)+ \cos\left( 2j\frac{\pi}{k(i)}\right)\right)
\end{align*} 
and
\begin{align*}
q_{ij}(\rho, a):=\frac{1}{\sqrt{2}}\frac{\sinh(a)\cdot \rho}{\sqrt{\cosh(\rho)-\cosh^2(a)\left(1-\cos\left( 2j\frac{\pi}{k(i)} \right)\right)- \cos\left( 2j\frac{\pi}{k(i)}\right)}}
\end{align*}
for all $a\geq 0$ and all $\rho\in(\psi_{ij}(a) ,\infty)$.

Then
\begin{align}
\int\limits_{0}^{R}  \int\limits_{\psi_{ij}(a)}^{\infty} q_{ij}(\rho, a)\cdot e^{-\frac{\rho^2}{4t}} d\rho \, da\text{ } \overset{t\downarrow 0}{\sim} \text{ }\sum\limits_{\eta=0}^{\infty} \widetilde{\nu}_{\eta}(i,j)\cdot  t^{\eta+\frac{3}{2}},
\end{align}
where 
\begin{align}
\widetilde{\nu}_{\eta}(i,j):&=\frac{\Gamma\left( \eta+\frac{3}{2} \right)}{2\eta+1}\cdot\sum\limits_{\ell=0}^{\eta}c_{\ell}(i,j)\cdot d_{2\ell}(2\eta) = \frac{\sqrt{\pi} (2\eta)!}{2\cdot 4^{\eta} \eta!}\cdot\sum\limits_{\ell=0}^{\eta}c_{\ell}(i,j)\cdot d_{2\ell}(2\eta), \label{equation:CoefficientsNuIJ}\\
c_{\ell}(i,j):&= \sum\limits_{\tau=0}^{\ell}\binom{\frac{1}{2}}{\tau}(-1)^{\ell-\tau}\cdot \frac{1}{\sin\left( j\cdot\frac{\pi}{k(i)} \right)^{2(\ell-\tau)+2}}, \label{equation:CoefficientsCellIJ}
\end{align}
and $d_{2\ell}(2\eta)$ are the coefficients defined in \eqref{equation:DefinitionCoefficientsD_k(n)}.
\end{lemma}

\begin{proof}
The function $\psi_{ij}$ is continuous and strictly increasing on the interval $[0,\infty)$ as a composition of continuous and strictly increasing functions, and thus its inverse
\begin{align*}
\psi_{ij}^{-1}:[0,\infty)\ni \omega \mapsto \arcosh\left(\sqrt{  \frac{\cosh(\omega)-\cos\left( 2j\frac{\pi}{k(i)} \right)}{1-\cos\left( 2j\frac{\pi}{k(i)}\right)} }\right)\in\mathbb{R}
\end{align*}
is continuous and strictly increasing as well. Both functions $\psi_{ij},\, \psi_{ij}^{-1}$ are smooth on the domain $(0,\infty)$, since they are compositions of smooth functions. Hence, by the Fubini-Tonelli theorem and substituting $\rho$ by $2\sqrt{\rho}$, we get
\begin{align*}
\int\limits_{0}^{R}  \int\limits_{\psi_{ij}(a)}^{\infty}  q_{ij}(\rho, a)\cdot e^{-\frac{\rho^2}{4t}}d\rho \, da &= \int\limits_{0}^{\infty}  \int\limits_{0}^{\min\lbrace{R,\text{ } \psi_{ij}^{-1}(\rho) \rbrace}} q_{ij}(\rho,a) da \, \cdot e^{-\frac{\rho^2}{4t}}d\rho \\
&=\int\limits_{0}^{\infty}  \int\limits_{0}^{\min\lbrace{R,\text{ } \psi_{ij}^{-1}(2\sqrt{\rho})\rbrace}}\frac{q_{ij}(2\sqrt{\rho},a)}{\sqrt{\rho}} da \,\cdot e^{-\frac{\rho}{t}} d\rho.
\end{align*}
We would like to apply Watson's lemma; recall Lemma \ref{lemma:WatsonsLemma}. So we need to determine the asymptotic expansion of the function
\begin{align*}
Q: (0,\infty) \ni \rho\mapsto \int\limits_{0}^{\min\lbrace{R,\text{ } \psi_{ij}^{-1}(2\sqrt{\rho})\rbrace}}\frac{q_{ij}(2\sqrt{\rho},a)}{\sqrt{\rho}}da \in\mathbb{R}
\end{align*}
as $\rho\searrow 0$.
Since $\psi_{ij}^{-1}$ is continuous in $0$ with $\lim\limits_{\rho\searrow 0}\psi_{ij}^{-1}(2\sqrt{\rho})=0$, for sufficiently small values $\rho>0$, we have
\begin{align*}
Q(\rho) &= \int\limits_{0}^{\psi_{ij}^{-1}(2\sqrt{\rho})} \frac{q_{ij}(2\sqrt{\rho},a)}{\sqrt{\rho}} da = \int\limits_{0}^{\arcosh \sqrt{  \frac{\cosh(2\sqrt{\rho})-\cos\left( 2j\frac{\pi}{k(i)} \right)}{1-\cos\left( 2j\frac{\pi}{k(i)}\right)} }} ...\\
&\qquad\qquad ... \frac{\sqrt{2}\cdot \sinh(a)}{\sqrt{\cosh(2\sqrt{\rho})-\cosh^2(a)\cdot \left( 1-\cos\left( 2j\frac{\pi}{k(i)} \right) \right)-\cos\left( 2j\frac{\pi}{k(i)} \right)}}da \\
&=\sqrt{2}\int\limits_{1}^{\sqrt{  \frac{\cosh(2\sqrt{\rho})-\cos\left( 2j\frac{\pi}{k(i)} \right)}{1-\cos\left( 2j\frac{\pi}{k(i)}\right)} }} \frac{1}{\sqrt{ \cosh(2\sqrt{\rho})-\cos\left( 2j\frac{\pi}{k(i)} \right) -\left(1- \cos(2j\frac{\pi}{k(i)}) \right)a^2}} da.
\end{align*}
This integral is of the form
\begin{align*}
\int\limits_{1}^{\sqrt{\frac{x}{y}}}\frac{1}{\sqrt{x-y a^2}} da
\end{align*}
with $ x > y>0$. By an elementary substitution we get
\begin{align*}
\int\limits_{1}^{\sqrt{\frac{x}{y}}}\frac{1}{\sqrt{x-ya^2}} da  = \frac{1}{\sqrt{y}}\arccos\left( \sqrt{\frac{y}{x}} \right).
\end{align*}
Therefore,
\begin{align}
\label{equation:FunctionQSimplified}
Q(\rho)&= \frac{\sqrt{2}}{\sqrt{1- \cos(2j\frac{\pi}{k(i)})}} \cdot \arccos  \sqrt{\frac{1- \cos(2j\frac{\pi}{k(i)})}{\cosh(2\sqrt{\rho})-\cos(2j\frac{\pi}{k(i)})} }  \nonumber \\
&=\frac{1}{\sin\left( j\cdot \frac{\pi}{k(i)} \right)}\cdot \arccos  \frac{\sin\left( j\frac{\pi}{k(i)} \right)}{\sqrt{\cosh^2(\sqrt{\rho})-1+\sin^2\left( j\cdot \frac{\pi}{k(i)} \right)}} .
\end{align}
This function can be expanded into an asymptotic power series, but we will postpone the proof to the next lemma. Applying Lemma \ref{lemma:HilfslemmaWinkelSpiegelungsmethode} with $c:=\sin\left(j\cdot\frac{\pi}{k(i)}  \right)\in\left( 0,1 \right)$, we obtain
\begin{align*}
Q(\rho)\overset{\rho \downarrow 0}{\sim} \sum\limits_{\eta=0}^{\infty}  \xi_{2\eta+1}(i,j)\cdot \rho^{\left(\eta+\frac{3}{2}\right)-1},
\end{align*}
where 
\begin{align*}
\xi_{2\eta+1}\left(i,j\right) &:= \frac{1}{2\eta+1}\cdot \sum\limits_{k=0}^{\eta}c_k(i,j)\cdot d_{2k}(2\eta).
\end{align*}

Using Watson's lemma we get
\begin{align*}
\int\limits_{0}^{\infty} Q(\rho)\cdot e^{-\frac{\rho}{t}}d\rho\text{ }\overset{t\downarrow 0}{\sim} \sum\limits_{\eta=0}^{\infty} \underbrace{\xi_{2\eta+1}(i,j)\cdot \Gamma\left(\eta+\frac{3}{2} \right)}_{=\, \widetilde{\nu}_{\eta}(i,j)}\cdot t^{\eta+\frac{3}{2}}.
\end{align*}
Finally, recall from \eqref{equation:GammaWert} that $\Gamma\left( \eta+\frac{3}{2} \right)=\frac{\sqrt{\pi}\cdot \left(2\eta+2\right)!}{4^{\eta+1}\left(\eta+1\right)!}$ for all $\eta\in\mathbb{N}_{0}$.
\end{proof}

We will give in the following lemma the details for the asymptotic expansion of \eqref{equation:FunctionQSimplified}.

\begin{lemma}
\label{lemma:HilfslemmaWinkelSpiegelungsmethode}
Let $c\in (0,1]$ and consider the function 
\begin{align*}
U: [0,\infty) \ni r \mapsto \frac{1}{c}\cdot \arccos  \frac{c}{\sqrt{\cosh^2(\sqrt{r})-1+c^2}}  \in\mathbb{R}.
\end{align*}
 Then
\begin{align}
U(r) \overset{r\downarrow 0}{\sim}\sum\limits_{\eta=0}^{\infty} \xi_{2\eta+1}\cdot r^{\eta+\frac{1}{2}},
\end{align}
where 
\begin{align}
\xi_{2\eta+1}:&=\frac{1}{2\eta+1}\cdot \sum\limits_{k=0}^{\eta}c_k\cdot d_{2k}(2\eta)\\
\text{ with }\, c_k:&=  \sum\limits_{\tau=0}^{k}\binom{\frac{1}{2}}{\tau}(-1)^{k-\tau}\cdot \frac{1}{c^{2(k-\tau)+2}}.
\end{align}
\end{lemma}

\begin{proof}
Let us consider $u(r):=U(r^2)$.
Observe that the function $u(r)$ can be extended continuously by the same formula to all of $\mathbb{R}$, and this continuation is smooth on the domain $\mathbb{R}\backslash\lbrace 0\rbrace$, but is not continuously differentiable at the point $0$. But since we are interested only in the asymptotic expansion for $r\searrow 0$, we can choose a different continuation to obtain a smooth function on $\mathbb{R}$. Consider $\tilde{u}:\mathbb{R}\rightarrow\mathbb{R}$ defined by
\begin{align}
\tilde{u}(r):=\begin{cases}
u(r) &\text{, if $r\in [0,\infty)$},\\
-u(-r) &\text{, if $r\in (-\infty,0)$}.
\end{cases}
\end{align}
Then $\tilde{u}$ is a smooth odd function on $\mathbb{R}$. This follows for instance from the following facts: $u$ is smooth on the domain $(0,\infty)$, satisfies $u(0)=0$, and the limit of $\frac{d^m}{dr^m}u(r)$ as $r\searrow 0$ exists in $\mathbb{R}$ for all $m\in\mathbb{N}$ and is equal to $0$ if $m$ is an even number. The last assertion can be easily seen from \eqref{equation:BeweisAbleitungu(r)} below. Hence $\tilde{u}$ has a Taylor series around $0$ which formally equals the asymptotic series for $u(r)$ as $r\searrow 0$; i.e., there exist constants $\widetilde{\xi}_k:=\frac{1}{k!}\frac{d^k \tilde{u}}{dx^k}(0)$ such that
\begin{align*}
u(r) \overset{r\downarrow 0}{\sim}\sum\limits_{k=0}^{\infty} \widetilde{\xi}_k\cdot r^k.
\end{align*}
It remains to show that $\widetilde{\xi}_{2\eta}=0$ and $\widetilde{\xi}_{2\eta+1}=\xi_{2\eta+1}$ for all $\eta\in\mathbb{N}_0$. Obviously $\widetilde{\xi}_0=\tilde{u}(0)=0.$ To compute the other coefficients, we first differentiate $u(r)$ for $r>0$ and get
\begin{align}
\label{equation:BeweisAbleitungu(r)}
\frac{d u}{dr}(r)&=\frac{1}{c}\frac{-1}{\sqrt{1-\frac{c^2}{\cosh^2(r)-1+c^2}}}\cdot \left( -\frac{1}{2} \right)\cdot \frac{c\cdot 2\cdot \cosh(r)\cdot \sinh(r)}{\left( \cosh^2(r)-1+c^2 \right)^{\frac{3}{2}}} \nonumber \\
&= \frac{1}{\cosh^2(r)-1+c^2}\cdot \frac{\cosh(r)\sinh(r)}{\sqrt{\cosh^2(r)-1}} \nonumber \\
&=\frac{ \cosh(r)}{\cosh^2(r)-1+c^2}.
\end{align}
Thus $\frac{d \tilde{u}}{dr}(r)=\frac{\cosh(r)}{\cosh^2(r)-1+c^2}$ for all $r\in\mathbb{R}$, since $\frac{d \tilde{u}}{dr}$ is an even function. The Taylor series for this function can be derived as follows: For $r\in\mathbb{R}$ such that $\vert \sinh(r) \vert<\min\{1, c\}$, we have
\begin{align*}
\frac{\cosh(r)}{\cosh^2(r)-1+c^2}&=\sqrt{1+\sinh^2(r)}\cdot\frac{1}{\sinh^2(r)+c^2} \\
&=\frac{1}{c^2}\cdot \sqrt{1+\sinh^2(r)}\cdot\frac{1}{1+ \left(\frac{\sinh(r)}{c}\right)^2}  \\
&=\frac{1}{c^2} \sum\limits_{\ell=0}^{\infty} \binom{\frac{1}{2}}{\ell} \cdot \sinh^{2\ell}(r) \cdot \sum\limits_{m=0}^{\infty}(-1)^{m}\frac{\sinh(r)^{2m}}{c^{2m}}  \\
&=\frac{1}{c^2}\sum\limits_{k=0}^{\infty}  \sum\limits_{\tau=0}^{k}\binom{\frac{1}{2}}{\tau}(-1)^{k-\tau}\cdot \frac{1}{c^{2(k-\tau)}}  \cdot \sinh^{2k}(r) \\
&=\sum\limits_{k=0}^{\infty} c_k\cdot \sinh^{2k}(r). 
\end{align*}

The last expression becomes by \eqref{equation:DefinitionCoefficientsD_k(n)}, the fact that $d_{2k}(\eta)=0$ for odd $\eta$,  and \eqref{equation:CoefficientsD_kVanish}
\begin{align*}
\sum\limits_{k=0}^{\infty}  c_k\cdot \sum\limits_{\eta=0}^{\infty} d_{2k}(2\eta) r^{2\eta} =\sum\limits_{\eta=0}^{\infty}  \sum\limits_{k=0}^{\infty} c_{k}\cdot d_{2k}(2\eta) \cdot r^{2\eta} = \sum\limits_{\eta=0}^{\infty}  \sum\limits_{k=0}^{\eta} c_{k}\cdot d_{2k}(2\eta) \cdot r^{2\eta}.
\end{align*}
Since this was $\tilde{u}'(r)$ for $r\in\mathbb{R}$ with $\vert \sinh(r) \vert<\min\{1,c \}$ and the coefficients $\widetilde{\xi}_k$ are given as the coefficients of the Taylor series $T(\tilde{u})$ of $\tilde{u}$, we get
\begin{align*}
\sum\limits_{\eta=0}^{\infty}\widetilde{\xi}_{\eta} \cdot r^{\eta} =  T(\tilde{u})(r) = \sum\limits_{\eta=0}^{\infty}\frac{1}{2\eta+1}   \sum\limits_{k=0}^{\eta}c_k\cdot d_{2k}(2\eta) \cdot r^{2\eta+1},
\end{align*}
hence $\widetilde{\xi}_{2\eta+1}= \frac{1}{2\eta+1}\cdot \sum\limits_{k=0}^{\eta}c_k\cdot d_{2k}(2\eta)=\xi_{2\eta+1}$ for all $\eta\in\mathbb{N}_0$ and $\widetilde{\xi}_{2\eta}=0$ for all $\eta\in\mathbb{N}_0$.

\end{proof}

From \eqref{equation:FormelI_{ij}PolarCoordinates} and Lemma \ref{lemma:AsymptoticExpansionQWatson} we immediately get:

\begin{corollary}
\begin{align*}
I_{ij}(t)\overset{t\downarrow 0}{\sim} \sum\limits_{\eta=0}^{\infty} \nu_{\eta}(i,j) t^{\eta},
\end{align*}
where
\begin{align*}
 \nu_{\eta}(i,j)=\frac{1}{2k(i)}\frac{(2\eta)!}{4^{\eta+1}\eta !}  \sum\limits_{\ell=0}^{\eta} c_{\ell}(i,j)\cdot d_{2\ell}(2\eta)
\end{align*}
with $c_{\ell}(i,j)$ as in \eqref{equation:CoefficientsCellIJ}.
\end{corollary}

Recall that $\tilde{I}_{\gamma_i}^{\nicefrac{1}{4}}(t)=\sum_{j=1}^{k(i)-1}I_{ij}(t)$.

\begin{corollary}
 We have
\begin{align}
\label{equation:ContributionFromOneVerticeTrigSum}
\tilde{I}_{\gamma_i}^{\nicefrac{1}{4}}(t)\overset{t\downarrow 0}{\sim} \sum\limits_{\eta=0}^{\infty} c_{\eta}(\gamma_i) t^{\eta}, 
\end{align}
where
\begin{align*}
c_{\eta}(\gamma_i)&:=\frac{1}{2k(i)}\frac{(2\eta)!}{4^{\eta+1}\eta !}  \sum\limits_{\ell=0}^{\eta}  C_{\ell}(i) \cdot d_{2\ell}(2\eta),\\
C_{\ell}(i)&:=  \sum\limits_{j=1}^{k(i)-1} c_{\ell}(i,j)
\end{align*}
with $c_{\ell}(i,j)$ as in \eqref{equation:CoefficientsCellIJ}.
\end{corollary}

\begin{corollary}
\begin{align*}
\tilde{Z}_V^{\nicefrac{1}{4}}(t)\overset{t\downarrow 0}{\sim} \sum\limits_{\eta=0}^{\infty} c_{\eta} t^{\eta},
\end{align*}
where $c_{\eta}:=\sum\limits_{i=1}^{M}c_{\eta}(\gamma_i)$.
\end{corollary}

Recall that by Corollary \ref{corollary:HeatKernelsHalfSpaceTheSame}, $Z_V^{\nicefrac{1}{4}}$ and $\tilde{Z}_V^{\nicefrac{1}{4}}$ necessarily have the same asymptotic expansion. So we can compare \eqref{equation:ContributionFromOneVerticeTrigSum} with \eqref{equation:DefinitionC_k(Gamma)} and obtain the following theorem as an aside.

\begin{theorem}
\label{theorem:FiniteTrigonometricSums}
For all $k\in\mathbb{N}_{\geq 2}$ and $\eta\in\mathbb{N}_{0}$ we have the identities
\begin{align}
\label{equation:FiniteTrigonometricSums}
\sum\limits_{\ell=1}^{\eta+1}\binom{2\eta+2}{2\ell} B_{2\eta-2\ell+2}\left(\frac{1}{2}\right)\cdot B_{2\ell} \cdot  \left( k^{2\ell}-1\right) = \frac{(2\eta+2)!}{4^{\eta+1}}  \sum\limits_{\ell=0}^{\eta} C_{\ell,k} \cdot d_{2\ell}(2\eta),
\end{align}
where
\begin{align*}
C_{\ell,k}&:= \sum\limits_{\tau=0}^{\ell}\binom{\frac{1}{2}}{\tau}(-1)^{\ell-\tau}\cdot \sum\limits_{j=1}^{k-1} \frac{1}{\sin\left( j\cdot\frac{\pi}{k} \right)^{2(\ell-\tau)+2}} .
\end{align*}
\end{theorem}

\begin{example}
\label{example:TrigonometricSums}
Let $k\in\mathbb{N}_{\geq 2}$ be arbitrary. By \eqref{equation:FiniteTrigonometricSums}, we can recursively evaluate the finite trigonometric sums of the form
\begin{align*}
\sum_{j=1}^{k-1}\frac{1}{\sin^{2n}\left( j\cdot \frac{\pi}{k} \right)} \text{ for all } n\in\mathbb{N}.
\end{align*}
For example, equation \eqref{equation:FiniteTrigonometricSums} reduces to the following fomulas for $\eta=0,1,2:$
\item[$\eta=0:$]
\begin{align*}
\frac{1}{3}(k^2-1) = \sum\limits_{j=1}^{k-1}\frac{1}{\sin^{2}\left( j\cdot \frac{\pi}{k} \right)}.
\end{align*}
\item[$\eta=1:$]
\begin{align*}
\frac{1}{45}(k^4-1) + \frac{2}{9}(k^2-1) = \sum\limits_{j=1}^{k-1}\frac{1}{\sin^{4}\left( j\cdot \frac{\pi}{k} \right)}.
\end{align*}
\item[$\eta=2:$]
\begin{align*}
\frac{2}{945}(k^6-1) + \frac{1}{45}(k^4-1) + \frac{8}{45}(k^2-1)   = \sum\limits_{j=1}^{k-1}\frac{1}{\sin^{6}\left( j\cdot \frac{\pi}{k} \right)}.
\end{align*}
\end{example}

The last three formulas in Example \ref{example:TrigonometricSums} can also be found in \cite[p. $148$]{Chu} and \cite[p. $19/20$]{Yaep}. Both articles actually give explicit formulas for all $\sum_{j=1}^{k-1}\frac{1}{\sin^{2n}\left( j\cdot \frac{\pi}{k} \right)}$, but these have a different form from the ones that we get recursively: They do not involve $B_{2\ell}(\frac{1}{2})$, and the formulas from \cite{Chu} do not involve Bernoulli numbers at all. So Theorem \ref{theorem:FiniteTrigonometricSums} can be seen as yet another way of evaluating these trigonometric sums.

%% file: chapter04.tex
\chapter{Spectral invariants for orbisurfaces}
\label{chapter:orbisurface}

This chapter has two central goals. The first one is to compute all heat invariants for orbisurfaces of constant curvature, which is done in the first two sections. Our strategy is to investigate two examples of orbisurfaces thoroughly and explore the relation between their heat traces (see Section \ref{section:ExamplesOrbifolds}). In Section \ref{section:InvariantsOrbisurfaces} we use this relation to compute the heat invariants for general orbisurfaces of constant curvature. 

The second goal is to give an alternative explanation why an angle $\gamma = \frac{\pi}{k}$ ($k\in\mathbb{N}_{\geq 2}$) of a polygon of constant curvature contributes to the heat invariants as stated in Corollary \ref{corollary:HeatAsymptoticConstantCurvaturePolygon} (see formula \eqref{equation:CoefficientsVerticesHeatAsymptoticConstantCurvaturePolygon}). In particular, this will, in the case of an angle $\gamma = \frac{\pi}{k}$, give an alternative proof for the formula of the coefficients $c_{\ell}^{\mathbb{H}}(\gamma)$ from Theorem \ref{theorem:AsymptoticExpansionHeatTraceHyperbolPolygon}. Namely, by Theorem \ref{theorem:SphaerischenKoeffizientenOrbifolds} below, they can in this case be computed directly from Watson's formula for the $c_{\ell}^{\mathbb{S}}(\gamma)$ (see the Remark after Corollary \ref{corollary:EckenbeitragWinkelPolygon}).

\section{Two examples of orbifolds}
\label{section:ExamplesOrbifolds}

In this section, we want to study two examples of orbifolds and, in particular, how their heat traces relate to each other. Our treatment presupposes some familiarity with basic orbifold theory. Anything we presuppose in this respect can be found in \cite{Gordon08} or \cite{Gordon12}.

Whenever appropriate, we identify the orthogonal group $O(2)$ with the subgroup of $O(3)$ given by the image of the monomorphism 
\begin{align*}
i:O(2) \ni  A\mapsto \begin{pmatrix}
    A & \begin{matrix} 0 \\ 0 \end{matrix} \\
    \begin{matrix} 0 & 0 \end{matrix} & 1
\end{pmatrix} \in O(3).
\end{align*}
Moreover, let $M:=\mathbb{S}^2(r):=\{\, x\in\mathbb{R}^3 \mid \Vert x \Vert = r \,\}$ be the sphere of radius $r>0$ equipped with an $O(2)$-invariant Riemannian metric, i.e. $O(2)$ is contained in the isometry group of $M$.

Let $k\in\mathbb{N}_{\geq 2}$ be fixed. We now introduce the two orbifolds we are primarily interested in. We denote the cyclic group of order $k$ by $\mathbb{Z}_{k}$ and we think of it as a subset $\mathbb{Z}_{k}\subset O(2)$ generated by a rotation about the origin through the angle $\frac{2\pi}{k}$. We can describe the group explicitly as
\begin{align*}
\mathbb{Z}_{k}:=\left\lbrace \tilde{D}_{\ell}:=D_{\frac{2\pi}{k}\ell }:= 
\begin{pmatrix}
   \cos(\frac{2\pi}{k}\ell) &  -\sin(\frac{2\pi}{k}\ell)  \\
   \sin(\frac{2\pi}{k}\ell) & \text{ }\cos(\frac{2\pi}{k}\ell)
\end{pmatrix} \text{ }\Big\vert\text{ }\ell=0,...,k-1
\right\rbrace.
\end{align*} 
Each element $D_{\frac{2\pi}{k}\ell }$ with $\ell\in \{ 0,...,k-1 \}$ acts on the Euclidean plane as a rotation about the origin through the angle ${\frac{2\pi}{k}\ell }$. Thus the group $\mathbb{Z}_{k}$ acts effectively on $M$, so $M/\mathbb{Z}_{k}$ becomes a good orbifold. Its only singularities are two cone points of order $\vert \mathbb{Z}_{k} \vert = k$.

Next, consider the dihedral group $\mathbb{D}_{k}\subset O(2)$, generated by the reflection in the $x$-axis and a rotation about the origin through the angle $\frac{2\pi}{k}$. Obviously we have the inclusion $\mathbb{Z}_{k}\subset \mathbb{D}_{k}$ and we can likewise describe the dihedral group explicitly by
\begin{align*}
\mathbb{D}_{k}:= \mathbb{Z}_{k}\cup \left\lbrace \tilde{S}_{\ell}:=S_{\frac{\pi}{k}\ell }:= 
\begin{pmatrix}
   \cos(\frac{2\pi}{k}\ell) & \text{ } \sin(\frac{2\pi}{k}\ell)  \\
   \sin(\frac{2\pi}{k}\ell) &  -\cos(\frac{2\pi}{k}\ell)
\end{pmatrix} \text{ }\Big\vert\text{ }\ell=0,...,k-1
\right\rbrace.
\end{align*} 
Any element $S_{\frac{\pi}{k}\ell }$ with $\ell\in \{ 0,..., k-1 \}$ acts on the Euclidean plane as a reflection in a line forming an angle of $\frac{\pi}{k}\ell$ with the $x$-axis. Thus the dihedral group acts effectively on $M = \mathbb{S}^2(r)$ as well, so $M/\mathbb{D}_{k}$ becomes a good orbifold. The singular points consist of two dihedral points with isotropy order $\vert \mathbb{D}_{k} \vert = 2k$ and of mirror points which form two connected reflector edges.

Since we are concerned mainly with those two orbifolds, it is appropriate for our purposes to replace some of the more general definitions given in \cite{Gordon08}, \cite{Gordon12} by simpler ones. It can be easily shown that for our orbifolds the following definitions are equivalent to those given in \cite{Gordon08}, \cite{Gordon12}.

Let $G\in\lbrace \mathbb{Z}_{k}, \mathbb{D}_{k} \rbrace$ be one of the two groups introduced above and let $\pi_G:M\rightarrow M/G$ be the canonical projection. First we want to stress that the group $G$ acts by isometries on $M$, and therefore the orbifold $M/G$ inherits from $M$ a Riemannian metric. If the metric of $M$ is given by the standard metric, which is induced by Euclidean $\mathbb{R}^3$, then $M/G$ is of constant curvature $\kappa=\frac{1}{r^2}$. Orbifolds with such a metric are sometimes called \emph{spherical} orbifolds.

\begin{definition}
\label{definition:SmoothFunctions}
Let $n \in\mathbb{N}_0\cup \lbrace \infty \rbrace$. A function $\tilde{f}:M/G\rightarrow \mathbb{R}$ is called a \emph{$C^{n}$-function} or \emph{of class $C^{n}$} if and only if $\tilde{f} \circ \pi_G :M\rightarrow \mathbb{R}$ is a $C^{n}$-function. We denote the set of all $C^{n}$-functions by $C^{n}\left( M/G \right)$ and set $C(M/G):=C^{0}(M/G)$.

Let $C^{n}_G (M):= \{\, f\in C^{n}(M) \mid \gamma^{\ast}f:=f\circ\gamma =f\, \text{ for all } \gamma\in G \,\}$ be the set of all $G$-invariant $C^{n}$-functions on $M$.
\end{definition}

We point out some simple relations between the function spaces of Definition \ref{definition:SmoothFunctions}. Let  $n \in\mathbb{N}_0\cup \lbrace \infty \rbrace$ be fixed until Definition \ref{definition:Laplacian} below. It is easy to show that  the map 
\begin{align*}
\Phi:C^{n}(M/G)\rightarrow C^{n}_G (M), \quad \tilde{f}\mapsto \tilde{f}\circ \pi_G =:f
\end{align*}
is a vector space isomorphism, i.e. it is bijective and linear. Therefore we can identify the set of all $C^{n}$-functions on the orbifold $M/G$ with the set of all $G$-invariant $C^{n}$-functions on $M$. In particular, for any $G$-invariant $C^{n}$-function $f:M\rightarrow \mathbb{R}$ there exists a unique $C^{n}$-function $\tilde{f}:M/G\rightarrow \mathbb{R}$ with the property $\tilde{f}\circ \pi_G = f$.

Obviously we have the relation $C^{\infty}(M/G)\subset C^{\ell}(M/G)\subset C^{\ell-1}(M/G)$ for all $\ell\in \mathbb{N}$. We want to remark that for our purposes the most important function spaces are those consisting of continuous functions and smooth functions, respectively.

\begin{definition}
\label{definition:Integration}
For any continuous function $\tilde{f}\in C(M/G)$ and any open set $\tilde{U}\subset M/G$, we define integration as follows:
\begin{align*}
\int\limits_{\tilde{U}} \tilde{f} :=\int\limits_{\tilde{U}} \tilde{f}(\tilde{x})d\tilde{x}:=\frac{1}{\vert G \vert}\int\limits_{\pi_G^{-1}(\tilde{U})} \left( \tilde{f}\circ \pi_G\right) (x)dx,
\end{align*}
where $\vert G \vert$ denotes the order of the group. Analogously to the $L^2$-inner product on $C(M)$, we can introduce the following inner product on $C(M/G)$:
\begin{align*}
\langle \tilde{f}, \tilde{g} \rangle:&=\int\limits_{M/G} \tilde{f}\cdot \tilde{g} =\frac{1}{\vert G \vert}\int\limits_{M}\tilde{f}(\pi_G(x))\cdot \tilde{g}(\pi_G(x)) dx  \\
&=\frac{1}{\vert G \vert}\langle \Phi(\tilde{f}),\Phi(\tilde{g}) \rangle_{L^2(M)}
\end{align*}
for all $\tilde{f},\tilde{g}\in C(M/G)$.
\end{definition}

The space $L^2(M/G)$ is defined as the completion of $\left( C(M/G),\langle \cdot,\cdot \rangle\right )$. If in the sequel an inner product appears in connection with functions, it will either refer to the inner product of $L^2(M/G)$ or $L^2(M)$, depending on the context.

Notice that the isomorphism $\Phi$ preserves orthogonality of functions. Moreover, the map 
\begin{align}
\label{equation:IsometryPHI}
\frac{1}{\sqrt{\vert G \vert}}\cdot \Phi  :C^{n}(M/G)\rightarrow C^{n}_G(M), \quad \tilde{f}\mapsto \frac{1}{\sqrt{\vert G \vert}}\cdot \Phi(\tilde{f})
\end{align}
is an isometry between the spaces $C^{n}(M/G)$ and $C^{n}_G(M)$. In particular, for any orthonormal set $\lbrace u_i \rbrace_{i\in I}\subset C^{n}_G (M)$ of $G$-invariant functions the set $\lbrace \sqrt{\vert G \vert}\cdot \tilde{u}_i \rbrace_{i\in I} \subset C^{n}(M/G)$ is an orthonormal set as well, where $\tilde{u}_i$ is, as usual, the unique map on $M/G$ such that $\tilde{u}_i\circ \pi_G =u_i$.

\begin{definition}
\label{definition:Laplacian}
Let $\Delta_M$ denote the Dirichlet Laplacian of $M$. The \emph{Laplacian} of $M/G$ is the linear operator 
\begin{align*}
\Delta_{M/G}:C^{\infty}(M/G)\rightarrow C^{\infty}(M/G), \quad \tilde{f}\mapsto \Delta_{M/G}(\tilde{f}),
\end{align*}
where for any $\tilde{f}\in C^{\infty}(M/G)$ the function $\Delta_{M/G}(\tilde{f})$ is defined as the unique smooth function such that $\Delta_{M/G}(\tilde{f}) \circ\pi_G = \Delta_M(\tilde{f}\circ \pi_G)$.

Note that the smooth function $\Delta_M(\tilde{f}\circ \pi_G)$ is $G$-invariant since $G$ acts by isometries on $M$. Another way to write $\Delta_{M/G}$ is given by 
\begin{align*}
\Delta_{M/G}=\Phi^{-1}\circ \Delta_M\circ \Phi.
\end{align*}
\end{definition}

The following theorem is well-known.

\begin{theorem} \emph{(see \citep[Proposition 3.1]{Gordon08})}
\label{theorem:SpectrumLaplacianOrbifolds}
The Laplacian of $M/G$ has a discrete spectrum $0 = \lambda_1\leq\lambda_2\leq...$ with $\lambda_j\rightarrow \infty$ as $j\rightarrow\infty$ and with each eigenvalue having finite multiplicity. The normalised eigenfunctions are smooth and form an orthonormal basis of $L^2(M/G).$
\end{theorem}

Let $\sigma(M/G):=\lbrace \lambda_1,\lambda_2,... \rbrace$ denote the set of all eigenvalues of $\Delta_{M/G}$ (without multiplicities) and let 
\begin{align*}
E_{\lambda}(M/G):=\{\, \tilde{f}\in C^{\infty}(M/G)\mid \Delta_{M/G} \tilde{f}=\lambda \tilde{f} \,\}
\end{align*}
 be the eigenspace corresponding to an eigenvalue $\lambda\in \sigma(M/G)$. The sets $\sigma(M), E_{\lambda}(M)$ are defined correspondingly. Further, let 
\begin{align*}
E_{\lambda}(M)_G:=E_{\lambda}(M)\cap C^{\infty}_G (M)= \{\,  f\in E_{\lambda}(M) \mid \gamma^{\ast} f = f\, \text{ for all }\, \gamma\in G \,\}
\end{align*}
be the set of all $G$-invariant eigenfunctions for any eigenvalue $\lambda\in \sigma(M)$.

By elementary observations, one can recognise a close relationship between the eigenvalues and eigenfunction of $\Delta_{M/G}$ and $\Delta_{M}$: Suppose $\tilde{f}\in$ $E_{\lambda}(M/G)$, i.e. $\Delta_{M/G}(\tilde{f})=\lambda \tilde{f}$. Then $\lambda$ is also an eigenvalue of $\Delta_{M}$ and $f:=\tilde{f}\circ \pi_G$ is a corresponding $G$-invariant eigenfunction on $M$. Conversely, suppose $f\in E_{\lambda}(M)_G$ is a $G$-invariant eigenfunction. Then $\lambda$ will be an eigenvalue of $\Delta_{M/G}$ and there exists a unique eigenfunction $\tilde{f}\in$ $E_{\lambda}(M/G)$ such that $\tilde{f}\circ \pi_G = f$. In short,
\begin{align}
\label{equation:IsomorphismEigenspaces}
\Phi_{\vert E_{\lambda}(M/G)}: E_{\lambda}(M/G)\rightarrow E_{\lambda}(M)_G
\end{align}
is a linear isomorphism. In analogy to \eqref{equation:IsometryPHI}, we have the isometry
\begin{align}
\label{equation:IsometryEigenspaces}
\frac{1}{\sqrt{\vert G \vert}}\cdot \Phi_{\vert E_{\lambda}(M/G)}: E_{\lambda}(M/G)\rightarrow E_{\lambda}(M)_G.
\end{align}

Recall that $\tilde{S}_0\in \mathbb{D}_{k}$ is a reflection in Euclidean $\mathbb{R}^3$. The (Riemannian) isometry $\tilde{S}_0:M\rightarrow M$ induces the linear and symmetric map 
\begin{align*}
\tilde{S}_0^{\ast}:\left( C^{\infty}(M),\langle \cdot, \cdot \rangle_{L^2(M)}\right) \ni f\mapsto f\circ \tilde{S}_0 \in  \left( C^{\infty}(M),\langle \cdot, \cdot\rangle_{L^2(M)}\right).
\end{align*}
This map is an involution, i.e. it satisfies $\tilde{S}_0^{\ast}\circ \tilde{S}_0^{\ast}=\Id$. Furthermore, the spaces $C^{\infty}_G (M)$ and $E_{\lambda}(M)_G$ are invariant under  $\tilde{S}_0^{\ast}$ (see below). Thus one can replace the linear space $C^{\infty}(M)$ by its linear subspaces $C^{\infty}_G (M)$ and $E_{\lambda}(M)_G$ respectively, without losing any of the above properties of $\tilde{S}_0^{\ast}$. More precisely, the maps
\begin{align*}
{\tilde{S}_0^{\ast}} : C^{\infty}_G (M) \rightarrow C^{\infty}_G(M),\quad f\mapsto f\circ \tilde{S}_0
\end{align*}
and
\begin{align*}
\tilde{S}_0^{\ast}:E_{\lambda}(M)_G \rightarrow E_{\lambda}(M)_G,\quad f\mapsto f\circ \tilde{S}_0
\end{align*}
are linear symmetric involutions. The spaces $C^{\infty}_G (M)$ are indeed invariant under $\tilde{S}_0^{\ast}$: If $G=\mathbb{D}_{k}$, then $\tilde{S}_0^{\ast}$ restricted to $C^{\infty}_{\mathbb{D}_{k}}(M)$ is just the identity map. If $G=\mathbb{Z}_{k}$, then one can make use of the identity $\tilde{S}_0\circ\tilde{D}_{\ell} = \tilde{D}_{k-\ell}\circ \tilde{S}_0$ for any $\ell\in\lbrace 0,...,k-1 \rbrace$. Thus for any $f\in  C^{\infty}_{\mathbb{Z}_{k}} (M)$ and $\ell\in\lbrace 0,...,k-1 \rbrace$ we have
 \begin{align*}
 \left( f\circ\tilde{S}_0 \right) \circ \tilde{D}_{\ell} =  f\circ \tilde{D}_{k-\ell}\circ\tilde{S}_0 = f\circ\tilde{S}_0.
 \end{align*}

Furthermore, for any G-invariant eigenfunction $f\in E_{\lambda}(M)_G$ we indeed have $f\circ \tilde{S}_0 \in E_{\lambda}(M)_G$: On the one hand, $f\circ \tilde{S}_0$ is G-invariant by the discussion above and on the other hand $f\circ \tilde{S}_0$ is an eigenfunction since $\tilde{S}_0$ is an isometry.
 
With all the basic observations described above, we are now ready to define and relate the heat traces for our orbifolds.

\begin{lemma}
\label{lemma:AntiInvariantEigenfunctions}
Let $\lambda\in \sigma(M/\mathbb{Z}_{k})$ be an eigenvalue. The \emph{(}finite dimensional\emph{)} inner product space $E_{\lambda}(M)_{\mathbb{Z}_{k}}=C^{\infty}_{\mathbb{Z}_{k}} (M)\cap E_{\lambda}(M)$ has an orthonormal basis consisting of $\tilde{S}_0$-invariant and $\tilde{S}_0$-anti-invariant functions. More precisely, there exists an orthonormal basis $\lbrace u_i^{\lambda}\rbrace_{i=1}^{n_{\lambda}}$ with $n_{\lambda}\in\mathbb{N}$ such that any function $u_i^{\lambda}$ is either $\tilde{S}_0$-invariant, i.e. $\tilde{S}_0^{\ast}u_i^{\lambda} =  u_i^{\lambda}\circ \tilde{S}_0=u_i^{\lambda}$, or $\tilde{S}_0$-anti-invariant, i.e. $\tilde{S}_0^{\ast}u_i^{\lambda}=u_i^{\lambda}\circ \tilde{S}_0=-u_i^{\lambda}$.  
\end{lemma}

\begin{proof}
Because of the map given in \eqref{equation:IsomorphismEigenspaces} the vector space $E_{\lambda}(M)_{\mathbb{Z}_{k}}$ is isomorphic to $E_{\lambda}(M/\mathbb{Z}_{k})$, which in turn is finite dimensional by Theorem \ref{theorem:SpectrumLaplacianOrbifolds}. Hence the space $E_{\lambda}(M)_{\mathbb{Z}_{k}}$ must be a finite dimensional vector space for any $\lambda\in \sigma(M/\mathbb{Z}_{k})$.

The map $\tilde{S}_0^{\ast}:E_{\lambda}(M)_{\mathbb{Z}_{k}}\rightarrow E_{\lambda}(M)_{\mathbb{Z}_{k}}$ is symmetric and linear, and thus by the spectral theorem there exists an orthonormal basis of eigenfunctions of $\tilde{S}_0^{\ast}$ for $E_{\lambda}(M)_{\mathbb{Z}_{k}}$. But $\tilde{S}_{0}^{\ast}$ is an involution, i.e. the only eigenvalues are $1, -1$. The eigenfunctions corresponding to the eigenvalue $1$ are the $\tilde{S}_0$-invariant functions, and the eigenfunctions corresponding to the eigenvalue $-1$ are the $\tilde{S}_0$-anti-invariant functions. 
\end{proof}

By using the isometry \eqref{equation:IsometryEigenspaces} with $G=\mathbb{Z}_{k}$ and $\vert \mathbb{Z}_{k} \vert=k$, we get the following corollary:

\begin{corollary}
\label{corollary:OrthonormalBasisEigenfunctions}

For any eigenvalue $\lambda\in \sigma(M/\mathbb{Z}_{k})$ let $\lbrace u_i^{\lambda}\rbrace_{i=1}^{n_{\lambda}}$ be an orthonormal basis of $E_{\lambda}(M)_{\mathbb{Z}_{k}}$ like in the previous lemma.
\begin{itemize}
\item[$(i)$]  Then the set $\lbrace \sqrt{k}\cdot \tilde{u}_i^{\lambda}\rbrace_{i=1}^{n_{\lambda}}$ is an orthonormal basis for $E_{\lambda}(M/\mathbb{Z}_{k})$, where the functions $\tilde{u}_i^{\lambda}$ with $i=1,...,n_{\lambda}$ are defined by the relation $\tilde{u}_i^{\lambda}\circ \pi_{\mathbb{Z}_{k}} = u_i^{\lambda}$.
\item[$(ii)$] Furthermore, the collection 
\begin{align}
\label{equation:OrthonormalBasisCorollar}
\bigcup\limits_{\lambda\in\sigma(M/\mathbb{Z}_{k})}\lbrace \sqrt{k}\cdot\tilde{u}_i^{\lambda}\rbrace_{i=1}^{n_{\lambda}}
\end{align}
is an orthonormal basis for $L^2(M/\mathbb{Z}_{k})$ consisting of smooth eigenfunctions of $\Delta_{M/\mathbb{Z}_{k}}$.
\end{itemize}
\end{corollary}

Consider the collection of all bases $\lbrace u_i^{\lambda} \rbrace_{i=1}^{n_{\lambda}}$ described in Lemma \ref{lemma:AntiInvariantEigenfunctions} over all $\lambda\in \sigma(M/\mathbb{Z}_{k})$, i.e.
\begin{align*}
\bigcup_{\lambda\in \sigma(M/\mathbb{Z}_{k})}\lbrace u_i^{\lambda} \rbrace_{i=1}^{n_{\lambda}}.
\end{align*}
Any element $u_i^{\lambda}$ of this collection is a smooth $\mathbb{Z}_{k}$-invariant eigenfunction on $M$ which happens to be either $\tilde{S}_0$-invariant or $\tilde{S}_0$-anti-invariant. We reorder and rename this collection in order to obtain a better notation for the sequel. Let $\lbrace \varphi_i \rbrace_{i\in I}$, $I\subset\mathbb{N}$, be all those functions in this collection which are $\tilde{S}_0$-invariant, and let $\tau_i\in \sigma(M/\mathbb{Z}_{k})$ be such that $\Delta_M(\varphi_i)=\tau_i\cdot \varphi_i$. Similarly, let $\lbrace \psi_j \rbrace_{j\in J}$, $J\subset\mathbb{N}$, be all those functions in the collection which are $\tilde{S}_0$-anti-invariant, and let $\mu_j\in \sigma(M/\mathbb{Z}_{k})$ be such that $\Delta_M(\psi_j)=\mu_j\cdot \psi_j$.

As usual, let $\tilde{\psi}_j\in C^{\infty}(M/\mathbb{Z}_{k})$, $j\in J$, be the unique smooth function on $M/\mathbb{Z}_{k}$ such that $\tilde{\psi}_j\circ \pi_{\mathbb{Z}_{k}}=\psi_j$.
Each of the functions $\varphi_i :M\rightarrow \mathbb{R}$ is $\mathbb{Z}_{k}$-invariant and $\tilde{S}_0$-invariant. Thus it is $\mathbb{D}_{k}$-invariant as well. As any $\mathbb{D}_{k}$-invariant smooth function on $M$, it induces unique smooth functions $\tilde{\varphi}_i\in C^{\infty}( M/\mathbb{Z}_{k})$ and $\hat{\varphi}_i \in C^{\infty} (M/\mathbb{D}_{k})$ such that $\tilde{\varphi}_i \circ \pi_{\mathbb{Z}_{k}}=\varphi_i$ and $\hat{\varphi}_i\circ \pi_{\mathbb{D}_{k}}=\varphi_i$.

\begin{proposition}
\label{proposition:OrthonormalBasisEigenfunctionsQuadratintegrierbar}
We have the following orthonormal basis\emph{:}
\begin{itemize}
\item[$(i)$] The system $ \lbrace \sqrt{k}\cdot \tilde{\varphi}_i \rbrace_{i\in I} \cup \lbrace \sqrt{k}\cdot  \tilde{\psi}_i \rbrace_{j\in J}$ is an orthonormal basis for $L^2(M/\mathbb{Z}_{k})$ consisting of smooth eigenfunctions of $\Delta_{M/\mathbb{Z}_{k}}$.
\item[$(ii)$] The set of functions $\lbrace \sqrt{2k}\cdot  \hat{\varphi}_{i} \rbrace_{i\in I}$ is an orthonormal basis for $L^2(M/\mathbb{D}_{k})$ consisting of smooth eigenfunctions of $\Delta_{M/\mathbb{D}_{k}}$.
\end{itemize}
\end{proposition}

\begin{proof}
The first statement is obvious, since by definition 
\begin{align*}
\lbrace \sqrt{k}\cdot  \tilde{\varphi}_i \rbrace_{i\in I} \cup \lbrace \sqrt{k}\cdot \tilde{\psi}_j \rbrace_{j\in J}=\bigcup\limits_{\lambda\in\sigma(M/\mathbb{Z}_{k})}\lbrace \sqrt{k}\cdot\tilde{u}_i^{\lambda}\rbrace_{i=1}^{n_{\lambda}}.
\end{align*}
Hence by Corollary \ref{corollary:OrthonormalBasisEigenfunctions}, this collection is an orthonormal basis for $L^2(M/\mathbb{Z}_{k})$ as claimed.

The second statement follows similarly by using the isometry \eqref{equation:IsometryEigenspaces} with $G=\mathbb{D}_{k}$ and $\vert \mathbb{D}_{k} \vert = 2k$.
\end{proof}

The following theorem is a special case of a well-known, more general theorem (see e.g. \cite{Gordon08}, \cite{Gordon12} or Theorem $4.5$ in \cite{Donnelly}).

\begin{theorem}
\label{theorem:HeatKernelHeatTraceOrbifolds}
Let $G\in\lbrace \mathbb{Z}_{k}, \mathbb{D}_{k} \rbrace$ and let $0=\lambda_1\leq \lambda_2\leq ... \nearrow \infty$ be the sequence of eigenvalues \emph{(}with multiplicities\emph{)} of the Laplacian $\Delta_{M/G}$ as in Theorem \emph{\ref{theorem:SpectrumLaplacianOrbifolds}} above. Furthermore, let $\lbrace \phi_i \rbrace_{i=1}^{\infty}$ be an orthonormal basis for $L^2(M/G)$ consisting of smooth eigenfunctions of the Laplacian such that $\Delta_{M/G} \phi_i=\lambda_i \cdot \phi_i$ for all $i\in\mathbb{N}$. Then the heat kernel of $M/G$ exists and is given by
\begin{align*}
K_{M/G}(\tilde{x},\tilde{y};t)=\sum\limits_{i=1}^{\infty}e^{-\lambda_i t}\phi_i (\tilde{x})\phi_i(\tilde{y})\, \text{ for all } \tilde{x},\tilde{y}\in M/G \text{ and }\,t>0.
\end{align*}

The heat trace can be written as
\begin{align*}
Z_{M/G}(t):=\sum\limits_{i=1}^{\infty}e^{-\lambda_i t} = \int\limits_{M/G} K_{M/G}(\tilde{x},\tilde{x};t)d\tilde{x}\, \text{ for } t>0.
\end{align*}
\end{theorem}

Because of the above theorem and Proposition \ref{proposition:OrthonormalBasisEigenfunctionsQuadratintegrierbar} we get:
\begin{align*}
K_{M/\mathbb{Z}_{k}}(\tilde{x},\tilde{x};t)&=\sum\limits_i e^{-\tau_i t}\left( \sqrt{k}\cdot \tilde{\varphi}_i(\tilde{x})\right)^2 + \sum\limits_j e^{-\mu_j t}\left(\sqrt{k}\cdot \tilde{\psi}_j(\tilde{x})\right)^2\\
&= k\cdot  \sum\limits_i e^{-\tau_i t} \tilde{\varphi}_i(\tilde{x})^2 + k\cdot \sum\limits_j e^{-\mu_j t} \tilde{\psi}_j (\tilde{x})^2 
\end{align*}
for all $ \tilde{x}\in M/\mathbb{Z}_{k}$ and $t>0,$ and
\begin{align*}
K_{M/\mathbb{D}_{k}}(\hat{x},\hat{x};t)&=\sum\limits_i e^{-\tau_i t}\left( \sqrt{2k}\cdot \hat{\varphi}_i(\hat{x})\right)^2 = 2k\cdot \sum\limits_i e^{-\tau_i t}\hat{\varphi}_i(\hat{x})^2
\end{align*}
for all $\hat{x}\in M/\mathbb{D}_{k}$ and $t>0.$
Thus, for all $t>0$ and any open set $U\subset M$, which is invariant under $O(2)$, we have:
\begin{align*}
\int_{\proj{Z}{U}}K_{M/\mathbb{Z}_{k}}\left( \tilde{x},\tilde{x};t\right)d\tilde{x}&=\frac{1}{\vert \mathbb{Z}_{k} \vert}\int_{U}K_{M/\mathbb{Z}_{k}}\left( \pi_{\mathbb{Z}_{k}}(x),\pi_{\mathbb{Z}_{k}}(x);t\right)dx \\
&= \int_U  \sum\limits_i e^{-\tau_i t}\left( \tilde{\varphi}_i(\pi_{\mathbb{Z}_{k}}(x))\right)^2dx + \int_U  \sum\limits_j e^{-\mu_j t} \left( \tilde{\psi}_j(\pi_{\mathbb{Z}_{k}}(x))\right)^2dx\\
&= \int_U  \sum\limits_i e^{-\tau_i t} \varphi_i(x)^2dx + \int_U  \sum\limits_j e^{-\mu_j t}\psi_j(x)^2dx,\\
\int_{\proj{D}{U}}K_{M/\mathbb{D}_{k}}\left( \hat{x},\hat{x};t\right)d\hat{x}&=\frac{1}{\vert \mathbb{D}_{k} \vert}\int_{U}K_{M/\mathbb{D}_{k}}\left( \pi_{\mathbb{D}_{k}}(x),\pi_{\mathbb{D}_{k}}(x);t\right)dx \\
&= \int_U  \sum\limits_i e^{-\tau_i t}\left( \hat{\varphi}_i(\pi_{\mathbb{D}_{k}}(x))\right)^2 dx\\
&=\int_U  \sum\limits_i e^{-\tau_i t}\varphi_i(x)^2 dx.
\end{align*}

Hence,

\begin{align}
\label{equation:HeatTraceRelation1}
\int_{\proj{Z}{U}}K_{M/\mathbb{Z}_{k}}\left( \tilde{x},\tilde{x};t\right)d\tilde{x} = \int_{\proj{D}{U}}K_{M/\mathbb{D}_{k}}\left( \hat{x},\hat{x};t\right)d\hat{x} + \underbrace{\int\limits_U  \sum\limits_j e^{-\mu_j t}\psi_j(x)^2 dx}_{=:\, G(t)}.
\end{align}

Next, we want to investigate the function $t\mapsto G(t)$. We will show that this function is closely related to the Dirichlet heat kernel for a spherical lune with angle $\frac{\pi}{k}$.

For $i\in\{0,...,k-1\}$, let $\text{Fix}(\tilde{S}_i)\subset M$ be the fixed point set of the reflection $\tilde{S}_i\in \mathbb{D}_{k}$. Let $\Omega \subset M$ be a spherical lune or $2$-gon with vertices being the north pole and south pole, angle $\frac{\pi}{k}$ and with $\partial \Omega\subset \text{Fix}(\tilde{S}_0)\cup\text{Fix}(\tilde{S}_1)$. Using spherical coordinates
\begin{align*}
\xi:(0,2\pi)\times \left( -\frac{\pi}{2}, \frac{\pi}{2}\right) \ni (\theta,\nu)\mapsto \begin{pmatrix}
r\cos(\nu)\cos(\theta)\\
r\cos(\nu)\sin(\theta)\\
r\sin(\nu)
\end{pmatrix}\in M,
\end{align*}
the set $\Omega$ can be chosen as
\begin{align*}
\Omega:=\left\lbrace\, (\theta,\nu) \,\, \big\vert \,\, 0<\theta<\frac{\pi}{k};\, -\frac{\pi}{2}<\nu<\frac{\pi}{2} \,\right\rbrace.
\end{align*}
The boundary can be decomposed into two sides $\gamma_0:=\text{Fix}(\tilde{S}_0)\cap \partial\Omega$ and $\gamma_1:=\text{Fix}(\tilde{S}_1)\cap \partial\Omega$.

Since $\Omega\subset M$ is a relatively compact domain, the spectrum of the Dirichlet Laplacian $\Delta_{\Omega}$ consists of a sequence of non-negative eigenvalues and any eigenvalue has finite multiplicity (see Proposition \ref{proposition:heat trace bounded domain}). Let us enumerate the eigenvalues of $\Delta_{\Omega}$ with multiplicities as follows:
\begin{align*}
 0<\nu_1 < \nu_2\leq \nu_3\leq...\nearrow \infty.
\end{align*}  
Moreover, there exists an orthonormal basis $\lbrace \omega_{j} \rbrace_{j=1}^{\infty}$ for $L^2(\Omega)$, such that for any $j\in\mathbb{N}$ the function $\omega_{j}$ is a smooth eigenfunction of $\Delta_{\Omega}$ corresponding to the eigenvalue $\nu_j$.

The Dirichlet heat kernel for $\Omega$ can be written as 
\begin{align*}
K_{\Omega}(x,y;t)=\sum\limits_{j=1}^{\infty} e^{-\nu_j t} \omega_j(x)\omega_j(y)\, \text{ for all } x,y\in\Omega \text{ and }t>0,
\end{align*}
and the trace of the Dirichlet heat kernel is given by
\begin{align*}
Z_{\Omega}(t)=\sum\limits_{j=1}^{\infty}e^{-\nu_j t} = \int\limits_{\Omega} K_{\Omega}(x,x;t) dx\, \text{ for all } t>0.
\end{align*}

%
%
%
%

\begin{lemma}
\label{lemma:EigenfunctionsLuneAntiInvariantFunctions}
Let $\Delta:C^{\infty}(\Omega)\rightarrow C^{\infty}(\Omega)$, $\Delta(f):=-\emph{div}(\nabla f)$ \emph{(}where $\nabla$ and \emph{div} are taken with respect to the metric on $M=\mathbb{S}^2(r)$\emph{)} be the Laplace operator. Recall that for any $j\in J$ the function $\psi_j\in C^{\infty}_{\mathbb{Z}_{k}}(M)$ is an $\tilde{S}_0$-anti-invariant eigenfunction of the Dirichlet Laplacian $\Delta_M$ corresponding to the eigenvalue $\mu_j$.
\begin{itemize}
\item[$(i)$]  For any $j\in J$ the function ${\psi_{j}}_{\vert \Omega}$ is a classical solution to the Dirichlet eigenvalue problem on $\Omega$. This means that ${\psi_{j}}_{\vert \Omega}\in C^{\infty}(\Omega)\cap C(\overline{\Omega})$ and
\begin{align*}
\Delta{\psi_{j}}_{\vert \Omega}(x) & = \mu_j\cdot {\psi_{j}}_{\vert \Omega}(x) \,\text{ for all } x\in \Omega \\
\text{ and }\quad {\psi_{j}}_{\vert \overline{\Omega}}(x) &= 0 \quad\quad\quad\quad\,\, \text{ for all } x\in\partial\Omega. 
\end{align*}

\item[$(ii)$]  Moreover, the set $\left\lbrace \sqrt{2k}\cdot {\psi_{j}}_{\vert\Omega} \right\rbrace_{j\in J}$ is an orthonormal set in $L^2(\Omega)$ consisting of smooth eigenfunctions for the Dirichlet Laplacian $\Delta_\Omega$.
\end{itemize}
\end{lemma}

\begin{proof}
$(i)$. Let $j\in J$ be arbitrary. All statements of $(i)$ are obvious except possibly that $\psi_j$ vanishes pointwise on the boundary $\partial\Omega$, i.e. $\psi_j(x)=0$ for all $x\in\partial\Omega$. Since $\partial\Omega \subset \text{Fix}(\tilde{S}_0)\cup \text{Fix}(\tilde{S}_1)$ it suffices to show that $\psi_j$ vanishes on the sets $\text{Fix}(\tilde{S}_i)$, $i=0,1$:

Let $x\in \text{Fix}(\tilde{S}_0)$. Then we have $\psi_j(x) = \psi_j( \tilde{S}_0 x) = \left( {\tilde{S}_0}^{\ast} \psi_j \right) (x) = -\psi_j (x)$. Hence $\psi_j(x)=0$ and $\psi_j$ vanishes on $\text{Fix}(\tilde{S}_0)$.

In order to show that $\psi_{j}$ vanishes on $\text{Fix}(\tilde{S}_1)$, observe that the function $\psi_{j}$ is $\tilde{S}_1$-anti-invariant as well. In fact $\tilde{D}_1\circ \tilde{S}_0=\tilde{S}_1$, and thus $\tilde{S}_1$-anti-invariance of $\psi_{j}$ follows from its $\mathbb{Z}_{k}$-invariance and $\tilde{S}_0$-anti-invariance.

$(ii)$.
As any classical solution of the Dirichlet eigenvalue problem, the function ${\psi_{j}}_{\vert \Omega}$ is contained in the domain of the Dirichlet Laplacian and is also a smooth eigenfunction of $\Delta_{\Omega}$ (see \citep[Exercise 8.3 and the remark thereafter]{Grigoryan}).
It remains to show that 
\begin{align*}
\left\lbrace \sqrt{2k}\cdot {\psi_{j}}_{\vert\Omega} \right\rbrace_{j\in J}
\end{align*}
 is an orthonormal set in $L^2(\Omega)$. By using the decomposition 
\begin{align*}
M\backslash \bigcup_{i=0}^{k-1}\text{Fix}(\tilde{S}_{i})  =\bigcupdot_{i=0}^{k-1} \tilde{D}_i(\Omega\cup \tilde{S}_0(\Omega)),
\end{align*}
 and the $\mathbb{Z}_{k}$-invariance and $\tilde{S}_0$-anti-invariance of $\psi_{j}$, we get for all $j,\ell\in J$:
\begin{align*}
\langle \psi_{j}, \psi_{\ell} \rangle_{L^2(M)} &= \int\limits_{M} \psi_{j}(x)\cdot  \psi_{\ell}(x) dx  = k  \int\limits_{\Omega\cup \tilde{S}_0(\Omega)} \psi_{j}(x)\cdot  \psi_{\ell}(x) dx \\
& = 2k \int\limits_{\Omega} \psi_{j}(x)\cdot  \psi_{\ell}(x) dx = \int\limits_{\Omega} {\sqrt{2k}\cdot \psi_{j}}_{\vert\Omega}(x) \cdot  \sqrt{2k}\cdot {\psi_{\ell}}_{\vert\Omega}(x)  dx \\
& = \left\langle \sqrt{2k}\cdot {\psi_{j}}_{\vert\Omega}, \sqrt{2k}\cdot {\psi_{\ell}}_{\vert\Omega} \right\rangle_{L^2(\Omega)}.
\end{align*}
Therefore $\left\lbrace \sqrt{2k}\cdot {\psi_{j}}_{\vert\Omega} \right\rbrace_{j\in J}\subset L^2(\Omega)$ is an orthonormal set in $L^2(\Omega)$ consisting of smooth eigenfunctions of the Dirichlet Laplacian. 
\end{proof}

\begin{theorem}
\label{theorem:GEqualHeatTraceLune}
For all $U\subset M$, which are $O(2)$-invariant, we have 
\begin{align*}
G(t) = \int_{U\cap \Omega}K_{\Omega}(x,x;t)dx \, \text{ for all } t>0,
\end{align*}
where $G(t)=\int\limits_U  \sum\limits_{j\in J} e^{-\mu_j t}\psi_j(x)^2 dx$ as defined above. Thus, by \eqref{equation:HeatTraceRelation1}
\begin{align}
\label{equation:GeneralHeatTraceRelation2}
\int_{\proj{Z}{U}}K_{M/\mathbb{Z}_{k}}\left( \tilde{x},\tilde{x};t\right)d\tilde{x} - \int_{\proj{D}{U}}K_{M/\mathbb{D}_{k}}\left( \hat{x},\hat{x};t\right)d\hat{x} = \int_{U\cap \Omega}K_{\Omega}(x,x;t)dx.
\end{align}
In particular, we have for $U=M$:
\begin{align}
\label{equation:HeatTraceRelation2}
Z_{M/\mathbb{Z}_{k}}(t)-Z_{M/\mathbb{D}_{k}}(t)=Z_{\Omega}(t) \, \text{ for all } t>0.
\end{align}
\end{theorem}

\begin{proof}
Because the functions $\psi_{j}$ are $\mathbb{Z}_{k}$-invariant and $\tilde{S}_0$-anti-invariant, we have
\begin{align*}
G(t)=\int\limits_{U\cap \Omega}\sum\limits_{j\in J} e^{-\mu_j t} \left( \sqrt{2k}\cdot {\psi_j(x)}_{\vert\Omega} \right)^2 dx.
\end{align*}
%
Therefore it suffices to prove that the set $\left\lbrace \sqrt{2k}\cdot {\psi_{j}}_{\vert\Omega} \right\rbrace_{j\in J}$ is an orthonormal basis for $L^2(\Omega)$ of eigenfunctions of the Dirichlet Laplacian $\Delta_{\Omega}$. We know from Lemma  \ref{lemma:EigenfunctionsLuneAntiInvariantFunctions} $(ii)$ that this set is an orthonormal set in $L^2(\Omega)$ consisting of smooth eigenfunctions of the Dirichlet Laplacian. Hence it remains to show that this set is complete in $L^2(\Omega)$.

We want to emphasise first the following relations. We know by Proposition \ref{proposition:OrthonormalBasisEigenfunctionsQuadratintegrierbar} that the set 
\begin{align*}
\lbrace \sqrt{k}\cdot \tilde{\varphi}_i\rbrace_{i\in I}\cup \lbrace \sqrt{k}\cdot \tilde{\psi}_j\rbrace_{j\in J}
\end{align*}
is an orthonormal basis for $L^2(M/{\mathbb{Z}_{k}})$ consisting of smooth eigenfunctions. In particular, any continuous function $\tilde{f}\in C(M/\mathbb{Z}_{k})$ orthogonal to this set must vanish identically, i.e., if $\tilde{f}\perp \lbrace \sqrt{k}\cdot \tilde{\varphi}_i\rbrace_{i\in I}\cup \lbrace \sqrt{k}\cdot \tilde{\psi}_j\rbrace_{j\in J}$, then $\tilde{f}\equiv 0$.

By using the isometry \eqref{equation:IsometryPHI} for $n=0,\infty$ one gets immediately that the set
\begin{align*}
\lbrace \varphi_i\rbrace_{i\in I}\cup \lbrace \psi_j\rbrace_{j\in J}
\end{align*}
is an orthonormal set in the space $C_{\mathbb{Z}_{k}}(M)$, consisting of smooth and $\mathbb{Z}_{k}$-invariant eigenfunctions. Moreover, any continuous function $f\in C_{\mathbb{Z}_{k}}(M)$ with $f\perp \lbrace \varphi_i\rbrace_{i\in I}\cup \lbrace \psi_j\rbrace_{j\in J}$ must vanish, i.e. $f\equiv 0$.

Let $f\in C(\overline{\Omega})$ be given such that $f$ vanishes on the boundary $\partial\Omega$ and $f\perp \lbrace \sqrt{2k}\cdot {\psi_{j}}_{\vert\Omega} \rbrace_{j\in J}$. We now extend $f$ continuously on $M$ in two steps. Firstly, extend it onto $\tilde{S}_1(\overline{\Omega})$ by setting $f(\tilde{S}_1 x):=-f(x)$ for all $x\in \overline{\Omega}$. Secondly, extend this function continuously onto $M$ such that it becomes $\mathbb{Z}_{k}$-invariant.

This function is now perpendicular to $\lbrace \psi_j\rbrace_{j\in J}$ in $L^2(M)$ by construction and the initial assumption on $f_{\vert \Omega}$. Similary, it is perpendicular to the set $\lbrace \varphi_i\rbrace_{i\in I}$ in $L^2(M)$ by construction.
Hence, by the above considerations, $f$ must vanish on $M$.

\end{proof}

\section{Heat invariants for orbisurfaces}
\label{section:InvariantsOrbisurfaces}

In this section, we assume throughout that $M=\mathbb{S}^2(r)$ is equipped with the standard Riemannian metric induced by Euclidean $\mathbb{R}^3$. Recall that we proved in the first section (see \eqref{equation:HeatTraceRelation2}):
\begin{align*}
Z_{M/\mathbb{Z}_{k}}(t)-Z_{M/\mathbb{D}_{k}}(t)=Z_{\Omega}(t)\, \text{ for all } t>0.
\end{align*}
We aim to compute the coefficients of the asymptotic expansions as $t\searrow 0$ for all three heat traces appearing in the above equation. We will thereby obtain explicit formulas for all coefficients and it will be possible to recognise the contribution of any singular stratum to the asymptotic expansion. 

Throughout this section, we will use the definitions and results given in \cite{Gordon08}. In particular, Definition $2.3$ and the beginning of Section $4$ up to the end of the proof of Theorem $4.8$ of that article are relevant. Sometimes we will change the notation slightly compared to \cite{Gordon08}, mainly because the statements in \cite{Gordon08} are formulated for general orbifolds, whereas we are dealing only with two special orbifolds.

The following theorem is a special case of Theorem 4.8 in \cite{Gordon08}.
\begin{theorem}
\label{theorem:HeatTraceAsymptoticsOrbifoldQualitatively}
Let $\mathcal{O}$ be a two-dimensional closed Riemannian orbifold and let $\lambda_1\leq\lambda_2\leq...$ be the spectrum of the Lapacian $\Delta_{\mathcal{O}}$. The heat trace $Z_{\mathcal{O}}(t)=\sum_{i=1}^{\infty}e^{-\lambda_{i} t}$ of $\mathcal{O}$ is asymptotic as $t\searrow 0$ to
\begin{align*}
I_0(\mathcal{O})+\sum\limits_{N\in S(\mathcal{O})}\frac{I_N (\mathcal{O})}{\vert \Iso(N)\vert},
\end{align*}
where $S(\mathcal{O})$ is the set of all singular $\mathcal{O}$-strata and where $\vert \Iso(N) \vert$ is the order of the isotropy at each $p\in N$. The symbols $I_0 (\mathcal{O}),\text{ } I_N (\mathcal{O})$ are defined to be:
\begin{align*}
I_0 (\mathcal{O})&=\left( 4\pi t \right)^{-1}\cdot \sum\limits_{i=0}^{\infty} a_i(\mathcal{O}) t^i,\\
I_N (\mathcal{O})&=\left( 4\pi t \right)^{-\frac{\text{dim}(N)}{2}} \cdot \sum\limits_{i=0}^{\infty} t^{i}\int\limits_{N}b_i(N,x)d\vol_{N}(x),
\end{align*}
where $a_i(\mathcal{O})$ and $b_i(N,x)$ are given as in $\emph{\citep[Definition $4.7$]{Gordon08}}$. 
\end{theorem}

We will apply the above theorem to the orbifolds $M/G$ with $G\in\lbrace \mathbb{Z}_{k}, \mathbb{D}_{k} \rbrace$. Note that $M/G$ is finitely covered by $M$ and thus $a_i(\mathcal{O})=\frac{1}{\vert G \vert}a_i(M)$, where $a_i(M)$ denote the familiar heat invariants for the manifold $M$ (see the comment at the end of Definition $4.7$ in \citep{Gordon08}). The $M/G$-strata are defined as the connected components of the \emph{G-isotropy equivalence classes} in $M/G$. Let us investigate what the G-isotropy equivalence classes look like for our orbifolds and then apply Theorem \ref{theorem:HeatTraceAsymptoticsOrbifoldQualitatively}.

Let $x_N:=(0,0,r)$ and $x_S:=(0,0,-r)$. The orbifold $M/\mathbb{Z}_{k}$ has two $\mathbb{Z}_{k}$-isotropy equivalence classes. If $\tilde{x}_N:=\pi_{{\mathbb{Z}_{k}}}(x_N)$, $\tilde{x}_S:=\pi_{{\mathbb{Z}_{k}}}(x_S)$ are the cone points, one isotropy equivalence class is given by the set $\lbrace \tilde{x}_N, \tilde{x}_S \rbrace$ which has two connected components $\lbrace \tilde{x}_N \rbrace, \lbrace \tilde{x}_S \rbrace$. The isotropy group of $\tilde{x}_N,\, \tilde{x}_S$ is given by $\Iso_{\mathbb{Z}_{k}}\lbrace \tilde{x}_N \rbrace=\mathbb{Z}_{k}=\Iso_{\mathbb{Z}_{k}}\lbrace \tilde{x}_S\rbrace$ so that the isotropy order is $\vert \mathbb{Z}_{k} \vert = k$. All the other points in $M/\mathbb{Z}_{k}$ are regular and constitute the other $\mathbb{Z}_{k}$-isotropy equivalence class. Hence, we get
\begin{align}
\label{equation:AsymptotikTraceFootballOrbifold}
Z_{M/\mathbb{Z}_{k}}(t)\overset{t\downarrow 0}{\sim} I_0(M/\mathbb{Z}_{k}) + \frac{1}{k} I_{\lbrace \tilde{x}_N \rbrace} (M/\mathbb{Z}_{k}) + \frac{1}{k} I_{\lbrace \tilde{x}_S \rbrace} (M/\mathbb{Z}_{k}).
\end{align}

The number of $\mathbb{D}_{k}$-isotropy equivalence classes of the orbifold $M/\mathbb{D}_{k}$ depends on whether $k$ is even or odd. Let us first assume that $k$ is even. Then the orbifold $M/\mathbb{D}_{k}$ has four $\mathbb{D}_{k}$-isotropy equivalence classes. Let $\hat{x}_N:=\pi_{\mathbb{D}_{k}}(x_N)$, and $\hat{x}_S:=\pi_{\mathbb{D}_{k}}(x_S)$ be the two dihedral points. Then $\lbrace \hat{x}_N, \hat{x}_S \rbrace$ is an isotropy equivalence class with two connected components $\{\hat{x}_N\},\, \{\hat{x}_S\}$ and isotropy group $\Iso_{\mathbb{D}_{k}}\lbrace \hat{x}_N \rbrace=\mathbb{D}_{k}=\Iso_{\mathbb{D}_{k}}\lbrace \hat{x}_S \rbrace$. Thus the isotropy order is $\vert \mathbb{D}_{k} \vert = 2k$.

Two of the remaining isotropy equivalence classes consist of mirror points. Recall that $\gamma_0$, $\gamma_1\subset M$ are the two sides of the lune $\Omega$. The two isotropy equivalence classes are represented by the sets $\text{Mir}_0:= \pi_{\mathbb{D}_{k}}\left( \gamma_0\right) \backslash \lbrace \hat{x}_N, \hat{x}_S \rbrace $ and $\text{Mir}_1 := \pi_{\mathbb{D}_{k}}\left( \gamma_1\right) \backslash \lbrace \hat{x}_N, \hat{x}_S \rbrace $, respectively. Both isotropy equivalence classes are connected. The conjugacy class of the isotropy group of any $\tilde{p}\in \text{Mir}_0$ is given by $\lbrace \lbrace e, \tilde{S}_{\ell} \rbrace \mid \ell=0,2,4,...,k-2 \rbrace$, where $e$ denotes the identity element of the dihedral group $\mathbb{D}_{k}$. Similarly the conjugacy class of the isotropy group of any $\tilde{p}\in \text{Mir}_1$ is given by $\lbrace \lbrace e, \tilde{S}_{\ell} \rbrace \mid \ell=1,3,5,...,k-1 \rbrace$. Hence the isotropy order is $\vert \Iso_{\mathbb{D}_{k}}(\tilde{p}) \vert = 2$ for any $\tilde{p}\in \text{Mir}_0 \cup \text{Mir}_1$.

The fourth $M/\mathbb{D}_{k}$-isotropy equivalence class consists of the regular points
\begin{align*}
\left( M/\mathbb{D}_{k} \right) \backslash \bigcup\limits_{i=0}^1 \pi_{\mathbb{D}_{k}}(\gamma_i).
\end{align*}

Hence we get
\begin{align}
\label{equation:AsymptotikTraceDihedralOrbifold}
\begin{split}
Z_{M/\mathbb{D}_{k}} (t) \overset{t\downarrow 0}{\sim} I_0(M/\mathbb{D}_{k}) &+ \frac{1}{2k} I_{\lbrace \hat{x}_N \rbrace}(M/\mathbb{D}_{k}) + \frac{1}{2k} I_{\lbrace \hat{x}_S \rbrace}(M/\mathbb{D}_{k})  \\
&+ \frac{1}{2} \left( I_{\text{Mir}_0}(M/\mathbb{D}_{k}) + I_{\text{Mir}_1}(M/\mathbb{D}_{k})\right).
\end{split}
\end{align}

If $k$ is odd, then we have three $\mathbb{D}_{k}$-isotropy equivalence classes. Two of them are given as above by the set of dihedral points and regular points, respectively. However, for odd $k$ all reflections $\tilde{S}_{\ell}$ are conjugate to each other, such that the third $\mathbb{D}_{k}$-isotropy equivalence class is given by the set of all mirror points 
\begin{align*}
\text{Mir}_0 \bigcupdot \text{Mir}_1.
\end{align*}
This isotropy equivalence class has two connected components, namely, $\text{Mir}_0$ and $\text{Mir}_1$. The conjugacy class of the isotropy group of any $\tilde{p}\in \text{Mir}_0 \cup \text{Mir}_1$ is given by $\lbrace \lbrace e, \tilde{S}_{\ell} \rbrace \mid \ell=0,...,k-1 \rbrace$ and hence the isotropy order is $\vert \Iso_{\mathbb{D}_{2k}}(\tilde{p}) \vert = 2$. We conclude that the asymptotic expansion \eqref{equation:AsymptotikTraceDihedralOrbifold} holds for odd $k$ as well. 

\begin{remark}
The above analysis together with Theorem \ref{theorem:HeatTraceAsymptoticsOrbifoldQualitatively} shows that there will be no half-powers of $t$ in the asymptotic expansion of $Z_{M/\mathbb{Z}_k}(t)$. In particular, by \eqref{equation:HeatTraceRelation2}, the mirror-edge contribution $B$ to $Z_{M/\mathbb{D}_k}(t)$ must equal the negative of the boundary contribution to $Z_{\Omega}(t)$. Moreover, we will show that the contribution $C$ of the dihedral points to the asymptotic expansion of $Z_{M/\mathbb{D}_k}(t)$ equals half the contribution of the cone points to $Z_{M/\mathbb{Z}_k}(t)$. The same obviously holds for the contribution $A$ of the interior to the asymptotic expansion of $Z_{M/\mathbb{D}_k}(t)$. So it will follow that $Z_{\Omega}(t) \overset{t\downarrow 0}{\sim} (2A+2C)-(A+B+C) = A-B+C$. We will state this more explicitly in Corollary \ref{corollary:AsymptoticExpansionDifferenceHeatTracesOrbifolds} further below.
\end{remark}

Before we compute the coefficients in the above asymptotic expansions explicitly, we want to simplify the asymptotic expansions given in \eqref{equation:AsymptotikTraceFootballOrbifold} and \eqref{equation:AsymptotikTraceDihedralOrbifold}.

 As it is remarked in \cite{Gordon08} (see the proof of Theorem $4.8$), we can express $I_{N}(M/G)$ by 
\begin{align}
\label{equation:GoodOrbifoldsSimplificationContributionStrata}
\tilde{I}_{\tilde{N}_1} + ... + \tilde{I}_{\tilde{N}_n} = \frac{\vert G \vert}{\vert \Iso_G(N) \vert} I_N (M/G),
\end{align}
where $\tilde{N}_1,...,\tilde{N}_n$ are the mutually isometric $M$-strata with $\pi_{G}^{-1}(N) = \bigcupdot_{i=1}^{n}\tilde{N}_i$, and for all $M$-strata $\tilde{N}\subset M$:
\begin{align}
\tilde{I}_{\tilde{N}} :&= \left( 4\pi t \right)^{-\frac{\text{dim}(\tilde{N})}{2}}\sum\limits_{i=0}^{\infty} t^{i}\int\limits_{\tilde{N}}b_i(\tilde{N},x)d\vol_{\tilde{N}}(x), \label{equation:StrataHeatTrace} \\
b_{i}(\tilde{N},x)&:= \sum\limits_{\gamma\in \Iso^{\text{max}}(\tilde{N})}b_i\left( \gamma, x \right). \label{equation:CoefficientsOfDonelly1}
\end{align}
By definition, for any $M$-stratum $\tilde{N}$, the set $\Iso^{\text{max}}(\tilde{N})\subset G$ is defined as the set of all $\gamma\in G$ such that $\tilde{N}$ is open in the fixed point set $\text{Fix}(\gamma)$ of $\gamma$. 

The functions $b_i\left( \gamma, x \right)$, first defined by H. Donnelly in his article \cite{Donnelly2}, have two important properties, which are called ``locality'' and ``universality'' in \cite{Gordon08}. Because of the universality property, for any isometry $\sigma:M\rightarrow M$ satisfying $\sigma \circ \gamma =  \gamma' \circ \sigma$, we have
\begin{align}
\label{equation:UniversalityDonnellyCoefficients}
b_i(\gamma,x)=b_i(\gamma', \sigma(x)) \text{ for all } x\in \text{Fix}(\gamma).
\end{align}

\begin{lemma}
\label{lemma:IntegralsOrbifoldIdentifications}
\text{ }
\begin{itemize}
\item[$(i)$]$\begin{aligned}[t]
 I_{\lbrace \tilde{x}_N \rbrace } (M/\mathbb{Z}_{k}) &= I_{\lbrace \tilde{x}_S \rbrace } (M/\mathbb{Z}_{k}) = I_{\lbrace \hat{x}_S \rbrace} (M/\mathbb{D}_{k}) = I_{\lbrace \hat{x}_N \rbrace} (M/\mathbb{D}_{k})\, \text{ and }\\
I_{\lbrace \tilde{x}_N \rbrace} (M/\mathbb{Z}_{k}) &=  \sum\limits_{i=0}^{\infty} t^{i} \cdot \sum\limits_{\ell=1}^{k-1} b_i(\tilde{D}_{\ell} ,x_N).
\end{aligned}$

\item[$(ii)$]$\begin{aligned}[t]
I_{ \text{Mir}_0 } (M/\mathbb{D}_{k}) &= I_{\text{Mir}_1 } (M/\mathbb{D}_{k})\, \text{ and }\\
I_{\text{Mir}_0}(M/\mathbb{D}_{k}) &=  \left( 4\pi t \right)^{-\frac{1}{2}}\sum\limits_{i=0}^{\infty} t^{i} \cdot \frac{1}{2}\int\limits_{\text{Fix}(\tilde{S}_0)}b_i(\tilde{S}_0,x)d\vol_{ \text{Fix}(\tilde{S}_{0}) }(x).
\end{aligned}$

\end{itemize}
\end{lemma}

\begin{proof}
$(i)$. Obviously, $\pi_{\mathbb{Z}_{k}}^{-1}(\tilde{x}_N)=\lbrace x_N \rbrace = \pi_{\mathbb{D}_{k}}^{-1}(\hat{x}_N)$, $\pi_{\mathbb{Z}_{k}}^{-1}(\tilde{x}_S)=\lbrace x_S \rbrace = \pi_{\mathbb{D}_{k}}^{-1}(\hat{x}_S)$, and 
\begin{align*}
\Iso^{\text{max}}\left( \lbrace \tilde{x}_N \rbrace \right) &= \mathbb{Z}_{k}\backslash \{ e \} = \Iso^{\text{max}}\left( \lbrace \tilde{x}_S \rbrace \right), \\
\Iso^{\text{max}}\left( \lbrace \hat{x}_N \rbrace \right) &= \mathbb{Z}_{k}\backslash \{ e \} = \Iso^{\text{max}}\left( \lbrace \hat{x}_S \rbrace \right).
\end{align*}
Thus by using \eqref{equation:GoodOrbifoldsSimplificationContributionStrata}-\eqref{equation:CoefficientsOfDonelly1} we get, noting that $\int_{\lbrace x_N \rbrace}b_i(\tilde{D}_{\ell} ,x)d\vol_{\lbrace x_N \rbrace}(x)=b_i(\tilde{D}_{\ell} ,x_N)$ by the definition of the integral over a one-point set:
\begin{align*}
I_{\lbrace \tilde{x}_N \rbrace}(M/\mathbb{Z}_{k}) &= \sum\limits_{i=0}^{\infty} t^{i} \cdot \sum\limits_{\ell=1}^{k-1} b_i(\tilde{D}_{\ell} ,x_N)  =I_{\lbrace \hat{x}_N \rbrace}(M/\mathbb{D}_{k}),\\
I_{\lbrace \tilde{x}_S \rbrace}(M/\mathbb{Z}_{k}) &= \sum\limits_{i=0}^{\infty} t^{i} \cdot  \sum\limits_{\ell=1}^{k-1} b_i(\tilde{D}_{\ell} ,x_S) =I_{\lbrace \hat{x}_S \rbrace}(M/\mathbb{D}_{k}).
\end{align*}
Therefore it suffices to show that for any $\tilde{D}_{\ell}\in \mathbb{Z}_{k}$:
\begin{align*}
b_i(\tilde{D}_{\ell} ,x_N) = b_i(\tilde{D}_{\ell} ,x_S).
\end{align*}
Using $\sigma\circ \tilde{D}_{\ell} = \tilde{D}_{\ell}\circ \sigma$ for the reflection $\sigma$ in the $(x,y)$-plane, this follows trivially by the universality property \eqref{equation:UniversalityDonnellyCoefficients} of $b_i(\gamma ,x)$.
%

We now turn to $(ii)$. We have 
\begin{align*}
\pi_{\mathbb{D}_{k}}^{-1}(\text{Mir}_0) = \left( \bigcupdot\limits_{\ell=0}^{k-1} \tilde{D}_{\ell} \gamma_0 \right) \bigg\backslash \left\lbrace x_N, x_S \right\rbrace, \\
\pi_{\mathbb{D}_{k}}^{-1}(\text{Mir}_1) = \left( \bigcupdot\limits_{\ell=0}^{k-1} \tilde{D}_{\ell} \gamma_1 \right) \bigg\backslash \left\lbrace x_N, x_S \right\rbrace,
\end{align*}
where $ \gamma_{0}\backslash \lbrace x_N, x_S \rbrace$ and $\gamma_{1}\backslash \lbrace x_N, x_S\rbrace$ are connected components of $\text{Fix}(\tilde{S}_0)\backslash\lbrace x_N, x_S \rbrace$ and $\text{Fix}(\tilde{S}_1)\backslash\lbrace x_N, x_S \rbrace$, respectively.
Again by using universality of the functions $b_i(\gamma ,x)$ and  equations \eqref{equation:GoodOrbifoldsSimplificationContributionStrata}-\eqref{equation:CoefficientsOfDonelly1}, we obtain:
\begin{align*}
I_{\text{Mir}_0}(M/\mathbb{D}_{k}) &=  \left( 4\pi t \right)^{-\frac{1}{2}}\sum\limits_{i=0}^{\infty} t^{i} \cdot \frac{1}{2}\int\limits_{\text{Fix}(\tilde{S}_0)}b_i(\tilde{S}_0,x)d\vol_{\text{Fix}(\tilde{S}_{0})}(x), \\
I_{\text{Mir}_1}(M/\mathbb{D}_{k}) &=  \left( 4\pi t \right)^{-\frac{1}{2}}\sum\limits_{i=0}^{\infty} t^{i} \cdot \frac{1}{2}\int\limits_{\text{Fix}(\tilde{S}_1)}b_i(\tilde{S}_1,x)d\vol_{ \text{Fix}(\tilde{S}_{1}) }(x).
\end{align*}
By using universality once again, we get $I_{\text{Mir}_0}(M/\mathbb{D}_{k}) = I_{\text{Mir}_1}(M/\mathbb{D}_{k})$.
\end{proof}

Because of the universality property of the functions $b_i\left( \gamma, x \right)$, it is easy to show that $b_i\left( \tilde{S}_0, x \right)$ is constant on $\text{Fix}(\tilde{S}_0)$. Hence by setting $b_{i}(\tilde{S}_0):=b_{i}(\tilde{S}_0, a^{\ast})$ for some fixed $a^{\ast}\in\text{Fix}(\tilde{S}_0)$, we get
\begin{align*}
\int\limits_{\text{Fix}(\tilde{S}_0)}b_i(\tilde{S}_0,x)d\vol_{\text{Fix}(\tilde{S}_{0})}(x)  = b_i \left( \tilde{S}_0 \right)\cdot \vert \text{Fix}(\tilde{S}_0) \vert = b_i \left( \tilde{S}_0 \right)\cdot 2\pi r,
\end{align*}
hence $I_{\text{Mir}_0}(M/\mathbb{D}_{k}) = (4\pi t)^{-\frac{1}{2}}\sum_{i=0}^{\infty} t^{i}\cdot \pi r b_i(\tilde{S}_0)$.
Furthermore, for any $\ell\in\lbrace 0,...,k-1 \rbrace$, we set $b_{i}(\tilde{D}_{\ell}):=b_{i}(\tilde{D}_{\ell}, x_N )$.

\begin{corollary}
\label{corollary:AsymptoticExpansionDifferenceHeatTracesOrbifolds}
Letting
\begin{align*}
A:&=I_{0}(M/\mathbb{D}_{k}) =  \frac{1}{4\pi t}\sum\limits_{i=0}^{\infty} \frac{a_i (M)}{2k} t^i = \frac{1}{2}I_{0}(M/\mathbb{Z}_{k}), \\
B:&=I_{\text{Mir}_0}(M/\mathbb{D}_{k}) = \frac{1}{\sqrt{4\pi t}}\sum\limits_{i=0}^{\infty} t^i \cdot  \pi r b_i\left( \tilde{S}_0 \right), \text{ and }\\
C:&=\frac{1}{k} I_{\lbrace{ \hat{x}_N\rbrace} }(M/\mathbb{D}_{k}) =  \sum\limits_{i=0}^{\infty}t^i \cdot \frac{1}{k} \sum\limits_{\ell=1}^{k-1} b_i\left( \tilde{D}_{\ell} \right),
\end{align*}
we obtain from \eqref{equation:AsymptotikTraceFootballOrbifold}, \eqref{equation:AsymptotikTraceDihedralOrbifold}, and Lemma \emph{\ref{lemma:IntegralsOrbifoldIdentifications}}:
\begin{align*}
Z_{M/\mathbb{Z}_{k}} (t) &\overset{t\downarrow 0}{\sim} 2A + 2C,\\
Z_{M/\mathbb{D}_{k}}(t) &\overset{t\downarrow 0}{\sim} A + B + C.
\end{align*}
In particular, by \eqref{equation:HeatTraceRelation2}:
\begin{align}
\label{equation:AsymptoticExpansionDifferenceHeatTracesOrbifolds}
Z_{\Omega}(t) \overset{t\downarrow 0}{\sim} A-B+C.
\end{align}
\end{corollary}

In order to get explicit formulas for the coefficients $b_i\left( \tilde{S}_0 \right)$ and $\sum_{\ell=1}^{k-1} b_i\left( \tilde{D}_{\ell} \right)$ for all $i\in\mathbb{N}_0$, we will now compute the asymptotic expansion of $Z_{\Omega}(t)$ and the coefficients $a_i(M)$ explicitly. Both the coefficients $a_i(M)$ and the asymptotic expansion of $Z_{\Omega}(t)$ were computed in \cite{Watson} for the radius  $r=1$. But since the statements and proofs in \cite{Watson} have some typographical errors, we will reproduce the relevant parts with corrections in Proposition \ref{proposition:HeatTraceAsymptoticSphere} and Proposition \ref{proposition:AsymptoticExpansionLuneWatson} below. We will comment shortly on the errors contained in the corresponding statements of the article \citep{Watson} afterwards.

Recall the Bernoulli polynomials $B_k (x)$ and Bernoulli numbers $B_k$ from \eqref{equation:DefinitionBernoulli}.

\begin{lemma}
\label{lemma:RelationBernoulliPolynomials}
\begin{align}
\label{equation:RelationBernoulliPolynomials}
-B_n\left( \frac{1}{2} \right) = B_n \cdot \left( 1-\frac{1}{2^{n-1}} \right)\, \text{ for all } n\in\mathbb{N}_0.
\end{align}
\end{lemma}

\begin{proof}
Consider first the sum of the generating functions for $B_n \left( \frac{1}{2} \right)$ and $B_n$:
\begin{align*}
\frac{t\cdot e^{\frac{t}{2}}}{e^{t}-1} + \frac{t}{e^{t}-1} &= \frac{t\left( e^{\frac{t}{2}}+1\right)}{e^{t}-1} = \frac{t}{e^{\frac{t}{2}}-1} \cdot \frac{e^{\frac{t}{2}}+1}{e^{\frac{t}{2}}+1} \\
&= \frac{t}{e^{\frac{t}{2}}-1} = 2\cdot \frac{\frac{t}{2}}{e^{\frac{t}{2}}-1}.
\end{align*}
Hence, 
\begin{align*}
B_{n}\left( \frac{1}{2} \right) + B_n = 2\cdot \frac{B_n}{2^n}\, \text{ for all }\, n\in \mathbb{N}_0,
\end{align*}
from which \eqref{equation:RelationBernoulliPolynomials} follows.
\end{proof}

\begin{proposition}
\label{proposition:HeatTraceAsymptoticSphere}
Let $Z_{\mathbb{S}^2(r)}(t)$ be the heat trace for the sphere of radius $r>0$. Then
\begin{align}
\label{equation:HeatTraceAsymptoticSphere}
Z_{\mathbb{S}^2(r)}(t) \overset{t\downarrow 0}{\sim}  \frac{1}{\kappa \cdot t} +  \sum\limits_{\nu=0}^{\infty} i_{\nu}^{\mathbb{S}} \cdot \kappa^{\nu} \cdot t^{\nu},
\end{align}  
where $\kappa=\frac{1}{r^2}$ denotes the Gaussian \emph{(}or sectional\emph{)} curvature and
\begin{align}
i_{\nu}^{\mathbb{S}} :&=   \sum\limits_{\ell=0}^{\nu+1}\frac{1}{4^{\ell}\cdot \ell !}i_{\nu-\ell}^{(s)}\, \text{ for all } \, \nu\in\mathbb{Z}_{\geq -1}, \label{equation:HeatCoefficientsSphere1}\\
 i_{\ell}^{(s)}:&=\frac{(-1)^{\ell +1}}{(\ell+1)!}B_{2\ell+2}\left( \frac{1}{2} \right)\, \text{ for all } \, \ell \in \mathbb{Z}_{\geq -1}. \label{equation:HeatCoefficientsSphere2}
\end{align}
\end{proposition}

\begin{proof}
First we assume that $r=1$. The eigenvalues of the sphere $S^2(1)$ are the numbers $\ell \cdot \left( \ell + 1 \right),$
where $\ell\in\mathbb{N}_{0}$, and each $\ell \cdot \left( \ell + 1 \right)$ has multiplicity $2\ell+1$. Thus the heat trace is given by
\begin{align}
Z_{\mathbb{S}^2(1)}(t) = \sum\limits_{\ell=0}^{\infty}\left( 2\ell+1 \right)e^{-\ell(\ell+1)t} \text{ for } t>0.
\end{align}
As before, it is more convenient to consider a shifted heat trace; in contrast to the hyperbolic case, the convenient factor is now $e^{-\frac{t}{4}}$ instead of $e^{\frac{t}{4}}$. The shifted trace function is given by
\begin{align*}
Z_{\mathbb{S}^2(1)}^{-\nicefrac{1}{4}}(t) &:= e^{-\frac{t}{4}}\sum\limits_{\ell=0}^{\infty}\left( 2\ell+1 \right)e^{-\ell(\ell+1)t} = \sum\limits_{\ell=0}^{\infty}\left( 2\ell+1 \right)e^{-\left( \ell(\ell+1) +\frac{1}{4} \right) t} \\
&= \sum\limits_{\ell=0}^{\infty}\left( 2\ell+1 \right)e^{-\left( \ell + \frac{1}{2} \right)^2 t}.
\end{align*}
The asymptotic expansion of the function $Z_{\mathbb{S}^2 (1)}^{-\nicefrac{1}{4}}(t)$ has been known for a long time and can be found in \cite{Mulholland}, where the expansion is given as
\begin{align*}
Z_{\mathbb{S}^2 (1)}^{-\nicefrac{1}{4}}(t)\overset{t\downarrow 0}{\sim}\frac{1}{t} + \sum\limits_{n=0}^{\infty}  \frac{(-1)^{n}}{(n+1)!}B_{2n+2} \cdot \left( 1 -  \frac{1}{2^{2n+1}} \right) t^n.
\end{align*}
By using \eqref{equation:RelationBernoulliPolynomials} from Lemma \ref{lemma:RelationBernoulliPolynomials}, we can rewrite the coefficients as
\begin{align*}
\frac{(-1)^{n}}{(n+1)!}B_{2n+2} \cdot \left( 1 -  \frac{1}{2^{2n+1}} \right) = \frac{(-1)^{n+1}}{(n+1)!}B_{2n+2}\left( \frac{1}{2} \right) =  i_n^{(s)}
\end{align*}
and we obtain
\begin{align}
\label{equation:AsymptotikShiftedTraceSphere}
Z_{\mathbb{S}^2 (1)}^{-\nicefrac{1}{4}}(t)\overset{t\downarrow 0}{\sim}\frac{1}{t} + \sum\limits_{n=0}^{\infty}  i_n^{(s)} t^n.
\end{align}
Hence we get
\begin{align}
\label{equation:AsymptoticExpansionTraceSphere}
Z_{\mathbb{S}^2 (1)}(t) \overset{t\downarrow 0}{\sim} \frac{1}{t} + \sum\limits_{\nu=0}^{\infty} \underbrace{\left( \frac{1}{4^{\nu+1}\left( \nu+1 \right)!} + \sum\limits_{\ell=0}^{\nu}\frac{1}{4^{\ell}\cdot \ell !}i_{\nu-\ell}^{(s)} \right)}_{=\, i_{\nu}^{\mathbb{S}}} t^{\nu}.
\end{align}
Thus the case $r=1$ is established. By scaling the Riemannian metric one can now easily deduce the general case $r>0$ from the special case $r=1$, just as we did in the proof of Corollary \ref{corollary:HeatAsymptoticConstantCurvaturePolygon}. If $g_{\mathbb{S}^2(1)}$ denotes the standard metric on $\mathbb{S}^2(1)$, then $\lambda_1\leq \lambda_2\leq... $ are the eigenvalues of $\left( \mathbb{S}^2(1), g_{\mathbb{S}^2(1)}\right)$ listed with multiplicities, if and only if $ \frac{\lambda_1}{r^2}\leq \frac{\lambda_2}{r^2}\leq... $ are the eigenvalues of $\left( \mathbb{S}^2(1), r^2  g_{\mathbb{S}^2(1)}\right)$ listed with multiplicities. Thus, if $Z(r^2 g_{\mathbb{S}^2(1)}, t)$ denotes the heat trace of the space $\left( \mathbb{S}^2(1), r^2  g_{\mathbb{S}^2(1)}\right)$, we have
\begin{align*}
Z_{\mathbb{S}^2(r)}(t) = Z(r^2 g_{\mathbb{S}^2(1)}, t) = \sum\limits_{i=1}^{\infty} e^{-\frac{\lambda_i}{r^2}\cdot t} = \sum\limits_{i=1}^{\infty} e^{-\lambda_i\frac{t}{r^2}} = Z_{\mathbb{S}^2(1)} \left(\frac{t}{r^2}\right) = Z_{\mathbb{S}^2(1)} \left(\kappa t\right).
\end{align*}

%
%
\end{proof}

\begin{remark}
Basically the asymptotic expansion \eqref{equation:AsymptotikShiftedTraceSphere} for the shifted heat trace is stated in \citep[Lemma 16]{Watson} by referring to \citep{Mulholland}. Let us point out some aspects of Lemma 16 which might be confusing. First of all, the function considered in Lemma 16, which is denoted there by $Z_{\text{int}}^{(s)}(\Omega, t)$, is not defined anywhere in the article \citep{Watson}. We assume that Watson means the following function:
\begin{align*}
Z_{\text{int}}^{(s)}(\Omega, t):=\frac{\vert \Omega \vert}{4\pi} Z_{\mathbb{S}^2 (1)}^{(s)}(t),
\end{align*}
where $Z_{\mathbb{S}^2 (1)}^{(s)}(t):=Z_{\mathbb{S}^2 (1)}^{-\nicefrac{1}{4}}(t)$ is the function we introduced above and $\vert \Omega \vert$ is the volume of a spherical polygon fixed throughout in \citep{Watson}. This would at least be in accordance with the notation given previously in \citep[formula (11)]{Watson}. Under this assumption, the coefficients $i_n^{(s)}$ given in \citep[formula (80)]{Watson} have the wrong sign (compare them with our formula above). Lastly, we want to remark that Watson uses two different definitions for the coefficients $i_n^{(s)}$ (compare formula $(80)$ with $(25)$ in \citep{Watson}) which do not coincide and thus might be confusing at first sight.
\end{remark}

\begin{corollary}
\label{corollary:HeatCoefficientsLocallySymmetricManifolds}
Let $A = I_{0}(M/\mathbb{D}_k)$ be the formal series as defined in Corollary \emph{\ref{corollary:AsymptoticExpansionDifferenceHeatTracesOrbifolds}} \emph{(}recall $M=\mathbb{S}^2(r)$\emph{)}, and let $\kappa=\frac{1}{r^2}$ as before. Then
\begin{align}
\label{equation:ABC}
A&= \frac{1}{4\pi t}\sum\limits_{\nu=0}^{\infty} \frac{a_{\nu} (M)}{2k} t^{\nu} = \frac{1}{4\pi t}\cdot \frac{1}{2k} \vert \mathbb{S}^2(r) \vert \sum\limits_{\nu=0}^{\infty} i_{\nu-1}^{\mathbb{S}} \kappa^{\nu} t^{\nu} \nonumber \\
&=\frac{1}{4\pi t} \cdot  \frac{1}{2k} \vert \mathbb{S}^{2}(r) \vert \sum\limits_{\nu =0}^{\infty} \sum\limits_{\ell=0}^{\nu} \frac{1}{4^{\nu-\ell} (\nu -\ell)!} \frac{(-1)^{\ell}}{\ell !} B_{2\ell}\left( \frac{1}{2} \right) \cdot \kappa^{\nu}\cdot t^{\nu}.
\end{align}
In particular, we have for all $\nu\in\mathbb{N}_{0}$
\begin{align}
a_{\nu} (M) = \int\limits_{M}  \sum\limits_{\ell=0}^{\nu} \frac{1}{4^{\nu-\ell} (\nu -\ell)!} \frac{(-1)^{\ell}}{\ell !} B_{2\ell}\left( \frac{1}{2} \right) \cdot \kappa^{\nu} dx. \label{equation:HeatCoefficientsLocallySymmetricManifolds}
\end{align}
The last formula for the coefficient $a_{\nu} (M)$ is universal in the sense that the heat coefficients for any two-dimensional closed Riemannian manifold $M$ of constant Gaussian curvature $\kappa$ are given by that formula.
\end{corollary} 

\begin{proof}
Recall that the $a_{\nu}(M)$ are given by $Z_M(t)\overset{t\downarrow 0}{\sim} \frac{1}{4\pi t}\sum_{\nu=0}^{\infty}a_{\nu}(M) t^{\nu}$. Comparing this with \eqref{equation:HeatTraceAsymptoticSphere}, one easily gets the formulas \eqref{equation:ABC} for $A$ and \eqref{equation:HeatCoefficientsLocallySymmetricManifolds} for $a_{\nu} (M)$ from $\vert \mathbb{S}^2(r) \vert = \frac{4\pi}{\kappa}$ and by substituting the coefficients $i_{\nu-1},\, i_{\ell}^{(s)}$ from \eqref{equation:HeatCoefficientsSphere1},  \eqref{equation:HeatCoefficientsSphere2}.

In general, for any $\nu\in\mathbb{N}_{0}$ there exists a universal polynomial in the Gaussian curvature and its covariant derivatives
such that the heat invariant $a_{\nu}(M)$ for any two-dimensional closed Riemannian manifold $M$ is given as the integral of that polynomial over $M$ (see e.g. \citep{GilkeyBranson}). More precisely, that polynomial is a so-called curvature invariant of order $2\nu$. The covariant derivatives of the Gaussian curvature vanish if the Gaussian curvature is constant. So in this case, the polynomial is just a monomial of the form $\alpha_{\nu}\kappa^{\nu}$, and since it is universal, the case of $M=\mathbb{S}^2(r)$ tells us that $\alpha_{\nu}$ equals $i_{\nu-1}^{\mathbb{S}}$. Thus the formula \eqref{equation:HeatCoefficientsLocallySymmetricManifolds} must hold for any two-dimensional Riemannian manifold of constant curvature $\kappa$.
\end{proof}

The next result from \citep{Watson} which we will need, and prove here with minor corrections, is the following:

\begin{proposition}
\label{proposition:AsymptoticExpansionLuneWatson}
Let $k\in\mathbb{N}$ \emph{(}we allow $k=1$ in this proposition\emph{)}. As before, let $\Omega\subset\mathbb{S}^2(r)$ be a lune with angle $\frac{\pi}{k}$, and write $\kappa=\frac{1}{r^2}$. Then
\begin{align}
 \label{equation:AsymptoticExpansionLuneWatson}
 \begin{split}
Z_{\Omega}(t)\overset{t\downarrow 0}{\sim}  \frac{1}{2k}\cdot \frac{1}{\kappa \cdot t} - \frac{\sqrt{\pi}}{4 \sqrt{\kappa}}\cdot \frac{1}{\sqrt{t}} + \sum\limits_{\nu=0}^{\infty} & \left\lbrace  \frac{1}{2k} \cdot i_{\nu}^{\mathbb{S}} - \frac{1}{4^{\nu+1}(\nu + 1)!} \cdot  \frac{\sqrt{\pi}\cdot \sqrt{\kappa}}{4}\sqrt{t} \right.  \\
&  \left.  + 2\cdot \sum\limits_{\ell=0}^{\nu}\frac{1}{4^{\ell}\cdot \ell !}\cdot c_{\nu-\ell}^{\mathbb{S}}\left( \frac{\pi}{k}\right)   \right\rbrace \kappa^{\nu} t^{\nu},
\end{split}
\end{align}
where $i_{\nu}^{\mathbb{S}}$ is defined as in the previous proposition and
\begin{align}
\label{equation:DefinitionCoefficientsALune}
c_{\ell}^{\mathbb{S}}\left( \frac{\pi}{k}\right)=\frac{1}{4k}\cdot \frac{(-1)^{\ell}}{(\ell+1)!}\cdot \frac{1}{2\ell+1}\sum\limits_{j=0}^{\ell+1}  \binom{2\ell+2}{2j}\left(  k ^{2j}-1 \right)B_{2j}B_{2\ell+2-2j}\left(\frac{1}{2}\right)
\end{align}
\emph{(}recall $c_{\ell}^{\mathbb{S}}(\gamma)$ from the Remark after Corollary \emph{\ref{corollary:EckenbeitragWinkelPolygon})}.
\end{proposition}

\begin{proof}

As in the proof of the previous proposition we first focus on the case $r=1$.

The eigenvalues for a lune with arbitrary angle are known (see \cite{Gromes}). For our lune $\Omega$ with angle $\frac{\pi}{k}$, the set of eigenvalues is given by
\begin{align*}
\{\, \lambda_{n,m}\cdot \left( \lambda_{n,m} + 1 \right) \mid m\in\mathbb{N},\, n\in\mathbb{N}_0\,\},
\end{align*}
where $\lambda_{n,m}:=k\cdot m + n$. Furthermore, each eigenvalue is simple, i.e. has multiplicity one. 

Thus we get for the heat trace of the lune:
\begin{align*}
Z_{\Omega}(t)&=\sum\limits_{m=1}^{\infty}\sum\limits_{n=0}^{\infty} e^{-(k\cdot m + n)(k\cdot m+n+1)\, t} = \sum\limits_{m=1}^{\infty}\sum\limits_{n=0}^{\infty} e^{-\left( (k\cdot m + n)^2 + (k\cdot m+n) + \frac{1}{4} - \frac{1}{4}\right)\, t}\\
&= e^{\frac{t}{4}}\cdot \sum\limits_{m=1}^{\infty}\sum\limits_{n=0}^{\infty}e^{-\left( k\cdot m + n + \frac{1}{2}\right)^2\, t}\\
&= e^{\frac{t}{4}}\cdot \underbrace{\left( \sum\limits_{m=0}^{\infty}\sum\limits_{n=0}^{\infty}e^{-\left( k\cdot m + n + \frac{1}{2}\right)^2\, t}  - \sum\limits_{n=0}^{\infty} e^{-\left( n+\frac{1}{2} \right)^2 \, t}\right)}_{=\, :Z_{\Omega}^{-\nicefrac{1}{4}}(t)}
\end{align*}

The asymptotic expansion of $Z_{\Omega}^{-\nicefrac{1}{4}}(t)$ can be found in \cite[Lemma 15]{Watson}, and is given by
\begin{align}
\label{equation:ShiftedHeatTraceLuneSphereWatson}
Z_{\Omega}^{-\nicefrac{1}{4}}(t)\overset{t\downarrow 0}{\sim}  \frac{\pi}{k} \cdot \frac{1}{2\pi  t} - \frac{\sqrt{\pi}}{4\sqrt{t}} + \sum\limits_{\ell=0}^{\infty} \left( 2\cdot c_{\ell}^{\mathbb{S}}\left( \frac{\pi}{k}\right) + \frac{1}{2k}i_{\ell}^{(s)}\right)\cdot t^{\ell}
\end{align}
where $i_{\ell}^{(s)} = \left( -1 \right)^{\ell+1}\cdot \frac{B_{2\ell+2}\left( \frac{1}{2} \right)}{\left(\ell+1\right)!}$ as in \eqref{equation:HeatCoefficientsSphere2}. Multiplying by the series $e^{\frac{t}{4}}$, we get
\begin{align*}
Z_{\Omega}(t)\overset{t\downarrow 0}{\sim}  \frac{1}{2k}\cdot \frac{1}{t} - \frac{\sqrt{\pi}}{4}\cdot \frac{1}{\sqrt{t}} + \sum\limits_{\nu=0}^{\infty} &\left\lbrace - \frac{1}{4^{\nu+1}(\nu + 1)!} \cdot \frac{\sqrt{\pi}}{4}\sqrt{t} \right.  \\
&   \left. + 2\sum\limits_{\ell=0}^{\nu}\frac{1}{4^{\ell}\cdot \ell !}\cdot c_{\nu-\ell}^{\mathbb{S}}\left( \frac{\pi}{k}\right) + \frac{1}{2k}\sum\limits_{\ell=0}^{\nu+1}\frac{1}{4^{\ell}\cdot \ell !} i_{\nu-\ell}^{(s)}    \right\rbrace t^{\nu}.
\end{align*}
By definition of $i_{\nu}^{\mathbb{S}}$ given in \eqref{equation:HeatCoefficientsSphere1}, this simplifies to
\begin{align}
Z_{\Omega}(t)\overset{t\downarrow 0}{\sim}  \frac{1}{2k}\cdot \frac{1}{t} - \frac{\sqrt{\pi}}{4}\cdot \frac{1}{\sqrt{t}} + \sum\limits_{\nu=0}^{\infty} &\left\lbrace \frac{1}{2k}\cdot i_{\nu}^{\mathbb{S}} - \frac{1}{4^{\nu+1}(\nu + 1)!}  \cdot \frac{\sqrt{\pi}}{4}\sqrt{t}  \right. \nonumber\\
&   \left. + 2 \sum\limits_{\ell=0}^{\nu}\frac{1}{4^{\ell}\cdot \ell !}\cdot c_{\nu-\ell}^{\mathbb{S}}\left( \frac{\pi}{k}\right)    \right\rbrace t^{\nu}.
\end{align}
Hence the proposition is proven for $r=1$.

To get the asymptotic expansion for arbitrary $r>0$, proceed as usual, replacing $t$ by $\frac{t}{r^2}=\kappa t$.

\end{proof}

\begin{remark}
The asymptotic expansion \eqref{equation:ShiftedHeatTraceLuneSphereWatson} is stated and proven in \citep[Lemma 15]{Watson}. The statement of Lemma 15 is correct and the proof is basically correct as well. The only issue occurs implicitly in the last step of the proof, where Watson uses Part $(iii)$ of his Lemma 11. The problem with Lemma 11 is that Watson does not explain explicitly his definition for the coefficients $B_{2n}^{(2)}(x \mid a)$ and does not really prove part $(i),\, (iii)$ of that Lemma. From his proof of part $(ii)$ of Lemma 11 and the usage of the coefficients $B_{2n}^{(2)}(x\mid a)$ in the sequel of his article it seems as if he defines $B_{2n}^{(2)}(x\mid a) := B_{2n}^{(2)}(x \mid a \cdot 1)$ in the sense of formula $(20)$ of his article. With this definition, the formulas of part $(i)$ and $(iii)$ of Lemma 11 are wrong. However, Watson uses part $(iii)$ of Lemma 11 only for the special case $x=\frac{1}{2}$, where because of $B_{2n-1}\left(\frac{1}{2}\right)=0$ it simplifies to
\begin{align*}
B_{2n}^{(2)}\left( \frac{1}{2} \, \middle| \, a\cdot 1 \right) = \sum\limits_{j=0}^{n} \binom{2n}{2j}\left( a^{2j} - 1 \right)B_{2j} B_{2n-2j}\left( \frac{1}{2} \right) - (2n-1)B_{2n}\left( \frac{1}{2} \right),
\end{align*} 
for all $n\in\mathbb{N}$, $a>0$. One can show that this formula is indeed correct.
\end{remark}

From Proposition \ref{proposition:AsymptoticExpansionLuneWatson} we can deduce formulas for the heat coefficients for any smooth manifold of constant curvature and smooth totally geodesic boundary.

\begin{corollary}
Let $\mathcal{S}$ be any compact two-dimensional Riemannian manifold of constant curvature with totally geodesic boundary $\partial \mathcal{S}$. Let $K_{\mathcal{S}}$ be the \emph{(}Dirichlet\emph{)} heat kernel and $K_{\mathcal{S}}^{N}$ be the Neumann heat kernel of $\mathcal{S}$. Then the heat invariants are given by the following formulas
\begin{align}
Z_{\mathcal{S}}(t)&= \int\limits_{M} K_{\mathcal{S}}(x,x;t) dx \overset{t\downarrow 0}{\sim} \frac{1}{4\pi t} \sum\limits_{\nu=0}^{\infty} \left( a_{\nu} (\mathcal{S}) - \beta_{\nu}(\mathcal{S})\cdot \sqrt{4\pi t} \right) t^{\nu}, \label{equation:HeatCoefficientsTotallyGeodesicBoundaryDirichlet} \\
Z_{\mathcal{S}}^{N}(t):&= \int\limits_{M} K_{\mathcal{S}}^{N}(x,x;t) dx \overset{t\downarrow 0}{\sim} \frac{1}{4\pi t} \sum\limits_{\nu =0}^{\infty} \left( a_{\nu} (\mathcal{S}) + \beta_{\nu}(\mathcal{S})\cdot \sqrt{4\pi t} \right) t^{{\nu}}, \label{equation:HeatCoefficientsTotallyGeodesicBoundaryNeumann}
\end{align} 
where $a_{\nu}(\mathcal{S})$ is defined as in \eqref{equation:HeatCoefficientsLocallySymmetricManifolds} \emph{(}with $M$ replaced by $\mathcal{S}$\emph{)} and 
\begin{align}
\beta_{\nu}(\mathcal{S}) := \int\limits_{\partial S} \frac{1}{4^{\nu+1} \nu !}\kappa^{\nu} dx.
\end{align}
\end{corollary}

\begin{proof}
Suppose $\mathcal{S}\subset \mathbb{S}^2(r)$ is a hemisphere (a spherical lune with angle $\pi$). This lune is an example of a compact two-dimensional Riemannian manifold of constant curvature with a totally geodesic boundary. When we set $k=1$ in Proposition \ref{proposition:AsymptoticExpansionLuneWatson} we immediately obtain \eqref{equation:HeatCoefficientsTotallyGeodesicBoundaryDirichlet} for the hemisphere (observe that for $k=1$ the coefficients in \eqref{equation:DefinitionCoefficientsALune} vanish: $c_{\ell}^{\mathbb{S}}(\pi)=0$).

 The asymptotic expansion \eqref{equation:HeatCoefficientsTotallyGeodesicBoundaryNeumann} for the hemisphere follows easily from \eqref{equation:HeatCoefficientsTotallyGeodesicBoundaryDirichlet}, since the Dirichlet and Neumann heat kernels can be written as:
 \begin{align*}
 K_{\mathcal{S}}(x,y;t) &=  K_{\mathbb{S}^2(r)}(x,y;t) - K_{\mathbb{S}^2(r)}(x,\sigma (y) ;t), \\
  K_{\mathcal{S}}^{N}(x,y;t) &=  K_{\mathbb{S}^2(r)}(x,y;t) + K_{\mathbb{S}^2(r)}(x,\sigma (y) ;t)
 \end{align*}
for all $x, y\in \mathcal{S}$ and $t>0$, where $\sigma$ denotes the reflection in the boundary $\partial \mathcal{S}$; so $Z_{\mathcal{S}}(t) + Z_{\mathcal{S}}^{N}(t) = 2 Z_{\mathbb{S}^2(r)}(t) \overset{t\downarrow 0}{\sim} \frac{1}{4\pi t}\sum_{\nu=0}^{\infty} 2 a_{\nu}(\mathcal{S}) t^{\nu}$.

By a similar argument as in the proof of Corollary \ref{corollary:HeatCoefficientsLocallySymmetricManifolds}, the formulas for $a_{\nu} (\mathcal{S}), \beta_{\nu}(\mathcal{S})$ now follow for all two-dimensional compact manifolds of constant curvature with totally geodesic boundary because of \citep{GilkeyBranson}.
\end{proof}

Of course, the asymptotic expansion \eqref{equation:HeatCoefficientsTotallyGeodesicBoundaryDirichlet} coincides with the asymptotic expansion given in \eqref{equation:HeatAsymptoticConstantCurvaturePolygon} for polygons with zero vertices.

Let us resume our discussion of the orbifolds. We now have two expressions for the asymptotic expansion of $Z_{\Omega}(t)$ given by \eqref{equation:AsymptoticExpansionLuneWatson} and \eqref{equation:AsymptoticExpansionDifferenceHeatTracesOrbifolds}. The coefficients corresponding to the same power of $t$ must be equal in both asymptotic expansions. Recall that from \eqref{equation:AsymptoticExpansionDifferenceHeatTracesOrbifolds} we have
\begin{align*}
Z_{\Omega}(t)\overset{t\downarrow 0}{\sim} A - B + C,
\end{align*}
and  \eqref{equation:AsymptoticExpansionLuneWatson} can be written in the form
\begin{align*}
Z_{\Omega}(t) \overset{t\downarrow 0}{\sim} A - \frac{1}{\sqrt{4 \pi t}}\cdot  \sum\limits_{\nu=0}^{\infty} \frac{2\pi }{\sqrt{\kappa}} \cdot \frac{\kappa^{\nu}}{4^{\nu+1}\nu !} t^{\nu} + \sum\limits_{\nu=0}^{\infty}  \sum\limits_{\ell=0}^{\nu}\frac{2}{4^{\ell}\cdot \ell !}\cdot c_{\nu-\ell}^{\mathbb{S}}\left(\frac{\pi}{k}\right)   \kappa^{\nu} t^{\nu}.
\end{align*}
Since $B$ contains only half-integer powers and $C$ only integer powers, we conclude the following formulas for $B$ and $C$:

\begin{corollary}
\label{corollary:ABCandCoefficients}
Let $B, C$ be defined as in Corollary \emph{\ref{corollary:AsymptoticExpansionDifferenceHeatTracesOrbifolds}}. Then:

\begin{itemize}
\item[$(i)$] 
\begin{align}
B = \frac{1}{\sqrt{4\pi t}} \cdot \sum\limits_{\nu=0}^{\infty} \frac{2\pi r}{4^{\nu+1} \nu!}\kappa^{\nu}\cdot t^{\nu} = \frac{\vert \ml \vert}{\sqrt{4\pi t}} \cdot \sum\limits_{\nu=0}^{\infty} \frac{1}{4^{\nu+1} \nu!}\kappa^{\nu}\cdot t^{\nu},
\end{align}
where $\vert\ml\vert$ denotes the length of the mirror locus; in particular,
\begin{align}
b_{\nu}\left( \tilde{S}_0 \right) =  \frac{2}{4^{\nu+1}\cdot \nu !}\cdot \kappa^{\nu}\, \text{ for all }\, \nu\in\mathbb{N}_{0}.
\end{align}
\item[$(ii)$]
\begin{align}
\label{equation:DihedralPointContribution}
C = \sum\limits_{\nu=0}^{\infty} \sum\limits_{\ell=0}^{\nu}\frac{2}{4^{\ell}\cdot \ell !}\cdot c_{\nu-\ell}^{\mathbb{S}}\left(\frac{\pi}{k}\right)  \kappa^{\nu} t^{\nu}
\end{align}
with $c_{\ell}^{\mathbb{S}}\left(\frac{\pi}{k}\right)$ from \eqref{equation:DefinitionCoefficientsALune}; in particular,
\begin{align}
\frac{1}{k} \sum\limits_{\ell=1}^{k-1} b_{\nu}\left( \tilde{D}_{\ell} \right)  =  \sum\limits_{\ell=0}^{\nu} \frac{2}{4^{\ell}\cdot \ell !}\cdot c_{\nu-\ell}^{\mathbb{S}}\left(\frac{\pi}{k}\right)  \kappa^{\nu}\, \text{ for all }\, \nu\in\mathbb{N}_{0}.
\end{align}
\end{itemize}
\end{corollary}

%

Recall that the orbifolds $M/\mathbb{Z}_{k}$ and $M/\mathbb{D}_{k}$ have the asymptotic expansions
\begin{align*}
Z_{M/\mathbb{Z}_{k}}(t) \overset{t\downarrow 0}{\sim} 2A + 2C;\quad Z_{M/\mathbb{D}_{k}}(t) \overset{t\downarrow 0}{\sim} A + B + C.
\end{align*}
Together with the formulas for $A,B,C$ from the above corollary and Corollary \ref{corollary:HeatCoefficientsLocallySymmetricManifolds}, this completes our discussion of the heat invariants for the spherical orbifolds $M/\mathbb{Z}_{k}$ and $M/\mathbb{D}_{k}$.

The computations made so far can be generalised to more general orbisurfaces. The reason is that there exist only three kinds of singular points for two-dimensional orbifolds, which we have already met in the discussion so far. These are: \emph{mirror points}, \emph{cone points} and \emph{dihedral points} (see \cite{Gordon12}). Recall from \citep[Theorem 4.8]{Gordon08} (see Theorem \ref{theorem:HeatTraceAsymptoticsOrbifoldQualitatively} above) that for each two-dimensional closed Riemannian orbifold $\mathcal{O}$,  the heat trace $Z_{\mathcal{O}}(t)$ has an asymptotic expansion, as $t\searrow 0$, of the form
\begin{align*}
\frac{1}{4\pi t} \sum\limits_{\nu=0}^{\infty} a_{\nu}(\mathcal{O}) t^{\nu}+\sum\limits_{N\in S(\mathcal{O})}\frac{I_N (\mathcal{O})}{\vert \Iso(N)\vert}.
\end{align*}

\begin{theorem}
\label{theorem:OrbifoldAsymptotikKonstanteKruemmung}
Let $\mathcal{O}$ be a two-dimensional closed Riemannian orbifold of constant curvature $\kappa\in\mathbb{R}$. Then, with $C$ as in \eqref{equation:DihedralPointContribution} we have:

\begin{itemize}
\item[$(i)$]
\begin{align}
\label{equation:OrbifoldHeatKoefficientsInterior}
a_{\nu}(\mathcal{O}) =\frac{\volume(\mathcal{O})}{ \nu! \cdot 4^{\nu}} \sum\limits_{\ell=0}^{\nu} \binom{\nu}{\ell} (-4)^{\ell}B_{2\ell}\left( \frac{1}{2} \right) \cdot \kappa^{\nu}\, \text{ for all }\, \nu\in\mathbb{N}_{0}.
\end{align}
\item[$(ii)$] If $N$ consists of a cone point of order $k\in\mathbb{N}$, then $\frac{I_N (\mathcal{O})}{\vert \Iso(N)\vert} = C$.
\item[$(iii)$] If $N$ consists of a dihedral point of order $2k\in\mathbb{N}$, then $\frac{I_N (\mathcal{O})}{\vert \Iso(N)\vert} = \frac{1}{2} C$.
\item[$(iv)$] If $N$ is a mirror edge of length $\vert N \vert$, then 
\begin{align}
\frac{I_N (\mathcal{O})}{\vert \Iso(N)\vert} =  \frac{\vert N \vert}{\sqrt{4\pi t}}\sum\limits_{\nu=0}
^{\infty} \frac{1}{4^{\nu+1}\cdot \nu !}\cdot \kappa^{\nu}\cdot t^{\nu}.
\end{align}
\end{itemize}
\end{theorem}

\begin{proof}
Part $(i)$ is clear from Corollary \ref{corollary:HeatCoefficientsLocallySymmetricManifolds} since $a_{\nu}(\mathcal{O})$ is the integral over $\mathcal{O}$ of the same curvature invariant as for manifolds. The remaining statements follow from Corollary \ref{corollary:ABCandCoefficients} (and the definition of $C$ as the contribution of one of the cone points of $M/\mathbb{Z}_k$) via the following fact; see \citep[Theorem $5.1$]{Donnelly2}, cited also in \citep{Gordon08}: The functions $b_i(\gamma,x)$, whose integrals over $N$ make up the coefficients in $I_N(\mathcal{O})$, are of the form $\varphi(\gamma,x)\cdot \psi(\gamma,x)$, where $\varphi(\gamma,x)$ only depends on the Euclidean isometry $d\gamma_x$ and $\psi(\gamma,x)$ is a universal polynomial in the curvature tensor of $\mathcal{O}$ and its covariant derivatives.
\end{proof}

\begin{remark}
Actually the results in $(ii)-(iv)$ in the above theorem do not require the whole orbifold to have constant curvature. It suffices that the curvature is constant in a sufficiently small neighborhood of $N$. For example, if we consider a teardrop with cone point of order $k$ such that the curvature is constant in a neighborhood of the cone point, then the contribution of that cone point is given by $(ii)$ above.
\end{remark}

\section{Applications}
\label{section:ApplicationsOrbifolds}

In this section we first consider some applications of the heat invariants computed in Section \ref{section:InvariantsOrbisurfaces}. 
Finally, we will use them to give an alternative proof for the formulas from Theorem \ref{theorem:AsymptoticExpansionHeatTraceHyperbolPolygon}/Corollary \ref{corollary:HeatAsymptoticConstantCurvaturePolygon} for the angle contributions in the special case of polygons with angles of the form $\frac{\pi}{k}$ with $k \in\mathbb{N}$.

Let $\mathcal{O}$ be a closed Riemannian orbisurface of constant curvature $\kappa\in\mathbb{R}$. There exists a strong connection between the heat invariants for $\mathcal{O}$ computed in the previous section and the heat invariants for a polygon furnished with the same curvature $\kappa$, which we have computed in Section \ref{section:ConsequencesPolygons}. We want to compare those coefficients. Recall the notation introduced in Corollary \ref{corollary:HeatAsymptoticConstantCurvaturePolygon} and Theorem \ref{theorem:OrbifoldAsymptotikKonstanteKruemmung}. First of all, the formal series $V_{\kappa}(\gamma)$ for $\gamma=\frac{\pi}{k}$, $k\in\mathbb{N}$, is the same as the contribution of a dihedral point of order $2k$ to the orbifold heat coefficients of $\mathcal{O}$. In other words, we have $V_{\kappa}(\gamma)=\frac{1}{2} C$, where $C$ is as in \eqref{equation:DihedralPointContribution}. Consequently, a cone point of order $k$ contributes to the heat invariants of $\mathcal{O}$ as much as two angles of magnitude $\frac{\pi}{k}$ to the heat invariants for a polygon of curvature $\kappa$. Lastly, the mirror locus of $\mathcal{O}$ contributes the negative of the contribution of the boundary of a polygon, provided their lengths are equal.

Therefore we obtain, just as in Corollary \ref{corollary:SpectralInvariantsVolumePerimeterCurvature} and Theorem \ref{theorem:SpectralInvariantsAnglesEulerCharacteristic}, the following spectral invariants:

\begin{corollary}
\label{corollary:SpectralInvariantsOrbisurfaces}
Let $\mathcal{O}$ be a closed orbisurface of constant curvature $\kappa\in\mathbb{R}$. Let $M\in\mathbb{N}_0$ denote the number of dihedral points of $\mathcal{O}$ and let $2m_1,...,2m_M$ denote their individual orders, where $m_1,...,m_M\in\mathbb{N}$. Similarly, let $N\in\mathbb{N}_{0}$ denote the number of cone points and $n_1,...,n_N\in\mathbb{N}$ their orders. Then:

\begin{itemize}
\item[$(i)$]  The volume of $\mathcal{O}$ and the length of the mirror locus are spectral invariants. 
\item[$(ii)$] If the mirror locus is non-trivial, i.e. the length of the mirror locus is positive, then the curvature of the orbifold is determined by the spectrum.
\item[$(iii)$] If $\kappa\neq 0$, then $\kappa$ together with the spectrum determines the number $M+2N$ as well as the multiset $ \lbrace m_1 ,..., m_M, n_1, n_1,  n_2, n_2, \allowbreak ...,n_N, n_N  \rbrace$.
\item[$(iv)$] If the mirror locus is trivial and $\kappa\neq 0$, then $\kappa$ together with the spectrum determines the number of cone points as well as the multiset of all orders $\lbrace n_1,...,n_N \rbrace$. \emph{(}Note that trivial mirror locus implies absence of dihedral points.\emph{)}
\end{itemize}
\end{corollary}

To the best of our knowledge, statement $(iii)$ is a new result. Similarly, statement $(iv)$ is also new in this generality and was only known for the special case of $\kappa=-1$ (see the comments after Corollary \ref{corollary:OrientableOrbisurfacesSameUnderlyingSpace} below).

As usual, we call two orbifolds \emph{isospectral} if their Laplacians have the same spectrum, including multiplicities. Corollary \ref{corollary:SpectralInvariantsOrbisurfaces} has some interesting consequences for isospectral orbisurfaces, which we now want to demonstrate.

\begin{definition} (see \citep[p. 311]{thurston})
\label{definition:EulerCharacteristicOrbifolds}
Suppose $\mathcal{O}$ is a closed orbisurface. Let $X_{\mathcal{O}}$ be the underlying topological space of $\mathcal{O}$, and let $\chi(X_{\mathcal{O}})$ denote its Euler characteristic. Then, using the notation of Corollary \ref{corollary:SpectralInvariantsOrbisurfaces}, the Euler characteristic of $\mathcal{O}$ is defined as
\begin{align}
\label{equation:EulerCharacteristicOrbifolds}
\chi(\mathcal{O}):=\chi(X_{\mathcal{O}}) - \frac{1}{2}\sum\limits_{i=1}^M \left(1-\frac{1}{m_i}\right) - \sum\limits_{i=1}^{N}\left( 1-\frac{1}{n_i} \right).
\end{align}
\end{definition}

Thus the Euler characteristic of an orbifold depends on the Euler characteristic of the underlying space and the orders of all its  dihedral and cone points. Note that we use the term \emph{order} always in the sense of isotropy order of a point. Instead, W. Thurston defines the order of a dihedral point as half of its isotropy order (see \citep[Proposition 13.3.1]{thurston}).

It is well-known that the Euler characteristic of a closed orbisurface $\mathcal{O}$ is related to its curvature $\kappa$ by the Gau{\ss}-Bonnet theorem exactly as for closed manifolds, i.e.
\begin{align}
\label{equation:GaussBonnetOrbifolds}
\int\limits_{\mathcal{O}} \kappa dA = 2\pi \chi(\mathcal{O}).
\end{align}
If the curvature is constant then the Gau{\ss}-Bonnet formula reduces to the equation $\volume(\mathcal{O}) \kappa  = 2\pi \chi(\mathcal{O})$. As we see, it follows from Corollary \ref{corollary:SpectralInvariantsOrbisurfaces} $(i)$ that two isospectral orbisurfaces have the same curvature if and only if they have the same Euler characteristic.


\begin{corollary}
\label{corollary:OrientableOrbisurfacesSameUnderlyingSpace}
Let $\mathcal{O}$ be a closed orientable orbisurface with constant curvature $\kappa\neq 0$. Then $\kappa$ together with the spectrum of $\mathcal{O}$ determines the Euler characteristic of $\mathcal{O}$ as well as the Euler characteristic of the underlying space $X_{\mathcal{O}}$.
\end{corollary}

\begin{proof}
By assumption, we know the value of $\kappa$. Thus by Corollary \ref{corollary:SpectralInvariantsOrbisurfaces} $(i)$ and the Gau{\ss}-Bonnet formula, the spectrum determines the Euler characteristic $\chi(\mathcal{O})$. 

Furthermore, since the orbisurface $\mathcal{O}$ is orientable, its mirror locus is trivial and all singular points are cone points. Thus, by Corollary \ref{corollary:SpectralInvariantsOrbisurfaces} $(iv)$, we can deduce from the spectrum the total number of all cone points as well as the multiset of all their orders. Therefore the spectrum determines the Euler characteristic of the underlying space through \eqref{equation:EulerCharacteristicOrbifolds}.
\end{proof}

It is well-known that the underlying space of an orbisurface is always a topological surface, possibly with boundary (see \citep{thurston}). Further, this underlying topological surface has to be orientable and without boundary if the orbifold is known to be orientable. Hence, Corollary \ref{corollary:OrientableOrbisurfacesSameUnderlyingSpace} implies that if two closed orientable orbisurfaces of constant curvature $\kappa\neq 0$ are isospectral, then they have the same underlying topological space, i.e. the underlying topological spaces are homeomorphic. This fact was previously known in the special case of closed orientable orbisurfaces of constant curvature $\kappa = -1$ (see \citep[Proposition 3.3]{Strohmaier}). In their article \citep{Strohmaier}, E. Dryden and A. Strohmaier generalise Huber's theorem to orientable orbisurfaces of constant curvature $\kappa=-1$, using the Selberg trace formula for the wave kernel, from which, in particular, the statement of Corollary \ref{corollary:SpectralInvariantsOrbisurfaces} $(iv)$ follows. Then they use this result with Weyl's asymptotic formula in order to obtain Corollary \ref{corollary:OrientableOrbisurfacesSameUnderlyingSpace} for $\kappa=-1$. Thus, Corollary \ref{corollary:OrientableOrbisurfacesSameUnderlyingSpace} is a new proof of \citep[Proposition 3.3]{Strohmaier} using only heat invariants, in addition to being a generalisation of it. 

%


%

Suppose now that the mirror locus of the orbisurface $\mathcal{O}$ is not trivial. Then by Corollary \ref{corollary:SpectralInvariantsOrbisurfaces} $(i), (ii)$ and the Gau{ss}-Bonnet theorem, the Euler characteristic $\chi(\mathcal{O})$ is determined by the spectrum. Furthermore, if two isospectral closed orbisurfaces with non-empty mirror locus are given, then Corollary \ref{corollary:SpectralInvariantsOrbisurfaces} $(i), (iii)$ provide some restrictions on the singular sets. Unfortunately it is not possible in general to detect which elements of the multiset in Corollary \ref{corollary:SpectralInvariantsOrbisurfaces} $(iii)$ correspond to dihedral points and which to cone points. The simplest classes of orbifolds for which the multiset of orders can be used to obtain all information about the orbifold are the following.

\begin{corollary}
\label{corollary:IsospectralOrbifoldsNoConeNoDihedral}
Let $\mathcal{D}_{\kappa}$, $\mathcal{C}_{\kappa}$ denote the set of all orbisurfaces with constant curvature $\kappa\neq 0$ and without any cone point, respectively dihedral point. Within each of these classes the spectrum determines $\chi(\mathcal{O})$, $\chi(X_{\mathcal{O}})$, and the singular set of the orbifold; that is, the length of the mirror locus and the total number of dihedral points together with their orders, respectively the total number of cone points with all orders.
\end{corollary}

\begin{proof}
This follows easily from Corollary \ref{corollary:SpectralInvariantsOrbisurfaces}. Note that within each class the numbers in the  multiset given in Corollary \ref{corollary:SpectralInvariantsOrbisurfaces} $(iii)$ correspond by assumption solely to dihedral points, respectively solely to cone points. Thus all information about the singular set and the topology of the orbifold can be deduced from the spectrum as claimed.
\end{proof}

What happens if we consider orbisurfaces with dihedral as well as cone points? Sometimes a partial answer about the singular set is possible. For example, a cone point always contributes its order twice to this multiset. Thus, if a number in the multiset appears only once, then we know it must come from a dihedral point. More generally, we can prove the following result.

\begin{corollary}
\label{corollary:SpectrumOfNonOrientableOrbisurfaces}
Let $\mathcal{A}_{1}$ denote the class of all orbisurfaces with constant nonzero curvature which have at least one dihedral point and at least one cone point and satisfy the following condition: Any two dihedral points, respectively any two cone points, have different orders. Then within $\mathcal{A}_1$, the spectrum again determines $\chi(\mathcal{O}),$ $\chi(X_{\mathcal{O}})$, the length of the mirror locus, the number of dihedral points with their order, and the number of cone points with their orders.
\end{corollary}

\begin{proof}
Since $\mathcal{O}\in \mathcal{A}_{1}$ has a dihedral point, it must have non-trivial mirror locus. Thus by Corollary \ref{corollary:SpectralInvariantsOrbisurfaces} $(ii)$ the value of the curvature is determined and therefore, by the Gau{ss}-Bonnet theorem, $\chi(\mathcal{O})$ is also determined. 

Since each dihedral point has a unique order among all dihedral points and similarly, each cone point has a unique order among all cone points, we can detect which elements in the multiset in Corollary \ref{corollary:SpectralInvariantsOrbisurfaces} $(iii)$ come from dihedral points and which from cone points. Consequently we can detect the number of dihedral points with orders as well as the number of cone points with orders. Thus, $\chi(X_{\mathcal{O}})$ can be computed by \eqref{equation:EulerCharacteristicOrbifolds}.
\end{proof}

In \citep[Proposition 5.22]{Gordon08} the authors show that within the class of orbisurfaces of constant curvature $\kappa>0$ the spectrum determines the orbifold. They prove this result by comparing the first few heat invariants of each such orbisurface (see \citep[Table 1]{Gordon08}). Most of those orbisurfaces can easily be distinguished through Corollary \ref{corollary:SpectralInvariantsOrbisurfaces} $(iii)$ and taking into account whether the mirror loci are trivial or not. In fact, all but three pairs of orbisurfaces can be distinguished this way, namely (in the notation of \citep{Gordon08}): $\mathcal{O}(\ast m,m)$ and $\mathcal{O}(m,\ast)$; $\mathcal{O}(\ast 2,2,m)$ and $\mathcal{O}(2, \ast m)$; $\mathcal{O}(\ast 2,3,3)$ and $\mathcal{O}(3, \ast 2)$. In those remaining pairs, in which the orbifolds have the same multiset as well as non-trivial mirror loci, the orbifolds can be distinguished by comparing the length of the mirror loci as it is said in the proof of \citep[Proposition 5.22]{Gordon08}. 


One can restrict the class of orbifolds in different ways in order to obtain further results. For example, if one fixes the curvature $\kappa\neq 0$ and the value of $M+N$, where $M$ and $N$ are defined as in Corollary \ref{corollary:SpectralInvariantsOrbisurfaces}, then the spectrum determines the value of $M$ as well as $N$ (note that $M, N$ can be deduced from the values $M+N$ and $M+2N$, where the latter is determined by the spectrum). \medskip

Recall that an angle $\gamma=\frac{\pi}{k}$, $k\in\mathbb{N}$, contributes to the heat invariants of a polygon as much as a dihedral point of order $2k$ contributes to the orbifold heat coefficients, provided the curvature of the polygon and the orbifold are constant and equal. In other words, $V_{\kappa}\left(\frac{\pi}{k}\right)=\frac{1}{2} C$. In the remaining part of this section we will give an alternative proof for the equation $V_{\kappa}(\frac{\pi}{k}) = \frac{1}{2}C$ for $\kappa<0$ by using the formulas in \citep{Watson} and orbifold theory (more precisely Theorem \ref{theorem:OrbifoldAsymptotikKonstanteKruemmung} and the Remark after that theorem).

\begin{lemma}
\label{lemma:MetrikVariationSphaere}
Let $\mathbb{S}^2\subset\mathbb{R}^3$ be the standard sphere, let $x_N:=(0,0,1)$, and let $\kappa\in\mathbb{R}$. Then there exists a Riemannian metric $g$ on $\mathbb{S}^2$ which is $O(2)$-invariant and has constant curvature $\kappa$ in a neighborhood of $x_N$. \emph{(}Of course, this lemma is trivial for $\kappa>0$.\emph{)}
\end{lemma}

\begin{proof}
Let $(H, h)$ be a two-dimensional complete Riemannian manifold of constant curvature $\kappa$, let $P\in H$, and let $g_{\mathbb{S}^2}$ be the standard metric on $\mathbb{S}^2$. We choose $R\in (0,\pi)$ smaller than the injectivity radius of $H$ at $P$. Choose a (linear) Euclidean isometry from $T_{x_N}\mathbb{S}^2$ to $T_P H$, and let $\varphi: B_{R}(x_N)\rightarrow B_{R}(P)$ be the corresponding diffeomorphism given by the geodesic exponential maps. Let $f:\mathbb{S}^2\rightarrow [0,1]$ be a smooth $O(2)$-invariant function such that $f_{\vert B_{\frac{R}{2}}(x_N)}\equiv 1$ and $f$ vanishes outside $B_{R}(x_N)$. Then $g:=f\cdot \varphi^{\ast}h + (1-f)g_{\mathbb{S}^2}$ has the desired properties.
\end{proof}

Note that the results of Section \ref{section:ExamplesOrbifolds} hold for $(\mathbb{S}^2, g)$, if $g$ denotes a Riemannian metric as in Lemma \ref{lemma:MetrikVariationSphaere}.

\begin{theorem}
\label{theorem:SphaerischenKoeffizientenOrbifolds}
Let $(N,h)$ be a two-dimensional complete Riemannian manifold whose Gaussian curvature is bounded. Let $\Omega \subset N$ be a polygon which has an angle $\gamma := \frac{\pi}{k}$ at some vertex $P$ of $\Omega$, where $k\in\mathbb{N}$. Let $W_{\gamma}(P)$ denote the wedge corresponding to the angle $\gamma$ and suppose that the curvature is constant and equal to $\kappa\in\mathbb{R}$ in some neighborhood of $P$. Let $g$ be a Riemannian metric on $\mathbb{S}^2$ as in Lemma \emph{\ref{lemma:MetrikVariationSphaere}}. Suppose $\tilde{R}>0$ is such that $(B_{\tilde{R}}(P),h)$ and $(B_{\tilde{R}}(x_N),g)$ have constant curvature and are isometric, where $x_N=(0,0,1)$. Further, let $i(P)$ be the injectivity radius of $N$ at $P$ and let $i(x_N)$ denote the injectivity radius of $(\mathbb{S}^2,g)$ at $x_N$. For any $r\in (0, i(P)]$ let $W_r(P)$ denote the circle sector at $P$ with angle $\gamma$ and radius $r$, such that $W_r(P)\subset W_{\gamma}(P)$. Then for all $R>0$ such that $R< \frac{1}{2}\cdot \min\{ i(P), i(x_N), \tilde{R} \}$ and such that $W_{2R}(P)\cap \Omega$ is a disjoint union of circle sectors at $P$, the functions
\begin{align*}
t\mapsto \int\limits_{W_{R}(P)} K_{\Omega}(x,x;t) dx
\end{align*}
and
\begin{align*}
t\mapsto \int\limits_{B_R(\proj{Z}{x_N})} K_{(M/\mathbb{Z}_k,g)}(x,x;t) dx - \int\limits_{B_R(\proj{D}{x_N})} K_{(M/\mathbb{D}_k,g)}(x,x;t) dx
\end{align*}
have the same asymptotic expansion as $t\searrow 0$.
\end{theorem}

\begin{proof}
First, when Lemma \ref{lemma:PNFBlokal} $(ii)$ is applied to $N$, it follows that the function \begin{align*}
t\mapsto \int_{W_{R}(P)} K_{\Omega}(x,x;t) dx
\end{align*}  
has the same asymptotic expansion as the function 
\begin{align*}
t\mapsto \int_{W_{R}(P)} K_{W_{2R}(P)}(x,x;t) dx.
\end{align*}
 By assumption, $B_{2R}(P)$ and $B_{2R}(x_N)$ are isometric and thus we identify $B_{2R}(P)$ with $B_{2R}(x_N)$ (and accordingly the subsets $W_{R}(P)$ and $W_{2R}(P)$). Let $\Omega'\subset (\mathbb{S}^2,g)$ denote a $2$-gon as in Section \ref{section:ExamplesOrbifolds} with vertex $x_N$ and angle $\frac{\pi}{k}$ such that $W_{R}(P)\subset \Omega'$. Hence, by Lemma \ref{lemma:PNFBlokal} $(ii)$ (applied to $(\mathbb{S}^2,g)$) the function $t\mapsto \int_{W_{R}(P)} K_{W_{2R}(P)}(x,x;t) dx$ has the same asymptotic expansion as 
\begin{align*}
 t\mapsto \int_{W_{R}(x_N)} K_{\Omega'}(x,x;t) dx.
\end{align*} 
Note that $W_{R}(x_N)=B_R(x_N)\cap \Omega'$ and thus by Theorem \ref{theorem:GEqualHeatTraceLune} with $U:=B_R(x_N)$ the proof is complete. Note that $\proj{Z}{B_R(x_N)} = B_R(\proj{Z}{x_N})$ and $\proj{D}{B_R(x_N)} = B_R(\proj{D}{x_N})$.
\end{proof}

\begin{lemma}
\label{lemma:LemmaSchuethKegelpunktundDiederpunkt}
Let $\mathcal{O}$ be a two-dimensional closed Riemannian orbifold, $Q\in\mathcal{O}$ be a cone point or dihedral point. Suppose there exists some $R>0$ such that $B_R(Q)$ has constant curvature $\kappa$, and $\overline{B_R(Q)}$ does not contain any singular points except for $Q$ and, if $Q$ is a dihedral point, mirror points of the two reflector edges which end at $Q$.
\begin{itemize}
\item[$(i)$] If $Q$ is a cone point of order $k$, then
\begin{align*}
\int\limits_{B_R(Q)} K_{\mathcal{O}}(x,x;t)dx \overset{t\downarrow 0}{\sim} \frac{1}{4\pi t} \sum\limits_{\nu=0}^{\infty} a_{\nu}(B_R(Q)) t^{\nu} + C,
\end{align*}
where $a_{\nu}$ is defined as in \eqref{equation:OrbifoldHeatKoefficientsInterior} and $C$ as in \eqref{equation:DihedralPointContribution}.
\item[$(ii)$] If $Q$ is a dihedral point of order $2k$, then
\begin{align*}
\int\limits_{B_R(Q)} K_{\mathcal{O}}(x,x;t)dx \overset{t\downarrow 0}{\sim} \frac{1}{4\pi t} \sum\limits_{\nu=0}^{\infty} a_{\nu}(B_R(Q)) t^{\nu} + \frac{2R}{\sqrt{4\pi t}}\sum\limits_{\nu=0}^{\infty} \frac{1}{4^{\nu+1}\nu!}\kappa^{\nu}t^{\nu} + \frac{1}{2} C.
\end{align*}
\end{itemize}
\end{lemma}

\begin{proof}
The proof follows directly from Theorem \ref{theorem:OrbifoldAsymptotikKonstanteKruemmung} and the Remark given after that theorem, if one takes into account the following: In order to obtain the asymptotic expansion of $Z_{\mathcal{O}}(t)$ in \citep[Theorem $4.8$]{Gordon08}, the authors consider the function $x\mapsto K_{\mathcal{O}}(x,x;t)$ in local orbifold charts (so called ``distinguished'' charts) and compute locally the heat kernel asymptotic expansion as $t\searrow 0$. The resulting coefficients only depend on the local geometry and are put together by a partition of unity and then integrated. Hence, the asymptotic expansion of $\int_U K_{\mathcal{O}}(x,x;t)dx$ captures exactly the contribution of the points (to the integrals which occur in $I_0$ and $I_N$), which are contained in any open cover of $\overline{U}$, i.e. exactly the points of $\overline{U}$. 
\end{proof}

From Theorem \ref{theorem:SphaerischenKoeffizientenOrbifolds} and Lemma \ref{lemma:LemmaSchuethKegelpunktundDiederpunkt} we obtain immediately the following corollary.

\begin{corollary}
\label{corollary:AlternativbeweisKoeffizientenEckeMittelsOrbifolds}
Let $(N,h)$ be a complete Riemannian manifold whose Gaussian curvature is bounded. Let $\Omega\subset N$ be a polygon which has an angle $\gamma:=\frac{\pi}{k}$ at some vertex $P$ of $\Omega$, where $k\in\mathbb{N}$. Let $W_r(P)$ be as in Theorem \emph{\ref{theorem:SphaerischenKoeffizientenOrbifolds}} and suppose the Gaussian curvature is constant in some neighborhood of $P$. Then there exists some $R_0>0$ such that for all $R\in (0,R_0)$:
\begin{align*}
\int\limits_{W_R(P)} K_{\Omega}(x,x;t)dx\overset{t\downarrow 0}{\sim} \frac{1}{4\pi t} \sum\limits_{\nu=0}^{\infty} a_{\nu}(W_R(P)) t^{\nu} - \frac{2R}{\sqrt{4\pi t}}\sum\limits_{\nu=0}^{\infty} \frac{1}{4^{\nu+1}\nu!}\kappa^{\nu}t^{\nu} + \frac{1}{2} C.
\end{align*}
\end{corollary}

\begin{remark}
The formula \eqref{equation:DihedralPointContribution} for $C$ contains the coefficients $c_{\ell}^{\mathbb{S}}\left(\frac{\pi}{k}\right)$, which were computed by S. Watson in \citep{Watson}. Furthermore, that formula for $C$ was derived using only orbifold theory. Thus, we obtain from the above corollary (which is based on Theorem \ref{theorem:SphaerischenKoeffizientenOrbifolds}) and the formula of S. Watson for $c_{\ell}^{\mathbb{S}}\left(\frac{\pi}{k}\right)$ a \emph{new proof} of the formula for $V_{\kappa}\left( \frac{\pi}{k} \right)$ (which is given in Corollary \ref{corollary:HeatAsymptoticConstantCurvaturePolygon}), in particular for $\kappa=-1$ (and for the coefficients $c_{\ell}^{\mathbb{H}}(\frac{\pi}{k})$).
\end{remark}

%% file: declaration.tex
\chapter*{Selbst\"andigkeitserkl\"arung}

Ich erkl\"are, dass ich die vorliegende Arbeit selbst\"andig und nur unter Verwendung der angegebenen Literatur und Hilfsmittel angefertigt habe.

\vspace{2\baselineskip}
\noindent Berlin, 26.05.2017\hfill\authorfirstname \authorsurname